\RequirePackage{fix-cm}
\documentclass[smallextended]{svjour3}       % onecolumn (second format)
\smartqed  % flush right qed marks, e.g. at end of proof
\usepackage{graphicx}
\usepackage{mathrsfs}
\usepackage{amsfonts}
\usepackage{graphicx}
\usepackage{caption}
\usepackage[figuresright]{rotating}
\usepackage{array}
\usepackage{diagbox}
\usepackage{multirow}
\usepackage{amssymb}
\usepackage{amsmath}
\usepackage{color}
\usepackage[misc]{ifsym}

\usepackage{diagbox}
\allowdisplaybreaks
\journalname{}
\begin{document}
	
	\title{Association schemes arising from non-weakly regular bent  functions%\thanks{Grants or other notes
	}
	
	
	\author{Yadi Wei  \and
		Jiaxin Wang \and
		Fang-Wei Fu%etc.
	}
	
	
	\institute{Yadi Wei\at
		Chern Institute of Mathematics and LPMC, Nankai University, Tianjin, 300071, China\\
		\email{wydecho@mail.nankai.edu.cn}           %  \\
		\and 
		Jiaxin Wang\at
		Chern Institute of Mathematics and LPMC, Nankai University, Tianjin, 300071, China\\
		\email{wjiaxin@mail.nankai.edu.cn}
		\and
		Fang-Wei Fu\at
		Chern Institute of Mathematics and LPMC, Nankai University, Tianjin, 300071, China\\
		\email{fwfu@nankai.edu.cn}}
	\date{Received: date / Accepted: date}

	\maketitle
	\begin{abstract}
		Association schemes play an important role in algebraic combinatorics and have important applications in coding theory, graph theory and design theory. The methods to construct association schemes by using bent functions have been extensively studied. Recently, in \cite{Ozbudak2}, {\"O}zbudak and Pelen constructed infinite families of symmetric association schemes of classes $5$ and $6$ by using ternary non-weakly regular bent functions. They also stated that “constructing $2p$-class association schemes from $p$-ary non-weakly regular bent functions is an interesting problem", where $p>3$ is an odd prime. In this paper, using non-weakly regular bent functions, we construct infinite families of symmetric association schemes of classes $2p$, $(2p+1)$ and $\frac{3p+1}{2}$ for any odd prime $p$. Fusing those association schemes, we obtain $t$-class symmetric association schemes, where $t=4,5,6,7$. In addition, we give the sufficient and necessary conditions for the partitions $P$, $D$, $T$, $U$ and $V$ (defined in this paper) to induce symmetric association schemes. 
		
	\end{abstract}

	\section{Introduction}
	
	\quad Bent functions, first introduced by Rothaus  in \cite{Rothaus}, are maximally nonlinear Boolean functions and were extended to any prime $p$ by Kumar, Scholtz and Welch \cite{Kumar}. They have been intensively studied as they have numerous applications in cryptography, coding theory, design theory and sequence theory. All Boolean bent functions are regular bent functions, which are also dual-bent functions, i.e., the functions and their dual functions are both bent. In odd characteristics,  bent functions are classified into regular bent functions, weakly regular but not regular bent functions, and non-weakly regular bent functions. All weakly regular bent functions are dual-bent. However, the dual functions of non-weakly regular bent functions may be bent or not. In \cite{Cesmelioglu1,Cesmelioglu2,Cesmelioglu3,Ozbudak}, infinitely many dual-bent and non-dual-bent functions were constructed.
	
	Association schemes, introduced by Bose \emph{et  al.} in \cite{Bose,Bose1}, are an important part of algebraic combinatorics and have applications in many areas of research such as coding theory, graph theory, design theory and so on. It is a good way to construct association schemes from bent functions. In \cite{Tan}, Tan \emph{et al.} constructed $2$-class association schemes using ternary bent functions. Later, Pott \emph {et al.} in \cite{Pott} obtained $p$-class association schemes from weakly regular $p$-ary bent functions for any odd prime $p$. In \cite{Chee}, Chee \emph {et al.}  constructed $2$-class and $3$-class association schemes from $p$-ary weakly regular bent functions. In \cite{Wu}, Wu \emph{et al.} proved that a $p$-ary bent function induces a $p$-class association scheme if and only if the function is weakly regular. Recently, {\"O}zbudak and Pelen in  \cite{Ozbudak2} constructed infinite families of symmetric association schemes of classes $3$, $4$, $5$ and $6$ by using ternary non-weakly regular bent functions. They also proposed a problem of constructing $2p$-class association schemes from $p$-ary non-weakly regular bent functions, where $p>3$ is an odd prime. In \cite{Anbar}, some amorphic association schemes were constructed from bent partitions. Later, in \cite{Wang}, Wang \emph{et al.} constructed some association schemes from vectorial dual-bent functions.
	
	Let $p$ be an odd prime and $n$ be a positive integer. In this paper, using  non-weakly regular bent function $f(x):\mathbb{F}_p^n\rightarrow \mathbb{F}_p$ belonging to  $\mathcal{DBF}$ (defined in Section 2) and satisfying the Condition 1 (defined in Section 4),	we construct $2p$-class, $(2p+1)$-class and $\frac{3p+1}{2}$-class symmetric association schemes. Furthermore, fusing those association schemes, we obtain $4$-class, $5$-class, $6$-class and $7$-class symmetric association schemes. In particular, the association schemes constructed in \cite{Ozbudak2} can also be obtained by our constructions. In addition, we give the sufficient and necessary conditions for the partitions $P$, $D$, $T$, $U$ and $V$ (defined in Section 4) to induce symmetric association schemes.
	
	The rest of this paper is arranged as follows. In Section 2, we introduce the main notation and give some necessary preliminaries. In Section 3, we present some auxiliary results. In Section 4, we construct some infinite families of symmetric association schemes and consider the fusions of them, by which we obtain more symmetric association schemes. In addition, we give the sufficient and necessary conditions for the partitions  $P$, $D$, $T$, $U$ and $V$ to induce symmetric association schemes.  In Section 5, we make a conclusion.
	
	\section{Preliminaries}\label{Section 2}
	\quad
	In this section, we introduce some basic results on cyclotomic field, non-weakly regular bent functions and association schemes. We begin this section by setting some basic notation which will be used throughout this paper.
	\subsection{Notation}\label{Section2.1}
	\quad Let $p$ be an odd prime, $n$ be a positive integer and $\mathbb{F}_{p^n}$ be the finite field of order $p^n$. Since $\mathbb{F}_{p^n}$ is an $n$-dimensional vector space over $\mathbb{F}_p$, we also use the notation $\mathbb{F}_p^n$ consisting of $n$-tuples over the prime field $\mathbb{F}_p$. Throughout this paper, we set the following notation.
	\begin{itemize}
		\item[$\bullet$] $|D|$: The cardinality of a set $D$;
		\item[$\bullet$] $\mathbb{F}_p^*$: The multiplicative group of the finite field $\mathbb{F}_p$;
		\item[$\bullet$] $SQ$ and $NSQ$: The set of squares and nonsquares in $\mathbb{F}_p^*$, respectively;
		\item[$\bullet$] $\eta$: The quadratic character of $\mathbb{F}_p^*$, that is, $\eta(a)=1$ for $a \in SQ$ and $\eta(a)=-1$ for $a \in NSQ$;
		\item[$\bullet$] $p^*$: $p^*=\eta(-1)p=(-1)^{(p-1)/2}p$;
		\item[$\bullet$] $\xi_p=e^{\frac{2\pi \sqrt{-1}}{p}} $: The $p$-th complex primitive root of unity.
	\end{itemize}
	\subsection{Cyclotomic field\ $\mathbb{Q}(\xi_p)$}\label{Section2.2}
	\quad The cyclotomic field $\mathbb{Q}(\xi_p)$ is obtained from the rational field $\mathbb{Q}$ by adjoining $\xi_p$. The ring of algebraic integers in $\mathbb{Q}(\xi_p)$ is $\mathcal{O}_{ {\mathbb{Q}_{(\xi_p)}} }=\mathbb{Z}[\xi_p]$ and  the set $\{\xi_p^j\ :\ 1\le j\le p-1\}$ is an integral basis of it. The field extension $\mathbb{Q}(\xi_p)/\mathbb{Q}$ is a Galois extension of degree $p-1$, and $Gal(\mathbb{Q}(\xi_p)/\mathbb{Q})=\{\sigma_a:\ a\in \mathbb{F}_p^*\}$ is the Galois group, where the automorphism $\sigma_a$ of $\mathbb{Q}(\xi_p)$ is defined by $\sigma_a(\xi_p)=\xi_p^a.$ The cyclotomic field $\mathbb{Q}(\xi_p)$ has a unique quadratic subfield $\mathbb{Q}(\sqrt{p^{*}})$, and so $Gal(\mathbb{Q}(\sqrt{p^{*}})/\mathbb{Q})=\{1,\ \sigma_\gamma\}$, where $\gamma$ is a nonsquare in $\mathbb{F}_p^{*}.$ Obviously, for $a\in \mathbb{F}_p^{*}$ and $b \in \mathbb{F}_p$, we have that $\sigma_a(\xi_p^b)=\xi_p^{ab}$ and $\sigma_a(\sqrt{p^*})=\eta(a)\sqrt{p^*}$. For more information on cyclotomic field, the reader is referred to \cite{Ireland}.
	\subsection{Non-weakly regular bent functions}\label{Section 2.3}
	\quad In this subsection, we give the definition of non-weakly regular bent functions and introduce some properties of them. Let $f(x):\mathbb{F}_p^n\rightarrow \mathbb{F}_p$ be a $p$-ary function. The \emph{Walsh transform} $\mathcal{W}_f$ of $f(x)$ is a complex-valued function of $\mathbb{F}_p^n$ defined by 
	\[\mathcal{W}_f (\alpha)=\sum\limits_{x \in \mathbb{F}_p^n} \xi_p^{f(x)-\alpha\cdot x},\ \alpha\in\mathbb{F}_p^n,\]
	where $\cdot$ is the ordinary inner product in $\mathbb{F}_p^n$. The \emph{inverse Walsh transform} of $f(x)$ is given by
	\[
	\xi^{f(\alpha)}_p=p^{-n}\sum\limits_{x\in \mathbb{F}_{p}^n}\mathcal{W}_f(x) \xi^{\alpha\cdot x}_p,\ \alpha\in \mathbb{F}_p^n.
	\]
	The function $f(x)$ is called \emph{bent} if $|\mathcal{W}_f(\alpha)|=p^{\frac{n}{2}}$ for any $\alpha\in \mathbb{F}_p^n$. The Walsh transform of a bent function $f(x)$ at $\alpha\in \mathbb{F}_p^n$ is given as follows \cite{Kumar}.
	\begin{align*}
	\mathcal{W}_f(\alpha)=
	\begin{cases}
	\pm p^{\frac{n}{2}} \xi^{f^*(\alpha)}_p, & \mathrm{if}\ p^n\equiv 1 \ (\mathrm{mod} \ 4),\\
	\pm \sqrt{-1}p^{\frac{n}{2}} \xi^{f^*(\alpha)}_p, & \mathrm{if}\ p^n\equiv 3 \ (\mathrm{mod}\ 4),
	\end{cases}
	\end{align*}
	where $f^*(x)$ is a function from $\mathbb{F}_p^n$ to $\mathbb{F}_p$ called the \emph{dual} of $f(x)$.
	A bent function $f(x): \mathbb{F}_p^n \rightarrow \mathbb{F}_p$ is called \emph{regular} if, for all $\alpha\in \mathbb{F}_p^n$, 
	$\mathcal{W}_f(\alpha)=p^{\frac{n}{2}} \xi_p^{f^*(\alpha)},$
	and is called \emph{weakly regular} if, for all $\alpha\in \mathbb{F}_p^n$,
	$\mathcal{W}_f(\alpha)=\lambda p^{\frac{n}{2}} \xi_p^{f^*(\alpha)},$
	where $\lambda\in \{\pm 1,\ \pm \sqrt{-1}\}$ is independent of $\alpha$;
	otherwise it is called \emph{non-weakly regular}. It is known that weakly regular bent functions appear in pairs, i.e., the dual of a weakly regular bent function is also weakly regular bent. Given a weakly regular bent function $f(x)$, we have $f^{**}(x)=f(-x)$, where $f^{**}(x)$ is the dual of $f^*(x)$. However, for a non-weakly regular bent function, its dual may not be bent \cite{Cesmelioglu1,Cesmelioglu2}. In \cite{Ozbudak1}, {\"Ozbudak} \emph {et al.} proved that for a non-weakly regular bent function $f(x)$, if its dual $f^*(x)$ is bent, then $f^*(x)$ is non-weakly regular bent and satisfies $f^{**}(x)=f(-x)$, where $f^{**}(x)$ is the dual of $f^*(x)$. For a bent function $f(x)$, if its dual function $f^*(x)$ is bent, then $f(x)$ is called \emph{dual-bent}, otherwise it is called \emph{non-dual-bent}.
	
	Let $\mu=1$ if $p^{n}\equiv1$ (mod $4$) and  $\mu=\sqrt{-1}$ if $p^{n}\equiv3$ (mod $4$). For a bent function  $f(x):\mathbb{F}_{p}^n\longrightarrow\mathbb{F}_p$, we define $B_+(f)\ \text{and}\ B_{-}(f)$ as follows.
	\begin{align*}
	B_+(f)&:=\{\alpha\in \mathbb{F}_p^n:\ \mathcal{W}_f(\alpha)=\mu p^{\frac{n}{2}}\xi_p^{f^*(\alpha)}\},\\
	B_-(f)&:=\{\alpha\in \mathbb{F}_p^n:\ \mathcal{W}_f(\alpha)=-\mu p^{\frac{n}{2}}\xi_p^{f^*(\alpha)}\}.
	\end{align*}
	Define the type of $f(x)$ as 
	$f(x)\  \text{is of}\  Type\  (+) \ \text{if\  }\  \mathcal{W}_f(0)=\mu p^{\frac{n}{2}}\xi_p^{f^*(0)}\ \text{and of}\  Type\\ (-)\ \text{if\ }\  \mathcal{W}_f(0)=-\mu p^{\frac{n}{2}}\xi_p^{f^*(0)}.$ 
	
	Let $f(x)$ be a dual-bent function and $f^*(x)$ be its dual function, define $\epsilon_\alpha=1$ (respectively $-1$) if $\alpha\in B_+(f)$ (respectively $ B_-(f))$ and $\epsilon_\alpha^*=1$ (respectively $-1$) if $\alpha\in B_+(f^*)$ (respectively $B_-(f^*))$ for any $\alpha\in \mathbb{F}_p^n$. For any $j\in\mathbb{F}_p$, define \begin{align*}
	N_j(f)&:=\{x\in \mathbb{F}_p^n:\ f(x)=j\},\\
	c_j(f)&:=\{x\in B_+(f^*):\ f(x)=j\},\\
	d_j(f)&:=\{x\in B_-(f^*):\ f(x)=j\},\\
	c_j(f^*)&:=\{x\in B_+(f):\ f^*(x)=j\},\\
	d_j(f^*)&:=\{x\in B_-(f):\ f^*(x)=j\}.
	\end{align*}
	
	Let $\mathcal{DBF}$ be the set of $p$-ary dual-bent functions satisfying the following two conditions:\\
	(1) $f(0)$=0;\\
	(2) There exists a positive integer $h$ with $\mathrm{gcd}(h-1,\ p-1)=1$ such that for any $a\in\mathbb{F}_p^*$, $ x\in\mathbb{F}_{p}^n,\  f(ax)=a^hf(x).$
	
	We give the definition of non-degenerate subspace of $\mathbb{F}_p^n$ in the following.
	\begin{definition}
		Let $V$ be a subspace of $\mathbb{F}_p^n$. $V$ is said to be a non-degenerate subspace with respect to the ordinary inner product in $\mathbb{F}_p^n$ if $\{x\in V: x\cdot y =0\ \text{for\ all}\ y \in V\}=\{0\}$.
	\end{definition}
	
	By the results in \cite{Pelen1,Tang1,Wang1,Wei}, we give some properties of the functions belonging to  $\mathcal{DBF}$ in the following lemmas.
	\begin{lemma}\label{le 1}
		Let $f(x):\mathbb{F}_{p}^n\longrightarrow\mathbb{F}_p$ be a dual-bent function belonging to $\mathcal{DBF}$. Then the following statements hold.
		\begin{itemize}
			\item[$\mathrm{(1)}$]
			For any $ a\in\mathbb{F}_p^*,\ x\in\mathbb{F}_p^n$, there exists a positive integer $t$ with $\mathrm{gcd}(t-1,\ p-1)=1$ such that $f^*(ax)=a^tf^*(x)$.
			\item[$\mathrm{(2)}$] $f^*(0)=0$.
			\item[$\mathrm{(3)}$] For any $ x\in\mathbb{F}_p^n$,  $x\in B_+(f)$ (respectively $B_-(f)$) if and only if $ax\in B_+(f)$ (respectively $B_-(f)$) for any $a\in\mathbb{F}_p^*.$
			\item[$\mathrm{(4)}$] The types of $f(x)$ and $f^*(x)$ are the same if $p^n\equiv 1\ (\mathrm{mod}\ 4)$ and different if $p^n\equiv 3\ (\mathrm{mod}\ 4)$.
			\item[$\mathrm{(5)}$] If $f(x)$ is a non-weakly regular bent function, then we have the following.
			\begin{itemize} 
				\item [$\bullet$] If $B_+(f)$ is a non-degenerate subspace of dimension $r$ over $\mathbb{F}_p$, then 
				for $p^n\equiv 1\ (\mathrm{mod}\ 4)$, we have that $|B_+(f^*)|=p^r$ and $B_+(f)^{\bot}\subseteq c_0(f)$; for $p^n\equiv 3\ (\mathrm{mod}\ 4)$, we have that $|B_-(f^*)|=p^r$ and $B_+(f)^{\bot}\subseteq d_0(f)$.
				\item [$\bullet$] If $B_-(f)$ is a non-degenerate subspace of dimension $r$ over $\mathbb{F}_p$, then 
				for $p^n\equiv 1\ (\mathrm{mod}\ 4)$, we have that $|B_-(f^*)|=p^r$ and $B_-(f)^{\bot}\subseteq d_0(f)$; for $p^n\equiv 3\ (\mathrm{mod}\ 4)$, we have that $|B_+(f^*)|=p^r$ and $B_-(f)^{\bot}\subseteq c_0(f)$.
			\end{itemize}
		\end{itemize}
	\end{lemma}
	\begin{proof}
		$\mathrm{(1)}$ The result follows from \cite[Proposition 5]{Tang1} when $f(x)$ is weakly regular, and follows from 
		\cite[Lemma 9]{Wei} when $f(x)$ is non-weakly regular.\\
		$\mathrm{(2)}$ The result follows from \cite[Proposition 4]{Tang1} when $f(x)$ is weakly regular, and follows from 
		\cite[Lemma 10]{Wei} when $f(x)$ is non-weakly regular.\\
		$\mathrm{(3)}$ The result obviously holds when $f(x)$ is weakly regular, and follows from 
		\cite[Remark 6]{Wei} when $f(x)$ is non-weakly regular.\\
		$\mathrm{(4)}$ The result follows from \cite[Proposition 2]{Wang1}.\\
		$\mathrm{(5)}$ If $B_+(f)$ is a non-degenerate subspace of dimension $r$ over $\mathbb{F}_p$, then by part (4), we know that for $p^n\equiv 1\ (\mathrm{mod}\ 4)$, $0\in B_+(f)\cap B_+(f^*)$ and for $p^n\equiv 3\ (\mathrm{mod}\ 4)$, $0\in B_+(f)\cap B_-(f^*)$. If $B_-(f)$ is a non-degenerate subspace of dimension $r$ over $\mathbb{F}_p$, then by part (4), we know that for $p^n\equiv 1\ (\mathrm{mod}\ 4)$, $0\in B_-(f)\cap B_-(f^*)$ and for $p^n\equiv 3\ (\mathrm{mod}\ 4)$, $0\in B_-(f)\cap B_+(f^*)$. Since $f(x)\in\mathcal{DBF}$, then $f(x)=f(-x)$  for any $x\in\mathbb{F}_p^n$ and $f(0)=0$. Thus, the result follows from \cite[Proposition 2.4, Lemma 2.3]{Pelen1}.\qed
	\end{proof}

	Let $f(x):\mathbb{F}_{p}^n\longrightarrow\mathbb{F}_p$ be a dual-bent function belonging to $\mathcal{DBF}$ and $f^*(x)$ be its dual. In the sequel, denote $t$ as a positive integer satisfying $\rm{gcd}(t-1,\ p-1)=1$ and $f^*(ax)=a^tf^*(x)$ for any $a\in\mathbb{F}_p^*$, $x\in\mathbb{F}_p^n$. Define $S_0(f^*,y)=\sum\limits_{\alpha\in B_+(f^*)}\xi_p^{f(\alpha)+\alpha\cdot y}$ and $ S_1(f^*,y)=\sum\limits_{\alpha\in B_-(f^*)}\xi_p^{f(\alpha)+\alpha\cdot y} $ for any $y\in \mathbb{F}_p^n$. Let $S_0(f^*,y)=0$ (respectively $S_1(f^*,y)=0$) if $B_+(f^*)=\emptyset$ (respectively $B_-(f^*)=\emptyset)$. In the following lemma, we give the values of $S_0(f^*,y)$ and $S_1(f^*,y)$.
	
	\begin{lemma}\label{le 2}
		Let $f(x):\mathbb{F}_{p}^n\longrightarrow\mathbb{F}_p$ be a dual-bent function belonging to $\mathcal{DBF}$. Then we have the following.
		\begin{itemize}
			\item [$\bullet$] If $n$ is even, then 
			$S_0(f^*,y)=\frac{\epsilon_{y}+1}{2}p^{\frac{n}{2}}\xi_p^{f^*(y)}$ and  $S_1(f^*,y)=\frac{\epsilon_{y}-1}{2}p^{\frac{n}{2}}\xi_p^{f^*(y)}$.
			\item [$\bullet$] If $n$ is odd, then $S_0(f^*,y)=\frac{\epsilon_{y}+\eta(-1)}{2}\sqrt{p^*}p^{\frac{n-1}{2}}\xi_p^{f^*(y)}$ and  $S_1(f^*,y)=\frac{\epsilon_{y}-\eta(-1)}{2}\sqrt{p^*}p^{\frac{n-1}{2}}\xi_p^{f^*(y)}$.
		\end{itemize}
	\end{lemma}
	\begin{proof} According to Lemma \ref{le 1}, we know that for any $y\in\mathbb{F}_p^n$, $f^*(-y)=f^*(y)$, $f(-y)=f(y)$ and  $\epsilon_y=\epsilon_{-y}$ .
		When $f(x)$ is weakly regular, again by Lemma \ref{le 1}, we know that if $\epsilon_y=1$ for any $y\in\mathbb{F}_p^n$, then $B_+(f^*)=\mathbb{F}_p^n$, $B_-(f^*)=\emptyset$ for $p^n\equiv 1\ (\mathrm{mod}\ 4)$, and $B_+(f^*)=\emptyset$, $B_-(f^*)=\mathbb{F}_p^n$ for $p^n\equiv 3\ (\mathrm{mod}\ 4)$; if  $\epsilon_y=-1$ for any $y\in\mathbb{F}_p^n$, then $B_+(f^*)=\emptyset$, $B_-(f^*)=\mathbb{F}_p^n$ for $p^n\equiv 1\ (\mathrm{mod}\ 4)$, and $B_+(f^*)=\mathbb{F}_p^n$, $B_-(f^*)=\emptyset$ for $p^n\equiv 3\ (\mathrm{mod}\ 4)$. Thus, the result directly follows from the definition of Walsh transform of $f(x)$. When $f(x)$ is non-weakly regular, by the definition of Walsh transform, we have that 
		\begin{equation}\begin{split}S_0(f^*,y)+S_1(f^*,y)&=\mathcal{W}_{f}(-y)=\mathcal{W}_{f}(y)\\
		&=\begin{cases} \epsilon_y p^{\frac{n}{2}}\xi_p^{f^*(y)}, &\text{if\ } n\  \text{is\ even},\\
		\epsilon_y\sqrt{p^*}p^{\frac{n-1}{2}}\xi_p^{f^*(y)}, &\text{if\ } n\ \text{is\ odd.}
		\end{cases}\end{split}\end{equation}\\
		Since $f^{**}(\alpha)=f(-\alpha)=f(\alpha)$ for any $\alpha\in \mathbb{F}_p^n$, then\\ \begin{equation}\mathcal{W}_{f^*}(\alpha)=\begin{cases}
		\epsilon_{\alpha}^*p^{\frac{n}{2}}\xi_p^{f(\alpha)},&\text{if}\ n\ \text{is\ even},\\
		\epsilon_{\alpha}^*\sqrt{p^*} p^{\frac{n-1}{2}}\xi_p^{f(\alpha)},&\text{if}\ n\ \text{is\ odd}.
		\end{cases}\end{equation}\\
		By the inverse Walsh transform, we have that \begin{equation}\sum\limits_{\alpha\in \mathbb{F}_p^n}\mathcal{W}_{f^*}(\alpha)\xi_p^{\alpha\cdot y}=p^n\xi_p^{f^*(y)}.\end{equation}
		Thus, by Equations (2) and (3), we have that 
		\begin{equation}S_0(f^*,y)-S_1(f^*,y)=\begin{cases}
		p^{\frac{n}{2}}\xi_p^{f^*(y)},&\text{if}\ n\ \text{is\ even},\\
		\eta(-1)\sqrt{p^*}p^{\frac{n-1}{2}}\xi_p^{f^*(y)},&\text{if}\ n\ \text{is\ odd}.\\
		\end{cases}\end{equation}
		Combining Equations (1) and (4), we easily obtain the values of $S_0(f^*,y)$ and $S_1(f^*,y)$.\qed
	\end{proof}
	
	The following lemma gives the value distributions of $p$-ary bent functions.
	\begin{lemma}\label{le 3} \cite{Ozbudak1} Let $f(x):\mathbb{F}_{p}^n\longrightarrow\mathbb{F}_p$ be a $p$-ary bent function with $f^{*}(0)=j_0$, then we have the following.                                           	\begin{itemize}
			\item[$\bullet$]If $n$ is even, then \\$|N_{j_0}(f)|=p^{n-1}\pm p^{\frac{n}{2}}\mp p^{\frac{n}{2}-1},$ $|N_j(f)|=p^{n-1}\mp p^{\frac{n}{2}-1},$ for $j\ne j_0\in \mathbb{F}_p$.
			\item[$\bullet$]If $n$ is odd, then\\
			$|N_{j_0}(f)|=p^{n-1},\ |N_{j_0+j}(f)|=p^{n-1}\pm\eta(j)p^{\frac{n-1}{2}},$ for $ j\in\mathbb{F}_p^*$.
		\end{itemize}
		Here the sign is $+$ (respectively $-$) if and only if the type of $f(x)$ is $(+)$ (respectively $(-)$).
	\end{lemma}
	
	For $p$-ary bent functions belonging to $\mathcal{DBF}$, we give the values of $|c_i(f)|$ and $|d_i(f)|$ for any $i\in \mathbb{F}_p$ in the following lemma.
	\begin{lemma} \label{le 4}
		Let $f(x):\mathbb{F}_{p}^n\longrightarrow\mathbb{F}_p$ be a dual-bent function belonging to $\mathcal{DBF}$ and $|B_{+}(f^*)|=k\ (0\le k\le p^n)$. Then we have the following.
		\begin{itemize}
			\item[$\bullet$]
			If $n$ is even, then  \\
			$|c_{0}(f)|=\begin{cases}\frac{k}{p} +(p-1)p^{\frac{n}{2}-1},&\hspace{1.2cm}         \text{if}\ 0\in B_+(f),\\
			\frac{k}{p},&\hspace{1.2cm}     \text{if}\ 0\in B_-(f);
			\end{cases}$\\ 
			$|d_{0}(f)|=\begin{cases}p^{n-1}-\frac{k}{p},\hspace{2.8cm}   &\text{if}\ 0\in B_+(f),\\
			p^{n-1}-(p-1)p^{\frac{n}{2}-1}-\frac{k}{p},\ &\text{if}\ 0\in B_-(f);\end{cases}$\\
			and for $j\in \mathbb{F}_p^*$\\
			$|c_j(f)|\ =\begin{cases}
			\frac{k}{p}-p^{\frac{n}{2}-1},&\hspace{2.2cm} \text{if}\ 0\in B_+(f),\\
			\frac{k}{p},\ &\hspace{2.2cm}\text{if}\ 0\in B_-(f);\end{cases}$\\
			$|d_j(f)|\ =\begin{cases}
			p^{n-1}-\frac{k}{p},&\hspace{1cm} \text{if}\ 0\in B_+(f),\\
			p^{n-1}+p^{\frac{n}{2}-1}-\frac{k}{p},\ &\hspace{1cm}\text{if}\ 0\in B_-(f).\end{cases}$\\
			\item[$\bullet$] If $p\equiv 1$ $\mathrm{(mod\  4)}$ and $n$ is odd, then \\
			$|c_{0}(f)|=\frac{k}{p};$\\
			$|d_{0}(f)|=p^{n-1}-\frac{k}{p};$\\
			and for $j\in\mathbb{F}_p^*$\\
			$|c_{j}(f)|\ =\begin{cases}
			\frac{k}{p}+\eta(j)p^{\frac{n-1}{2}},&\hspace{1.5cm} \text{if}\ 0\in B_+(f),\\
			\frac{k}{p},\ &\hspace{1.5cm} \text{if}\ 0\in B_-(f);\end{cases}$\\
			$|d_{j}(f)|\ =\begin{cases} p^{n-1}-\frac{k}{p},&\hspace{0.3cm} \text{if}\ 0\in B_+(f),\\
			p^{n-1}-\eta(j)p^{\frac{n-1}{2}}-\frac{k}{p},\ &\hspace{0.3cm}\text{if}\ 0\in B_-(f).\end{cases}$\\
			\item[$\bullet$] If $p\equiv 3$ $\mathrm{(mod\  4)}$ and $n$ is odd, then\\
			$|c_{0}(f)|=\frac{k}{p};$\\
			$|d_{0}(f)|=p^{n-1}-\frac{k}{p};$\\
			and for $j\in\mathbb{F}_p^*$\\
			$|c_{j}(f)|\ =\begin{cases}
			\frac{k}{p},\ &\hspace{1.5cm}\text{if}\ 0\in B_+(f),\\
			\frac{k}{p}-\eta(j)p^{\frac{n-1}{2}},&\hspace{1.5cm} \text{if}\ 0\in B_-(f);\end{cases}$\\
			$|d_{j}(f)|\ =\begin{cases} p^{n-1}+\eta(j)p^{\frac{n-1}{2}}-\frac{k}{p},&\hspace{0.3cm} \text{if}\ 0\in B_+(f),\\
			p^{n-1}-\frac{k}{p},\ &\hspace{0.3cm}\text{if}\ 0\in B_-(f).\end{cases}$\\	
		\end{itemize}
	\end{lemma}
	\begin{proof}
		When $f(x)$ is weakly regular, the results follow from Lemma \ref{le 3}. When $f(x)$ is non-weakly regular, since $f(x)\in \mathcal{DBF}$, we know that $f^{**}(x)=f(-x)=f(x)$ for any $x\in \mathbb{F}_p^n$, then the results follow from \cite[Lemma 4]{Wei}. \qed
	\end{proof}
	\begin{remark}\label{re 1}
		
		$\mathrm{(1)}$ If $f(x)$ is a dual-bent function belonging to $\mathcal{DBF}$, then by the results of \cite[Lemma 4]{Wei}, Lemmas \ref{le 1}, \ref{le 3} and \ref{le 4}, we easily get that for $p^n\equiv 1\ (\mathrm{mod}\ 4)$, if $|B_+(f)|=|B_+(f^*)|=k\ (0\le k\le p^n)$, then $|c_j(f)|=|c_j(f^*)|, |d_j(f)|=|d_j(f^*)|$ for any  $j\in\mathbb{F}_p$. For $p^n\equiv 3\ (\mathrm{mod}\ 4)$, if $0\in B_+(f)$ and $|B_-(f)|=|B_+(f^*)|=k\ (0\le k< p^n)$, then $|d_0(f)|=|c_0(f^*)|,$ $|c_i(f)|=|d_i(f^*)|$ for any $i\in \mathbb{F}_p$, and $|d_j(f)|=|c_s(f^*)|$, where $j,s\in\mathbb{F}_p^*$ and $\eta(js)=-1$; if $0\in B_-(f)$ and $|B_-(f)|=|B_+(f^*)|=k\ (0< k\le p^n)$, then $|c_0(f)|=|d_0(f^*)|$,  $|d_i(f)|=|c_i(f^*)|$ for any $i\in \mathbb{F}_p$, and $|c_j(f)|=|d_s(f^*)|$, where $j,s\in\mathbb{F}_p^*$ and $\eta(js)=-1$.\\
		$\mathrm{(2)}$ If $f(x)$ is non-weakly regular, then by Lemma \ref{le 4}, we have that when $n$ is even, if $B_+(f^*)$ is an $r$-dimensional subspace over $\mathbb{F}_p$, then $\frac{n}{2}\le r\le n-1$ and if $B_-(f^*)$ is an $r$-dimensional subspace over $\mathbb{F}_p$, then $\frac{n}{2}+1\le r\le n-1$; when $n$ is odd, if $B_+(f^*)$ or $B_-(f^*)$ is an $r$-dimensional subspace over $\mathbb{F}_p$, then $\frac {n+1}{2}\le r\le n-1$.

	\end{remark}
	\subsection{Association schemes}
	\quad 
	\begin{definition}
		Let $X$ be a nonempty finite set. A \emph {$d$-class association scheme} $(X, \{R_i\}_{i=0}^d)$ is a partition of $X\times X$ into binary relations $R_0,\ R_1,\ \dots,\ R_d$ satisfying the following properties:
		\begin{itemize}
			\item [$\bullet$] $R_0=\{(x,\ x): x\in X\}$;
			\item [$\bullet$] For any $j\in \{0,\ 1,\ \dots, d\}$, $R_j^{T}=R_{j'}$ for some $j'\in \{0,\ 1,\ \dots, d\}$, where $R_j^{T}=\{(y,\ x): (x,\ y)\in R_j\}$. If $j=j'$, we call $R_j$ symmetric;
			\item [$\bullet$] For $u,\ v,\ w\in\{0,\ 1,\ \dots, d\}$  and $x,\ y \in X$ with $(x,\ y)\in R_w$, the number $p_{u,v}^w=|\{z\in X: (x,\ z)\in R_u,\ (z,\ y)\in R_v\}|$ depends only on $u,\ v,\ w$. These numbers $p_{u,v}^w$ are called \emph{intersection numbers}.
		\end{itemize}
	\end{definition}

	If $p_{u,v}^w=p_{v,u}^w$ for any $u,\ v,\ w\in\{0,\ 1,\ \dots, d\}$, then the association scheme is called \emph{commutative}. If $R_u$ is symmetric for any $u\in \{0,\ 1,\ \dots, d\}$, then the association scheme is called \emph{symmetric}. Clearly, a symmetric association scheme is commutative. A \emph{fusion} of an association scheme $(X, \{R_i\}_{i=0}^d)$ is a partition $\{S_0,\ S_1,\ \dots,S_t\}$ of $X\times X$ such that $S_0=R_0$ and $S_j$, $1\le j \le t$, is the union of some of the relations $R_s$, $1\le s \le d$.
	
	Let $(X, \{R_i\}_{i=0}^d)$ be a $d$-class association scheme, and denote $A_i$ as the adjacency matrix of $R_i$ whose $(x,y)$ entry is 1 if $(x,y)\in R_i$ and $0$ otherwise for any $i\in\{0,\ 1,\ \dots, d\}$. The $\mathbb{C}$-linear span of $A_0,A_1,\dots,A_d$ forms an algebra $\mathcal{A}$ of dimension of $d+1$, called the \emph{Bose-Mesner algebra} of the association scheme.
	When $(X, \{R_i\}_{i=0}^d)$ is commutative, the algebra $\mathcal{A}$ has a basis of idempotents $E_0, E_1,\dots,E_d$ such that $E_uE_v=\delta_{uv}E_u$ for any $u,v\in \{0,\ 1,\ \dots, d\}$, where $\delta_{uv}=1$ if $u=v$ and $\delta_{uv}=0$ if $u\ne v$. The $(d+1)\times(d+1)$ matrix $P=[P_{ij}]$, where $P_{ij}$ is the element in the $i$-th row and $j$-th column of $P$, such that 
	\begin{equation}
	(A_0,A_1,\dots,A_d)=(E_0,E_1,\dots,E_d)P,
	\end{equation}
	is called the \emph{first eigenmatrix} of the association scheme. Dually, the $(d+1)\times(d+1)$ matrix $Q=[Q_{ij}]$, where $Q_{ij}$ is the element in the $i$-th row and $j$-th column of $Q$, such that 
	\begin{equation}
	(E_0,E_1,\dots,E_d)=\frac{1}{|X|}(A_0,A_1,\dots,A_d)Q,
	\end{equation}
	is called the \emph{second eigenmatrix} of the association scheme. There exist numbers $q_{u,v}^w$, called the \emph{Krein parameters}, such that 
	\begin{equation}
	E_u\circ E_v=\frac{1}{|X|}\sum\limits_{w=0}^dq_{u,v}^wE_w,\ \text{for}\ u,v\in\{0,\ 1,\ \dots, d\},
	\end{equation}
	where $\circ$ is the Hadamard product, defined by $(A\circ B)_{ij}=A_{ij}B_{ij}$  for two matrices $A,B$ of order $m$.
	
	\begin{remark}\label{re 2}
		For a $d$-class symmetric association scheme $(X, \{R_i\}_{i=0}^d)$, by the definitions of adjacency matrix and association scheme, we easily know that $A_uA_v=\sum\limits_{w=0}^dp_{u,v}^wA_w$  and $A_u\circ A_v=\delta_{uv}A_u$ for $u,v\in\{0,\ 1,\ \dots, d\}$. Since $E_uE_v=\delta_{uv}E_u$ for $u,v\in\{0,1,\dots,d\}$,  by Equations (5)-(7), for any $u,\ v,\ w\in\{0,\ 1,\ \dots, d\}$, we have that 
		\begin{align}
		p_{u,v}^w&=\frac{1}{|X|}\sum\limits_{i=0}^dQ_{wi}P_{iu}P_{iv},\\
		q_{u,v}^w&=\frac{1}{|X|}\sum\limits_{i=0}^dP_{wi}Q_{iu}Q_{iv}.
		\end{align} 
		
	\end{remark}
	
	A character $\chi$ of a finite abelian group $G$ is a homomorphism from $G$ into the multiplicative group of complex numbers of absolute value $1$. All characters form a group denoted by $\widehat{G}$ and $\widehat{G}$ is isomorphic to $G$. Every additive character of $\mathbb{F}_p^n$ can be expressed as 
	$\chi_\alpha(x)=\xi_p^{\alpha\cdot x},\ x\in \mathbb{F}_p^n$ for $\alpha\in \mathbb{F}_p^n$. For any $\alpha\in \mathbb{F}_p^n$, define $\chi_\alpha(\emptyset)=0$ and  $\chi_\alpha(D)=\sum_{x\in D}\xi_p^{\alpha\cdot x}$ for any nonempty subset $D$ of $\mathbb{F}_p^n$. Using the additive characters of $\mathbb{F}_p^n$, one can determine whether a partition of $\mathbb{F}_p^n\times\mathbb{F}_p^n$ is an association scheme. Firstly, we give the definition of dual partition in the following.
	\begin{definition}
		Let  $U=\{U_i\}_{i=0}^d$ be a partition of $\mathbb{F}_p^n$. The dual partition, denoted by $\widehat{U}$, is the partition of the character group $\widehat{\mathbb{F}_p^n}$ defined by the equivalence relation
		$$ \chi\sim_{\widehat{U}} \chi' \Longleftrightarrow \chi (U_i)=\chi'(U_i) \mbox{ for all } i=0, \ldots, d,$$ where $\chi$ and $\chi'$ are  additive characters of ${\mathbb{F}_p^n}$. By identifying the group $\mathbb{F}_p^n$ with its character group $\widehat{\mathbb{F}_p^n}$ via $\alpha\mapsto\chi_{\alpha}$, the equivalence relation is restated as follows:
		$$ \alpha\sim_{\widehat{U}} \alpha' \Longleftrightarrow \chi_{\alpha} (U_i)=\chi_{\alpha'}(U_i) \mbox{ for all } i=0, \ldots, d,$$
		where $\chi_{\alpha}$ and $\chi_{\alpha'}$ are additive characters of $\mathbb{F}_p^n$. The partition $U$ is called \emph{Fourier-reflexive} if and only if  $\widehat{\widehat{U}}=U$.
	\end{definition}
	
	\begin{remark}\label {re 3}
		By the discussion in \cite {Wu}, we know that for any partition $U=\{U_i\}_{i=0}^d$ of $\mathbb{F}_p^n$, $\alpha$ is not equivalent to $0$ for any $\alpha\in\mathbb{F}_p^n\setminus\{0\}$. 
	\end{remark}
	
	For Fourier-reflexive partitions, we have the following lemma.
	\begin{lemma}\cite{Glu}\label{le 5}
		A partition $U=\{U_i\}_{i=0}^{d}$ of $\mathbb{F}_p^n$ is Fourier-reflexive if and only if $|U|=|\widehat{U}|$.
	\end{lemma}
	
	The following lemma presents a sufficient and necessary condition for a partition of $\mathbb{F}_p^n\times\mathbb{F}_p^n$ to be a symmetric association scheme.
	\begin{lemma}\cite{Glu,Zino}\label{le 6}
		Let  $U=\{U_i\}_{i=0}^d$ be a partition of $\mathbb{F}_p^n$,
		where $U_0=\{0\}$, and let $\widehat{U}=\{V_j\}_{j=0}^{d'}$ be the dual partition of $U$, where $V_0=\{0\}$.
		Let $\{R_i\}_{i=0}^d$ be a partition of $\mathbb{F}_p^n \times\mathbb{F}_p^n$ defined by
		\begin{equation*}
		(\alpha,\beta)\in R_i \Longleftrightarrow \alpha-\beta\in U_i, \; \; i=0,1,\ldots,d.
		\end{equation*}
		Then,  $(\mathbb{F}_p^n, \{R_i\}_{i=0}^d)$ is a symmetric association scheme if and  only if $d=d'$. 
	\end{lemma}
	\begin{remark}\label{re 4}
		(1)
		Let $U=\{U_i\}_{i=0}^d$ be a partition of $\mathbb{F}_p^n$,
		where $U_0=\{0\}$, and $\{R_i\}_{i=0}^d$ be a partition of $\mathbb{F}_p^n\times\mathbb{F}_p^n$ defined in Lemma \ref{le 6}. If $(\mathbb{F}_p^n, \{R_i\}_{i=0}^d)$ is a symmetric association scheme, then we say that the partition $U$ induces a symmetric association scheme.\\
		(2) By Lemmas \ref{le 5} and \ref{le 6}, if $(\mathbb{F}_p^n, \{R_i\}_{i=0}^d)$ defined in Lemma \ref{le 6} is a symmetric association scheme, then the dual partition of $U$ can also induce a symmetric association scheme, called the \emph{dual} of $(\mathbb{F}_p^n, \{R_i\}_{i=0}^d)$.\\
		(3) By the results in \cite{Bannai,Bridges,Ozbudak2}, if $(\mathbb{F}_p^n, \{R_i\}_{i=0}^d)$ defined in Lemma \ref{le 6} is a symmetric association scheme, then the first eigenmatrix of the dual of $(\mathbb{F}_p^n, \{R_i\}_{i=0}^d)$ is equal to the second eigenmatrix of $(\mathbb{F}_p^n, \{R_i\}_{i=0}^d)$. For $\alpha\in V_i,\ 1\le i\le d$, let $u_{ij}=\chi_\alpha(U_j)$, $\ 1\le j\le d$, then the first eigenmatrix of $(\mathbb{F}_p^n, \{R_i\}_{i=0}^d)$ is given by 
		\[\begin{bmatrix}
		1&|U_1|&|U_2|&\dots&|U_d|\\
		1&u_{11}&u_{12}&\dots&u_{1d}\\
		\vdots&\vdots&\vdots&\vdots&\vdots\\
		1&u_{d1}&u_{d2}&\dots&u_{dd}
		\end{bmatrix}.\]
		
	\end{remark}
	\section{Auxiliary results}
	\quad In this section, we give some exponential sums related to dual-bent functions, which will be used to construct association schemes.  We begin this section by recalling the following two lemmas.
	\begin{lemma}\cite{Lidl}\label{le 7}
		With the defined notation, we have \begin{itemize} 
			\item[$\mathrm{(1)}$] $\sum\limits_{y\in \mathbb{F}_p^*}\eta(y)=0.$
			\item[$\mathrm{(2)}$] $\sum\limits_{y \in \mathbb{F}_p^*}\xi_p^{zy}=-1\ \text{for any}\ z\in \mathbb{F}_p^{*}.$
			\item[$\mathrm{(3)}$] $\sum\limits_{y\in \mathbb{F}_p^*}\eta(y)\xi_p^y=\sqrt{p^*}=\begin{cases}
			\sqrt{p},&\text{if}\ p\equiv 1\ (\mathrm{mod}\ 4),\\
			\sqrt{-p},&\text{if}\ p\equiv3 \ (\mathrm{mod}\ 4).
			\end{cases}$
			\item[$\mathrm(4)$] $\sum\limits_{y\in SQ}\xi_p^{ya}=\frac{\eta(a)\sqrt{p^*}-1}{2}\ \text {for any}\ a \in \mathbb{F}_p^*.$
			
			\item[$\mathrm(5)$] $\sum\limits_{y\in NSQ}\xi_p^{ya}=\frac{-\eta(a)\sqrt{p^*}-1}{2}\  \text {for any}\ a \in \mathbb{F}_p^*.$
		\end{itemize}
	\end{lemma}
	\begin{lemma}\cite{Lidl}\label{le 8}
		Let $H\subseteq \mathbb{F}_p^n$ be a subspace of dimension $r$, then for any $\alpha\in \mathbb{F}_p^n$, we have 
		\[\chi_\alpha(H)=\begin{cases}
		p^r, &\text{if}\ \alpha\in H^{\bot},\\
		0, &\text{otherwise}.
		\end{cases}\]
	\end{lemma}
	
	In what follows, a series of auxiliary results are presented, which will be useful for our constructions.
	
	\begin{lemma}\label{le 9}
		For any $i,\ j\in \mathbb{F}_p$, define $K^{(m)}(i,j)=\sum\limits_{c\in \mathbb{F}_p^*}\xi_p^{c^{1-m}j-ci}$ and $S^{(m)}(i,j)=\sum\limits_{c\in\mathbb{F}_p^*}\eta(c)\xi_p^{c^{1-m}j-ci}$, where $m$ is a positive integer satisfying $\mathrm{gcd}(m-1,\ p-1)=1$, then we have the following.
		\begin{itemize}
			\item[$\mathrm(1)$] $K^{(m)}(0,0)=p-1$, and for any $i,\ j\in \mathbb{F}_p^*,\ K^{{(m)}}(i,0)=-1$ and $K^{{(m)}}(0,j)=-1$.
			\item[$\mathrm(2)$]	
			$\sum\limits_{i\in SQ}K^{{(m)}}(i,j)=
			\begin{cases}
			\frac{-(p-1)}{2}, &\text{if}\ j=0,\\
			\frac{\eta(j)p+1}{2},&\text{if}\ j\ne0,
			\end{cases}$
			and \\$\sum\limits_{i\in NSQ}K^{{(m)}}(i,j)=
			\begin{cases}
			\frac{-(p-1)}{2},& \text{if}\ j=0,\\
			\frac{-\eta(j)p+1}{2},&\text{if}\ j\ne0.
			\end{cases}$
			\item[$\mathrm(3)$]$\sum\limits_{i\in\mathbb{F}_p^*}K^{(m)}(i,j)=
			\begin{cases}
			-(p-1),&\text{if}\ j=0,\\
			1,&\text{if}\ j\ne 0.
			\end{cases}$
			\item[$\mathrm(4)$] $S^{(m)}(0,0)=0$, and for any $i,\ j\in \mathbb{F}_p^*,\ S^{(m)}(i,0)=\eta(-i)\sqrt{p^*}$ and $S^{(m)}(0,j)=\eta(j)\sqrt{p^*}$.
			\item[$\mathrm(5)$]	
			$\sum\limits_{i\in SQ}S^{(m)}(i,j)=
			\begin{cases}
			\frac{(p-1)}{2}\eta(-1)\sqrt{p^*}, &\text{if}\ j=0,\\
			\frac{-\eta(-1)-\eta(j)}{2}\sqrt{p^*},&\text{if}\ j\ne0,
			\end{cases}$
			and \\$\sum\limits_{i\in NSQ}S^{(m)}(i,j)=
			\begin{cases}
			\frac{-(p-1)}{2}\eta(-1)\sqrt{p^*},& \text{if}\ j=0,\\
			\frac{\eta(-1)-\eta(j)}{2}\sqrt{p^*},&\text{if}\ j\ne0.
			\end{cases}$
			\item[$\mathrm(6)$]$\sum\limits_{i\in\mathbb{F}_p^*}S^{(m)}(i,j)=
			\begin{cases}
			0,&\text{if}\ j=0,\\
			-\eta(j)\sqrt{p^*},&\text{if}\ j\ne 0.
			\end{cases}$
		\end{itemize}
	\end{lemma}
	\begin{proof}
		Since $m$ is a positive integer satisfying $\mathrm{gcd}(m-1,p-1)=1$, then for any $c\in \mathbb{F}_p^*$, when $c$ runs through $\mathbb{F}_p^*$, $c^{1-m}$ runs through $\mathbb{F}_p^*$ and $\eta(c)=\eta(c^{1-m})$. Then by the definitions of $K^{(m)}(i,j)$, $S^{(m)}(i,j)$ and Lemma \ref{le 7}, we easily obtain the results.\qed
	\end{proof}
	
	\begin{lemma}\label{le 10}
		Let $f(x):\mathbb{F}_{p}^n\longrightarrow\mathbb{F}_p$ be a dual-bent function belonging to $\mathcal{DBF}$. Then  we have the following.
		\begin{itemize}
			\item [$\bullet$]
			When $n$ is even, then, for any $\alpha\in \mathbb{F}_p^n$, $i\in\mathbb{F}_p$,
			\begin{align*}
			\chi_\alpha(c_i(f))&=\frac{1+\epsilon_\alpha}{2}p^{\frac{n}{2}-1}K^{(t)}(i, f^*(\alpha))+p^{-1}\chi_\alpha(B_+(f^*)),\\ \chi_\alpha(d_i(f))&=\frac{\epsilon_\alpha-1}{2}p^{\frac{n}{2}-1}K^{(t)}(i,f^*(\alpha))+p^{-1}\chi_\alpha(B_-(f^*)).
			\end{align*}
			\item [$\bullet$]
			When $n$ is odd, then, for any $\alpha\in\mathbb{F}_p^n$, $i\in\mathbb{F}_p$,
			\begin{align*} \chi_\alpha(c_i(f))=\frac{\epsilon_\alpha+\eta(-1)}{2}\sqrt{p^*}p^{\frac{n-3}{2}}S^{(t)}(i,f^*(\alpha))+p^{-1}\chi_\alpha(B_+(f^*)),\\ \chi_\alpha(d_i(f))=\frac{\epsilon_\alpha-\eta(-1)}{2}\sqrt{p^*}p^{\frac{n-3}{2}}S^{(t)}(i,f^*(\alpha))+p^{-1}\chi_\alpha(B_-(f^*)).
			\end{align*}
		\end{itemize}
	\end{lemma}
	\begin{proof}
		For any $\alpha\in\mathbb{F}_p^n$, we have 
		\begin{align*}
		\chi_\alpha(c_i(f))&=\sum\limits_{x\in B_+(f^*), f(x)=i}\xi_p^{\alpha\cdot x}\\
		&=p^{-1}\sum\limits_{x\in B_+(f^*)}\xi_p^{\alpha\cdot x}\sum\limits_{c\in\mathbb{F}_p}\xi_p^{(f(x)-i)c}\\
		&=p^{-1}\sum\limits_{c\in\mathbb{F}_p^*}\xi_p^{-ci}\sum\limits_{x\in B_+(f^*)}\xi_p^{cf(x)+\alpha\cdot x}+p^{-1}\sum\limits_{x\in B_+(f^*)}\xi_p^{\alpha\cdot x}\\
		&=p^{-1}\sum\limits_{c\in\mathbb{F}_p^*}\xi_p^{-ci}\sigma_c(\sum\limits_{x\in B_+(f^*)}\xi_p^{f(x)+c^{-1}\alpha\cdot x})+p^{-1}\sum\limits_{x\in B_+(f^*)}\xi_p^{\alpha\cdot x},\\
		\chi_\alpha(d_i(f))&=\sum\limits_{x\in B_-(f^*), f(x)=i}\xi_p^{\alpha\cdot x}\\
		&=p^{-1}\sum\limits_{x\in B_-(f^*)}\xi_p^{\alpha\cdot x}\sum\limits_{c\in\mathbb{F}_p}\xi_p^{(f(x)-i)c}\\
		&=p^{-1}\sum\limits_{c\in\mathbb{F}_p^*}\xi_p^{-ci}\sum\limits_{x\in B_-(f^*)}\xi_p^{cf(x)+\alpha\cdot x}+p^{-1}\sum\limits_{x\in B_-(f^*)}\xi_p^{\alpha\cdot x}\\
		&=p^{-1}\sum\limits_{c\in\mathbb{F}_p^*}\xi_p^{-ci}\sigma_c(\sum\limits_{x\in B_-(f^*)}\xi_p^{f(x)+c^{-1}\alpha\cdot x})+p^{-1}\sum\limits_{x\in B_-(f^*)}\xi_p^{\alpha\cdot x}.
		\end{align*}
		According to Lemma \ref{le 1}, since $f(x)\in\mathcal{DBF}$,  then  for any $c\in\mathbb{F}_p^*$, $\alpha\in\mathbb{F}_p^n$, $\alpha\in B_+(f)$ (respectively $B_-(f)$) if and only if $c\alpha\in B_+(f)$ (respectively $B_-(f)$), i.e., $\epsilon_\alpha=\epsilon_{c\alpha}$. Thus, by Lemma \ref{le 2}, 
		if $n$ is even, we get
		\begin{align*}
		\chi_\alpha(c_i(f))&=p^{-1}\sum_{c\in\mathbb{F}_p^*}\xi_p^{-ci}\sigma_c(\frac{1+\epsilon_\alpha}{2}p^{\frac{n}{2}}\xi_p^{c^{-t}f^*(\alpha)})+p^{-1}\sum\limits_{x\in B_+(f^*)}\xi_p^{\alpha\cdot x}\\
		&=\frac{1+\epsilon_\alpha}{2}p^{\frac{n}{2}-1}\sum\limits_{c\in\mathbb{F}_p^*}\xi_p^{c^{1-t}f^*(\alpha)-ci}+p^{-1}\sum\limits_{x\in B_+(f^*)}\xi_p^{\alpha\cdot x}\\
		&=\frac{1+\epsilon_\alpha}{2}p^{\frac{n}{2}-1}K^{(t)}(i,f^*(\alpha))+p^{-1}\chi_\alpha(B_+(f^*)),\\	
		\chi_\alpha(d_i(f))&=p^{-1}\sum_{c\in\mathbb{F}_p^*}\xi_p^{-ci}\sigma_c(\frac{\epsilon_\alpha-1}{2}p^{\frac{n}{2}}\xi_p^{c^{-t}f^*(\alpha)})+p^{-1}\sum\limits_{x\in B_-(f^*)}\xi_p^{\alpha\cdot x}\\
		&=\frac{\epsilon_\alpha-1}{2}p^{\frac{n}{2}-1}\sum\limits_{c\in\mathbb{F}_p^*}\xi_p^{c^{1-t}f^*(\alpha)-ci}+p^{-1}\sum\limits_{x\in B_-(f^*)}\xi_p^{\alpha\cdot x}
		\end{align*}
		\[
		=\frac{\epsilon_\alpha-1}{2}p^{\frac{n}{2}-1}K^{(t)}(i,f^*(\alpha))+p^{-1}\chi_\alpha(B_-(f^*)).\]
		
		If $n$ is odd, we get
		\begin{align*}
		\chi_\alpha(c_i(f))&=p^{-1}\sum\limits_{c\in\mathbb{F}_p^*}\xi_p^{-ci}\sigma_c(\frac{\epsilon_\alpha+\eta(-1)}{2}\sqrt{p^*}p^{\frac{n-1}{2}}\xi_p^{c^{                                                            -t}f^*(\alpha)})+p^{-1}\sum\limits_{x\in B_+(f^*)}\xi_p^{\alpha\cdot x}\\
		&=\frac{\epsilon_\alpha+\eta(-1)}{2}\sqrt{p^*}p^{\frac{n-3}{2}}\sum\limits_{c\in\mathbb{F}_p^*}\eta(c)\xi_p^{c^{1-t}f^*(\alpha)-ci}+p^{-1}\sum\limits_{x\in B_+(f^*)}\xi_p^{\alpha\cdot x}\\
		&=\frac{\epsilon_\alpha+\eta(-1)}{2}\sqrt{p^*}p^{\frac{n-3}{2}}S^{(t)}(i,f^*(\alpha))+p^{-1}\chi_\alpha(B_+(f^*)),\\
		\chi_\alpha(d_i(f))&=p^{-1}\sum\limits_{c\in\mathbb{F}_p^*}\xi_p^{-ci}\sigma_c(\frac{\epsilon_\alpha-\eta(-1)}{2}\sqrt{p^*}p^{\frac{n-1}{2}}\xi_p^{c^{                                                            -t}f^*(\alpha)})+p^{-1}\sum\limits_{x\in B_-(f^*)}\xi_p^{\alpha\cdot x}\\
		&=\frac{\epsilon_\alpha-\eta(-1)}{2}\sqrt{p^*}p^{\frac{n-3}{2}}\sum\limits_{c\in\mathbb{F}_p^*}\eta(c)\xi_p^{c^{1-t}f^*(\alpha)-ci}+p^{-1}\sum\limits_{x\in B_-(f^*)}\xi_p^{\alpha\cdot x}\\
		&=\frac{\epsilon_\alpha-\eta(-1)}{2}\sqrt{p^*}p^{\frac{n-3}{2}}S^{(t)}(i,f^*(\alpha))+p^{-1}\chi_\alpha(B_-(f^*)).
		\end{align*}
		The proof is now completed.\qed
	\end{proof}
	
	The following corollary follows from Lemma \ref{le 10}.
	\begin{corollary}\label{co 1}
		Let $f(x):\mathbb{F}_{p}^n\longrightarrow\mathbb{F}_p$ be a dual-bent function belonging to $\mathcal{DBF}$. Then  we have the following.
		\begin{itemize}
			\item [$\bullet$]
			When $n$ is even, then, for any $\alpha\in \mathbb{F}_p^n$, $i\in\mathbb{F}_p$,
			\[\chi_\alpha(N_i(f))=\epsilon_\alpha p^{\frac{n}{2}-1}K^{(t)}(i, f^*(\alpha))+p^{n-1}\delta_0(\alpha).\]
			\item [$\bullet$]
			When $n$ is odd, then, for any $\alpha\in\mathbb{F}_p^n$, $i\in\mathbb{F}_p$, \[\chi_\alpha(N_i(f))=\epsilon_\alpha\sqrt{p^*}p^{\frac{n-3}{2}}S^{(t)}(i,f^*(\alpha))+p^{n-1}\delta_0(\alpha).\]
			
			Here $\delta_0(0)=1$ and $\delta_0(\alpha)=0$ for any $\alpha\in \mathbb{F}_p^n\setminus\{0\}$.
		\end{itemize}
	\end{corollary}
	
	\begin{lemma}\label{le 11}
		Let $f(x):\mathbb{F}_{p}^n\longrightarrow\mathbb{F}_p$ be a dual-bent function belonging to $\mathcal{DBF}$. Then  we have the following.
		\begin{itemize}
			\item [$\bullet$] When $n$ is even, then, for any $\alpha\in\mathbb{F}_p^n$,\\ $\chi_\alpha(\bigcup\limits_{i\in SQ} c_i(f))=
			\begin{cases}
			\frac{(p-1)}{2}(p^{-1}\chi_\alpha(B_+(f^*))-\frac{1+\epsilon_\alpha}{2}p^{\frac{n}{2}-1}),&\hspace{0.5cm}\text{if}\ f^*(\alpha)=0,\\
			\frac{(1+\epsilon_\alpha)(\eta(f^*(\alpha))p+1)}{4}p^{\frac{n}{2}-1}+\\\frac{(p-1)}{2p}\chi_\alpha(B_+(f^*)),&\hspace{0.5cm}\text{if}\ f^*(\alpha)\ne0;
			\end{cases} $\\
			$\chi_\alpha(\bigcup\limits_{i\in NSQ} c_i(f))=
			\begin{cases}
			\frac{(p-1)}{2}(p^{-1}\chi_\alpha(B_+(f^*))-\frac{1+\epsilon_\alpha}{2}p^{\frac{n}{2}-1}),&\hspace{0.2cm}\text{if}\ f^*(\alpha)=0,\\
			\frac{(1+\epsilon_\alpha)(-\eta(f^*(\alpha))p+1)}{4}p^{\frac{n}{2}-1}+\\\frac{(p-1)}{2p}\chi_\alpha(B_+(f^*)),&\hspace{0.2cm}\text{if}\ f^*(\alpha)\ne0;
			\end{cases} $\\
			$\chi_\alpha(\bigcup\limits_{i\in SQ} d_i(f))=
			\begin{cases}
			\frac{(p-1)}{2}(p^{-1}\chi_\alpha(B_-(f^*))-\frac{\epsilon_\alpha-1}{2}p^{\frac{n}{2}-1}),&\hspace{0.4cm}\text{if}\ f^*(\alpha)=0,\\
			\frac{(\epsilon_\alpha-1)(\eta(f^*(\alpha))p+1)}{4}p^{\frac{n}{2}-1}+\\\frac{(p-1)}{2p}\chi_\alpha(B_-(f^*)),&\hspace{0.4cm}\text{if}\ f^*(\alpha)\ne0;
			\end{cases} $\\
			$\chi_\alpha(\bigcup\limits_{i\in NSQ} d_i(f))=
			\begin{cases}
			\frac{(p-1)}{2}(p^{-1}\chi_\alpha(B_-(f^*))-\frac{\epsilon_\alpha-1}{2}p^{\frac{n}{2}-1}),&\hspace{0.2cm}\text{if}\ f^*(\alpha)=0,\\
			\frac{(\epsilon_\alpha-1)(-\eta(f^*(\alpha))p+1)}{4}p^{\frac{n}{2}-1}+\\\frac{(p-1)}{2p}\chi_\alpha(B_-(f^*)),&\hspace{0.2cm}\text{if}\ f^*(\alpha)\ne0.
			\end{cases} $\\
			\item [$\bullet$] When $n$ is odd, then, for any $\alpha\in\mathbb{F}_p^n$,\\ $\chi_\alpha(\bigcup\limits_{i\in SQ} c_i(f))=
			\begin{cases}
			\frac{(p-1)}{2}(p^{-1}\chi_\alpha(B_+(f^*))+\frac{\eta(-1)+\epsilon_\alpha}{2}p^{\frac{n-1}{2}}),&\text{if}\ f^*(\alpha)=0,\\
			-\frac{(\eta(-1)+\epsilon_\alpha)(\eta(f^*(\alpha))+\eta(-1))}{4}p^*p^{\frac{n-                     3}{2}}+\\\frac{(p-1)}{2p}\chi_\alpha(B_+(f^*)),&\text{if}\ f^*(\alpha)\ne0;
			\end{cases} $
			
			$\chi_\alpha(\bigcup\limits_{i\in NSQ} c_i(f))=
			\begin{cases}
			\frac{(p-1)}{2}(p^{-1}\chi_\alpha(B_+(f^*))-\frac{\eta(-1)+\epsilon_\alpha}{2}p^{\frac{n-1}{2}}), &\hspace{-0.2cm}\text{if}\ f^*(\alpha)=0,\\
			\frac{(\eta(-1)+\epsilon_\alpha)(\eta(-1)-\eta (f^*(\alpha)))}{4}p^*p^{\frac{n-                     3}{2}}+\\\frac{(p-1)}{2p}\chi_\alpha(B_+(f^*)), &\hspace{-0.2cm}\text{if}\ f^*(\alpha)\ne0;
			\end{cases} $
			
			$\chi_\alpha(\bigcup\limits_{i\in SQ} d_i(f))=
			\begin{cases}
			\frac{(p-1)}{2}(p^{-1}\chi_\alpha(B_-(f^*))+\frac{\epsilon_\alpha-\eta(-1)}{2}p^{\frac{n-1}{2}}),&\hspace{0cm}\text{if}\ f^*(\alpha)=0,\\
			-\frac{(\epsilon_\alpha-\eta(-1))(\eta(f^*(\alpha))+\eta(-1))}{4}p^*p^{\frac{n-                     3}{2}}+\\\frac{(p-1)}{2p}\chi_\alpha(B_-(f^*)),&\hspace{0cm}\text{if}\ f^*(\alpha)\ne0;
			\end{cases} $
			$\chi_\alpha(\bigcup\limits_{i\in NSQ} d_i(f))=
			\begin{cases}
			\frac{(p-1)}{2}(p^{-1}\chi_\alpha(B_-(f^*))-\frac{\epsilon_\alpha-\eta(-1)}{2}p^{\frac{n-1}{2}}), &\hspace{-0.2cm}\text{if}\ f^*(\alpha)=0,\\
			\frac{(\epsilon_\alpha-\eta(-1))(\eta(-1)-\eta (f^*(\alpha)))}{4}p^*p^{\frac{n-                     3}{2}}+\\\frac{(p-1)}{2p}\chi_\alpha(B_-(f^*)), &\hspace{-0.2cm}\text{if}\ f^*(\alpha)\ne0.
			\end{cases} $
		\end{itemize}
	\end{lemma}
	\begin{proof} Note that for any $\alpha\in\mathbb{F}_p^n$, we have
		\begin{align*}
		\chi_\alpha(\bigcup\limits_{i\in SQ} c_i(f))&=\sum\limits_{i\in SQ}\chi_\alpha(c_i(f)),\\
		\chi_\alpha(\bigcup\limits_{i\in SQ} d_i(f))&=\sum\limits_{i\in SQ}\chi_\alpha(d_i(f)),\\
		\chi_\alpha(\bigcup\limits_{i\in NSQ} c_i(f))&=\sum\limits_{i\in NSQ}\chi_\alpha(c_i(f)),\\
		\chi_\alpha(\bigcup\limits_{i\in NSQ} d_i(f))&=\sum\limits_{i\in NSQ}\chi_\alpha(d_i(f)).
		\end{align*}
		
		Then the desired results directly follow from Lemmas  \ref{le 9} and \ref{le 10}.\qed 
	\end{proof}
	
	The following corollaries follow from Lemma \ref{le 11}.
	\begin{corollary}\label{co 2}
		Let $f(x):\mathbb{F}_{p}^n\longrightarrow\mathbb{F}_p$ be a dual-bent function belonging to $\mathcal{DBF}$. Then  we have the following.
		\begin{itemize}
			\item [$\bullet$] When $n$ is even, then, for any $\alpha\in\mathbb{F}_p^n$,\\
			$\chi_\alpha(\bigcup\limits_{i\in\mathbb{F}_p^*}c_i(f))=\begin{cases}
			(p-1)(p^{-1}\chi_\alpha(B_+(f^*)) -\frac{\epsilon_\alpha+1}{2}p^{\frac{n}{2}-1}),&\hspace{0.3cm}\text{if}\ f^*(\alpha)=0,\\
			\frac{\epsilon_\alpha+1}{2}p^{\frac{n}{2}-1}+\frac{p-1}{p}\chi_\alpha(B_+(f^*)), &\hspace{0.3cm}\text{if}\ f^*(\alpha)\ne0;
			\end{cases}$
			$\chi_\alpha(\bigcup\limits_{i\in\mathbb{F}_p^*}d_i(f))=\begin{cases}
			(p-1)(p^{-1}\chi_\alpha(B_-(f^*)) -\frac{\epsilon_\alpha-1}{2}p^{\frac{n}{2}-1}),&\hspace{0.3cm}\text{if}\ f^*(\alpha)=0,\\
			\frac{\epsilon_\alpha-1}{2}p^{\frac{n}{2}-1}+\frac{p-1}{p}\chi_\alpha(B_-(f^*)), &\hspace{0.3cm}\text{if}\ f^*(\alpha)\ne0.
			\end{cases}$
			\item [$\bullet$]	When $n$ is odd, then, for any $\alpha\in\mathbb{F}_p^n$,\\
			$\chi_\alpha(\bigcup\limits_{i\in\mathbb{F}_p^*}c_i(f))=\begin{cases}
			\frac{p-1}{p}\chi_\alpha(B_+(f^*)) ,&\hspace{1.7cm}\text{if}\ f^*(\alpha)=0,\\
			\frac{-(\epsilon_\alpha+\eta(-1))}{2}\eta(f^*(\alpha))p^*p^{\frac{n-3}{2}}+\\\frac{p-1}{p}\chi_\alpha(B_+(f^*)), &\hspace{1.7cm}\text{if}\ f^*(\alpha)\ne0;
			\end{cases}$
			$\chi_\alpha(\bigcup\limits_{i\in\mathbb{F}_p^*}d_i(f))=\begin{cases}
			\frac{p-1}{p}\chi_\alpha(B_-(f^*)) ,&\hspace{1.7cm}\text{if}\ f^*(\alpha)=0,\\
			\frac{-(\epsilon_\alpha-\eta(-1))}{2}\eta(f^*(\alpha))p^*p^{\frac{n-3}{2}}+\\\frac{p-1}{p}\chi_\alpha(B_-(f^*)), &\hspace{1.7cm}\text{if}\ f^*(\alpha)\ne0.
			\end{cases}$
		\end{itemize}
	\end{corollary}
	\begin{corollary}\label{co 3}
		Let $f(x):\mathbb{F}_{p}^n\longrightarrow\mathbb{F}_p$ be a dual-bent function belonging to $\mathcal{DBF}$. Then  we have the following.
		\begin{itemize}
			\item [$\bullet$] When $n$ is even, then, for any $\alpha\in\mathbb{F}_p^n$,\\ $\chi_\alpha(\bigcup\limits_{i\in SQ} N_i(f))=
			\begin{cases}
			\frac{(p-1)}{2}(p^{n-1}\delta_0(\alpha)-\epsilon_\alpha p^{\frac{n}{2}-1}),&\hspace{1.5cm}\text{if}\ f^*(\alpha)=0,\\
			\epsilon_\alpha\frac{(\eta(f^*(\alpha))p+1)}{2}p^{\frac{n}{2}-1}+\\\frac{(p-1)}{2}p^{n-1}\delta_0(\alpha),&\hspace{1.5cm}\text{if}\ f^*(\alpha)\ne0;
			\end{cases} $\\
			$\chi_\alpha(\bigcup\limits_{i\in NSQ} N_i(f))=
			\begin{cases}
			\frac{(p-1)}{2}(p^{n-1}\delta_0(\alpha)-\epsilon_\alpha p^{\frac{n}{2}-1}),&\hspace{1.3cm}\text{if}\ f^*(\alpha)=0,\\
			\epsilon_\alpha\frac{(-\eta(f^*(\alpha))p+1)}{2}p^{\frac{n}{2}-1}+\\\frac{(p-1)}{2}p^{n-1}\delta_0(\alpha),&\hspace{1.3cm}\text{if}\ f^*(\alpha)\ne0.
			\end{cases} $\\
			
			\item [$\bullet$] When $n$ is odd, then, for any $\alpha\in\mathbb{F}_p^n$,\\ $\chi_\alpha(\bigcup\limits_{i\in SQ} N_i(f))=
			\begin{cases}
			\frac{(p-1)}{2}(p^{n-1}\delta_0(\alpha)+\epsilon_\alpha p^{\frac{n-1}{2}}),&\hspace{1.6cm}\text{if}\ f^*(\alpha)=0,\\
			-\epsilon_\alpha\frac{(\eta(f^*(\alpha))+\eta(-1))}{2}p^*p^{\frac{n-                     3}{2}}+\\\frac{(p-1)}{2}	p^{n-1}\delta_0(\alpha),&\hspace{1.6cm}\text{if}\ f^*(\alpha)\ne0;
			\end{cases} $\\
			$\chi_\alpha(\bigcup\limits_{i\in NSQ} N_i(f))=
			\begin{cases}
			\frac{(p-1)}{2}(p^{n-1}\delta_0(\alpha)-\epsilon_\alpha p^{\frac{n-1}{2}}), &\hspace{1.4cm}\text{if}\ f^*(\alpha)=0,\\
			\epsilon_\alpha\frac{(\eta(-1)-\eta (f^*(\alpha)))}{2}p^*p^{\frac{n-                     3}{2}}+\\\frac{(p-1)}{2}p^{n-1}\delta_0(\alpha), &\hspace{1.4cm}\text{if}\ f^*(\alpha)\ne0.
			\end{cases} $	
		\end{itemize}
		
		Here $\delta_0(0)=1$ and $\delta_0(\alpha)=0$ for any $\alpha\in \mathbb{F}_p^n\setminus\{0\}$.
	\end{corollary}	
	\begin{corollary}\label{co 4}
		Let $f(x):\mathbb{F}_{p}^n\longrightarrow\mathbb{F}_p$ be a dual-bent function belonging to $\mathcal{DBF}$. Then  we have the following.
		\begin{itemize}
			\item [$\bullet$] When $n$ is even, then, for any $\alpha\in\mathbb{F}_p^n$,\\ $\chi_\alpha(\bigcup\limits_{i\in \mathbb{F}_p^*} N_i(f))=
			\begin{cases}
			{(p-1)}(p^{n-1}\delta_0(\alpha)-\epsilon_\alpha p^{\frac{n}{2}-1}),&\hspace{1.4cm}\text{if}\ f^*(\alpha)=0,\\
			\epsilon_\alpha p^{\frac{n}{2}-1}+(p-1)	p^{n-1}\delta_0(\alpha),&\hspace{1.4cm}\text{if}\ f^*(\alpha)\ne0.
			\end{cases} $\\
			\item [$\bullet$] When $n$ is odd, then, for any $\alpha\in\mathbb{F}_p^n$,\\ $\chi_\alpha(\bigcup\limits_{i\in \mathbb{F}_p^*} N_i(f))=
			\begin{cases}
			{(p-1)}p^{n-1}\delta_0(\alpha),&\hspace{-0.1cm}\text{if}\ f^*(\alpha)=0,\\
			-\epsilon_\alpha\eta(f^*(\alpha))p^*p^{\frac{n-                     3}{2}}+{(p-1)}	p^{n-1}\delta_0(\alpha),&\hspace{-0.1cm}\text{if}\ f^*(\alpha)\ne0.
			\end{cases} $
		\end{itemize}
		
		Here $\delta_0(0)=1$ and $\delta_0(\alpha)=0$ for any $\alpha\in \mathbb{F}_p^n\setminus\{0\}$.
	\end{corollary}	
	
	\section{The constructions of association schemes}
	\hspace{0.5cm}In this section, we construct some infinite families of symmetric association schemes by using non-weakly regular bent function $f(x)$, which belongs to $\mathcal{DBF}$ and satisfies the following condition.\\
	\textbf{Condition 1:} When $p^n\equiv 1\ (\mathrm{mod}\ 4)$, $B_+(f)$ and $B_+(f^*)$ (or $B_-(f)$ and $B_-(f^*)$) are both non-degenerate $r$-dimensional subspaces over $\mathbb{F}_p$; when $p^n\equiv 3\ (\mathrm{mod}\ 4)$, $B_-(f)$ and $B_+(f^*)$ (or $B_+(f)$ and $B_-(f^*)$) are both non-degenerate $r$-dimensional subspaces over $\mathbb{F}_p$.
	
	In fact, by the construction method introduced in \cite[Remark 4]{Ozbudak}, we can obtain infinitely many non-weakly regular bent functions belonging to $\mathcal{DBF}$ and  satisfying \textbf{Condition 1}.\\
	Define 
	\begin{align*}
	\mathcal{P}&:=\{\{0\},c_0(f)\setminus\{0\},c_1(f), \dots,c_{p-1}(f),d_0(f)\setminus\{0\},d_1(f),\dots,d_{p-1}(f)\},\\
	\mathcal{D}&:=\{\{0\}, c_0(f)\setminus\{0\}, \bigcup\limits_{i\in SQ}c_i(f),\bigcup\limits_{i\in NSQ}c_i(f), d_0(f)\setminus\{0\}, \bigcup\limits_{i\in SQ}d_i(f),\\&\qquad\bigcup\limits_{i\in NSQ}d_i(f)\},\\
	\mathcal{T}&:=\{\{0\}, c_0(f)\setminus\{0\}, \bigcup\limits_{i\in \mathbb{F}_p^*}c_i(f), d_0(f)\setminus\{0\}, \bigcup\limits_{i\in \mathbb{F}_p^*}d_i(f)\}.
	\end{align*} 
	Let $P$, $D$  and $T$ be the sets of all nonempty sets in $\mathcal{P}$, $\mathcal{D}$ and $\mathcal{T}$, respectively. 
	\hspace{0.5cm}  We begin this section with the following lemma.
	\begin{lemma}\label{le 12}
		Let $f(x):\mathbb{F}_{p}^n\longrightarrow\mathbb{F}_p$ be a dual-bent function belonging to $\mathcal{DBF}$. For any two distinct elements $\alpha, \alpha'\in \mathbb{F}_p^n$, if $\chi_\alpha(N_i(f))=\chi_{\alpha'}(N_i(f))$ for any $i\in\mathbb{F}_p$, then $\epsilon_\alpha=\epsilon_{\alpha'}$ and $f^*(\alpha)=f^*(\alpha')$.
	\end{lemma}
	\begin{proof}
		Since $f(x)\in\mathcal{DBF}$, then by Lemma \ref{le 1}, we have that $\epsilon_\beta=\epsilon_{-\beta}$ and $f^{*}(-\beta)=f^*(\beta)$ for any $\beta\in\mathbb{F}_p^n$. Note that 
		$\mathcal{W}_{f}(-\alpha)=\sum\limits_{i\in\mathbb{F}_p}\chi_\alpha(N_i(f))\xi_p^i=\sum\limits_{i\in\mathbb{F}_p}\chi_{\alpha'}(N_i(f))\xi_p^i=\mathcal{W}_{f}(-\alpha')$.
		Thus, if $n$ is even, then  $\epsilon_\alpha p^{\frac{n}{2}}\xi_p^{f^*(\alpha)}=\mathcal{W}_{f}(-\alpha)=\mathcal{W}_{f}(-\alpha')=\epsilon_{\alpha'}p^{\frac{n}{2}}\xi_p^{f^*(\alpha')}$, so $\epsilon_\alpha=\epsilon_{\alpha'}$ and $f^*(\alpha)=f^*(\alpha')$. If $n$ is odd, then $\epsilon_{\alpha}\sqrt{p^*}p^{\frac{n-1}{2}}\xi_p^{f^*(\alpha)}$ $=\mathcal{W}_{f}(-\alpha)=\mathcal{W}_{f}(-\alpha')=\epsilon_{\alpha'}\sqrt{p^*}p^{\frac{n-1}{2}}\xi_p^{f^*(\alpha')}$, so $\epsilon_\alpha=\epsilon_{\alpha'}$ and $f^*(\alpha)=f^*(\alpha')$. The proof is now completed.\qed
	\end{proof}
	
	In the following theorems, when $n$ is even, we present the constructions of association schemes from non-weakly regular bent functions.
	\begin{theorem}\label{th 1}
		Let $n$ be an even integer with $n\ge 4$,  $f(x):\mathbb{F}_{p}^n\longrightarrow\mathbb{F}_p$ be a non-weakly regular bent function belonging to $\mathcal{DBF}$, $B_+(f)$ and $B_+(f^*)$ both be non-degenerate $r$-dimensional subspaces over $\mathbb{F}_p$. Then 
		when $\frac{n}{2}+1\le r \le n-1$, $U_1=\{U_{1,i}\}_{i=0}^{2p+1}=\{\{0\}, B_+(f)^{\bot}\setminus\{0\}, c_0(f)\setminus B_+(f)^{\bot}, c_1(f), \dots, c_{p-1}(f), d_0(f), d_1(f),\\\dots, d_{p-1}(f)\}$ is a partition of $\mathbb{F}_p^n$ and induces a $(2p+1)$-class symmetric association scheme.
	\end{theorem}
	\begin{proof}
		By Lemma \ref{le 1}, we know that $B_+(f)^{\bot}\subseteq c_0(f)$,  $B_+(f^*)^{\bot}\subseteq c_0(f^*)$ and $0\in c_0(f)\cap c_0(f^*)$.  Since $r\ge \frac{n}{2}+1$, then by Lemma \ref{le 4} and Remark \ref{re 1}, we have that $|c_0(f)|=|c_0(f^*)|=p^{r-1}+(p-1)p^{\frac{n}{2}-1}>p^{n-r}$, $|c_i(f)|=|c_i(f^*)|=p^{r-1}-p^{\frac{n}{2}-1}>0$ for any $i\in\mathbb{F}_p^*$, and $|d_j(f)|=|d_j(f^*)|=p^{n-1}-p^{r-1}>0$ for any $j\in\mathbb{F}_p$. Thus, $U_1$ is a partition of $\mathbb{F}_p^n$. For any $\alpha,\alpha' \in \mathbb{F}_p^n$, if $\alpha\sim_{\widehat{U_1}}\alpha'$, where $\widehat{U_1}$ is the dual partition of $U_1$, then we have that $\chi_\alpha(N_i(f))=\chi_{\alpha'}(N_i(f))$ for any $i\in\mathbb{F}_p$. Thus, by Lemma \ref{le 12}, we know that  $\epsilon_\alpha=\epsilon_{\alpha'}$ and $f^*(\alpha)=f^*(\alpha')$. 
		According to Lemmas \ref{le 8} and \ref{le 10}, we know that the dual partition of $U_1$ is $\{\{0\}, B_+(f^*)^{\bot}\setminus\{0\}, c_0(f^*)\setminus B_+(f^*)^{\bot}, c_1(f^*), \dots, c_{p-1}(f^*), d_0(f^*), d_1(f^*), \dots, d_{p-1}(f^*)\}$. By Lemma \ref{le 6}, we have that $U_1$ can induce a $(2p+1)$-class symmetric association scheme.\qed
	\end{proof}
	\begin{remark}
		When $n$ is even, until now, we have not constructed a non-weakly regular bent function such that $B_+(f^*)$ or $B_+(f)$ is a non-degenerate $\frac{n}{2}$-dimensional subspace over $\mathbb{F}_p$. It is a question whether such a function exists, so in Theorem \ref{th 1}, we let $r\ge \frac{n}{2}+1$.
	\end{remark}
	\begin{theorem}\label{th 2}
		Let $n$ be an even integer with $n\ge 4$,  $f(x):\mathbb{F}_{p}^n\longrightarrow\mathbb{F}_p$ be a non-weakly regular bent function belonging to $\mathcal{DBF}$, $B_-(f)$ and $B_-(f^*)$ both be non-degenerate $r$-dimensional subspaces over $\mathbb{F}_p$. Then when $r=\frac{n}{2}+1$,  $U_2=\{U_{2,i}\}_{i=0}^{2p}=\{\{0\},c_0(f), c_1(f),$ $ \dots,$ $ c_{p-1}(f),d_0(f)\setminus\{0\},d_1(f), \dots, d_{p-1}(f)\}$ is a partition of $\mathbb{F}_p^n$ and induces a $2p$-class symmetric association scheme; when 
		$\frac{n}{2}+1<r\le n-1$, $U_3=\{U_{3,i}\}_{i=0}^{2p+1}=\{\{0\}, c_0(f), c_1(f), \dots,$ $c_{p-1}(f), B_-(f)^{\bot}\setminus\{0\}, d_0(f)\setminus B_-(f)^{\bot}, d_1$ $(f), \dots, d_{p-1}(f)\}$ is a partition of $\mathbb{F}_p^n$ and induces a $(2p+1)$-class symmetric association scheme.
		
	\end{theorem}
	\begin{proof}
		
		By Lemma \ref{le 1}, we know that $B_-(f)^{\bot}\subseteq d_0(f)$,  $B_-(f^*)^{\bot}\subseteq d_0(f^*)$ and $0\in d_0(f)\cap d_0(f^*)$.  By Lemma \ref{le 4} and Remark \ref{re 1}, we know that $|c_i(f)|=|c_i(f^*)|=p^{n-1}-p^{r-1}>0$ for any $i\in\mathbb{F}_p$ and  $|d_0(f)|=|d_0(f^*)|=p^{r-1}-(p-1)p^{\frac{n}{2}-1}\ge p^{n-r}$, $|d_i(f)|=|d_i(f^*)|=p^{r-1}+p^{\frac{n}{2}-1}>0$ for any $i\in\mathbb{F}_p^*$. 
		
		Thus, when $r=\frac{n}{2}+1$, $B_-(f^*)^{\bot}=d_0(f^*)$ and $U_2$ is a partition of $\mathbb{F}_p^n$; when $r>\frac{n}{2}+1$, $B_-(f)^{\bot}\subsetneqq d_0(f)$, $B_-(f^*)^{\bot}\subsetneqq d_0(f^*)$ and $U_3$ is a partition of $\mathbb{F}_p^n$. For any $\alpha,\alpha'\in\mathbb{F}_p^n$, if $\alpha\sim_{\widehat{U_2}}\alpha'$ (respectively $\alpha\sim_{\widehat{U_3}}\alpha'$), where $\widehat{U_2}$, $\widehat{U_3}$ are the dual partitions of $U_2$ and $U_3$, respectively, then $\chi_\alpha(N_i(f))=\chi_{\alpha'}(N_i(f))$ for any $i\in \mathbb{F}_p$. Thus, by Lemma \ref{le 12}, we know that $\epsilon_\alpha=\epsilon_{\alpha'}$ and $f^*(\alpha)=f^*(\alpha')$.  
		According to Lemmas \ref{le 8} and \ref{le 10}, we know that when $r=\frac{n}{2}+1$, the dual partition of $U_2$ is $\{\{0\},c_0(f^*), c_1(f^*), \dots, c_{p-1}(f^*),d_0(f^*)\setminus\{0\},d_1(f^*), \dots, d_{p-1}(f^*)\}$, and  when $\frac{n}{2}+1<r\le n-1$, the dual partition of $U_3$ is $\{\{0\},c_0(f^*), c_1(f^*), \dots, c_{p-1}(f^*),B_-(f^*)^{\bot}\setminus\{0\}, d_0(f^*)\setminus B_-(f^*)^{\bot} , d_1(f^*),$ $ \dots, d_{p-1}(f^*)\}$. By Lemma \ref{le 6}, we have that when $r=\frac{n}{2}+1$, $U_2$ can induce a $2p$-class symmetric association scheme; when $\frac{n}{2}+1<r\le n-1$, $U_3$ can induce a $(2p+1)$-class symmetric association scheme. \qed
	\end{proof}
	\begin{remark}
		$(1)$ According  to Remark \ref{re 4}, Lemmas \ref{le 8}, \ref{le 10}, and the proofs of Theorems \ref{th 1}, \ref{th 2}, we give the first and second eigenmatrices of the association schemes induced by $U_i$, $1\le i\le 3$, in Tables 1-6 of Appendix.\\
		$(2)$ The intersection numbers and the Krein parameters of the association schemes induced by $U_i$, $1\le i\le 3$, can be calculated by Equations (8), (9), and Tables 1-6. In Appendix, we give the intersection numbers and the Krein parameters of the association schemes induced by $U_i$, $1\le i\le 3$.
	\end{remark}
	
	Considering the fusions of the association schemes constructed in Theorems \ref{th 1} and \ref {th 2}, we can also obtain some association schemes in the following propositions.
	
	\begin{proposition}\label{po 1}
		Let $n$ be an even integer with $n\ge 4$,  $f(x):\mathbb{F}_{p}^n\longrightarrow\mathbb{F}_p$ be a non-weakly regular bent function belonging to $\mathcal{DBF}$, $B_+(f)$ and $B_+(f^*)$ both be non-degenerate $r$-dimensional subspaces over $\mathbb{F}_p$. Then 
		when $\frac{n}{2}+1 \le r \le n-1 $, $U_4=\{U_{4,i}\}_{i=0}^7=\{\{0\}, B_+(f)^{\bot}\setminus\{0\}, c_0(f)\setminus B_+(f)^{\bot}, \bigcup\limits_{i\in SQ}c_i(f), \bigcup\limits_{i\in NSQ}c_i(f), \\d_0(f),$ $ \bigcup\limits_{i\in SQ}d_i(f), \bigcup\limits_{i\in NSQ}d_i(f)\}$ is a partition of $\mathbb{F}_p^n$ and induces a $7$-class symmetric association scheme, and
		$U_5=\{U_{5,i}\}_{i=0}^5=\{\{0\},$ $ B_+(f)^{\bot}\setminus\{0\}, c_0(f)\setminus B_+(f)^{\bot}, \bigcup\limits_{i\in \mathbb{F}_p^*}c_i(f), $ $d_0(f), \bigcup\limits_{i\in \mathbb{F}_p^*}d_i(f)\}$ is a partition of $\mathbb{F}_p^n$ and induces a $5$-class symmetric association scheme.
	\end{proposition}
	\begin{proof}
		By Lemmas \ref{le 8}, \ref{le 10}, \ref{le 11}, Corollary \ref{co 2} and the proof of Theorem \ref{th 1}, we know that $U_4$ and $U_5$ are both partitions of $\mathbb{F}_p^n$, and the dual partitions of $U_4$ and $U_5$ are $\{\{0\}, B_+(f^*)^{\bot}\setminus\{0\}, c_0(f^*)\setminus B_+(f^*)^{\bot}, \bigcup\limits_{i\in SQ}c_i(f^*), \bigcup\limits_{i\in NSQ}c_i(f^*), d_0(f^*), $ $\bigcup\limits_{i\in SQ}d_i(f^*), \bigcup\limits_{i\in NSQ}d_i(f^*)\}$ and $\{\{0\}, B_+(f^*)^{\bot}\setminus\{0\}, c_0(f^*)\setminus B_+(f^*)^{\bot}, \bigcup\limits_{i\in \mathbb{F}_p^*}c_i$ $(f^*), d_0(f^*), \bigcup\limits_{i\in \mathbb{F}_p^*}d_i(f^*)\}$, respectively. Thus, according to Lemma \ref{le 6}, we have that $U_4$ and $U_5$ can induce $7$-class and $5$-class symmetric association schemes, respectively.	
		\qed
	\end{proof}
	
	\begin{proposition}\label{po 2}
		Let $n$ be an even integer with $n\ge 4$,  $f(x):\mathbb{F}_{p}^n\longrightarrow\mathbb{F}_p$ be a non-weakly regular bent function belonging to $\mathcal{DBF}$, $B_-(f)$ and $B_-(f^*)$ both be non-degenerate $r$-dimensional subspaces over $\mathbb{F}_p$. Then when $r= \frac{n}{2}+1$, $U_6=\{U_{6,i}\}_{i=0}^6=\{\{0\}, c_0(f), \bigcup\limits_{i\in SQ}c_i(f),\bigcup\limits_{i\in NSQ}c_i(f), d_0(f)\setminus\{0\}, \bigcup\limits_{i\in SQ}d_i(f), \bigcup\limits_{i\in NSQ}d_i\\(f)\}$ is a partition of $\mathbb{F}_p^n$ and induces a $6$-class symmetric association scheme, and
		$U_7=\{U_{7,i}\}_{i=0}^4=\{\{0\}, c_0(f), \bigcup\limits_{i\in \mathbb{F}_p^*}c_i(f),d_0(f)\setminus\{0\}, \bigcup\limits_{i\in \mathbb{F}_p^*}d_i(f)\}$ is a partition of $\mathbb{F}_p^n$ and induces a $4$-class symmetric association scheme; when  $\frac{n}{2}+1<r\le n-1$, $U_8=\{U_{8,i}\}_{i=0}^7=\{\{0\}, c_0(f), \bigcup\limits_{i\in SQ}c_i(f), \bigcup\limits_{i\in NSQ}c_i(f),B_-(f)^{\bot}\setminus\{0\}, d_0(f)\setminus B_-(f)^{\bot},\bigcup\limits_{i\in SQ}d_i(f), $ $\bigcup\limits_{i\in NSQ}d_i(f)\}$ is a partition of $\mathbb{F}_p^n$ and induces a $7$-class symmetric association scheme, and
		$U_9=\{U_{9,i}\}_{i=0}^5=\{\{0\}, c_0(f), \bigcup\limits_{i\in \mathbb{F}_p^*}c_i\\(f), B_-(f)^{\bot}\setminus\{0\}, d_0(f)\setminus B_-(f)^{\bot}, \bigcup\limits_{i\in \mathbb{F}_p^*}d_i$ $(f)\}$ is a partition of $\mathbb{F}_p^n$ and induces a $5$-class symmetric association scheme.
	\end{proposition}
	\begin{proof}
		By Lemmas \ref{le 8}, \ref{le 10}, \ref{le 11}, Corollary \ref{co 2} and the proof of Theorem \ref{th 2}, when $r=\frac{n}{2}+1$, we know that $U_6$ and $U_7$ are both partitions of $\mathbb{F}_p^n$, and the dual partitions of $U_6$ and $U_7$ are $\{\{0\}, c_0(f^*), \bigcup\limits_{i\in SQ}c_i(f^*), \bigcup\limits_{i\in NSQ}c_i(f^*), d_0(f^*)\setminus\{0\}, \bigcup\limits_{i\in SQ}d_i(f^*), \bigcup\limits_{i\in NSQ}d_i(f^*)\}$ and $\{\{0\}, c_0(f^*), \bigcup\limits_{i\in \mathbb{F}_p^*}c_i(f^*), d_0(f^*)\setminus \{0\}, \bigcup\limits_{i\in \mathbb{F}_p^*}d_i$ $(f^*)\}$, respectively; when $\frac{n}{2}+1<r\le n-1$, we know that $U_8$ and $U_9$ are both partitions of $\mathbb{F}_p^n$, and the dual partitions of $U_8$ and $U_9$ are $\{\{0\}, c_0(f^*), \bigcup\limits_{i\in SQ}c_i$ $(f^*), \bigcup\limits_{i\in NSQ}c_i(f^*),  B_-(f^*)^{\bot}\setminus\{0\}, d_0(f^*)\setminus B_-(f^*)^{\bot}, \bigcup\limits_{i\in SQ}d_i(f^*), \bigcup\limits_{i\in NSQ}d_i(f^*)\}$ and $\{\{0\}, c_0(f^*), \bigcup\limits_{i\in \mathbb{F}_p^*}c_i(f^*), B_-(f^*)^{\bot}\setminus\{0\}, d_0(f^*)\setminus B_-(f^*)^{\bot}, \bigcup\limits_{i\in \mathbb{F}_p^*}d_i(f^*)\}$, respectively. Thus, according to Lemma \ref{le 6}, we have that when $r=\frac{n}{2}+1$, $U_6$ and $U_7$ can induce $6$-class and $4$-class symmetric association schemes, respectively; when $\frac{n}{2}+1<r\le n-1$, $U_8$ and $U_9$ can induce $7$-class and $5$-class symmetric association schemes, respectively. \qed
	\end{proof}
	\begin{remark}\label{re 7}
		$(1)$ Note that in Theorem \ref{th 2} and Proposition \ref{po 2}, when $B_-(f^*)$ is a  non-degenerate $(\frac{n}{2}+1)$-dimensional subspace over $\mathbb{F}_p$, by the proofs of Theorem \ref{th 2} and Proposition \ref{po 2}, we know that the  condition that $B_-(f)$ is a  non-degenerate $(\frac{n}{2}+1)$-dimensional subspace is not necessary. Thus, when $p=3$, the $6$-class association scheme constructed in \cite{Ozbudak2} can also be obtained by Theorem \ref{th 2}.\\
		$(2)$ Since $f^{**}(x)=f(x)$ for any $x\in\mathbb{F}_p^n$, according to Remark \ref{re 4}, Lemmas \ref{le 8}, \ref{le 10}, \ref{le 11}, Corollary \ref{co 2}, and the proofs of Propositions \ref{po 1}, \ref{po 2},  we have that the first and second eigenmatrices of the association scheme induced by $U_4$ (respectively $U_5$, $U_6$, $U_7$, $U_8$, $U_9$) are the same. The first (second) eigenmatrices of the association schemes induced by $U_i$, $4\le i\le 9$, are given in Tables 7-12 of Appendix.\\
		$(3)$ The intersection numbers and the Krein parameters of the association schemes induced by $U_i$, $4\le i\le 9$, can be calculated by Equations (8), (9), and Tables 7-12. We easily have that the intersection number $p_{u,v}^w$ and the Krein parameter $q_{u,v}^w$ of the association scheme induced by $U_4$ (respectively $U_5$, $U_6$, $U_7$, $U_8$, $U_9$)  are the same. In Appendix, we give the intersection numbers and the Krein parameters of the association schemes induced by $U_i$, $4\le i\le 9$.
	\end{remark}
	
	In the following proposition, when $n$ is even, we give the sufficient and necessary condition for the partitions $P$, $D$ and $T$ to induce symmetric association schemes.
	\begin{proposition} \label{po 3}
		Let $n$ be an even integer with $n\ge 4$,  $f(x):\mathbb{F}_{p}^n\longrightarrow\mathbb{F}_p$ be a non-weakly regular bent function belonging to $\mathcal{DBF}$. If $B_-(f^*)$ is a non-degenerate $r$-dimensional subspaces over $\mathbb{F}_p$ with $\frac{n}{2}+1\le r\le n-1$, then the partitions $P$, $D$ and $T$ induce symmetric association schemes if and only if $r=\frac{n}{2}+1$.
	\end{proposition}
	\begin{proof}
		We only prove the necessity and the sufficiency can directly follow from Theorem \ref{th 2}, Proposition \ref{po 2} and Remark \ref{re 7}. 
		
		By the proof of Theorem \ref{th 2}, we have that $P=\{\{0\},c_0(f),c_1(f),\dots,c_{p-1}(f),\\d_0(f)\setminus\{0\},d_1(f),\dots,d_{p-1}(f)\}$, $D=\{\{0\},c_0(f),\bigcup\limits_{i\in SQ}c_i(f),\bigcup\limits_{i\in NSQ}c_i(f),d_0(f)\setminus\{0\},\bigcup\limits_{i\in SQ}d_i(f), \bigcup\limits_{i\in NSQ}d_i(f)\}$,  $T=\{\{0\},c_0(f),\bigcup\limits_{i\in \mathbb{F}_p^*}c_i(f),d_0(f)\setminus\{0\},\bigcup\limits_{i\in \mathbb{F}_p^*}d_i(f)\}$, and for any $\alpha, \alpha'\in\mathbb{F}_p$, if $\alpha\sim_{\widehat{P}}\alpha'$, where $\widehat{P}$ is the dual partition of $P$, then $\epsilon_\alpha=\epsilon_{\alpha'}$ and $f^*(\alpha)=f^*(\alpha')$. By Lemma \ref{le 1}, we know that  $B_-(f^*)^{\bot}\subseteq d_0(f^*)$. Next, we prove that $B_-(f^*)^{\bot}=d_0(f^*)$. On the contrary, 
		assume that $B_-(f^*)^{\bot}\subsetneqq d_0(f^*)$, then by 
		Lemmas \ref{le 8}, \ref{le 10}, \ref{le 11} and  Corollary \ref{co 2}, we know that the dual partitions of $P$, $T$ and $D$ are $ \{\{0\}, c_0(f^*), c_1(f^*),\dots, c_{p-1}(f^*), B_-(f^*)^{\bot}\setminus \{0\}, d_0(f^*)\setminus B_-(f^*)^{\bot}, d_1(f^*), \dots, d_{p-1}(f^*)\}$, $ \{\{0\}, c_0(f^*), \bigcup\limits_{i\in SQ}c_i(f^*), \bigcup\limits_{i\in NSQ}c_i$ $(f^*), B_-(f^*)^{\bot}\setminus\{0\}, d_0(f^*)\setminus B_-(f^*)^{\bot}, \bigcup\limits_{i\in SQ}d_i(f^*),$ $\bigcup\limits_{i\in NSQ}d_i(f^*)\}$ and $ \{\{0\}, c_0$ $(f^*), \bigcup\limits_{i\in \mathbb{F}_p^*}c_i(f^*), B_-(f^*)^{\bot}\setminus\{0\}, d_0(f^*)\setminus B_-(f^*)^{\bot}, \bigcup\limits_{i\in \mathbb{F}_p^*}d_i(f^*)\}$, respectively. Thus, according to Lemma \ref{le 6}, we have that $P$, $D$ and $T$ can not induce symmetric association schemes, so $B_-(f^*)^{\bot}=d_0(f^*)$, which implies $r=\frac{n}{2}+1.$	\qed
	\end{proof}
	
	In the following, when $n$ is even, we give some examples.

	\begin{example}
		Let $f(x):\mathbb{F}_5^4\longrightarrow \mathbb{F}_5$, $f(x_1,x_2,x_3,x_4)=x_1^2+x_2^2x_4^4+x_2^2+x_3x_4$. Then
		\begin{itemize}
			\item [$\bullet$] $f(x_1,x_2,x_3,x_4)$ is a non-weakly regular bent function of Type $(+)$.
			\item[$\bullet$] $f^*(x_1,x_2,x_3,x_4)=x_1^2+2x_2^2x_3^4+x_2^2+4x_3x_4$ is a non-weakly regular bent function of Type $(+)$.
			\item[$\bullet$] $B_+(f)=\mathbb{F}_5^2\times\{0\}\times\mathbb{F}_5$ and $B_+(f^*)=\mathbb{F}_5^2\times\mathbb{F}_5\times\{0\}$ are both non-degenerate $3$-dimensional subspaces over $\mathbb{F}_5$.
			\item[$\bullet$] $\{\{0\}, B_+(f)^{\bot}\setminus\{0\}, c_0(f)\setminus B_+(f)^{\bot}, c_1(f), c_2(f),c_3(f),c_4(f),d_0(f),d_1(f),d_2$ $(f),d_3(f),d_4(f)\}$, $\{\{0\}, B_+(f)^{\bot}\setminus\{0\}, c_0(f)\setminus B_+(f)^{\bot}, c_1(f)\cup c_4(f), c_2(f)\cup c_3(f),d_0(f),d_1(f)\cup d_4(f),d_2(f)\cup d_3(f)\}$ and $\{\{0\}, B_+(f)^{\bot}\setminus\{0\}, c_0(f)\setminus B_+(f)^{\bot}, \bigcup\limits_{i\in\mathbb{F}_5^*}c_i(f),d_0(f),\bigcup\limits_{i\in\mathbb{F}_5^*}d_i(f)\}$ induce $11$-class, $7$-class and $5$-class symmetric association schemes, respectively.
		\end{itemize}
	\end{example}
	\begin{example}
		Let $f(x):\mathbb{F}_5^4\longrightarrow \mathbb{F}_5$, $f(x_1,x_2,x_3,x_4)=x_1^2+4x_2^2x_4^4+2x_2^2+x_3x_4$. Then
		\begin{itemize}
			\item [$\bullet$] $f(x_1,x_2,x_3,x_4)$ is a non-weakly regular bent function of Type $(-)$.
			\item[$\bullet$] $f^*(x_1,x_2,x_3,x_4)=x_1^2+3x_2^2x_3^4+3x_2^2+4x_3x_4$ is a non-weakly regular bent function of Type $(-)$.
			\item[$\bullet$] $B_-(f)=\mathbb{F}_5^2\times\{0\}\times\mathbb{F}_5$ and $B_-(f^*)=\mathbb{F}_5^2\times\mathbb{F}_5\times\{0\}$ are both non-degenerate $3$-dimensional subspaces over $\mathbb{F}_5$.
			\item[$\bullet$] $\{\{0\},  c_0(f), c_1(f), c_2(f),c_3(f),c_4(f),d_0(f)\setminus\{0\},d_1(f),d_2(f),d_3(f),d_4(f)\}$, $\{\{0\},  c_0(f), c_1(f)\cup c_4(f), c_2(f)\cup c_3(f),d_0(f)\setminus\{0\},d_1(f)\cup d_4(f),d_2(f)\cup d_3(f)\}$ and $\{\{0\}, c_0(f), \bigcup\limits_{i\in\mathbb{F}_5^*}c_i(f),d_0(f)\setminus\{0\},\bigcup\limits_{i\in\mathbb{F}_5^*}d_i(f)\}$ induce $10$-class, $6$-class and $4$-class symmetric association schemes, respectively.
		\end{itemize}
	\end{example}
	\begin{example}
		Let $f(x):\mathbb{F}_5^6\longrightarrow \mathbb{F}_5$, $f(x_1,x_2,x_3,x_4,x_5,x_6)=x_1^2+4x_2^2x_6^4+2x_2^2+x_3^2+x_4^2+x_5x_6$. Then
		\begin{itemize}
			\item [$\bullet$] $f(x_1,x_2,x_3,x_4,x_5,x_6)$ is a non-weakly regular bent function of Type $(-)$.
			\item[$\bullet$] $f^*(x_1,x_2,x_3,x_4,x_5,x_6)=x_1^2+3x_2^2x_5^4+3x_2^2+x_3^2+x_4^2+4x_5x_6$ is a non-weakly regular bent function of Type $(-)$.
			\item[$\bullet$] $B_-(f)=\mathbb{F}_5^4\times\{0\}\times\mathbb{F}_5$ and $B_-(f^*)=\mathbb{F}_5^4\times\mathbb{F}_5\times\{0\}$ are both non-degenerate $5$-dimensional subspaces over $\mathbb{F}_5$.
			\item[$\bullet$] $\{\{0\},  c_0(f), c_1(f), c_2(f),c_3(f),c_4(f),B_-(f)^{\bot}\setminus\{0\},d_0(f)\setminus B_-(f)^{\bot},d_1(f),d_2$ $(f),d_3(f),d_4(f)\}$, $\{\{0\},  c_0(f), c_1(f)\cup c_4(f), c_2(f)\cup c_3(f),B_-(f)^{\bot}\setminus\{0\},d_0$ $(f)\setminus B_-(f)^{\bot},d_1(f)\cup d_4(f),d_2(f)\cup d_3(f)\}$ and $\{\{0\}, c_0(f), \bigcup\limits_{i\in\mathbb{F}_5^*}c_i(f),B_-(f)^{\bot}$ $\setminus\{0\},d_0(f)\setminus B_-(f)^{\bot},\bigcup\limits_{i\in\mathbb{F}_5^*}d_i(f)\}$ induce $11$-class, $7$-class and $5$-class symmetric association schemes, respectively.
		\end{itemize}
	\end{example}
	
	Let $SQ=\{s_1, s_2, \dots, s_{\frac{p-1}{2}}\}$ with $s_1<s_2< \cdots <s_{\frac{p-1}{2}}$, and $NSQ=\{n_1, n_2, \dots, n_{\frac{p-1}{2}}\}$ with $n_1<n_2< \cdots <n_{\frac{p-1}{2}}$.	In the following theorems, when $n$ is odd, we present the constructions of association schemes from non-weakly regular bent functions. 
	
	\begin{theorem}\label{th 3}
		Let $n$ be an odd integer with $n\ge 3$,  $f(x):\mathbb{F}_{p}^n\longrightarrow\mathbb{F}_p$ be a non-weakly regular bent function belonging to $\mathcal{DBF}$. Let $B_+(f)$ and $B_+(f^*)$ both be non-degenerate $r$-dimensional subspaces over $\mathbb{F}_p$ for $p\equiv 1\ (\mathrm{mod}\ 4)$, and $B_-(f)$ and $B_+(f^*)$ both be non-degenerate $r$-dimensional subspaces over $\mathbb{F}_p$ for $p\equiv 3\ (\mathrm{mod}\ 4)$. Then when $r=\frac{n+1}{2}$, for $p\equiv 1\ (\mathrm{mod}\ 4)$, $U_{10}=\{U_{10,i}\}_{i=0}^{(3p+1)/2}=\{\{0\}, c_0(f)\setminus\{0\}, c_{s_1}(f), c_{s_2}(f), \dots, c_{s_{\frac{p-1}{2}}}(f), d_0(f), d_1(f), \dots, d_{p-1}(f)\}$ is a partition of $\mathbb{F}_p^n$ and induces a $\frac{3p+1}{2}$-class symmetric association scheme; for $p\equiv 3\ (\mathrm{mod}\ 4)$, $U_{11}=\{U_{11,i}\}_{i=0}^{(3p+1)/2}=\{\{0\}, c_0(f)\setminus\{0\}, c_{n_1}(f), c_{n_2}(f),\dots, c_{n_{\frac{p-1}{2}}}\\(f), d_0(f), d_1(f), \dots, d_{p-1}(f)\}$ is a partition of $\mathbb{F}_p^n$ and induces a $\frac{3p+1}{2}$-class symmetric association scheme. When  $\frac{n+1}{2}<r\le n-1$, for $p\equiv 1\ (\mathrm{mod}\ 4)$, $U_{12}=\{U_{12,i}\}_{i=0}^{2p+1}=\{\{0\}, B_+(f)^{\bot}\setminus\{0\}, c_0(f)\setminus B_+(f)^{\bot}, c_1(f), \dots, c_{p-1}(f), d_0\\(f), d_1(f),\dots, d_{p-1}(f)\}$ is a partition of $\mathbb{F}_p^n$ and induces a $(2p+1)$-class symmetric association scheme; for $p\equiv 3\ (\mathrm{mod}\ 4)$, $U_{13}=\{U_{13,i}\}_{i=0}^{2p+1}=\{\{0\}, B_-(f)^{\bot}\setminus\{0\}, c_0(f)\setminus B_-(f)^{\bot}, c_1(f), \dots,c_{p-1}(f), d_0(f),d_1(f), \dots, $ $d_{p-1}\\(f)\}$ is a partition of $\mathbb{F}_p^n$ and induces a $(2p+1)$-class symmetric association scheme.
	\end{theorem}
	\begin{proof}
		By Lemma \ref{le 1}, we know that for $p \equiv 1\ (\mathrm{mod}\ 4)$, $B_+(f)^{\bot}\subseteq c_0(f)$,  $B_+(f^*)^{\bot}\subseteq c_0(f^*)$ and $0\in c_0(f)\cap c_0(f^*)$; for $p \equiv 3\ (\mathrm{mod}\ 4)$, $B_-(f)^{\bot}\subseteq c_0(f)$, $B_+(f^*)^{\bot}\subseteq d_0(f^*)$ and $0\in c_0(f)\cap d_0(f^*)$. By Lemma \ref{le 4} and Remark \ref{re 1}, we know that for $p\equiv 1\ (\mathrm{mod}\ 4)$, $|c_0(f)|=|c_0(f^*)|=p^{r-1}>0$, $|d_i(f)|=|d_i(f^*)|=p^{n-1}-p^{r-1}>0$ for any $i\in\mathbb{F}_p$ and $|c_i(f)|=|c_i(f^*)|=p^{r-1}+\eta(i)p^{\frac{n-1}{2}}\ge 0$ for any $i\in\mathbb{F}_p^*$; for $p\equiv 3\ (\mathrm{mod}\ 4)$, $|c_0(f)|=|d_0(f^*)|=p^{r-1}>0$, $|d_i(f)|=|c_i(f^*)|=p^{n-1}-p^{r-1}>0$ for any $i\in\mathbb{F}_p$ and $|c_j(f)|=p^{r-1}-\eta(j)p^{\frac{n-1}{2}}\ge 0$, $|d_j(f^*)|=p^{r-1}+\eta(j)p^{\frac{n-1}{2}}\ge 0$ for any $j\in\mathbb{F}_p^*$. 
		
		Thus, when $r=\frac{n+1}{2}$, we know that $B_+(f^*)^{\bot}=c_0(f^*)$ for $p\equiv 1\ (\mathrm{mod}\ 4)$, $B_+(f^*)^{\bot}=d_0(f^*)$ for $p\equiv 3\ (\mathrm{mod}\ 4)$, and $U_{10}$ and $U_{11}$ are both partitions of $\mathbb{F}_p^n$. When $\frac{n+1}{2}<r\le n-1$, we know that $B_+(f)^{\bot}\subsetneqq c_0(f)$,  $B_+(f^*)^{\bot}\subsetneqq c_0(f^*)$ for $p \equiv 1\ (\mathrm{mod}\ 4)$,  $B_-(f)^{\bot}\subsetneqq c_0(f)$, $B_+(f^*)^{\bot}\subsetneqq d_0(f^*)$ for $p \equiv 3\ (\mathrm{mod}\ 4)$, and $U_{12}$ and $U_{13}$ are both partitions of $\mathbb{F}_p^n$. For any $\alpha,\alpha'\in \mathbb{F}_p^n$, if $\alpha\sim_{\widehat{U_{10}}}\alpha'$ (respectively $\alpha\sim_{\widehat{U_{11}}}\alpha'$, $\alpha\sim_{\widehat{U_{12}}}\alpha'$, $\alpha\sim_{\widehat{U_{13}}}\alpha'$), where $\widehat{U_{10}}$, $\widehat{U_{11}}$, $\widehat{U_{12}}$ and $\widehat{U_{13}}$ are dual partitions of $U_{10}$, $U_{11}$, $U_{12}$ and $U_{13}$, respectively, then $\chi_\alpha(N_i(f))=\chi_{\alpha'}(N_i(f))$ for any $i\in \mathbb{F}_p$. Thus, by Lemma \ref{le 12}, we know that $\epsilon_\alpha=\epsilon_{\alpha'}$ and $f^*(\alpha)=f^*(\alpha')$. According to Lemmas \ref{le 8} and \ref{le 10}, we know that when $r=\frac{n+1}{2}$, then for $p\equiv 1\ (\mathrm{mod}\ 4)$, the dual partition of $U_{10}$ is $\{\{0\}, c_0(f^*)\setminus\{0\}, c_{s_1}(f^*), c_{s_2}(f^*), \dots,$ $ c_{s_{\frac{p-1}{2}}}(f^*), d_0(f^*), d_1(f^*), \dots, d_{p-1}(f^*)\}$; for $p\equiv 3\ (\mathrm{mod}\ 4)$, the dual partition of $U_{11}$ is $\{\{0\}, c_0(f^*), c_1(f^*), \dots, c_{p-1}(f^*),d_0(f^*)\setminus\{0\},d_{s_1}(f^*), d_{s_2}(f^*),\dots,$ $ d_{s_{\frac{p-1}{2}}}(f^*)\}$. When $\frac{n+1}{2}<r\le n-1$, then for $p\equiv 1\ (\mathrm{mod}\ 4)$, the dual partition of $U_{12}$ is $\{\{0\}, B_+(f^*)^{\bot}\setminus\{0\}, c_0(f^*)\setminus B_+(f^*)^{\bot}, c_1(f^*), \dots, c_{p-1}(f^*), d_0(f^*), $ $d_1(f^*), \dots, d_{p-1}(f^*)\}$; for $p\equiv 3\ (\mathrm{mod}\ 4)$, the dual partition of $U_{13}$ is $\{\{0\}, $ $c_0(f^*), c_1(f^*), \dots, c_{p-1}(f^*), B_+(f^*)^{\bot}\setminus\{0\}, d_0(f^*)\setminus B_+(f^*)^{\bot}, d_1(f^*), \dots, d_{p-1}$ $(f^*)\}$. By Lemma \ref{le 6}, when $r=\frac{n+1}{2}$, we have that $U_{10}$ and $U_{11}$ can induce  $\frac{3p+1}{2}$-class symmetric association schemes; when $\frac{n+1}{2}<r\le n-1$, we have that $U_{12}$ and $U_{13}$ can induce $(2p+1)$-class symmetric association schemes. \qed
	\end{proof}
	
	\begin{theorem}\label{th 4}
		Let $n$ be an odd integer with $n\ge 3$,  $f(x):\mathbb{F}_{p}^n\longrightarrow\mathbb{F}_p$ be a non-weakly regular bent function belonging to $\mathcal{DBF}$. Let $B_-(f)$ and $B_-(f^*)$ both be non-degenerate  $r$-dimensional subspaces over $\mathbb{F}_p$ for $p\equiv 1\ (\mathrm{mod}\ 4)$, and $B_+(f)$ and $B_-(f^*)$ both be non-degenerate $r$-dimensional subspaces over $\mathbb{F}_p$ for $p\equiv 3\ (\mathrm{mod}\ 4)$. Then when $r=\frac{n+1}{2}$, for $p\equiv 1\ (\mathrm{mod}\ 4)$, $U_{14}=\{U_{14,i}\}_{i=0}^{(3p+1)/2}=\{\{0\}, c_0(f), c_1(f), \dots, c_{p-1}(f), d_0(f)\setminus\{0\}, d_{n_1}(f), d_{n_2}(f), \dots, d_{n_{\frac{p-1}{2}}}(f)\}$ is a partition of $\mathbb{F}_p^n$ and induces a $\frac{3p+1}{2}$-class symmetric association scheme; for $p\equiv 3\ (\mathrm{mod}\ 4)$, $U_{15}=\{U_{15,i}\}_{i=0}^{(3p+1)/2}=\{\{0\},  c_0(f), c_1(f), \dots, c_{p-1}(f), d_0(f)\setminus\{0\},d_{s_1}(f), d_{s_2}(f),\dots,$ $ d_{s_{\frac{p-1}{2}}}(f)\}$ is a partition of $\mathbb{F}_p^n$ and induces a $\frac{3p+1}{2}$-class symmetric association scheme.
		When $\frac{n+1}{2}<r\le n-1$, for $p\equiv 1\ (\mathrm{mod}\ 4)$, $U_{16}=\{U_{16,i}\}_{i=0}^{2p+1}=\{\{0\}, c_0(f), c_1(f), \dots,$ $ c_{p-1}(f), B_-(f)^{\bot}\setminus\{0\}, d_0(f)\setminus B_-(f)^{\bot}, d_1(f), \dots, d_{p-1}(f)\}$ is a partition of $\mathbb{F}_p^n$ and induces a $(2p+1)$-class symmetric association scheme; for $p\equiv 3\ (\mathrm{mod}\ 4)$, $U_{17}=\{U_{17,i}\}_{i=0}^{2p+1}=\{\{0\},c_0(f), c_1(f), \dots, c_{p-1}(f), B_+(f)^{\bot}\setminus\{0\}, d_0(f)\setminus B_+(f)^{\bot}, d_1(f), \dots,$ $ d_{p-1}\\(f)\}$ is a partition of $\mathbb{F}_p^n$ and induces a $(2p+1)$-class symmetric association scheme.
	\end{theorem}
	\begin{proof}
		By Lemma \ref{le 1}, we know that for $p \equiv 1\ (\mathrm{mod}\ 4)$, $B_-(f)^{\bot}\subseteq d_0(f)$, $B_-(f^*)^{\bot}\subseteq d_0(f^*)$ and $0\in d_0(f)\cap d_0(f^*)$; for $p \equiv 3\ (\mathrm{mod}\ 4)$, $B_+(f)^{\bot}\subseteq d_0(f)$, $B_-(f^*)^{\bot}\subseteq c_0(f^*)$ and $0\in d_0(f)\cap c_0(f^*)$. By Lemma \ref{le 4} and Remark \ref{re 1}, we know that for $p\equiv 1\ (\mathrm{mod}\ 4)$,  $|c_i(f)|=|c_i(f^*)|=p^{n-1}-p^{r-1}>0$ for any $i\in\mathbb{F}_p$, $|d_0(f)|=|d_0(f^*)|=p^{r-1}>0$ and $|d_i(f)|=|d_i(f^*)|=p^{r-1}-\eta(i)p^{\frac{n-1}{2}}\ge 0$ for any $i\in\mathbb{F}_p^*$; for $p\equiv 3\ (\mathrm{mod}\ 4)$, $|c_i(f)|=|d_i(f^*)|=p^{n-1}-p^{r-1}>0$ for any $i\in\mathbb{F}_p$ , $|d_0(f)|=|c_0(f^*)|=p^{r-1}>0$ and $|d_j(f)|=p^{r-1}+\eta(j)p^{\frac{n-1}{2}}\ge 0$, $|c_j(f^*)|=p^{r-1}-\eta(j)p^{\frac{n-1}{2}}\ge 0$ for any $j\in\mathbb{F}_p^*$. 
		
		Thus, when $r=\frac{n+1}{2}$, $B_-(f^*)^{\bot}=d_0(f^*)$ for $p\equiv 1\ (\mathrm{mod}\ 4)$, $B_-(f^*)^{\bot}=c_0(f^*)$ for $p\equiv 3\ (\mathrm{mod}\ 4)$, and $U_{14}$ and $U_{15}$ are both partitions of $\mathbb{F}_p^n$. When $\frac{n+1}{2}<r\le n-1$, we know that $B_-(f)^{\bot}\subsetneqq d_0(f)$,  $B_-(f^*)^{\bot}\subsetneqq d_0(f^*)$ for $p \equiv 1\ (\mathrm{mod}\ 4)$, $B_+(f)^{\bot}\subsetneqq d_0(f)$, $B_-(f^*)^{\bot}\subsetneqq c_0(f^*)$ for $p \equiv 3\ (\mathrm{mod}\ 4)$, and  $U_{16}$ and $U_{17}$ are both partitions of $\mathbb{F}_p^n$. For any $\alpha,\alpha'\in \mathbb{F}_p^n$, if $\alpha\sim_{\widehat{U_{14}}}\alpha'$ (respectively $\alpha\sim_{\widehat{U_{15}}}\alpha'$, $\alpha\sim_{\widehat{U_{16}}}\alpha'$, $\alpha\sim_{\widehat{U_{17}}}\alpha'$), where $\widehat{U_{14}}$, $\widehat{U_{15}}$, $\widehat{U_{16}}$ and $\widehat{U_{17}}$ are dual partitions of $U_{14}$, $U_{15}$, $U_{16}$ and $U_{17}$, respectively, then $\chi_\alpha(N_i(f))=\chi_{\alpha'}(N_i(f))$ for any $i\in \mathbb{F}_p$. Thus, by Lemma \ref{le 12}, we know that $\epsilon_\alpha=\epsilon_{\alpha'}$ and $f^*(\alpha)=f^*(\alpha')$. According to Lemmas \ref{le 8} and \ref{le 10}, we know that when $r=\frac{n+1}{2}$, then for $p\equiv 1\ (\mathrm{mod}\ 4)$, the dual partition of  $U_{14}$ is $\{\{0\}, c_0(f^*), c_1(f^*), \dots, c_{p-1}(f^*), d_0(f^*)\setminus\{0\}, d_{n_1}(f^*), d_{n_2}(f^*), \dots, d_{n_{\frac{p-1}{2}}}(f^*)\}$; for $p\equiv 3\ (\mathrm{mod}\ 4)$, the dual partition of $U_{15}$ is $\{\{0\}, c_0(f^*)\setminus\{0\},c_{n_1}(f^*), c_{n_2}(f^*),\\\dots, c_{n_{\frac{p-1}{2}}}(f^*),d_0(f^*), d_1(f^*), \dots, d_{p-1}$ $(f^*)\}$. When $\frac{n+1}{2}<r\le n-1$, then for $p\equiv 1\ (\mathrm{mod}\ 4)$, the dual partition of $U_{16}$ is $\{\{0\}, c_0(f^*), c_1(f^*), \dots, c_{p-1}(f^*),\\ B_-(f^*)^{\bot}\setminus\{0\}, d_0(f^*)\setminus B_-(f^*)^{\bot}, d_1(f^*),$ $ \dots, d_{p-1}(f^*)\}$; for $p\equiv 3\ (\mathrm{mod}\ 4)$, the dual partition of  $U_{17}$ is $\{\{0\},B_-(f^*)^{\bot}\setminus\{0\}, c_0(f^*)\setminus B_-(f^*)^{\bot}, c_1(f^*), \dots, c_{p-1}\\(f^*), d_0(f^*), d_1(f^*), \dots, d_{p-1}(f^*)\}$. By Lemma \ref{le 6}, when $r=\frac{n+1}{2}$, we have that $U_{14}$ and $U_{15}$ can induce $\frac{3p+1}{2}$-class symmetric association schemes; when $\frac{n+1}{2}<r\le n-1$, we have that $U_{16}$ and $U_{17}$ can induce $(2p+1)$-class symmetric association schemes. \qed
	\end{proof}
	\begin{remark}
		$(1)$ According to Remark \ref{re 4}, Lemmas \ref{le 8}, \ref{le 10}, and the proofs of Theorems \ref{th 3}, \ref{th 4}, we give the first and second eigenmatrices of the association schemes induced by $U_i$, $10\le i\le 17$, in Tables 13-28 of Appendix. \\
		$(2)$ The intersection numbers and the Krein parameters of the association schemes induced by $U_i$, $10\le i\le 17$, can be calculated by Equations (8), (9), and Tables 13-28. In Appendix, we give the intersection numbers and the Krein parameters of the association schemes induced by $U_i$, $10\le i\le 17$.
		
	\end{remark}
	
	Considering the fusions of the association schemes constructed in Theorems \ref{th 3} and \ref{th 4}, we can also obtain some association schemes in the following propositions.
	\begin{proposition}\label{po 4}
		Let $n$ be an odd integer with $n\ge 3$,  $f(x):\mathbb{F}_{p}^n\longrightarrow\mathbb{F}_p$ be a non-weakly regular bent function belonging to $\mathcal{DBF}$. Let $B_+(f)$ and $B_+(f^*)$ both be non-degenerate  $r$-dimensional subspaces over $\mathbb{F}_p$ for $p\equiv 1\ (\mathrm{mod}\ 4)$, and $B_-(f)$ and $B_+(f^*)$ both be non-degenerate $r$-dimensional subspaces over $\mathbb{F}_p$ for $p\equiv 3\ (\mathrm{mod}\ 4)$. Then when $r=\frac{n+1}{2}$, for $p\equiv 1\ (\mathrm{mod}\ 4)$, $U_{18}=\{U_{18,i}\}_{i=0}^{5}=\{\{0\}, c_0(f)\setminus\{0\}, \bigcup\limits_{i\in SQ}c_i(f), d_0(f), \bigcup\limits_{i\in SQ} d_i(f), \bigcup\limits_{i\in NSQ} d_i(f)\}$ is a partition of $\mathbb{F}_p^n$ and induces a $5$-class symmetric association scheme; for $p\equiv 3\ (\mathrm{mod}\ 4)$, $U_{19}=\{U_{19,i}\}_{i=0}^{5}=\{\{0\}, c_0(f)\setminus\{0\}, \bigcup\limits_{i\in NSQ}c_i(f), d_0(f), \bigcup\limits_{i\in SQ}d_i(f), \bigcup\limits_{i\in NSQ}d_i(f)\}$ is a partition of $\mathbb{F}_p^n$ and induces a $5$-class symmetric association scheme. When $\frac{n+1}{2}<r\le n-1$, for $p\equiv 1\ (\mathrm{mod}\ 4)$, $U_{20}=\{U_{20,i}\}_{i=0}^{7}=\{\{0\}, B_+(f)^{\bot}\setminus\{0\}, c_0(f)\setminus B_+(f)^{\bot}, \bigcup\limits_{i\in SQ}c_i(f), \bigcup\limits_{i\in NSQ}c_i(f), d_0$ $(f),\bigcup\limits_{i\in SQ}d_i(f), \bigcup\limits_{i\in NSQ}d_i(f)\}$ is a partition of $\mathbb{F}_p^n$ and induces a $7$-class symmetric association scheme; for $p\equiv 3\ (\mathrm{mod}\ 4)$, $U_{21}=\{U_{21,i}\}_{i=0}^{7}=\{\{0\}, B_-(f)^{\bot}\setminus\{0\}, c_0(f)\setminus B_-(f)^{\bot},\bigcup\limits_{i\in SQ}c_i(f),\\ \bigcup\limits_{i\in NSQ}c_i(f), d_0(f),\bigcup\limits_{i\in SQ}d_i(f), \bigcup\limits_{i\in NSQ}d_i(f)\}$ is a partition of $\mathbb{F}_p^n$ and induces a $7$-class symmetric association scheme.
	\end{proposition}
	\begin{proof}
		
		By Lemmas \ref{le 8}, \ref{le 10}, \ref{le 11} and the proof of Theorem \ref{th 3}, when $r=\frac{n+1}{2}$, we know that $U_{18}$ and $U_{19}$ are both partitions of $\mathbb{F}_p^n$, and the dual partitions of $U_{18}$ and $U_{19}$ are $\{\{0\}, c_0(f^*)\setminus\{0\}, \bigcup\limits_{i\in SQ}c_i(f^*), d_0(f^*), \bigcup\limits_{i\in SQ}d_i(f^*), \bigcup\limits_{i\in NSQ}d_i(f^*)\}$ and $\{\{0\}, c_0(f^*), \bigcup\limits_{i\in SQ}c_i(f^*), \bigcup\limits_{i\in NSQ}c_i(f^*), d_0(f^*)\setminus\{0\}, \bigcup\limits_{i\in SQ}d_i(f^*)\}$, respectively. When $\frac{n+1}{2}<r\le n-1$, we know that $U_{20}$ and $U_{21}$ are both partitions of $\mathbb{F}_p^n$, and the dual partitions of $U_{20}$ and $U_{21}$ are $\{\{0\}, B_+(f^*)^{\bot}\setminus\{0\}, c_0(f^*)\setminus B_+(f^*)^{\bot}, \bigcup\limits_{i\in SQ}c_i(f^*), \bigcup\limits_{i\in NSQ}c_i(f^*), d_0(f^*),  \bigcup\limits_{i\in SQ}d_i(f^*), \bigcup\limits_{i\in NSQ}d_i(f^*)\}$ and $\{\{0\},$ $  c_0(f^*), \bigcup\limits_{i\in SQ}c_i(f^*), \bigcup\limits_{i\in NSQ}c_i(f^*), B_+(f^*)^{\bot}\setminus\{0\}, d_0(f^*)\setminus B_+(f^*)^{\bot},  \bigcup\limits_{i\in SQ}d_i(f^*), $ $\bigcup\limits_{i\in NSQ}d_i(f^*)\}$, respectively. Thus, according to Lemma \ref{le 6}, we have that when $r=\frac{n+1}{2}$, $U_{18}$ and $U_{19}$ can induce $5$-class symmetric association schemes. When $\frac{n+1}{2}<r\le n-1$, $U_{20}$ and $U_{21}$ can induce $7$-class symmetric association schemes.\qed	\end{proof}
	\begin{remark}
		$(1)$ Since $f^{**}(x)=f(x)$ for any $x\in\mathbb{F}_p^n$, for $p\equiv 1\ (\mathrm{mod}\ 4)$, according to Remark \ref{re 4}, Lemmas \ref{le 8}, \ref{le 10}, \ref{le 11}, and the proof of Proposition \ref{po 4}, we have that the first and second eigenmatrices of the association scheme induced by $U_{18}$ (respectively $U_{20}$) are the same. The first (second) eigenmatrices of the association schemes induced by $U_{18}$ and $U_{20}$ are given in Tables 29 and 32 of Appendix. For $p\equiv 3\ (\mathrm{mod}\ 4)$, we give the first and second eigenmatrices of the association schemes induced by $U_{19}$ and $U_{21}$ in Tables 30, 31, 33 and 34  of Appendix. \\
		$(2)$ The intersection numbers and the Krein parameters of the association schemes induced by $U_i$, $18\le i\le 21$, can be calculated by Equations (8), (9), and Tables 29-34. We easily have that the intersection number $p_{u,v}^w$ and the Krein parameter $q_{u,v}^w$ of the association scheme induced by $U_{18}$ (respectively $U_{20}$) are the same. In Appendix, we give the intersection numbers and the Krein parameters of the association schemes induced by $U_i$, $18\le i\le 21$.
		
	\end{remark}

	\begin{proposition}\label{po 5}
		Let $n$ be an odd integer with $n\ge 3$,  $f(x):\mathbb{F}_{p}^n\longrightarrow\mathbb{F}_p$ be a non-weakly regular bent function belonging to $\mathcal{DBF}$. Let $B_-(f)$ and $B_-(f^*)$ both be non-degenerate  $r$-dimensional subspaces over $\mathbb{F}_p$ for $p\equiv 1\ (\mathrm{mod}\ 4)$, and $B_+(f)$ and $B_-(f^*)$ both be non-degenerate $r$-dimensional subspaces over $\mathbb{F}_p$ for $p\equiv 3\ (\mathrm{mod}\ 4)$. Then when $r=\frac{n+1}{2}$, for $p\equiv 1\ (\mathrm{mod}\ 4)$, $U_{22}=\{U_{22,i}\}_{i=0} ^5=\{\{0\}, c_0(f), \bigcup\limits_{i\in SQ}c_i(f), \bigcup\limits_{i\in NSQ}c_i(f), d_0(f)\setminus\{0\}, \bigcup\limits_{i\in NSQ}d_i(f)\}$ is a partition of  $\mathbb{F}_p^n$ and induces a $5$-class symmetric association scheme; for $p\equiv 3\ (\mathrm{mod}\ 4)$, $U_{23}=\{U_{23,i}\}_{i=0}^5=\{\{0\}, c_0(f), \bigcup\limits_{i\in SQ}c_i(f), \bigcup\limits_{i\in NSQ}c_i(f),d_0(f)\setminus\{0\}, \bigcup\limits_{i\in SQ} d_i(f)\}$ is a partition of $\mathbb{F}_p^n$ and induces a $5$-class symmetric association schemes. When $\frac{n+1}{2}<r\le n-1$, for $p\equiv 1\ (\mathrm{mod}\ 4)$, $U_{24}=\{U_{24,i}\}_{i=0} ^7=\{\{0\}, c_0(f), \bigcup\limits_{i\in SQ}c_i(f),\bigcup\limits_{i\in NSQ}c_i(f), B_-(f)^{\bot}\setminus\{0\}, d_0(f)\setminus B_-(f)^{\bot}, \bigcup\limits_{i\in SQ}d_i(f),\\ \bigcup\limits_{i\in NSQ}d_i(f)\}$ is a partition of $\mathbb{F}_p^n$ and induces a $7$-class symmetric association scheme; for $p\equiv 3\ (\mathrm{mod}\ 4)$, $U_{25}=\{U_{25,i}\}_{i=0}^7=\{\{0\},c_0(f),$ $ \bigcup\limits_{i\in SQ}c_i(f), \bigcup\limits_{i\in NSQ}c_i\\(f),B_+(f)^{\bot}\setminus\{0\}, d_0(f)\setminus B_+(f)^{\bot}, \bigcup\limits_{i\in SQ}d_i(f),\bigcup\limits_{i\in NSQ}d_i(f)\}$ is a partition of $\mathbb{F}_p^n$ and induces a $7$-class symmetric association scheme.
	\end{proposition}
	\begin{proof}
		
		By Lemmas \ref{le 8}, \ref{le 10}, \ref{le 11} and the proof of Theorem \ref{th 4}, when $r=\frac{n+1}{2}$, we know that $U_{22}$ and $U_{23}$ are both partitions of $\mathbb{F}_p^n$, and the dual partitions of $U_{22}$ and $U_{23}$ are $\{\{0\}, c_0(f^*), \bigcup\limits_{i\in SQ}c_i(f^*), \bigcup\limits_{i\in NSQ}c_i(f^*), d_0(f^*)\setminus\{0\}, \bigcup\limits_{i\in NSQ}d_i(f^*)\}$ and $\{\{0\}, c_0(f^*)\setminus\{0\}, \bigcup\limits_{i\in NSQ}c_i(f^*), d_0(f^*), \bigcup\limits_{i\in SQ}d_i(f^*), \bigcup\limits_{i\in NSQ}d_i(f^*)\}$, respectively. When $\frac{n+1}{2}<r\le n-1$, we know that $U_{24}$ and $U_{25}$ are both partitions of $\mathbb{F}_p^n$, and the dual partitions of $U_{24}$ and $U_{25}$ are $\{\{0\},  c_0(f^*), \bigcup\limits_{i\in SQ}c_i(f^*), $ $\bigcup\limits_{i\in NSQ}c_i(f^*), B_-(f^*)^{\bot}\setminus\{0\}, d_0(f^*)\setminus B_-(f^*)^{\bot}, \bigcup\limits_{i\in SQ}d_i(f^*), \bigcup\limits_{i\in NSQ}d_i(f^*)\}$ and $\{\{0\}, B_-(f^*)^{\bot}\setminus\{0\}, c_0(f^*)\setminus B_-(f^*)^{\bot}, \bigcup\limits_{i\in SQ}c_i(f^*), \bigcup\limits_{i\in NSQ}c_i(f^*), d_0(f^*),\bigcup\limits_{i\in SQ}d_i$ $(f^*),\bigcup\limits_{i\in NSQ}d_i(f^*)\}$, respectively. Thus, according to Lemma \ref{le 6}, we have that when $r=\frac{n+1}{2}$, $U_{22}$ and $U_{23}$ can induce $5$-class symmetric association schemes. When $\frac{n+1}{2}<r\le n-1$, $U_{24}$ and $U_{25}$ can induce $7$-class symmetric association schemes.\qed
	\end{proof}
	\begin{remark}\label{re 10} $\mathrm{(1)}$ Note that in Theorem \ref{th 3} and Proposition \ref{po 4}, when $B_+(f^*)$ is a non-degenerate $\frac{n+1}{2}$-dimensional subspace over $\mathbb{F}_p$, by the proofs of Theorem \ref{th 3} and Proposition \ref{po 4}, we know that the condition that $B_+(f)$ for $p\equiv 1\ (\mathrm{mod}\ 4)$ and $B_-(f)$ for $p\equiv 3\ (\mathrm{mod}\ 4)$ are non-degenerate $\frac{n+1}{2}$-dimensional subspaces is not necessary. In Theorem \ref{th 4} and Proposition \ref{po 5}, when $B_-(f^*)$ is a non-degenerate  $\frac{n+1}{2}$-dimensional subspace over $\mathbb{F}_p$, by the proofs of Theorem \ref{th 4} and Proposition \ref{po 5}, we know that the condition that $B_-(f)$ for $p\equiv 1\ (\mathrm{mod}\ 4)$ and $B_+(f)$ for $p\equiv 3\ (\mathrm{mod}\ 4)$ are non-degenerate $\frac{n+1}{2}$-dimensional subspaces is not necessary. Thus, when $p=3$, the $5$-class association schemes constructed in \cite{Ozbudak2} can also be obtained by Theorems \ref{th 3} and \ref{th 4}.\\
		$(2)$ Since $f^{**}(x)=f(x)$ for any $x\in\mathbb{F}_p^n$, according to Remark \ref{re 4}, Lemmas \ref{le 8}, \ref{le 10}, \ref{le 11}, and the proof of Proposition \ref{po 5}, when $p\equiv 1\ (\mathrm{mod}\ 4)$, we have that the first and second eigenmatrices of the association scheme induced by $U_{22}$ (respectively $U_{24}$) are the same. For fixed parameters $n$ and $p$, where $p\equiv 3\ (\mathrm{mod}\ 4)$, the first (respectively the second) eigenmatrix of the association scheme induced by $U_{23}$ is the same as the second (respectively the first) eigenmatrix of the association scheme induced by $U_{19}$. For fixed parameters $n$, $p$ and $r$, where $p\equiv 3\ (\mathrm{mod}\ 4)$, the first (respectively the second) eigenmatrix of the association scheme induced by $U_{25}$ is the same as the second (respectively the first) eigenmatrix of the association scheme induced by $U_{21}$. The first (second) eigenmatrices of the association schemes induced by $U_{22}$ and $U_{24}$ are given in Tables 35 and 36 of Appendix.
		\\
		$(3)$ When $p\equiv 1\ (\mathrm{mod}\ 4)$, the intersection numbers and the Krein parameters of the association schemes induced by $U_{22}$ and $U_{24}$ can be calculated by Equations (8), (9), and Tables 35 and 36. We easily have that the intersection number $p_{u,v}^w$ and the Krein parameter $q_{u,v}^w$ of the association scheme induced by $U_{22}$ (respectively $U_{24}$) are the same. For fixed parameters $n$ and $p$, where $p\equiv 3\ (\mathrm{mod}\ 4)$, the intersection number $p_{u,v}^w$ (respectively the Krein parameter $q_{u,v}^w$) of the association scheme induced by $U_{23}$ is the same as the Krein parameter $q_{u,v}^w$ (respectively the intersection number $p_{u,v}^w$) of the association scheme induced by $U_{19}$. For fixed parameters $n$, $p$ and $r$, where $p\equiv 3\ (\mathrm{mod}\ 4)$, the intersection number $p_{u,v}^w$ (respectively the Krein parameter $q_{u,v}^w$) of the association scheme induced by $U_{25}$ is the same as the Krein parameter $q_{u,v}^w$ (respectively the intersection number $p_{u,v}^w$) of the association scheme induced by $U_{21}$. In Appendix, we give the intersection numbers and the Krein parameters of the association schemes induced by $U_{22}$ and $U_{24}$.

	\end{remark}
	
	In the following proposition, when $n$ is odd, we give the sufficient and necessary conditions for the partitions $P$ and $D$ to induce symmetric association schemes.
	\begin{proposition}\label{po 6}
		Let $n$ be an odd integer with $n\ge 3$,  $f(x):\mathbb{F}_{p}^n\longrightarrow\mathbb{F}_p$ be a non-weakly regular bent function belonging to $\mathcal{DBF}$, then we have the following.
		\begin{itemize}
			\item[$\mathrm{(1)}$]If $B_+(f^*)$ is a non-degenerate $r$-dimensional subspace over $\mathbb{F}_p$ with $\frac{n+1}{2}\le r\le n-1$, then 
			the partitions $P$ and $D$ induce symmetric association schemes if and only if $r=\frac{n+1}{2}$.  
			\item[$\mathrm{(2)}$]If $B_-(f^*)$ is a non-degenerate $r$-dimensional subspace over $\mathbb{F}_p$ with $\frac{n+1}{2}\le r\le n-1$, then 
			the partitions $P$ and $D$ induce symmetric association schemes if and only if $r=\frac{n+1}{2}$.  
		\end{itemize}
	\end{proposition}
	\begin{proof}
		$\mathrm{(1)}$	When $B_+(f^*)$ is a non-degenerate $r$-dimensional subspace over $\mathbb{F}_p$, we only prove the necessity and the sufficiency can directly follow from Theorem \ref{th 3}, Proposition \ref{po 4} and Remark \ref{re 10}.
		
		By Lemma \ref{le 1}, we know that $B_+(f^*)^{\bot}\subseteq c_0(f^*)$ for $p\equiv 1\ (\mathrm{mod}\ 4)$ and $B_+(f^*)^{\bot}\subseteq d_0(f^*)$ for $p\equiv 3\ (\mathrm{mod}\ 4)$. 
		Assume that $B_+(f^*)^{\bot}\subsetneqq c_0(f^*)$  for $p\equiv 1\ (\mathrm{mod}\ 4)$ and $B_+(f^*)^{\bot}\subsetneqq d_0(f^*)$ for $p\equiv 3\ (\mathrm{mod}\ 4)$, which implies that $r>\frac{n+1}{2}$, then by the proof of Theorem \ref{th 3}, we have that $P=\{\{0\},c_0(f)\setminus\{0\},c_1(f),\dots,c_{p-1}(f),d_0(f),d_1(f),\dots,d_{p-1}(f)\}$ and  $D=\{\{0\},c_0(f)\setminus\{0\},$ $\bigcup\limits_{i\in SQ}c_i(f),\bigcup\limits_{i\in NSQ}c_i(f),d_0(f),\bigcup\limits_{i\in SQ}d_i(f), \bigcup\limits_{i\in NSQ}d_i(f)\}$. For any $\alpha, \alpha'\in\mathbb{F}_p^n$, if $\alpha\sim_{\widehat{P}}\alpha'$, where $\widehat{P}$ is the dual partition of $P$, then $\chi_\alpha(N_i(f))=\chi_{\alpha'}(N_i(f))$ for any $i\in \mathbb{F}_p$. Thus, by Lemma \ref{le 12}, we know that $\epsilon_\alpha=\epsilon_{\alpha'}$ and $f^*(\alpha)=f^*(\alpha')$.  By Lemmas \ref{le 8}, \ref{le 10} and \ref{le 11}, we know that the dual partition of $P$ is $ \{\{0\}, $ $B_+(f^*)^{\bot}\setminus\{0\}, c_0(f^*)\setminus B_+(f^*)^{\bot},c_1(f^*),\dots,c_{p-1}(f^*),d_0(f^*)\setminus\{0\},d_1(f^*),\dots,$ $d_{p-1}(f^*)\}$ for $p\equiv 1\ (\mathrm{mod}\ 4)$ and  $\{\{0\}, c_0(f^*),c_1(f^*),\dots,$ $c_{p-1}$ $(f^*),$ $ B_+(f^*)^{\bot}\setminus\{0\},d_0(f^*)\setminus B_+(f^*)^{\bot},d_1(f^*),\dots,d_{p-1}(f^*)\}$ for $p\equiv 3\ (\mathrm{mod}\ 4)$. The dual partition of $D$ is $ \{\{0\},B_+(f^*)^{\bot}\setminus\{0\}, c_0(f^*)\setminus B_+(f^*)^{\bot},\bigcup\limits_{i\in SQ}c_i(f^*),\bigcup\limits_{i\in NSQ}c_i(f^*),$ $d_0(f^*),\bigcup\limits_{i\in SQ}d_i(f^*), \bigcup\limits_{i\in NSQ}d_i(f^*)\}$ for $p\equiv 1\ (\mathrm{mod}\ 4)$ and $ \{\{0\}, c_0(f^*),\bigcup\limits_{i\in SQ}c_i$ $(f^*),\bigcup\limits_{i\in NSQ}c_i(f^*),B_+(f^*)^{\bot}\setminus\{0\},d_0(f^*)\setminus B_+(f^*)^{\bot},\bigcup\limits_{i\in SQ}d_i(f^*),\bigcup\limits_{i\in NSQ}d_i(f^*)\}$ for $p\equiv 3\ (\mathrm{mod}\ 4)$.  
		Thus, according to Lemma \ref{le 6}, we have that $P$ and $D$ can not induce symmetric association schemes, so $B_+(f^*)^{\bot}= c_0(f^*)$  for $p\equiv 1\ (\mathrm{mod}\ 4)$ and $B_+(f^*)^{\bot}=d_0(f^*)$ for $p\equiv 3\ (\mathrm{mod}\ 4)$, which implies that $r=\frac{n+1}{2}$.
		
		$\mathrm{(2)}$	When $B_-(f^*)$ is a non-degenerate $r$-dimensional subspace over $\mathbb{F}_p$, the proof is similar, so we omit it.\qed
	\end{proof}
	
	In the following, when $n$ is odd, we give some examples.
	\begin{example}
		Let $f(x):\mathbb{F}_5^3\longrightarrow \mathbb{F}_5$, $f(x_1,x_2,x_3)=x_1^2x_3^4+x_1^2+x_2x_3$. Then
		\begin{itemize}
			\item [$\bullet$] $f(x_1,x_2,x_3)$ is a non-weakly regular bent function of Type $(+)$.
			\item[$\bullet$] $f^*(x_1,x_2,x_3)=2x_1^2x_2^4+x_1^2+4x_2x_3$ is a non-weakly regular bent function of Type $(+)$.
			\item[$\bullet$] $B_+(f)=\mathbb{F}_5\times\{0\}\times\mathbb{F}_5$ and $B_+(f^*)=\mathbb{F}_5\times\mathbb{F}_5\times\{0\}$ are both non-degenerate $2$-dimensional subspaces over $\mathbb{F}_5$.
			\item[$\bullet$] $\{\{0\},  c_0(f)\setminus\{0\}, c_1(f), c_4(f),d_0(f),d_1(f),d_2(f),d_3(f),d_4(f)\}$ and $\{\{0\},  c_0$ $(f)\setminus\{0\}, c_1(f)\cup c_4(f),d_0(f),d_1(f)\cup d_4(f),d_2(f)\cup d_3(f)\}$ induce $8$-class and $5$-class  symmetric association schemes, respectively.
		\end{itemize}
	\end{example}
	\begin{example}
		Let $f(x):\mathbb{F}_5^5\longrightarrow \mathbb{F}_5$, $f(x_1,x_2,x_3,x_4,x_5)=x_1^2+x_2^2x_5^4+x_2^2+x_3^2+x_4x_5$. Then
		\begin{itemize}
			\item [$\bullet$] $f(x_1,x_2,x_3,x_4)$ is a non-weakly regular bent function of Type $(+)$.
			\item[$\bullet$] $f^*(x_1,x_2,x_3,x_4,x_5)=x_1^2+2x_2^2x_4^4+x_2^2+x_3^2+4x_4x_5$ is a non-weakly regular bent function of Type $(+)$.
			\item[$\bullet$] $B_+(f)=\mathbb{F}_5^3\times\{0\}\times\mathbb{F}_5$ and $B_+(f^*)=\mathbb{F}_5^3\times\mathbb{F}_5\times\{0\}$ are both non-degenerate $4$-dimensional subspaces over $\mathbb{F}_5$.
			\item[$\bullet$] $\{\{0\}, B_+(f)^{\bot}\setminus \{0\}, c_0(f)\setminus B_+(f)^{\bot}, c_1(f), c_2(f),c_3(f),c_4(f),d_0(f),d_1(f),d_2$ $(f),d_3(f),d_4(f)\}$ and $\{\{0\},  B_+(f)^{\bot}\setminus \{0\}, c_0(f)\setminus B_+(f)^{\bot}, c_1(f)\cup c_4(f), c_2$ $(f)\cup c_3(f),d_0(f),d_1(f)\cup d_4(f),d_2(f)\cup d_3(f)\}$ induce $11$-class and $7$-class symmetric association schemes, respectively.
		\end{itemize}
	\end{example}
	
	\begin{example}
		Let $f(x):\mathbb{F}_5^3\longrightarrow \mathbb{F}_5$, $f(x_1,x_2,x_3)=4x_1^2x_3^4+2x_1^2+x_2x_3$. Then
		\begin{itemize}
			\item [$\bullet$] $f(x_1,x_2,x_3)$ is a non-weakly regular bent function of Type $(-)$.
			\item[$\bullet$] $f^*(x_1,x_2,x_3)=3x_1^2x_2^4+3x_1^2+4x_2x_3$ is a non-weakly regular bent function of Type $(-)$.
			\item[$\bullet$] $B_-(f)=\mathbb{F}_5\times\{0\}\times\mathbb{F}_5$ and $B_-(f^*)=\mathbb{F}_5\times\mathbb{F}_5\times\{0\}$ are both non-degenerate $2$-dimensional subspaces over $\mathbb{F}_5$.
			\item[$\bullet$] $\{\{0\},  c_0(f), c_1(f),c_2(f),c_3(f), c_4(f),d_0(f)\setminus\{0\},d_2(f),d_3(f)\}$ and $\{\{0\},  c_0$ $(f), c_1(f)\cup c_4(f),c_2(f)\cup c_3(f),d_0(f)\setminus\{0\},d_2(f)\cup d_3(f)\}$ induce $8$-class and $5$-class  symmetric association schemes, respectively.
		\end{itemize}
	\end{example}
	\begin{example}
		Let $f(x):\mathbb{F}_5^5\longrightarrow \mathbb{F}_5$, $f(x_1,x_2,x_3,x_4,x_5)=x_1^2+4x_2^2x_5^4+2x_2^2+x_3^2+x_4x_5$. Then
		\begin{itemize}
			\item [$\bullet$] $f(x_1,x_2,x_3,x_4)$ is a non-weakly regular bent function of Type $(-)$.
			\item[$\bullet$] $f^*(x_1,x_2,x_3,x_4,x_5)=x_1^2+3x_2^2x_4^4+3x_2^2+x_3^2+4x_4x_5$ is a non-weakly regular bent function of Type $(-)$.
			\item[$\bullet$] $B_-(f)=\mathbb{F}_5^3\times\{0\}\times\mathbb{F}_5$ and $B_-(f^*)=\mathbb{F}_5^3\times\mathbb{F}_5\times\{0\}$ are both non-degenerate $4$-dimensional subspaces over $\mathbb{F}_5$.
			\item[$\bullet$] $\{\{0\}, c_0(f), c_1(f), c_2(f),c_3(f),c_4(f),B_-(f)^{\bot}\setminus \{0\},d_0(f)\setminus B_-(f)^{\bot},d_1(f),d_2$ $(f),d_3(f),d_4(f)\}$ and $\{\{0\}, c_0(f), c_1(f)\cup c_4(f), c_2$ $(f)\cup c_3(f),B_-(f)^{\bot}\setminus \{0\},d_0(f)\setminus B_-(f)^{\bot},d_1(f)\cup d_4(f),d_2(f)\cup d_3(f)\}$ induce $11$-class and $7$-class symmetric association schemes, respectively.
		\end{itemize}
	\end{example}
	\begin{example}
		Let $f(x):\mathbb{F}_7^3\longrightarrow\mathbb{F}_7$, $f(x_1,x_2,x_3)=5x_1^2x_3^6+3x_1^2+x_2x_3.$ Then
		\begin{itemize}
			\item[$\bullet$]$f(x_1,x_2,x_3)$ is a non-weakly regular bent function of Type $(-)$.
			\item[$\bullet$] $f^*(x_1,x_2,x_3)=x_1^2x_2^6+4x_1^2+6x_2x_3$ is  a non-weakly regular bent function of Type $(+)$.
			\item[$\bullet$] $B_-(f)=\mathbb{F}_7\times\{0\}\times\mathbb{F}_7$ and $B_+(f^*)=\mathbb{F}_7\times\mathbb{F}_7\times\{0\}$ are both non-degenerate $2$-dimensional subspaces over $\mathbb{F}_7$.
			\item[$\bullet$] $\{\{0\},c_0(f)\setminus\{0\},c_3(f),c_5(f),c_6(f),d_0(f),d_1(f),d_2(f),d_3(f),d_4(f),d_5(f),\\d_6(f)\}$ and $\{ \{0\},c_0(f)\setminus\{0\},c_3(f)\cup c_5(f)\cup c_6(f),d_0(f),d_1(f)\cup d_2(f)\cup d_4(f),d_3(f)\cup d_5(f)\cup d_6(f)\}$ induce $11$-class and $5$-class symmetric association schemes, respectively.
		\end{itemize}
	\end{example}
	\begin{example}
		Let $f(x):\mathbb{F}_7^5\longrightarrow \mathbb{F}_7$, $f(x_1,x_2,x_3,x_4,x_5)=2x_1^2x_5^6+x_1^2+x_2^2+x_3^2+x_4x_5.$ Then
		\begin{itemize}
			\item [$\bullet$] $f(x_1,x_2,x_3,x_4,x_5)$ is a non-weakly regular bent function of Type $(-)$.
			\item[$\bullet$] $f^*(x_1,x_2,x_3,x_4,x_5)=6x_1^2x_4^6+5x_1^2+5x_2^2+5x_3^2+6x_4x_5$ is a non-weakly regular bent function of Type $(+)$.
			\item[$\bullet$] $B_-(f)=\mathbb{F}_7^3\times\{0\}\times\mathbb{F}_7$ and $B_+(f^*)=\mathbb{F}_7^3\times\mathbb{F}_7\times\{0\}$ are both non-degenerate $4$-dimensional subspaces over $\mathbb{F}_7$.
			\item[$\bullet$] $\{\{0\},B_-(f)^{\bot}\setminus\{0\},c_0(f)\setminus B_-(f)^{\bot}, c_1(f),c_2(f),c_3(f),c_4(f),c_5(f),c_6(f),\\d_0(f),d_1(f),d_2(f),d_3(f),d_4(f),d_5(f),d_6(f)\}$ and $\{\{0\},B_-(f)^{\bot}\setminus\{0\},c_0\\(f)\setminus B_-(f)^{\bot},c_1(f)\cup c_2(f)\cup c_4(f), c_3(f)\cup c_5(f)\cup c_6(f), d_0(f),d_1(f)\cup d_2(f)\cup d_4(f),d_3(f)\cup d_5(f)\cup d_6(f)\}$ induce $15$-class and $7$-class symmetric association schemes, respectively.
		\end{itemize}
	\end{example}
	\begin{example}
		Let $f(x):\mathbb{F}_7^3\longrightarrow \mathbb{F}_7$, $f(x_1,x_2,x_3)=2x_1^2x_3^6+x_1^2+x_2x_3.$ Then
		\begin{itemize}
			\item[$\bullet$] $f(x_1,x_2,x_3)$ is a non-weakly regular bent function of Type $(+)$.
			\item[$\bullet$] $f^*(x_1,x_2,x_3)=6x_1^2x_2^6+5x_1^2+6x_2x_3$ is a non-weakly regular bent function of Type $(-)$.
			\item[$\bullet$] $B_+(f)=\mathbb{F}_7\times\{0\}\times\mathbb{F}_7$ and $B_-(f^*)=\mathbb{F}_7\times\mathbb{F}_7\times\{0\}$ are both non-degenerate $2$-dimensional subspaces over $\mathbb{F}_7$.
			\item[$\bullet$] $\{\{0\},c_0(f),c_1(f),c_2(f),c_3(f),c_4(f),c_5(f),c_6(f),d_0(f)\setminus\{0\},d_1(f),d_2(f),\\d_4(f)\}$ and $\{\{0\},c_0(f),c_1(f)\cup c_2(f)\cup c_4(f),c_3(f)\cup c_5(f)\cup c_6(f),d_0(f)\setminus\{0\},d_1(f)\cup d_2(f)\cup d_4(f)\}$ induce $11$-class and $5$-class symmetric association schemes, respectively.
		\end{itemize}
	\end{example} 
	
	\begin{example}
		Let $f(x):\mathbb{F}_7^5\longrightarrow \mathbb{F}_7$, $f(x_1,x_2,x_3,x_4,x_5)=5x_1^2x_5^6+3x_1^2+x_2^2+x_3^2+x_4x_5.$ Then
		\begin{itemize}
			\item [$\bullet$]$f(x_1,x_2,x_3,x_4,x_5)$ is a non-weakly regular bent function of Type $(+)$.
			\item[$\bullet$] $f^*(x_1,x_2,x_3,x_4,x_5)=x_1^2x_4^6+4x_1^2+5x_2^2+5x_3^2+6x_4x_5$ is a non-weakly regular bent function of Type $(-)$.
			\item[$\bullet$] $B_+(f)=\mathbb{F}_7^3\times\{0\}\times\mathbb{F}_7$ and $B_-(f^*)=\mathbb{F}_7^3\times\mathbb{F}_7\times\{0\}$ are both non-degenerate $4$-dimensional subspaces over $\mathbb{F}_7$.
			\item[$\bullet$] $\{\{0\},c_0(f),c_1(f),c_2(f),c_3(f),c_4(f),c_5(f),c_6(f),B_+(f)^{\bot}\setminus\{0\},d_0(f)\setminus B_+\\(f)^{\bot},d_1(f),d_2(f),d_3(f),d_4(f),d_5(f),d_6(f)\}$ and $\{\{0\},c_0(f),c_1(f)\cup c_2(f)\cup c_4(f), c_3(f)\cup c_5(f)\cup c_6(f),B_+(f)^{\bot}\setminus\{0\},d_0(f)\setminus B_+(f)^{\bot}, d_1(f)\cup d_2(f)\cup d_4(f),d_3(f)\cup d_5(f)\cup d_6(f)\}$ induce $15$-class and $7$-class symmetric association schemes, respectively.
		\end{itemize}
	\end{example}
	Define \begin{align*}
	U&:=\{\{0\},N_0(f)\setminus\{0\}, \bigcup\limits_{i\in SQ} N_i(f),\bigcup\limits_{i\in NSQ}N_i(f)\},\\  V&:=\{\{0\}, N_0(f)\setminus\{0\},\bigcup\limits_{i\in \mathbb{F}_p^*}N_i(f)\}.
	\end{align*}
	In \cite{Chee}, Chee \emph{et al.} proved that when $n$ is even, if $f(x)$ is a weakly regular bent function belonging to $\mathcal{DBF}$, then $U$ and $V$ can induce $3$-class and $2$-class symmetric association schemes, respectively. Actually, for any positive integer $n$, we have the following proposition.
	\begin{proposition}\label{po 7}
		Let $n$ be a positive integer with $n\ge 3$,  $f(x):\mathbb{F}_{p}^n\longrightarrow\mathbb{F}_p$ be a weakly regular bent function belonging to $\mathcal{DBF}$, then $U$ is a partition of $\mathbb{F}_p^n$ and induces a $3$-class symmetric association scheme.
	\end{proposition}
	\begin{proof}
		By Lemma \ref{le 3}, we know that $U$ is a partition of $\mathbb{F}_p^n$. According to  Corollaries \ref{co 1} and \ref{co 3}, we have that the dual partition of $U$ is $\{\{0\}, N_0(f^*)\setminus\{0\}, \bigcup\limits_{i\in SQ}N_i$ $(f^*),\bigcup\limits_{i\in NSQ}N_i(f^*)\}$. Thus, according to Lemma \ref{le 6}, we have that $U$ can induce a $3$-class symmetric association scheme.\qed
	\end{proof}

	In the following proposition, we give the sufficient and necessary conditions for the partitions $U$ and $V$ to induce symmetric association schemes.
	\begin{proposition}\label{po 8}
		Let $n$ be a positive integer with $n\ge 3$,  $f(x):\mathbb{F}_{p}^n\longrightarrow\mathbb{F}_p$ be a dual-bent function belonging to $\mathcal{DBF}$. Then the partition $U$ induces a symmetric association scheme if and only if $f(x)$ is weakly regular, and partition $V$ induces a symmetric association scheme if and only if $f(x)$ is weakly regular and $n$ is even.
	\end{proposition}
	\begin{proof}
		We only prove the necessity and the sufficiency can directly follow from Proposition \ref{po 7} and \cite[Theorem 4]{Chee}.
		
		Assume that $f(x)$ is non-weakly regular, by Corollaries \ref{co 1} and \ref{co 3}, we have that if $0\in B_+(f)$, then the dual partition of $U$ is the set of the all nonempty sets in $\{\{0\}, c_0(f^*)\setminus\{0\}, \bigcup\limits_{i\in SQ} c_i(f^*), \bigcup\limits_{i\in NSQ}c_i(f^*), d_0(f^*), \bigcup\limits_{i\in SQ} d_i(f^*), \bigcup\limits_{i\in NSQ}d_i$ $(f^*)\}$ and if $0\in B_-(f)$, then the dual partition of $U$ is the set of the all nonempty sets in  $\{\{0\}, c_0(f^*), \bigcup\limits_{i\in SQ} c_i(f^*), \bigcup\limits_{i\in NSQ}c_i(f^*), d_0(f^*)\setminus\{0\}, \bigcup\limits_{i\in SQ} d_i$ $(f^*), \bigcup\limits_{i\in NSQ}d_i(f^*)\}$. By Corollaries \ref{co 1} and \ref{co 4}, we have that when $n$ is even, if $0\in B_+(f)$, then the dual partition of $V$ is the set of the all nonempty sets in $\{\{0\}, c_0(f^*)\setminus\{0\},\bigcup\limits_{i\in\mathbb{F}_p^*}c_i(f^*), d_0(f^*), \bigcup\limits_{i\in\mathbb{F}_p^*}d_i(f^*)\}$ and if $0\in B_-(f)$, then the dual partition of $V$ is the set of the all nonempty sets in $\{\{0\}, c_0(f^*),$ $\bigcup\limits_{i\in\mathbb{F}_p^*}c_i(f^*), d_0(f^*)\setminus\{0\}, \bigcup\limits_{i\in\mathbb{F}_p^*}d_i(f^*)\}$.
		When $n$ is odd, the dual partition of $V$ is  the set of the all nonempty sets in $\{\{0\}, N_0(f^*)\setminus\{0\}, \bigcup\limits_{i\in SQ} c_i(f^*)\bigcup\bigcup\limits_{i\in NSQ}d_i$ $(f^*), \bigcup\limits_{i\in NSQ}c_i(f^*)\bigcup\bigcup\limits_{i\in SQ} d_i(f^*)\}$. Thus, according to Lemmas \ref{le 4}, \ref{le 6} and Remark \ref{re 1}, we know that $U$ and $V$ can not induce symmetric association schemes when $f(x)$ is non-weakly regular.
		
		When $f(x)$ is weakly regular and $n$ is odd, then by Corollaries \ref{co 1} and \ref{co 4}, we know that the dual partition of $V$ is $\{\{0\}, N_0(f)\setminus\{0\}, \bigcup\limits_{i\in SQ}N_i(f), \bigcup\limits_{i\in NSQ} N_i(f)\}$. Thus, by Lemma \ref{le 6}, we know that $V$ can not induce symmetric association scheme when $f(x)$ is weakly regular and $n$ is odd. The proof is now completed.\qed
	\end{proof}
	
	\section{Conclusion}
	\quad This paper made further endeavors to construct association schemes from bent functions. Using non-weakly regular bent functions belonging to $\mathcal{DBF}$, we constructed infinite families of $2p$-class, $(2p+1)$-class and $\frac{3p+1}{2}$-class symmetric association schemes. Thus, we gave an answer to the problem proposed in \cite{Ozbudak2}. Furthermore, by fusing those association schemes, we obtained symmetric association schemes of classes $4$, $5$, $6$ and $7$. It is worth to mention that the $5$-class and $6$-class association schemes constructed in \cite{Ozbudak2} can be obtained by our constructions. In addition, we gave the sufficient and necessary conditions for the partitions $P$, $D$, $T$, $U$ and $V$ to induce symmetric association schemes. 
	
	\section*{Acknowledgements}              
	\quad This research is supported by the National Key Research and Development Program of China (Grant Nos. 2022YFA1005000
	and 2018YFA0704703), the National Natural Science Foundation of China (Grant Nos. 12141108, 62371259, 12226336), the Postdoctoral Fellowship Program of CPSF (Grant No. GZC20231177), the Fundamental Research Funds for the Central Universities of China (Nankai University), the Nankai Zhide Foundation.  
	
\newpage
\section*{Appendix. The first and second eigenmatrices, the intersection numbers and the Krein parameters of the constructed association schemes}
 \qquad Let $P_m=[P_m(ij)]$ and $Q_m=[Q_m(ij)]$ be the first and second eigenmatrices of the association schemes induced by $U_m$, where $1\le m \le 25$, and $P_m(ij)$ and $Q_m(ij)$ are the elements in the $i$-th row and $j$-the column of $P_m$ and $Q_m$, respectively. \\
1. The first and second eigenmatrices, the intersection numbers and the Krein parameters of the association scheme induced by $U_1$

The first and second eigenmatrices of the association scheme induced by $U_1$ are given in Tables 1 and 2. 
		\begin{table}[h]
		\renewcommand\arraystretch{1.5}
		\vspace{-15pt}
		\centering
		\caption{{\small The first eigenmatrix of the association scheme induced by $U_1$}}
		\resizebox{\textwidth}{!}{
			\begin{tabular}{|m{1.3cm}|c|c|m{2cm}<{\centering}|m{2cm}<{\centering}|c|m{2cm}<{\centering}|}
				\hline
				\diagbox{$i$}{$P_1(ij)$}{$j$}&1&2&3&$4\le j\le p+2$&$p+3$&$p+4\le j\le 2p+2$\\
				\hline
				1&1&$p^{n-r}-1$&$p^{r-1}-p^{n-r}+(p-1)p^{\frac{n}{2}-1}$&$p^{r-1}-p^{\frac{n}{2}-1}$&$p^{n-1}-p^{r-1}$&$p^{n-1}-p^{r-1}$\\
				\hline
				2&1&$p^{n-r}-1$&$p^{r-1}-p^{n-r}+(p-1)p^{\frac{n}{2}-1}$&$p^{r-1}-p^{\frac{n}{2}-1}$&$-p^{r-1}$&$-p^{r-1}$\\
				\hline
				3&1&$p^{n-r}-1$&$(p-1)p^{\frac{n}{2}-1}-p^{n-r}$&$-p^{\frac{n}{2}-1}$&0&0\\
				
				\hline
				$4\le i\le p+2$&1&$p^{n-r}-1$&$-p^{\frac{n}{2}-1}-p^{n-r}$&$p^{\frac{n}{2}-1}K^{(t)}(j-3,i-3)$&0&0\\
				\hline
				$p+3$&1&$-1$&0&0&$-p^{\frac{n}{2}-1}(p-1)$&$p^{\frac{n}{2}-1}$\\
				\hline
				$p+4\le i\le 2p+2$&1&$-1$&0&0&$p^{\frac{n}{2}-1}$&$-p^{\frac{n}{2}-1}K^{(t)}(j-p-3, i-p-3)$\\
				\hline	\end{tabular}}
		\vspace{-15pt}
	\end{table}
	
	\begin{table}[h]
		\renewcommand\arraystretch{1.5}
		\vspace{-15pt}
		\centering
		\caption{{\small The second eigenmatrix of the association scheme induced by $U_1$}}
		\resizebox{\textwidth}{!}{
			\begin{tabular}{|m{1.3cm}|c|c|m{2cm}<{\centering}|m{2cm}<{\centering}|c|m{2cm}<{\centering}|}
				\hline
				\diagbox{$i$}{$Q_1(ij)$}{$j$}&1&2&3&$4\le j\le p+2$&$p+3$&$p+4\le j\le 2p+2$\\
				\hline
				1&1&$p^{n-r}-1$&$p^{r-1}-p^{n-r}+(p-1)p^{\frac{n}{2}-1}$&$p^{r-1}-p^{\frac{n}{2}-1}$&$p^{n-1}-p^{r-1}$&$p^{n-1}-p^{r-1}$\\
				\hline
				2&1&$p^{n-r}-1$&$p^{r-1}-p^{n-r}+(p-1)p^{\frac{n}{2}-1}$&$p^{r-1}-p^{\frac{n}{2}-1}$&$-p^{r-1}$&$-p^{r-1}$\\
				\hline
				3&1&$p^{n-r}-1$&$(p-1)p^{\frac{n}{2}-1}-p^{n-r}$&$-p^{\frac{n}{2}-1}$&0&0\\
				
				\hline
				$4\le i\le p+2$&1&$p^{n-r}-1$&$-p^{\frac{n}{2}-1}-p^{n-r}$&$p^{\frac{n}{2}-1}K^{(h)}(j-3,i-3)$&0&0\\
				\hline
				$p+3$&1&$-1$&0&0&$-p^{\frac{n}{2}-1}(p-1)$&$p^{\frac{n}{2}-1}$\\
				\hline
				$p+4\le i\le 2p+2$&1&$-1$&0&0&$p^{\frac{n}{2}-1}$&$-p^{\frac{n}{2}-1}K^{(h)}(j-p-3, i-p-3)$\\
				\hline	\end{tabular}}
		\vspace{-15pt}
	\end{table}

Since $p_{u,v}^{w}=p_{v,u}^{w}$ for any $u,v,w\in\{0,1,\dots,2p+1\}$, we only give the values of the  intersection numbers $p_{u,v}^w$ for $u\le v$ in the following eight cases. 
\begin{itemize}
	\item [$(\rm{\romannumeral1})$] $w=0$ and $0\le u\le v\le 2p+1$
	\begin{itemize}
		\item [$\bullet$] $u=v$\\
		$p_{0,0}^0=1$, $p_{1,1}^0=p^{n-r}-1$, $p_{2,2}^0=p^{r-1}+(p-1)p^{\frac{n}{2}-1}-p^{n-r}$, $p_{u,u}^0=p^{r-1}-p^{\frac{n}{2}-1}$ for $3\le u\le p+1$ and $p_{u,u}^0=p^{n-1}-p^{r-1}$ for $p+2\le u\le 2p+1$.
		\item [$\bullet$] $u\ne v$, $p_{u,v}^0=0$.
	\end{itemize}
	
	\item [$(\rm{\rm{\romannumeral2})}$] $u=0$, $1\le w\le 2p+1$ and $0\le v\le 2p+1$\\
	$p_{0,v}^{w}=\begin{cases}
	1,&\text{if}\ w=v,\\
	0,&\text{if}\ w\ne v.
	\end{cases}$
	\item[$(\rm{\rm{\romannumeral3}})$] $1\le w \le p+1$ and $1\le u\le v\le p+1$ 
	\begin{itemize}
		\item [$\bullet$] $w=1$\\
		$p_{u,v}^1=\begin{cases}
		p^{n-r}-2,&\text{if}\ u=v=1,\\
		0,&\text{if} \ u=1\ \text{and} \ v\ge2,\\
		p^{r-1}+(p-1)p^{\frac{n}{2}-1}-p^{n-r},&\text{if}\ u=v=2,\\
		p^{r-1}-p^{\frac{n}{2}-1},&\text{if}\ u\ge 3\ \text{and}\ u=v,\\
		0,&\text{if}\  u\ge 2\ \text{and}\ u\ne v.
		\end{cases}$
		\item[$\bullet$] $w=2$\\
		$p_{u,v}^2=\begin{cases}
		p^{n-r}-1,&\text{if}\ u=1\ \text{and}\ v=2,\\
		0,&\text{if}\ u=1\ \text{and}\ v\ne2,\\
		p^{r-2}+p^{\frac{n}{2}}-2p^{n-r}-p^{\frac{n}{2}-1},&\text{if}\ u=v=2,\\
		p^{r-2},&\text{if}\ u=2\ \text{and}\ v\ge 3,\\
		p^{r-2}-p^{\frac{n}{2}-1},&\text{if}\ u\ge3\ \text{and}\ u=v,\\
		p^{r-2},&\text{if}\ u\ge3\ \text{and}\ u\ne v.
		\end{cases}$
		\item[$\bullet$] $3\le w\le p+1$\\
		$p_{u,v}^w=\begin{cases}
		p^{n-r}-1,&\hspace{-8pt}\text{if}\ u=1\ \text{and}\ v=w,\\
		0,&\hspace{-8pt}\text{if}\ u=1\ \text{and}\ v\ne w ,\\
		p^{r-2}+p^{\frac{n}{2}-1},&\hspace{-8pt}\text{if}\ u=v=2,\\
		p^{r-2}-p^{\frac{n}{2}-2}+p^{\frac{n}{2}
			-1}-p^{n-r-1}-(p^{\frac{n}{2}-3}+\\p^{n-r-2})\sum\limits_{i\in\mathbb{F}_p}K^{(h)}(i,w-2)K^{(t)}(v-2,i),&\hspace{-8pt}\text{if}\ u=2\ \text{and}\ v\ge3,\\
		p^{r-2}-2p^{\frac{n}{2}-2}+p^{\frac{n}{2}-3}\sum\limits_{i\in\mathbb{F}_p}K^{(h)}(i,w-2)\times\\K^{(t)}(u-2,i)K^{(t)}(v-2,i),&\hspace{-8pt}\text{if}\ 3\le u\le v.
		\end{cases}$
	\end{itemize}
	\item[$(\rm{\rm{\romannumeral4}})$] $1\le w\le p+1$, $1\le u\le p+1$ and $p+2\le v\le 2p+1$, $p_{u,v}^w=0$.
	\item[$(\rm{\rm{\romannumeral5}})$] $1\le w \le p+1$ and $p+2\le u\le v\le 2p+1$\\
	$p_{u,v}^w=\begin{cases}
	p^{n-2}-p^{r-1},&\text{if}\ w=1\ \text{and}\ u=v,\\
	p^{n-2},&\text{if}\ w=1\ \text{and}\ u\ne v,\\
	p^{n-2}-p^{r-2},&\text{if}\ w\ge2.
	\end{cases}$
	\item[$(\rm{\romannumeral6})$] $p+2\le w\le 2p+1$ and $1\le u\le v\le p+1$, $p_{u,v}^w=0$.
	\item[$(\rm{\romannumeral7})$] $p+2\le w \le 2p+1$, $1\le u\le p+1$ and $p+2\le v\le 2p+1$\\
	$p_{u,v}^w=\begin{cases}
	p^{n-r-1}-p^{-1}-p^{-2}\sum\limits_{i\in\mathbb{F}_p}K^{(h)}(i,w-p-2)\times\\
	K^{(t)}(v-p-2,i),&\text{if}\ u=1,\\
	p^{r-2}+(p-1)p^{\frac{n}{2}-2}-p^{n-r-1},&\text{if}\ u=2,\\
	p^{r-2}-p^{\frac{n}{2}-2},&\text{if}\ u\ge 3.
	\end{cases}$
	
	\item[$(\rm{\romannumeral8})$] $p+2\le w\le 2p+1$ and $p+2\le u\le v\le 2p+1$\\
	$p_{u,v}^w=p^{n-2}-2p^{r-2}-p^{\frac{n}{2}-3}\sum\limits_{i\in\mathbb{F}_p}K^{(h)}(i,w-p-2)K^{(t)}(u-p-2,i)K^{(t)}(v-p-2,i).$
\end{itemize}

Since $q_{u,v}^{w}=q_{v,u}^{w}$ for any $u,v,w\in\{0,1,\dots,2p+1\}$, we only give the values of the  Krein parameters $q_{u,v}^w$ for $u\le v$ in the following eight cases.
\begin{itemize}
	\item [$(\rm{\romannumeral1})$] $w=0$ and $0\le u\le v\le 2p+1$
	\begin{itemize}
		\item [$\bullet$] $u=v$\\
		$q_{0,0}^0=1$, $q_{1,1}^0=p^{n-r}-1$, $q_{2,2}^0=p^{r-1}+(p-1)p^{\frac{n}{2}-1}-p^{n-r}$, $q_{u,u}^0=p^{r-1}-p^{\frac{n}{2}-1}$ for $3\le u\le p+1$ and $q_{u,u}^0=p^{n-1}-p^{r-1}$ for $p+2\le u\le 2p+1$.
		\item [$\bullet$] $u\ne v$, $q_{u,v}^0=0$.
	\end{itemize}
	
	\item [$(\rm{\rm{\romannumeral2})}$] $u=0$, $1\le w\le 2p+1$ and $0\le v\le 2p+1$\\
	$q_{0,v}^{w}=\begin{cases}
	1,&\text{if}\ w=v,\\
	0,&\text{if}\ w\ne v.
	\end{cases}$
	\item[$(\rm{\rm{\romannumeral3}})$] $1\le w \le p+1$ and $1\le u\le v\le p+1$ 
	\begin{itemize}
		\item [$\bullet$] $w=1$\\
		$q_{u,v}^1=\begin{cases}
		p^{n-r}-2,&\text{if}\ u=v=1,\\
		0,&\text{if} \ u=1\ \text{and} \ v\ge2,\\
		p^{r-1}+(p-1)p^{\frac{n}{2}-1}-p^{n-r},&\text{if}\ u=v=2,\\
		p^{r-1}-p^{\frac{n}{2}-1},&\text{if}\ u\ge 3\ \text{and}\ u=v,\\
		0,&\text{if}\  u\ge 2\ \text{and}\ u\ne v.
		\end{cases}$
		\item[$\bullet$] $w=2$\\
		$q_{u,v}^2=\begin{cases}
		p^{n-r}-1,&\text{if}\ u=1\ \text{and}\ v=2,\\
		0,&\text{if}\ u=1\ \text{and}\ v\ne2,\\
		p^{r-2}+p^{\frac{n}{2}}-2p^{n-r}-p^{\frac{n}{2}-1},&\text{if}\ u=v=2,\\
		p^{r-2},&\text{if}\ u=2\ \text{and}\ v\ge 3,\\
		p^{r-2}-p^{\frac{n}{2}-1},&\text{if}\ u\ge3\ \text{and}\ u=v,\\
		p^{r-2},&\text{if}\ u\ge3\ \text{and}\ u\ne v.
		\end{cases}$
		\item[$\bullet$] $3\le w\le p+1$\\
		$q_{u,v}^w=\begin{cases}
		p^{n-r}-1,&\hspace{-8pt}\text{if}\ u=1\ \text{and}\ v=w,\\
		0,&\hspace{-8pt}\text{if}\ u=1\ \text{and}\ v\ne w ,\\
		p^{r-2}+p^{\frac{n}{2}-1},&\hspace{-8pt}\text{if}\ u=v=2,\\
		p^{r-2}-p^{\frac{n}{2}-2}+p^{\frac{n}{2}
			-1}-p^{n-r-1}-(p^{\frac{n}{2}-3}+\\p^{n-r-2})\sum\limits_{i\in\mathbb{F}_p}K^{(t)}(i,w-2)K^{(h)}(v-2,i),&\hspace{-8pt}\text{if}\ u=2\ \text{and}\ v\ge3,\\
		p^{r-2}-2p^{\frac{n}{2}-2}+p^{\frac{n}{2}-3}\sum\limits_{i\in\mathbb{F}_p}K^{(t)}(i,w-2)\times\\K^{(h)}(u-2,i)K^{(h)}(v-2,i),&\hspace{-8pt}\text{if}\ 3\le u\le v.
		\end{cases}$
	\end{itemize}
	\item[$(\rm{\rm{\romannumeral4}})$] $1\le w\le p+1$, $1\le u\le p+1$ and $p+2\le v\le 2p+1$, $q_{u,v}^w=0$.
	\item[$(\rm{\rm{\romannumeral5}})$] $1\le w \le p+1$ and $p+2\le u\le v\le 2p+1$\\
	\newpage
	$q_{u,v}^w=\begin{cases}
	p^{n-2}-p^{r-1},&\text{if}\ w=1\ \text{and}\ u=v,\\
	p^{n-2},&\text{if}\ w=1\ \text{and}\ u\ne v,\\
	p^{n-2}-p^{r-2},&\text{if}\ w\ge2.
	\end{cases}$
	\item[$(\rm{\romannumeral6})$] $p+2\le w\le 2p+1$ and $1\le u\le v\le p+1$, $q_{u,v}^w=0$.
	\item[$(\rm{\romannumeral7})$] $p+2\le w \le 2p+1$, $1\le u\le p+1$ and $p+2\le v\le 2p+1$\\
	$q_{u,v}^w=\begin{cases}
	p^{n-r-1}-p^{-1}-p^{-2}\sum\limits_{i\in\mathbb{F}_p}K^{(t)}(i,w-p-2)\times\\
	K^{(h)}(v-p-2,i),&\text{if}\ u=1,\\
	p^{r-2}+(p-1)p^{\frac{n}{2}-2}-p^{n-r-1},&\text{if}\ u=2,\\
	p^{r-2}-p^{\frac{n}{2}-2},&\text{if}\ u\ge 3.
	\end{cases}$
	
	\item[$(\rm{\romannumeral8})$] $p+2\le w\le 2p+1$ and $p+2\le u\le v\le 2p+1$\\
	$q_{u,v}^w=p^{n-2}-2p^{r-2}-p^{\frac{n}{2}-3}\sum\limits_{i\in\mathbb{F}_p}K^{(t)}(i,w-p-2)K^{(h)}(u-p-2,i)K^{(h)}(v-p-2,i).$
\end{itemize}
2. The first and second eigenmatrices, the intersection numbers and the Krein parameters of the association scheme induced by  $U_2$

The first and second eigenmatrices of the association scheme induced by $U_2$ are given in Tables 3 and 4.
\begin{table}[h]
	\vspace{-15pt}
	\centering
	\caption{The first eigenmatrix of the association scheme induced by $U_2$}
	\renewcommand\arraystretch{1.5}	
	\resizebox{\textwidth}{!}{
		\begin{tabular}{|m{1.3cm}|c|c|m{3cm}<{\centering}|c|m{2.5cm}<{\centering}|}
			\hline
			\diagbox{$i$}{$P_2(ij)$}{$j$}&1&2&$3\le j\le p+1$&$p+2$&$p+3\le j\le 2p+1$\\
			\hline
			1&	1&$p^{n-1}-p^{\frac{n}{2}}$&$p^{n-1}-p^{\frac{n}{2}}$&$p^{\frac{n}{2}-1}-1$&$p^{\frac{n}{2}}+p^{\frac{n}{2}-1}$\\
			\hline
			2&1&$(p-1)p^{\frac{n}{2}-1}$&$-p^{\frac{n}{2}-1}$&$-1$&0\\
			\hline
			$3\le i\le p+1$&1&$-p^{\frac{n}{2}-1}$&$p^{\frac{n}{2}-1}K^{(t)}(j-2,i-2)$&$-1$&0\\
			\hline
			$p+2$&1&$-p^{\frac{n}{2}}$&$-p^{\frac{n}{2}}$&$p^{\frac{n}{2}-1}-1$&$p^{\frac{n}{2}-1}+p^{\frac{n}{2}}$\\
			\hline
			$p+3\le i\le 2p+1$&1&0&0&$p^{\frac{n}{2}-1}-1$&$-p^{\frac{n}{2}-1}K^{(t)}(j-p-2,i-p-2)$\\
			\hline
	\end{tabular}}
	\vspace{-15pt}
\end{table}	
\begin{table}[h]
	\vspace{-15pt}
	\centering
	\caption{The second eigenmatrix of the association scheme induced by $U_2$}
	\renewcommand\arraystretch{1.5}	
	\resizebox{\textwidth}{!}{
		\begin{tabular}{|m{1.3cm}|c|c|m{3cm}<{\centering}|c|m{2.5cm}<{\centering}|}
			\hline
			\diagbox{$i$}{$Q_2(ij)$}{$j$}&1&2&$3\le j\le p+1$&$p+2$&$p+3\le j\le 2p+1$\\
			\hline
			1&	1&$p^{n-1}-p^{\frac{n}{2}}$&$p^{n-1}-p^{\frac{n}{2}}$&$p^{\frac{n}{2}-1}-1$&$p^{\frac{n}{2}}+p^{\frac{n}{2}-1}$\\
			\hline
			2&1&$(p-1)p^{\frac{n}{2}-1}$&$-p^{\frac{n}{2}-1}$&$-1$&0\\
			\hline
			$3\le i\le p+1$&1&$-p^{\frac{n}{2}-1}$&$p^{\frac{n}{2}-1}K^{(h)}(j-2,i-2)$&$-1$&0\\
			\hline
			$p+2$&1&$-p^{\frac{n}{2}}$&$-p^{\frac{n}{2}}$&$p^{\frac{n}{2}-1}-1$&$p^{\frac{n}{2}-1}+p^{\frac{n}{2}}$\\
			\hline
			$p+3\le i\le 2p+1$&1&0&0&$p^{\frac{n}{2}-1}-1$&$-p^{\frac{n}{2}-1}K^{(h)}(j-p-2,i-p-2)$\\
			\hline
	\end{tabular}}
	\vspace{-15pt}
\end{table}	

Since $p_{u,v}^{w}=p_{v,u}^{w}$ for any $u,v,w\in\{0,1,\dots,2p\}$, we only give the values of the  intersection numbers $p_{u,v}^w$ for $u\le v$ in the following eight cases. 
\begin{itemize}
	\item [$(\rm{\romannumeral1})$] $w=0$ and $0\le u\le v\le 2p$
	\begin{itemize}
		\item [$\bullet$] $u=v$\\
		$p_{0,0}^0=1$, $p_{1,1}^0=p^{n-1}-p^{\frac{n}{2}}$, $p_{u,u}^0=p^{n-1}-p^{\frac{n}{2}}$ for $2\le u\le p$, $p_{p+1,p+1}^0=p^{\frac{n}{2}-1}-1$ and $p_{u,u}^0=p^{\frac{n}{2}}+p^{\frac{n}{2}-1}$ for $p+2\le u\le 2p$.
		\item[$\bullet$] $u\ne v$, $p_{u,v}^0=0$.
	\end{itemize}
	
\item [$(\rm{\romannumeral2})$] $u=0$, $1\le w\le 2p$ and $0\le v\le 2p$\\
$p_{0,v}^w=\begin{cases}
1,&\text{if}\ w=v,\\
0,&\text{if}\ w\ne v.
\end{cases}$
\item [$(\rm{\romannumeral3})$] $1\le w\le p$ and $1\le u\le v\le p$\\
$p_{u,v}^w=p^{n-2}-2p^{\frac{n}{2}-1}+p^{\frac{n}{2}-3}\sum\limits_{i\in\mathbb{F}_p}K^{(h)}(i,w-1)K^{(t)}(u-1,i)K^{(t)}(v-1,i).$
\item [$(\rm{\romannumeral4})$] $1\le w\le p$, $1\le u\le p$ and $p+1\le v\le 2p$\\
$p_{u,v}^w=\begin{cases}
p^{\frac{n}{2}-2}-p^{-1}-p^{-2}\sum\limits_{i\in\mathbb{F}_p}K^{(h)}(i,w-1)K^{(t)}(u-1,i),&\text{if}\ v=p+1,\\
p^{\frac{n}{2}-1}+p^{\frac{n}{2}-2},&\text{if}\ v\ge p+2.
\end{cases}$
\item [$(\rm{\romannumeral5})$] $1\le w\le p$ and $p+1\le u\le v\le 2p$, $p_{u,v}^w=0$.\\

\item [$(\romannumeral6)$] $p+1\le w\le 2p$ and $1\le u\le v\le p$\\
$p_{u,v}^w=\begin{cases}
p^{n-2}-p^{\frac{n}{2}},&\text{if}\ w=p+1\ \text{and}\ u=v,\\
p^{n-2},&\text{if}\ w=p+1\ \text{and}\ u\ne v,\\
p^{n-2}-p^{\frac{n}{2}-1},&\text{if}\ w\ge p+2.
\end{cases}$
\item [$(\rm{\romannumeral7})$] $p+1\le w\le 2p$, $1\le u\le p$ and $p+1\le v\le 2p$, $p_{u,v}^w=0$.
\item [$(\rm{\romannumeral8})$] $p+1\le w\le 2p$ and $p+1\le u\le v\le 2p$

\begin{itemize}
	\item [$\bullet$] $w=p+1$\\
	$p_{u,v}^{p+1}=\begin{cases}
	p^{\frac{n}{2}-1}-2,&\text{if}\ u=v=p+1,\\
	0,&\text{if}\ u=p+1\ \text{and}\ v\ge p+2,\\
	p^{\frac{n}{2}}+p^{\frac{n}{2}-1},&\text{if}\ u\ge p+2\ \text{and}\ u=v,\\
	0,&\text{if}\ u\ge p+2\ \text{and}\ u\ne v.
	\end{cases}$
	\item [$\bullet$] $p+2\le w \le 2p$\\
$p_{u,v}^w=\begin{cases}
0,&\hspace{-8pt}\text{if}\ u=v=p+1,\\
p^{\frac{n}{2}-2}+p^{\frac{n}{2}-3}-p^{-1}-p^{-2}+(p^{\frac{n}{2}-3}-p^{-2})&\hspace{-8pt}\text{if}\ u=p+1\ \text{and}\\
\sum\limits_{i\in\mathbb{F}_p^*}K^{(h)}(i,w-p-1)K^{(t)}(v-p-1,i),&\hspace{-8pt} v\ge p+2,\\
p^{\frac{n}{2}-1}+2p^{\frac{n}{2}-2}+p^{\frac{n}{2}-3}-p^{\frac{n}{2}-3}\sum\limits_{i\in\mathbb{F}_p^*}K^{(h)}(i,\\
w-p-1)K^{(t)}(u-p-1,i)K^{(t)}(v-p-1,i),&\hspace{-8pt}\text{if}\ p+2\le u\le v.
\end{cases}$
\end{itemize}
\end{itemize}

Since $q_{u,v}^{w}=q_{v,u}^{w}$ for any $u,v,w\in\{0,1,\dots,2p\}$, we only give the values of the  Krein parameters $q_{u,v}^w$ for $u\le v$ in the following eight cases.
\begin{itemize}
	\item [$(\rm{\romannumeral1})$] $w=0$ and $0\le u\le v\le 2p$
	\begin{itemize}
		\item [$\bullet$] $u=v$\\
		$q_{0,0}^0=1$, $q_{1,1}^0=p^{n-1}-p^{\frac{n}{2}}$, $q_{u,u}^0=p^{n-1}-p^{\frac{n}{2}}$ for $2\le u\le p$, $q_{p+1,p+1}^0=p^{\frac{n}{2}-1}-1$ and $q_{u,u}^0=p^{\frac{n}{2}}+p^{\frac{n}{2}-1}$ for $p+2\le u\le 2p$.
		\item[$\bullet$] $u\ne v$, $q_{u,v}^0=0$.
	\end{itemize}
	
	\item [$(\rm{\romannumeral2})$] $u=0$, $1\le w\le 2p$ and $0\le v\le 2p$\\
	$q_{0,v}^w=\begin{cases}
	1,&\text{if}\ w=v,\\
	0,&\text{if}\ w\ne v.
	\end{cases}$
	\item [$(\rm{\romannumeral3})$] $1\le w\le p$ and $1\le u\le v\le p$\\
	$q_{u,v}^w=p^{n-2}-2p^{\frac{n}{2}-1}+p^{\frac{n}{2}-3}\sum\limits_{i\in\mathbb{F}_p}K^{(t)}(i,w-1)K^{(h)}(u-1,i)K^{(h)}(v-1,i).$
	\item [$(\rm{\romannumeral4})$] $1\le w\le p$, $1\le u\le p$ and $p+1\le v\le 2p$\\
	$q_{u,v}^w=\begin{cases}
	p^{\frac{n}{2}-2}-p^{-1}-p^{-2}\sum\limits_{i\in\mathbb{F}_p}K^{(t)}(i,w-1)K^{(h)}(u-1,i),&\text{if}\ v=p+1,\\
	p^{\frac{n}{2}-1}+p^{\frac{n}{2}-2},&\text{if}\ v\ge p+2.
	\end{cases}$
	\item [$(\rm{\romannumeral5})$] $1\le w\le p$ and $p+1\le u\le v\le 2p$, $q_{u,v}^w=0$.\\
	
	\item [$(\romannumeral6)$] $p+1\le w\le 2p$ and $1\le u\le v\le p$\\
	$q_{u,v}^w=\begin{cases}
	p^{n-2}-p^{\frac{n}{2}},&\text{if}\ w=p+1\ \text{and}\ u=v,\\
	p^{n-2},&\text{if}\ w=p+1\ \text{and}\ u\ne v,\\
	p^{n-2}-p^{\frac{n}{2}-1},&\text{if}\ w\ge p+2.
	\end{cases}$
	\item [$(\rm{\romannumeral7})$] $p+1\le w\le 2p$, $1\le u\le p$ and $p+1\le v\le 2p$, $q_{u,v}^w=0$.
	\item [$(\rm{\romannumeral8})$] $p+1\le w\le 2p$ and $p+1\le u\le v\le 2p$
	
	\begin{itemize}
		\item [$\bullet$] $w=p+1$\\
		$q_{u,v}^{p+1}=\begin{cases}
		p^{\frac{n}{2}-1}-2,&\text{if}\ u=v=p+1,\\
		0,&\text{if}\ u=p+1\ \text{and}\ v\ge p+2,\\
		p^{\frac{n}{2}}+p^{\frac{n}{2}-1},&\text{if}\ u\ge p+2\ \text{and}\ u=v,\\
		0,&\text{if}\ u\ge p+2\ \text{and}\ u\ne v.
		\end{cases}$
		\item [$\bullet$] $p+2\le w \le 2p$\\
		$q_{u,v}^w=\begin{cases}
		0,&\hspace{-10pt}\text{if}\ u=v=p+1,\\
		p^{\frac{n}{2}-2}+p^{\frac{n}{2}-3}-p^{-1}-p^{-2}+(p^{\frac{n}{2}-3}-&\hspace{-10pt}\text{if}\ u=p+1\ \text{and}\\
		p^{-2})\sum\limits_{i\in\mathbb{F}_p^*}K^{(t)}(i,w-p-1)K^{(h)}(v-p-1,i),&\hspace{-10pt}\ v\ge p+2,\\
		p^{\frac{n}{2}-1}+2p^{\frac{n}{2}-2}+p^{\frac{n}{2}-3}-p^{\frac{n}{2}-3}\sum\limits_{i\in\mathbb{F}_p^*}K^{(t)}(i,w-\\
		p-1)K^{(h)}(u-p-1,i)K^{(h)}(v-p-1,i),&\hspace{-10pt}\text{if}\ p+2\le u\le v.
		\end{cases}$
	\end{itemize}
\end{itemize}
3. The first and second eigenmatrices, the intersection numbers and the Krein parameters of the association scheme induced by $U_3$

The first and second eigenmatrices of the association scheme induced by $U_3$ are given in Tables 5 and 6. 
\begin{table}[h]
	\vspace{0pt}
	\centering
	\caption{The first eigenmatrix of the association scheme induced by $U_3$}
	\renewcommand\arraystretch{1.5}	
	\resizebox{\textwidth}{!}{
		\begin{tabular}{|m{1.3cm}|c|c|m{2cm}<{\centering}|c|m{2cm}<{\centering}|m{2cm}<{\centering}|}
			\hline
			\diagbox{$i$}{$P_3(ij)$}{$j$}&1&2&$3\le j\le p+1$&$p+2$&$p+3$&$p+4\le j\le 2p+2$\\
			\hline
			1&1&$p^{n-1}-p^{r-1}$&$p^{n-1}-p^{r-1}$&$p^{n-r}-1$&$p^{r-1}-p^{n-r}-(p-1)p^{\frac{n}{2}-1}$&$p^{r-1}+p^{\frac{n}{2}-1}$\\
			\hline
			2&1&$(p-1)p^{\frac{n}{2}-1}$&$-p^{\frac{n}{2}-1}$&$-1$&0&0\\
			\hline
			$3\le i\le p+1$&1&$-p^{\frac{n}{2}-1}$&$p^{\frac{n}{2}-1}K^{(t)}(j-2,i-2)$&$-1$&0&0\\
			\hline
			$p+2$&1&$-p^{r-1}$&$-p^{r-1}$&$p^{n-r}-1$&$p^{r-1}-p^{n-r}-(p-1)p^{\frac{n}{2}-1}$&$p^{\frac{n}{2}-1}+p^{r-1}$\\
			\hline
			$p+3$&1&0&0&$p^{n-r}-1$&$-p^{n-r}-(p-1)p^{\frac{n}{2}-1}$&$p^{\frac{n}{2}-1}$\\
			\hline
			$p+4\le i\le 2p+2$&1&0&0&$p^{n-r}-1$&$p^{\frac{n}{2}-1}-p^{n-r}$&$-p^{\frac{n}{2}-1}K^{(t)}(j-p-3,i-p-3)$\\
			\hline
	\end{tabular}}
	\vspace{0pt}
\end{table}
\begin{table}[h]
	\vspace{-8pt}
	\centering
	\caption{The second eigenmatrix of the association scheme induced by $U_3$}
	\renewcommand\arraystretch{1.5}	
	\resizebox{\textwidth}{!}{
		\begin{tabular}{|m{1.3cm}|c|c|m{2cm}<{\centering}|c|m{2cm}<{\centering}|m{2cm}<{\centering}|}
			\hline
			\diagbox{$i$}{$Q_3(ij)$}{$j$}&1&2&$3\le j\le p+1$&$p+2$&$p+3$&$p+4\le j\le 2p+2$\\
			\hline
			1&1&$p^{n-1}-p^{r-1}$&$p^{n-1}-p^{r-1}$&$p^{n-r}-1$&$p^{r-1}-p^{n-r}-(p-1)p^{\frac{n}{2}-1}$&$p^{r-1}+p^{\frac{n}{2}-1}$\\
			\hline
			2&1&$(p-1)p^{\frac{n}{2}-1}$&$-p^{\frac{n}{2}-1}$&$-1$&0&0\\
			\hline
			$3\le i\le p+1$&1&$-p^{\frac{n}{2}-1}$&$p^{\frac{n}{2}-1}K^{(h)}(j-2,i-2)$&$-1$&0&0\\
			\hline
			$p+2$&1&$-p^{r-1}$&$-p^{r-1}$&$p^{n-r}-1$&$p^{r-1}-p^{n-r}-(p-1)p^{\frac{n}{2}-1}$&$p^{\frac{n}{2}-1}+p^{r-1}$\\
			\hline
			$p+3$&1&0&0&$p^{n-r}-1$&$-p^{n-r}-(p-1)p^{\frac{n}{2}-1}$&$p^{\frac{n}{2}-1}$\\
			\hline
			$p+4\le i\le 2p+2$&1&0&0&$p^{n-r}-1$&$p^{\frac{n}{2}-1}-p^{n-r}$&$-p^{\frac{n}{2}-1}K^{(h)}(j-p-3,i-p-3)$\\
			\hline
	\end{tabular}}
	\vspace{-15pt}
\end{table}

Since $p_{u,v}^{w}=p_{v,u}^{w}$ for any $u,v,w\in\{0,1,\dots,2p+1\}$, we only give the values of the  intersection numbers  $p_{u,v}^w$ for $u\le v$ in the following eight cases. 
	
\begin{itemize}
	\item [$(\rm{\romannumeral1})$] $w=0$ and $0\le u\le v\le 2p+1$
	\begin{itemize}
		\item [$\bullet$] $u=v$\\
		$p_{0,0}^0=1$, $p_{u,u}^0=p^{n-1}-p^{r-1}$ for $1\le u\le p$, $p_{p+1,p+1}^0=p^{n-r}-1$, $p_{p+2,p+2}^0=p^{r-1}-(p-1)p^{\frac{n}{2}-1}-p^{n-r}$ and $p_{u,u}^0=p^{r-1}+p^{\frac{n}{2}-1}$ for $p+3\le u\le 2p+1$.
		\item[$\bullet$] $u\ne v$, $p_{u,v}^0=0$.
	\end{itemize}
	\item [$(\rm{\romannumeral2})$] $u=0$, $1\le w\le 2p+1$ and $0 \le v\le 2p+1$\\
	$p_{u,v}^w=\begin{cases}
	1,&\text{if}\ w=v,\\
	0,&\text{if}\ w\ne v.
	\end{cases}$
	\item [$(\rm{\romannumeral3})$]
	$1\le w\le p$ and $1\le u\le v\le p$\\
	$p_{u,v}^w=p^{n-2}-2p^{r-2}+p^{\frac{n}{2}-3}\sum\limits_{i\in\mathbb{F}_p}K^{(h)}(i,w-1)K^{(t)}(u-1,i)K^{(t)}(v-1,i).$
	\item [$(\rm{\romannumeral4})$] $1\le w\le p$, $1\le u\le p$ and $p+1\le v\le 2p+1$\\
		$p_{u,v}^w=\begin{cases}
p^{n-r-1}-p^{-1}-p^{-2}\sum\limits_{i\in\mathbb{F}_p}K^{(h)}(i,w-1)K^{(t)}(u-1,i),&\text{if}\ v=p+1,\\
p^{r-2}-(p-1)p^{\frac{n}{2}-2}-p^{n-r-1},&\text{if}\ v=p+2,\\
p^{r-2}+p^{\frac{n}{2}-2},&\text{if}\ v\ge p+3.
		\end{cases}$
\item [$(\rm{\romannumeral5})$]$1\le w\le p$ and $p+1\le u\le v\le2p+1$, $p_{u,v}^w=0.$
\item [$(\rm{\romannumeral6})$] $p+1\le w\le 2p+1$ and $1\le u\le v\le p$\\
$p_{u,v}^w=\begin{cases}
p^{n-2}-p^{r-1},&\text{if}\ w=p+1\ \text{and}\ u=v,\\
p^{n-2},&\text{if}\ w=p+1\ \text{and}\ u\ne v,\\
p^{n-2}-p^{r-2},&\text{if}\ w\ge p+2.
\end{cases}$
\item [$(\rm{\romannumeral7})$] $p+1\le w\le 2p+1$, $1\le u\le p$ and $p+1\le v\le 2p+1$, $p_{u,v}^w=0.$
\item [$(\rm{\romannumeral8})$] $p+1\le w\le 2p+1$ and $p+1\le u\le v\le 2p+1$
\begin{itemize}
	\item [$\bullet$] $w=p+1$\\
	$p_{u,v}^{p+1}=\begin{cases}
	p^{n-r}-2,&\text{if}\ u=v=p+1,\\
	0,&\text{if}\ u=p+1\ \text{and}\ v\ge p+2,\\
	p^{r-1}-(p-1)p^{\frac{n}{2}-1}-p^{n-r},&\text{if}\ u=v=p+2,\\
	p^{r-1}+p^{\frac{n}{2}-1},&\text{if}\ u\ge p+3\ \text{and}\ u=v,\\
	0,&\text{if}\ u\ge p+2\ \text{and}\ u\ne v.
	\end{cases}$
	\item [$\bullet$] $w=p+2$\\
	$p_{u,v}^{p+2}=\begin{cases}
p^{n-r}-1,&\text{if}\ u=p+1\ \text{and}\ v=p+2,\\
0,&\text{if}\ u=p+1\ \text{and}\ v\ne p+2,\\
p^{r-2}+p^{\frac{n}{2}-1}-2p^{n-r}-p^{\frac{n}{2}},&\text{if}\ u=v=p+2,\\
p^{r-2},&\text{if}\ u=p+2\ \text{and}\ v\ge p+3,\\
p^{r-2}+p^{\frac{n}{2}-1},&\text{if}\ u\ge p+3\ \text{and}\ u=v,\\
p^{r-2},&\text{if}\ u\ge p+3\ \text{and}\ u\ne v.
	\end{cases}$
	\item [$\bullet$] $p+3\le w\le 2p+1$\\
	$p_{u,v}^w=\begin{cases}
p^{n-r}-1,&\hspace{-12pt}\text{if}\ u=p+1\ \text{and}\\&\hspace{-12pt} v=w,\\
0,&\hspace{-12pt}\text{if}\ u=p+1\ \text{and}\\&\hspace{-12pt} v\ne w,\\
p^{r-2}-p^{\frac{n}{2}-1},&\hspace{-12pt}\text{if}\ u=v=p+2,\\
p^{r-2}-(p-1)p^{\frac{n}{2}
-2}-p^{n-r-1}+(p^{\frac{n}{2}-3}-&\hspace{-12pt}\text{if}\ u=p+2\ \text{and}\\p^{n-r-2})\sum\limits_{i\in\mathbb{F}_p}K^{(h)}(i,w-p-2)K^{(t)}(v-p-2,i),&\hspace{-10pt} v\ge p+3,\\
p^{r-2}+2p^{\frac{n}{2}-2}-p^{\frac{n}{2}-3}\sum\limits_{i\in\mathbb{F}_p}K^{(h)}(i,w-p-2)\times\\
K^{(t)}(u-p-2,i)K^{(t)}(v-p-2,i)&\hspace{-12pt}\text{if}\ p+3\le u\le v.
	\end{cases}$
\end{itemize}
\end{itemize}
	
Since $q_{u,v}^{w}=q_{v,u}^{w}$ for any $u,v,w\in\{0,1,\dots,2p+1\}$, we only give the values of the  Krein parameters  $q_{u,v}^w$ for $u\le v$ in the following eight cases.

\begin{itemize}
	\item [$(\rm{\romannumeral1})$] $w=0$ and $0\le u\le v\le 2p+1$
	\begin{itemize}
		\item [$\bullet$] $u=v$\\
		$q_{0,0}^0=1$, $q_{u,u}^0=p^{n-1}-p^{r-1}$ for $1\le u\le p$, $q_{p+1,p+1}^0=p^{n-r}-1$, $q_{p+2,p+2}^0=p^{r-1}-(p-1)p^{\frac{n}{2}-1}-p^{n-r}$ and $q_{u,u}^0=p^{r-1}+p^{\frac{n}{2}-1}$ for $p+3\le u\le 2p+1$.
		\item[$\bullet$] $u\ne v$, $q_{u,v}^0=0$.
	\end{itemize}
	\item [$(\rm{\romannumeral2})$] $u=0$, $1\le w\le 2p+1$ and $0\le v\le 2p+1$\\
	\newpage
	$q_{u,v}^w=\begin{cases}
	1,&\text{if}\ w=v,\\
	0,&\text{if}\ w\ne v.
	\end{cases}$
	\item [$(\rm{\romannumeral3})$]
	$1\le w\le p$ and $1\le u\le v\le p$\\
	$q_{u,v}^w=p^{n-2}-2p^{r-2}+p^{\frac{n}{2}-3}\sum\limits_{i\in\mathbb{F}_p}K^{(t)}(i,w-1)K^{(h)}(u-1,i)K^{(h)}(v-1,i).$
	\item [$(\rm{\romannumeral4})$] $1\le w\le p$, $1\le u\le p$ and $p+1\le v\le 2p+1$\\
	$q_{u,v}^w=\begin{cases}
	p^{n-r-1}-p^{-1}-p^{-2}\sum\limits_{i\in\mathbb{F}_p}K^{(t)}(i,w-1)K^{(h)}(u-1,i),&\text{if}\ v=p+1,\\
	p^{r-2}-(p-1)p^{\frac{n}{2}-2}-p^{n-r-1},&\text{if}\ v=p+2,\\
	p^{r-2}+p^{\frac{n}{2}-2},&\text{if}\ v\ge p+3.
	\end{cases}$
	\item [$(\rm{\romannumeral5})$]$1\le w\le p$ and $p+1\le u\le v\le2p+1$, $q_{u,v}^w=0.$
	\item [$(\rm{\romannumeral6})$] $p+1\le w\le 2p+1$ and $1\le u\le v\le p$\\
	$q_{u,v}^w=\begin{cases}
	p^{n-2}-p^{r-1},&\text{if}\ w=p+1\ \text{and}\ u=v,\\
	p^{n-2},&\text{if}\ w=p+1\ \text{and}\ u\ne v,\\
	p^{n-2}-p^{r-2},&\text{if}\ w\ge p+2.
	\end{cases}$
	\item [$(\rm{\romannumeral7})$] $p+1\le w\le 2p+1$, $1\le u\le p$ and $p+1\le v\le 2p+1$, $q_{u,v}^w=0.$
	\item [$(\rm{\romannumeral8})$] $p+1\le w\le 2p+1$ and $p+1\le u\le v\le 2p+1$
	\begin{itemize}
		\item [$\bullet$] $w=p+1$\\
		$q_{u,v}^{p+1}=\begin{cases}
		p^{n-r}-2,&\text{if}\ u=v=p+1,\\
		0,&\text{if}\ u=p+1\ \text{and}\ v\ge p+2,\\
		p^{r-1}-(p-1)p^{\frac{n}{2}-1}-p^{n-r},&\text{if}\ u=v=p+2,\\
		p^{r-1}+p^{\frac{n}{2}-1},&\text{if}\ u\ge p+3\ \text{and}\ u=v,\\
		0,&\text{if}\ u\ge p+2\ \text{and}\ u\ne v.
		\end{cases}$
		\item [$\bullet$] $w=p+2$\\
		$q_{u,v}^{p+2}=\begin{cases}
		p^{n-r}-1,&\text{if}\ u=p+1\ \text{and}\ v=p+2,\\
		0,&\text{if}\ u=p+1\ \text{and}\ v\ne p+2,\\
		p^{r-2}+p^{\frac{n}{2}-1}-2p^{n-r}-p^{\frac{n}{2}},&\text{if}\ u=v=p+2,\\
		p^{r-2},&\text{if}\ u=p+2\ \text{and}\ v\ge p+3,\\
		p^{r-2}+p^{\frac{n}{2}-1},&\text{if}\ u\ge p+3\ \text{and}\ u=v,\\
		p^{r-2},&\text{if}\ u\ge p+3\ \text{and}\ u\ne v.
		\end{cases}$
		\item [$\bullet$] $p+3\le w\le 2p+1$\\
		$q_{u,v}^w=\begin{cases}
		p^{n-r}-1,&\hspace{-12pt}\text{if}\ u=p+1\ \text{and}\\&\hspace{-12pt} v=w,\\
		0,&\hspace{-12pt}\text{if}\ u=p+1\ \text{and}\\&\hspace{-12pt} v\ne w,\\
		p^{r-2}-p^{\frac{n}{2}-1},&\hspace{-12pt}\text{if}\ u=v=p+2,\\
		p^{r-2}-(p-1)p^{\frac{n}{2}
			-2}-p^{n-r-1}+(p^{\frac{n}{2}-3}-&\hspace{-12pt}\text{if}\ u=p+2\ \text{and}\\p^{n-r-2})\sum\limits_{i\in\mathbb{F}_p}K^{(t)}(i,w-p-2)K^{(h)}(v-p-2,i),&\hspace{-10pt} v\ge p+3,\\
		p^{r-2}+2p^{\frac{n}{2}-2}-p^{\frac{n}{2}-3}\sum\limits_{i\in\mathbb{F}_p}K^{(t)}(i,w-p-2)\times\\
		K^{(h)}(u-p-2,i)K^{(h)}(v-p-2,i)&\hspace{-12pt}\text{if}\ p+3\le u\le v.
		\end{cases}$
	\end{itemize}
\end{itemize}	
4.The first and second eigenmatrices, the intersection numbers and the Krein parameters of the association scheme induced by $U_4$

Note that the first and second eigenmatrices of the association scheme induced by $U_4$ are the same. The first (second) eigenmatrix of the association scheme induced by $U_4$ is given in Table 7.
\begin{table}[h]
	\vspace{-15pt}
	\centering
	\caption{The first (second) eigenmatrix of the association scheme induced by $U_4$}
	\renewcommand\arraystretch{1.5}	
	\resizebox{\textwidth}{!}{
		\begin{tabular}{|c|m{11cm}|}
			\hline
			$i$&\hspace{4.25cm}$P_4(ij)\ (Q_4(ij))$\\
			\hline
			1&$P_4(11)=1$, $P_4(12)=p^{n-r}-1$, $P_4(13)=p^{r-1}+(p-1)p^{\frac{n}{2}-1}-p^{n-r}$, $P_4(14)=P_4(15)=\frac{(p-1)}{2}(p^{r-1}-p^{\frac{n}{2}-1})$, $P_4(16)=p^{n-1}-p^{r-1}$, $P_4(17)=P_4(18)=\frac{(p-1)}{2}(p^{n-1}-p^{r-1})$\\
			\hline
			2&$P_4(21)=1$, $P_4(22)=p^{n-r}-1$, $P_4(23)=p^{r-1}+(p-1)p^{\frac{n}{2}-1}-p^{n-r}$, $P_4(24)=P_4(25)=\frac{(p-1)}{2}(p^{r-1}-p^{\frac{n}{2}-1)}$, $P_4(26)=-p^{r-1}$, $P_4(27)=P_4(28)=-\frac{(p-1)}{2}p^{r-1}$\\
			\hline
			3&$P_4(31)=1$, $P_4(32)=p^{n-r}-1$, $P_4(33)=(p-1)p^{\frac{n}{2}-1}-p^{n-r}$, $P_4(34)=P_4(35)=-\frac{(p-1)}{2}p^{\frac{n}{2}-1}$, $P_4(36)=P_4(37)=P_4(38)=0$\\
			\hline
			4&$P_4(41)=1$, $P_4(42)=p^{n-r}-1$, $P_4(43)=-p^{\frac{n}{2}-1}-p^{n-r}$, $P_4(44)=\frac{p+1}{2}p^{\frac{n}{2}-1}$, $P_4(45)=\frac{-p+1}{2}p^{\frac{n}{2}-1}$, $P_4(46)=P_4(47)=P_4(48)=0$\\
			\hline
			5&$P_4(51)=1$,$P_4(52)=p^{n-r}-1$, $P_4(53)=-p^{\frac{n}{2}-1}-p^{n-r}$, $P_4(54)=\frac{-p+1}{2}p^{\frac{n}{2}-1}$, $P_4(55)=\frac{p+1}{2}p^{\frac{n}{2}-1}$, $P_4(56)=P_4(57)=P_4(58)=0$\\
			\hline
			6&$P_4(61)=1$, $P_4(62)=-1$, $P_4(63)=P_4(64)=P_4(65)=0$, $P_4(66)=-(p-1)p^{\frac{n}{2}-1}$, $P_4(67)=P_4(68)=\frac{(p-1)}{2}p^{\frac{n}{2}-1}$\\
			\hline
			7&$P_4(71)=1$, $P_4(72)=-1$, $P_4(73)=P_4(74)=P_4(75)=0$, $P_4(76)=p^{\frac{n}{2}-1}$, $P_4(77)=-\frac{(p+1)}{2}p^{\frac{n}{2}-1}$, $P_4(78)=\frac{p-1}{2}p^{\frac{n}{2}-1}$\\
			\hline
			8&$P_4(81)=1$, $P_4(82)=-1$, $P_4(83)=P_4(84)=P_4(85)=0$, $P_4(86)=p^{\frac{n}{2}-1}$, $P_4(87)=\frac{p-1}{2}p^{\frac{n}{2}-1}$, $P_4(88)=-\frac{(p+1)}{2}p^{\frac{n}{2}-1}$\\
			\hline		
	\end{tabular}}
	\vspace{-15pt}
\end{table}

Since $p_{u,v}^{w}=p_{v,u}^{w}$ for any $u,v,w\in\{0,1,\dots,7\}$, we only give the values of the  intersection numbers $p_{u,v}^w$ for $u\le v$ in the following eight cases. The Krein parameter $q_{u,v}^w$ is the same as the intersection number $p_{u,v}^w$ for any $u,v,w\in\{0,1,\dots,7\}$.
\begin{itemize}
\item [$(\rm{\romannumeral1})$] $w=0$
\begin{itemize}
	\item [$\bullet$] $u=v$\\
	$p_{0,0}^{0}=1$, $p_{1,1}^0=p^{n-r}-1$, $p_{2,2}^0=p^{r-1}+(p-1)p^{\frac{n}{2}-1}-p^{n-r}$, $p_{3,3}^0=p_{4,4}^0=
\frac{(p-1)}{2}(p^{r-1}-p^{\frac{n}{2}-1})$, $p_{5,5}^0=p^{n-1}-p^{r-1}$,  $p_{6,6}^0=p_{7,7}^0=\frac{(p-1)}{2}(p^{n-1}-p^{r-1}).$
	\item[$\bullet$]$u\ne v$, 
	$p_{u,v}^0=0$ for $0\le u<v\le 7$.
\end{itemize}
\item [$(\rm{\romannumeral2})$] $w=1$
\begin{itemize}
	\item [$\bullet$] $p_{0,1}^1=1$ and $p_{0,v}^1=0$ for $0\le v\le 7$ and $v\ne 1$.
	\item [$\bullet$] $p_{1,1}^1=p^{n-r}-2$ and $p_{1,v}^1=0$ for $2\le v\le 7$.
	\item[$\bullet$] $p_{2,2}^1=p^{r-1}+(p-1)p^{\frac{n}{2}-1}-p^{n-r}$ and $p_{2,v}^1=0$ for $3\le v\le 7$.
	\item[$\bullet$] $p_{3,3}^1=\frac{(p-1)}{2}(p^{r-1}-p^{\frac{n}{2}-1})$ and $p_{3,v}^1=0$ for $4\le v\le 7$.
	\item[$\bullet$] $p_{4,4}^1=\frac{(p-1)}{2}(p^{r-1}-p^{\frac{n}{2}-1})$ and $p_{4,v}^1=0$ for $5\le v\le 7$.
	\item [$\bullet$] $p_{5,5}^1=p^{n-2}-p^{r-1}$ and $p_{5,6}^1=p_{5,7}^1=\frac{(p-1)}{2}p^{n-2}$.
	\item[$\bullet$] $p_{6,6}^1=\frac{(p-1)}{4}(p^{n-1}-p^{n-2}-2p^{r-1})$ and $p_{6,7}^1=\frac{(p-1)^2}{4}p^{n-2}$.
	\item[$\bullet$] $p_{7,7}^1=\frac{(p-1)}{4}(p^{n-1}-p^{n-2}-2p^{r-1})$.
\end{itemize}
\item [$(\rm{\romannumeral3})$] $w=2$
\begin{itemize}
	\item [$\bullet$] $p_{0,2}^2=1$ and $p_{0,v}^2=0$ for $0\le v\le 7$ and $v\ne 2$.
	\item[$\bullet$] $p_{1,2}^2=p^{n-r}-1$ and $p_{1,v}^2=0$ for $1\le v\le 7$ and $v\ne 2$.
	\item[$\bullet$] $p_{2,2}^2=p^{r-2}+p^{\frac{n}{2}}-2p^{n-r}-p^{\frac{n}{2}-1}$, $p_{2,3}^2=p_{2,4}^2=\frac{(p-1)}{2}p^{r-2}$ and $p_{2,v}^2=0$ for $5\le v\le 7$.
	\item[$\bullet$] $p_{3,3}^2=\frac{(p-1)}{4}(p^{r-1}-p^{r-2}-2p^{\frac{n}{2}-1})$, $p_{3,4}^2=\frac{(p-1)^2}{4}p^{r-2}$ and $p_{3,v}^2=0$ for $5\le v\le 7$.
	\item[$\bullet$] $p_{4,4}^2=\frac{(p-1)}{4}(p^{r-1}-p^{r-2}-2p^{\frac{n}{2}-1})$ and $p_{4,v}^2=0$ for $5\le v\le 7$.
	\item[$\bullet$] $p_{5,5}^2=p^{n-2}-p^{r-2}$ and $p_{5,6}^2=p_{5,7}^2=\frac{(p-1)}{2}(p^{n-2}-p^{r-2})$.
	\item[$\bullet$] $p_{6,6}^2=p_{6,7}^2=\frac{(p-1)^2}{4}(p^{n-2}-p^{r-2})$.
	\item[$\bullet$] $p_{7,7}^2=\frac{(p-1)^2}{4}(p^{n-2}-p^{r-2})$.
\end{itemize}
\item [$(\rm{\romannumeral4})$] $w=3$
\begin{itemize}
	\item [$\bullet$] $p_{0,3}^3=1$ and $p_{0,v}^3=0$ for $0\le v\le 7$ and $v\ne 3$.
	\item[$\bullet$] $p_{1,3}^3=p^{n-r}-1$ and $p_{1,v}^3=0$ for $1\le v\le 7$ and $v\ne 3$.
	\item[$\bullet$] $p_{2,2}^3=p^{r-2}+p^{\frac{n}{2}-1}$, $p_{2,3}^3=\frac{1}{2}(p^{r-1}+p^{\frac{n}{2}}-2p^{n-r}-p^{r-2}-3p^{\frac{n}{2}-1})$, $p_{2,4}^3=\frac{(p-1)}{2}(p^{r-2}+p^{\frac{n}{2}-1})$ and $p_{2,v}^3=0$ for $5\le v\le 7$.
	\item[$\bullet$] $p_{3,3}^3=\frac{1}{4}(p^{r-2}-2p^{r-1}+6p^{\frac{n}{2}-1}+p^r-2p^{\frac{n}{2}})$, $p_{3,4}^3=\frac{(p-1)}{4}(p^{r-1}-2p^{\frac{n}{2}-1}-p^{r-2})$ and $p_{3,v}^3=0$ for $5\le v\le 7$.
	\item[$\bullet$] $p_{4,4}^3=\frac{(p-1)}{4}(p^{r-1}-p^{r-2}-2p^{\frac{n}{2}-1})$ and $p_{4,v}^3=0$ for $5\le v\le 7$.
	\item[$\bullet$] $p_{5,5}^3=p^{n-2}-p^{r-2}$ and $p_{5,6}^3=p_{5,7}^3=\frac{(p-1)}{2}(p^{n-2}-p^{r-2})$.
	\item[$\bullet$] $p_{6,6}^3=p_{6,7}^3=\frac{(p-1)^2}{4}(p^{n-2}-p^{r-2})$.
	\item[$\bullet$] $p_{7,7}^3=\frac{(p-1)^2}{4}(p^{n-2}-p^{r-2})$.
\end{itemize}
\item [$(\rm{\romannumeral5})$] $w=4$
\begin{itemize}
	\item [$\bullet$] $p_{0,4}^4=1$ and $p_{0,v}^4=0$ for $0\le v\le 7$ and $v\ne 4$.
	\item [$\bullet$] $p_{1,4}^4=p^{n-r}-1$ and $p_{1,v}^4=0$ for $1\le v\le 7$ and $v\ne 4$.
	\item[$\bullet$] $p_{2,2}^4=p^{r-2}+p^{\frac{n}{2}-1}$, $p_{2,3}^4=\frac{(p-1)}{2}(p^{r-2}+p^{\frac{n}{2}-1})$, $p_{2,4}^4=\frac{1}{2}(p^{r-1}+p^{\frac{n}{2}}-2p^{n-r}-p^{r-2}-3p^{\frac{n}{2}-1})$ and $p_{2,v}^4=0$ for $5\le v\le 7$.
	\item[$\bullet$] $p_{3,3}^4=p_{3,4}^4=\frac{(p-1)}{4}(p^{r-1}-p^{r-2}-2p^{\frac{n}{2}-1})$ and $p_{3,v}^4=0$ for $5\le v\le 7$.
	\item[$\bullet$] $p_{4,4}^4=\frac{1}{4}(p^{r-2}-2p^{r-1}+6p^{\frac{n}{2}-1}+p^{r}-2p^{\frac{n}{2}})$ and $p_{4,v}^4=0$ for $5\le v\le 7$.
	\item[$\bullet$] $p_{5,5}^4=p^{n-2}-p^{r-2}$ and $p_{5,6}^4=p_{5,7}^4=\frac{(p-1)}{2}(p^{n-2}-p^{r-2})$.
	\item[$\bullet$] $p_{6,6}^4=p_{6,7}^4=\frac{(p-1)^2}{4}(p^{n-2}-p^{r-2})$.
	\item[$\bullet$] $p_{7,7}^4=\frac{(p-1)^2}{4}(p^{n-2}-p^{r-2})$.
\end{itemize}
\item [$(\romannumeral6)$] $w=5$
\begin{itemize}
	\item [$\bullet$] $p_{0,5}^5=1$ and $p_{0,v}^5=0$ for $0\le v\le 7$ and $v\ne 5$.
	\item[$\bullet$] $p_{1,v}^5=0$ for $1\le v\le 4$, $p_{1,5}^5=p^{n-r-1}-1$ and $p_{1,6}^5=p_{1,7}^5=\frac{(p-1)}{2}p^{n-r-1}$.
	\item[$\bullet$] $p_{2,v}^5=0$ for $2\le v\le 4$, $p_{2,5}^5=p^{r-2}+(p-1)p^{\frac{n}{2}-2}-p^{n-r-1}$ and $p_{2,6}^5=p_{2,7}^5=\frac{(p-1)}{2}(p^{r-2}+(p-1)p^{\frac{n}{2}-2}-p^{n-r-1})$.
	\item[$\bullet$] $p_{3,3}^5=p_{3,4}^5=0$, $p_{3,5}^5=\frac{(p-1)}{2}(p^{r-2}-p^{\frac{n}{2}-2})$ and $p_{3,6}^5=p_{3,7}^5=\frac{(p-1)^2}{4}(p^{r-2}\\-p^{\frac{n}{2}-2})$.
	\item[$\bullet$] $p_{4,4}^5=0$, $p_{4,5}^5=\frac{(p-1)}{2}(p^{r-2}-p^{\frac{n}{2}-2})$ and $p_{4,6}^5=p_{4,7}^5=\frac{(p-1)^2}{4}(p^{r-2}-p^{\frac{n}{2}-2})$.
		\item[$\bullet$] $p_{5,5}^5=p^{n-2}+3p^{\frac{n}{2}-1}-2p^{r-2}-2p^{\frac{n}{2}-2}-p^{\frac{n}{2}}$ and $p_{5,6}^5=p_{5,7}^5=\frac{(p-1)}{2}(p^{n-2}-2p^{r-2}+p^{\frac{n}{2}-1}-2p^{\frac{n}{2}-2})$.
		\item[$\bullet$] $p_{6,6}^5=\frac{(p-1)}{4}(p^{n-1}-2p^{r-1}-p^{n-2}+2p^{r-2}+2p^{\frac{n}{2}-2})$ and $p_{6,7}^5=\frac{(p-1)^2}{4}(p^{n-2}\\-2p^{r-2}-2p^{\frac{n}{2}-2})$.
		\item[$\bullet$] $p_{7,7}^5=\frac{(p-1)}{4}(p^{n-1}-2p^{r-1}-p^{n-2}+2p^{r-2}+2p^{\frac{n}{2}-2})$.
		\end{itemize}
	\item [$(\rm{\romannumeral7})$] $w=6$
	\begin{itemize}
		\item [$\bullet$] $p_{1,6}^6=1$ and $p_{1,v}^6=0$ for $0\le v\le7$ and $v\ne 6$.
		\item[$\bullet$] $p_{1,v}^6=0$ for $1\le v\le 4$, $p_{1,5}^6=p^{n-r-1}$, $p_{1,6}^6=\frac{(p-1)}{2}p^{n-r-1}-1$ and $p_{1,7}^6=\frac{(p-1)}{2}p^{n-r-1}$.
	\item[$\bullet$] $p_{2,v}^6=0$ for $2\le v\le 4$, $p_{2,5}^6=p^{r-2}+(p-1)p^{\frac{n}{2}-2}-p^{n-r-1}$ and $p_{2,6}^6=p_{2,7}^6=\frac{(p-1)}{2}(p^{r-2}+(p-1)p^{\frac{n}{2}-2}-p^{n-r-1})$.
	\item[$\bullet$] $p_{3,3}^6=p_{3,4}^6=0$, $p_{3,5}^6=\frac{(p-1)}{2}(p^{r-2}-p^{\frac{n}{2}-2})$ and $p_{3,6}^6=p_{3,7}^6=\frac{(p-1)^2}{4}(p^{r-2}\\-p^{\frac{n}{2}-2})$.
	\item[$\bullet$] $p_{4,4}^6=0$, $p_{4,5}^6=\frac{(p-1)}{2}(p^{r-2}-p^{\frac{n}{2}-2})$ and $p_{4,6}^6=p_{4,7}^6=\frac{(p-1)^2}{4}(p^{r-2}-p^{\frac{n}{2}-2})$.
	\item[$\bullet$] $p_{5,5}^6=p^{n-2}-2p^{r-2}+p^{\frac{n}{2}-1}-2p^{\frac{n}{2}-2}$, $p_{5,6}^6=\frac{1}{2}(p^{n-1}-2p^{r-1}-p^{n-2}+2p^{r-2}+2p^{\frac{n}{2}-2})$ and $p_{5,7}^6=\frac{(p-1)}{2}(p^{n-2}-2p^{r-2}-2p^{\frac{n}{2}-2})$.
	\item[$\bullet$] $p_{6,6}^6=\frac{1}{4}(p^{n-2}+p^n-2p^{n-1}-2p^{\frac{n}{2}-1}-2p^{\frac{n}{2}-2}+4p^{r-1}-2p^{r-2}-2p^r)$ and $p_{6,7}^6=\frac{(p-1)}{4}(p^{n-1}-2p^{r-1}-p^{n-2}+2p^{r-2}+2p^{\frac{n}{2}-2})$.
	\item[$\bullet$] $p_{7,7}^6=\frac{(p-1)}{4}(p^{n-1}-p^{n-2}+2p^{\frac{n}{2}-2}-2p^{r-1}+2p^{r-2})$.
	\end{itemize}
\item [$(\rm{\romannumeral8})$] $w=7$
\begin{itemize}
	\item [$\bullet$] $p_{0,7}^7=1$ and $p_{0,v}^7=0$ for $0\le v\le 6$.
	\item[$\bullet$] $p_{1,v}^7=0$ for $1\le v\le 4$, $p_{1,5}^7=p^{n-r-1}$, $p_{1,6}^7=\frac{(p-1)}{2}p^{n-r-1}$ and $p_{1,7}^7=\frac{(p-1)}{2}p^{n-r-1}-1$.
	\item[$\bullet$] $p_{2,v}^7=0$ for $2\le v\le 4$, $p_{2,5}^7=p^{r-2}+(p-1)p^{\frac{n}{2}-2}-p^{n-r-1}$ and $p_{2,6}^7=p_{2,7}^7=\frac{(p-1)}{2}(p^{r-2}+(p-1)p^{\frac{n}{2}-2}-p^{n-r-1})$.
	\item[$\bullet$] $p_{3,3}^7=p_{3,4}^7=0$, $p_{3,5}^7=\frac{(p-1)}{2}(p^{r-2}-p^{\frac{n}{2}-2})$ and $p_{3,6}^7=p_{3,7}^7=\frac{(p-1)^2}{4}(p^{r-2}\\-p^{\frac{n}{2}-2})$.
	\item[$\bullet$] $p_{4,4}^7=0$, $p_{4,5}^7=\frac{(p-1)}{2}(p^{r-2}-p^{\frac{n}{2}-2})$ and $p_{4,6}^7=p_{4,7}^7=\frac{(p-1)^2}{4}(p^{r-2}-p^{\frac{n}{2}-2})$.
	\item[$\bullet$] $p_{5,5}^7=p^{n-2}-2p^{r-2}+p^{\frac{n}{2}-1}-2p^{\frac{n}{2}-2}$, $p_{5,6}^7=\frac{(p-1)}{2}(p^{n-2}-2p^{r-2}-2p^{\frac{n}{2}-2})$, $p_{5,7}^7=\frac{1}{2}(p^{n-1}-2p^{r-1}-p^{n-2}+2p^{r-2}+2p^{\frac{n}{2}-2})$.
	\item[$\bullet$] $p_{6,6}^7=p_{6,7}^7=\frac{(p-1)}{4}(p^{n-1}-2p^{r-1}-p^{n-2}+2p^{r-2}+2p^{\frac{n}{2}-2})$.
	\item[$\bullet$] $p_{7,7}^7=\frac{1}{4}(p^{n-2}+p^n-2p^{n-1}-2p^{\frac{n}{2}-1}-2p^{\frac{n}{2}-2}+4p^{r-1}-2p^{r-2}-2p^r)$.	
	\end{itemize}
\end{itemize}	
5. The first and second eigenmatrices, the intersection numbers and the Krein parameters of the association scheme induced by $U_5$.

Note that the first and second eigenmatrices of the association scheme induced by $U_5$ are the same. The first (second) eigenmatrix of the association scheme induced by $U_5$ is given in Table 8. 
\begin{table}[h]
	\vspace{-15pt}
	\centering
	\caption{The first (second) eigenmatrix of the association scheme induced by $U_5$}
	\renewcommand\arraystretch{1.5}	
	\resizebox{\textwidth}{!}{
		\begin{tabular}{|m{2.1cm}<{\centering}|c|c|m{2cm}<{\centering}|m{2cm}<{\centering}|c|m{2cm}<{\centering}|}
			\hline
			\diagbox{$i$}{{\tiny $P_5(ij)\ (Q_5(ij))$}}{$j$}&1&2&3&4&5&6\\
			\hline	
			1&1&$p^{n-r}-1$&$(p-1)p^{\frac{n}{2}-1}+p^{r-1}-p^{n-r}$&$(p-1)(p^{r-1}-p^{\frac{n}{2}-1})$&$p^{n-1}-p^{r-1}$&$(p-1)(p^{n-1}-p^{r-1})$\\
			\hline
			2&1&$p^{n-r}-1$&$(p-1)p^{\frac{n}{2}-1}+p^{r-1}-p^{n-r}$&$(p-1)(p^{r-1}-p^{\frac{n}{2}-1})$&$-p^{r-1}$&$-(p-1)p^{r-1}$\\
			\hline
			3&1&$p^{n-r}-1$&$(p-1)p^{\frac{n}{2}-1}-p^{n-r}$&$-(p-1)p^{\frac{n}{2}-1}$&0&0\\
			\hline
			4&1&$p^{n-r}-1$&$-p^{\frac{n}{2}-1}-p^{n-r}$&$p^{\frac{n}{2}-1}$&0&0\\
			\hline
			5&1&$-1$&0&0&$-(p-1)p^{\frac{n}{2}-1}$&$(p-1)p^{\frac{n}{2}-1}$\\
			\hline
			6&1&$-1$&0&0&$p^{\frac{n}{2}-1}$&$-p^{\frac{n}{2}-1}$\\
			\hline	
	\end{tabular}}
	\vspace{-15pt}
\end{table}	

Since $p_{u,v}^{w}=p_{v,u}^{w}$ for any $u,v,w\in\{0,1,\dots,5\}$, we only give the values of the  intersection numbers $p_{u,v}^w$ for $u\le v$ in the following six cases. The Krein parameter $q_{u,v}^w$ is the same as the intersection number $p_{u,v}^w$ for any  $u,v,w\in\{0,1,\dots,5\}$.
\begin{itemize}
	\item [$(\rm{\romannumeral1})$] $w=0$
	\begin{itemize}
		\item [$\bullet$]$u=v$\\
	$p_{0,0}^0=1$, $p_{1,1}^0=p^{n-r}-1$, $p_{2,2}^0=p^{r-1}+(p-1)p^{\frac{n}{2}-1}-p^{n-r}$, $p_{3,3}^0=(p-1)(p^{r-1}-p^{\frac{n}{2}-1})$, $p_{4,4}^0=p^{n-1}-p^{r-1}$, $p_{5,5}^0=(p-1)(p^{n-1}-p^{r-1})$.
	\item[$\bullet$] $u\ne v$, $p_{u,v}^0=0$ for $0\le u<v\le 5$.
			\end{itemize}
	\item [$(\rm{\romannumeral2})$] $w=1$
	\begin{itemize}
		\item [$\bullet$] $p_{0,1}^1=1$ and $p_{0,v}^1=0$ for $0\le v\le 5$ and $v\ne 1$.
		\item[$\bullet$] $p_{1,1}^1=p^{n-r}-2$ and $p_{1,v}^1=0$ for $2\le v\le 5$.
		\item[$\bullet$] $p_{2,2}^1=p^{r-1}+(p-1)p^{\frac{n}{2}-1}-p^{n-r}$ and $p_{2,v}^1=0$ for $3\le v\le 5$.
		\item[$\bullet$] $p_{3,3}^1=(p-1)(p^{r-1}-p^{\frac{n}{2}-1})$ and $p_{3,4}^1=p_{3,5}^1=0$.
		\item[$\bullet$] $p_{4,4}^1=p^{n-2}-p^{r-1}$ and $p_{4,5}^1=(p-1)p^{n-2}$.
		\item[$\bullet$] $p_{5,5}^1=(p-1)(p^{n-1}-p^{r-1}-p^{n-2})$.
	\end{itemize}	
	\item [$(\rm{\romannumeral3})$] $w=2$	
	\begin{itemize}
		\item [$\bullet$] $p_{0,2}^2=1$ and $p_{0,v}^2=0$ for $0\le v\le 5$ and $v\ne 2$.
		\item[$\bullet$] $p_{1,2}^2=p^{n-r}-1$ and $p_{1,v}^2=0$ for $0\le v\le 5$ and $v\ne 2$.
		\item[$\bullet$] $p_{2,2}^2=p^{r-2}+p^{\frac{n}{2}}-2p^{n-r}-p^{\frac{n}{2}-1}$, $p_{2,3}^2=(p-1)p^{r-2}$ and $p_{2,4}^2=p_{2,5}^2=0$.
		\item[$\bullet$] $p_{3,3}^2=(p-1)(p^{r-1}-p^{r-2}-p^{\frac{n}{2}-1})$ and $p_{3,4}^2=p_{3,5}^2=0$.
		\item[$\bullet$] $p_{4,4}^2=p^{n-2}-p^{r-2}$ and $p_{4,5}^2=(p-1)(p^{n-2}-p^{r-2})$.
		\item[$\bullet$] $p_{5,5}^2=(p-1)^2(p^{n-2}-p^{r-2})$.
	\end{itemize}
	\item [$(\rm{\romannumeral4})$] $w=3$
	\begin{itemize}
		\item [$\bullet$] $p_{0,3}^3=1$ and $p_{0,v}^3=0$ for $0\le v\le 5$ and $v\ne 3$.
		\item[$\bullet$] $p_{1,3}^3=p^{n-r}-1$ and $p_{1,v}^3=0$ for $0\le v\le 5$ and $v\ne 3$.
		\item[$\bullet$] $p_{2,2}^3=p^{r-2}+p^{\frac{n}{2}-1}$, $p_{2,3}^3=p^{r-1}+p^{\frac{n}{2}}-p^{n-r}-p^{r-2}-2p^{\frac{n}{2}-1}$ and $p_{2,4}^3=p_{2,5}^3=0$.
		\item[$\bullet$] $p_{3,3}^3=p^{r-2}-2p^{r-1}+3p^{\frac{n}{2}-1}+p^{r}-2p^{\frac{n}{2}}$ and $p_{3,4}^3=p_{3,5}^3=0$.
		\item[$\bullet$] $p_{4,4}^3=p^{n-2}-p^{r-2}$ and $p_{4,5}^3=(p-1)(p^{n-2}-p^{r-2})$.
		\item[$\bullet$] $p_{5,5}^3=(p-1)^2(p^{n-2}-p^{r-2})$.
		\end{itemize}		
		\item [$(\rm{\romannumeral5})$] $w=4$	
		\begin{itemize}
			\item [$\bullet$] $p_{0,4}^4=1$ and $p_{0,v}^4=0$ for $0\le v\le 5$ and $v\ne 4$.
			\item[$\bullet$] $p_{1,v}^4=0$ for $1\le v\le 3$, $p_{1,4}^4=p^{n-r-1}-1$ and $p_{1,5}^4=(p-1)p^{n-r-1}$.
			
			\item[$\bullet$] $p_{2,2}^4=p_{2,3}^4=0$, $p_{2,4}^4=p^{r-2}+(p-1)p^{\frac{n}{2}-2}-p^{n-r-1}$ and $p_{2,5}^4=(p-1)(p^{r-2}+(p-1)p^{\frac{n}{2}-2}-p^{n-r-1})$.
			\item[$\bullet$] $p_{3,3}^4=0$, $p_{3,4}^4=(p-1)(p^{r-2}-p^{\frac{n}{2}-2})$ and $p_{3,5}^4=(p-1)^2
(p^{r-2}-p^{\frac{n}{2}-2})$.
\item[$\bullet$] $p_{4,4}^4=p^{n-2}+3p^{\frac{n}{2}-1}-2p^{r-2}-2p^{\frac{n}{2}-2}-p^{\frac{n}{2}}$ and $p_{4,5}^4=(p-1)(p^{n-2}-2p^{r-2}+p^{\frac{n}{2}-1})-2p^{\frac{n}{2}-2}$.
\item[$\bullet$] $p_{5,5}^4=(p-1)(p^{n-1}-2p^{r-1}-p^{n-2}+2p^{r-2}-p^{\frac{n}{2}-1}+2p^{\frac{n}{2}-2})$.				
		\end{itemize}
		\item [$(\romannumeral6)$] $w=5$
		\begin{itemize}
			\item [$\bullet$] $p_{0,5}^5=1$ and $p_{0,v}^5=0$ for $0\le v\le 4$.
			\item[$\bullet$] $p_{1,v}^5=0$ for $1\le v\le 3$, $p_{1,4}^5=p^{n-r-1}$ and $p_{1,5}^5=p^{n-r}-p^{n-r-1}-1$.
			\item[$\bullet$] $p_{2,2}^5=p_{2,3}^5=0$, $p_{2,4}^5=p^{r-2}+(p-1)p^{\frac{n}{2}-2}-p^{n-r-1}$ and $p_{2,5}^5=(p-1)(p^{r-2}+(p-1)p^{\frac{n}{2}-2}-p^{n-r-1})$.
			\item[$\bullet$] $p_{3,3}^5=0$, $p_{3,4}^5=(p-1)(p^{r-2}-p^{\frac{n}{2}-2})$ and $p_{3,5}^5=(p-1)^2(p^{r-2}-p^{\frac{n}{2}-2})$.
			\item[$\bullet$] $p_{4,4}^5=p^{n-2}-2p^{r-2}+p^{\frac{n}{2}-1}-2p^{\frac{n}{2}-2}$ and $p_{4,5}^5=p^{n-1}-2p^{r-1}-p^{n-2}+2p^{r-2}-p^{\frac{n}{2}-1}+2p^{\frac{n}{2}-2}$.
			\item[$\bullet$] $p_{5,5}^5=(p-1)^2(p^{n-2}-2p^{r-2})+(p-2)p^{\frac{n}{2}-2}$.
		\end{itemize}	
\end{itemize}	
6. The first and second eigenmatrices, the intersection numbers and the Krein parameters of the association scheme induced by  $U_6$.

Note that the first and second eigenmarices of the association scheme induced by $U_6$ are the same. The first (second) eigenmatrix of the association scheme induced by $U_6$ is given in Table 9. 
\begin{table}[h]
	\vspace{-15pt}
	\centering
	\caption{The first (second) eigenmatrix of the association scheme induced by $U_6$}
	\renewcommand\arraystretch{1.5}	
	\resizebox{\textwidth}{!}{
		\begin{tabular}{|c|m{11cm}|}
			\hline
			$i$&\hspace{4.25cm}$P_6(ij)\ (Q_6(ij))$\\
			\hline
			1&$P_6(11)=1$, $P_6(12)=p^{n-1}-p^{\frac{n}{2}}$, $P_6(13)=P_6(14)=\frac{(p-1)}{2}(p^{n-1}-p^{\frac{n}{2}})$, $P_6(15)=p^{\frac{n}{2}-1}-1$, $P_6(16)=P_6(17)=\frac{(p-1)}{2}(p^{\frac{n}{2}}+p^{\frac{n}{2}-1})$\\
			\hline
			2&$P_6(21)=1$, $P_6(22)=(p-1)p^{\frac{n}{2}-1}$, $P_6(23)=P_6(24)=-\frac{(p-1)}{2}p^{\frac{n}{2}-1}$, $P_{6}(25)=-1$, $P_{6}(26)=P_6(27)=0$\\
			\hline
			3&$P_6(31)=1$, $P_{6}(32)=-p^{\frac{n}{2}-1}$, $P_{6}(33)=\frac{p+1}{2}p^{\frac{n}{2}-1}$, $P_{6}(34)=\frac{-p+1}{2}p^{\frac{n}{2}-1}$, $P_{6}(35)=-1$, $P_{6}(36)=P_6(37)=0$\\
			\hline
			4&$P_6(41)=1$, $P_6(42)=-p^{\frac{n}{2}-1}$, $P_{6}(43)=\frac{-p+1}{2}p^{\frac{n}{2}-1}$, $P_6(44)=\frac{p+1}{2}p^{\frac{n}{2}-1}$, $P_6(45)=-1$, $P_6(46)=P_6(47)=0$
			\\
			\hline
			5&$P_6(51)=1$, $P_6(52)=-p^{\frac{n}{2}}$, $P_6(53)=P_6(54)=-\frac{(p-1)}{2}p^{\frac{n}{2}}$, $P_6(55)=p^{\frac{n}{2}-1}-1$, $P_6(56)=P_6(57)=\frac{p-1}{2}(p^{\frac{n}{2}-1}+p^{\frac{n}{2}})$\\
			\hline
			6&$P_6(61)=1$, $P_6(62)=P_6(63)=P_6(64)=0$, $P_6(65)=p^{\frac{n}{2}-1}-1$, $P_6(66)=-\frac{(p+1)}{2}p^{\frac{n}{2}-1}$, $P_6(67)=\frac{p-1}{2}p^{\frac{n}{2}-1}$\\
			\hline
			7&$P_6(71)=1$, $P_6(72)=P_6(73)=P_6(74)=0$, $P_6(75)=p^{\frac{n}{2}-1}-1$, $P_6(76)=\frac{p-1}{2}p^{\frac{n}{2}-1}$, $P_6(77)=-\frac{(p+1)}{2}p^{\frac{n}{2}-1}$\\
			\hline
			
	\end{tabular}}
	\vspace{-15pt}
\end{table}

Since $p_{u,v}^{w}=p_{v,u}^{w}$ for any $u,v,w\in\{0,1,\dots,6\}$, we only give the values of the  intersection numbers $p_{u,v}^w$ for $u\le v$ in the following seven cases. The Krein parameter $q_{u,v}^w$ is the same as the intersection number $p_{u,v}^w$ for any $u,v,w\in\{0,1,\dots,6\}$.
\begin{itemize}
	\item [$(\rm{\romannumeral1})$] $w=0$
	\begin{itemize}
		\item [$\bullet$] $u=v$\\
		$p_{0,0}^0=1$, $p_{1,1}^0=p^{n-1}-p^{\frac{n}{2}}$, $p_{2,2}^0=p_{3,3}^0=\frac{(p-1)}{2}(p^{n-1}-p^{\frac{n}{2}})$, $p_{4,4}^0=p^{\frac{n}{2}-1}-1$ and $p_{5,5}^0=p_{6,6}^0=\frac{(p-1)}{2}(p^{\frac{n}{2}}+p^{\frac{n}{2}-1})$.
		\item[$\bullet$] $u\ne v$, $p_{u,v}^0=0$ for $0\le u<v\le 6$.	
	\end{itemize}
\item [$(\rm{\romannumeral2})$] $w=1$
\begin{itemize}
	\item [$\bullet$] $p_{0,1}^1=1$ and $p_{0,v}^1=0$ for $0\le v\le 6$ and $v\ne 1$.
	\item[$\bullet$] $p_{1,1}^1=p^{\frac{n}{2}-2}(p^{\frac{n}{2}}-5p+p^2+2)$, $p_{1,2}^1=p_{1,3}^1=\frac{(p-1)}{2}p^{\frac{n}{2}-2}(p^{\frac{n}{2}}-3p+2)$, $p_{1,4}^1=p^{\frac{n}{2}-2}-1$ and $p_{1,5}^1=p_{1,6}^1=\frac{(p-1)}{2}(p^{\frac{n}{2}-1}+p^{\frac{n}{2}-2})$.
	\item[$\bullet$] $p_{2,2}^1=\frac{(p-1)}{4}p^{\frac{n}{2}-2}(p^{\frac{n}{2}+1}+2p-p^{\frac{n}{2}}-2p^2-2)$, $p_{2,3}^1=\frac{(p-1)^2}{4}p^{\frac{n}{2}-2}(p^{\frac{n}{2}}-2p+2)$, $p_{2,4}^1=\frac{(p-1)}{2}p^{\frac{n}{2}-2}$ and $p_{2,5}^1=p_{2,6}^1=\frac{(p-1)^2}{4}(p+1)p^{\frac{n}{2}-2}$.
	\item[$\bullet$] $p_{3,3}^1=\frac{(p-1)}{4}p^{\frac{n}{2}-2}(p^{\frac{n}{2}+1}+2p-p^{\frac{n}{2}}-2p^2-2)$, $p_{3,4}^1=\frac{(p-1)}{2}p^{\frac{n}{2}-2}$ and $p_{3,5}^1=p_{3,6}^1=\frac{(p-1)^2}{4}(p+1)p^{\frac{n}{2}-2}$.
	\item[$\bullet$] $p_{u,v}^1=0$ for $4\le u\le v\le 6$.
\end{itemize}
\item [$(\rm{\romannumeral3})$] $w=2$
\begin{itemize}
	\item [$\bullet$] $p_{0,2}^2=1$ and $p_{0,v}^2=0$ for $0\le v\le 6$ and $v\ne 2$.
	\item[$\bullet$] $p_{1,1}^2=p^{\frac{n}{2}-2}(p^{\frac{n}{2}}-3p+2)$, $p_{1,2}^2=\frac{1}{2}p^{\frac{n}{2}-2}(p^{\frac{n}{2}+1}+2p-p^{\frac{n}{2}}-2p^2-2)$, $p_{1,3}^2=\frac{(p-1)}{2}p^{\frac{n}{2}-2}(p^{\frac{n}{2}}-2p+2)$, $p_{1,4}^2=p^{\frac{n}{2}-2}$ and $p_{1,5}^2=p_{1,6}^2=\frac{(p-1)}{2}(p^{\frac{n}{2}-1}+p^{\frac{n}{2}-2})$.
\item[$\bullet$] $p_{2,2}^2=\frac{1}{4}p^{\frac{n}{2}-2}(p^{\frac{n}{2}}-2p^{\frac{n}{2}+1}+p^{\frac{n}{2}+2}+4p^2-2p^3+2)$, $p_{2,3}^2=\frac{(p-1)}	{4}p^{\frac{n}{2}-2}(p^{\frac{n}{2}+1}\\+2p-p^{\frac{n}{2}}-2p^2-2)$, $p_{2,4}^2=\frac{1}{2}(p^{\frac{n}{2}-1}-p^{\frac{n}{2}-2}-2)$ and $p_{2,5}^2=p_{2,6}^2=\frac{(p-1)^2}{4}(p+1)p^{\frac{n}{2}-2}$.
\item[$\bullet$] $p_{3,3}^2=\frac{(p-1)}{4}p^{\frac{n}{2}-2}(p^{\frac{n}{2}+1}+2p-p^{\frac{n}{2}}-2p^2-2)$, $p_{3,4}^2=\frac{(p-1)}{2}p^{\frac{n}{2}-2}$ and $p_{3,5}^2=p_{3,6}^2=\frac{(p-1)^2}{4}(p^{\frac{n}{2}-1}+p^{\frac{n}{2}-2})$.
\item[$\bullet$] $p_{u,v}^2=0$ for $4\le u\le v\le 6$.
		\end{itemize}	
\item [$(\rm{\romannumeral4})$] $w=3$
\begin{itemize}
	\item [$\bullet$] $p_{0,3}^3=1$ and $p_{0,v}^3=0$ for $0\le v\le 6$ and $v\ne 3$.
	\item[$\bullet$] $p_{1,1}^3=p^{\frac{n}{2}-2}(p^{\frac{n}{2}}-3p+2)$, $p_{1,2}^3=\frac{(p-1)}{2}p^{\frac{n}{2}-2}(p^{\frac{n}{2}}-2p+2)$, $p_{1,3}^3=\frac{1}{2}p^{\frac{n}{2}-2}(p^{\frac{n}{2}+1}+2p-p^{\frac{n}{2}}-2p^2-2)$, $p_{1,4}^3=p^{\frac{n}{2}-2}$ and $p_{1,5}^3=p_{1,6}^3=\frac{(p-1)}{2}(p^{\frac{n}{2}-1}+p^{\frac{n}{2}-2})$.
	\item[$\bullet$] $p_{2,2}^3=p_{2,3}^3=\frac{(p-1)}{4}p^{\frac{n}{2}-2}(p^{\frac{n}{2}+1}+2p-p^{\frac{n}{2}}-2p^2-2)$, $p_{2,4}^3=\frac{(p-1)}{2}p^{\frac{n}{2}-2}$ and $p_{2,5}^3=p_{2,6}^3=\frac{(p-1)^2}{4}(p^{\frac{n}{2}-1}+p^{\frac{n}{2}-2})$.
	\item[$\bullet$] $p_{3,3}^3=\frac{1}{4}p^{\frac{n}{2}-2}(p^{\frac{n}{2}}-2p^{\frac{n}{2}+1}+p^{\frac{n}{2}+2}+4p^2-2p^3+2)$, $p_{3,4}^3=\frac{1}{2}(p^{\frac{n}{2}-1}-p^{\frac{n}{2}-2}-2)$ and $p_{3,5}^3=p_{3,6}^3=\frac{(p-1)^2}{4}(p^{\frac{n}{2}-1}+p^{\frac{n}{2}-2})$.
	\item[$\bullet$] $p_{u,v}^3=0$ for $4\le u\le v\le 6$.
	\end{itemize}
	\item [$(\rm{\romannumeral5})$] $w=4$
	\begin{itemize}
		\item [$\bullet$] $p_{0,4}^4=1$ and $p_{0,v}^4=0$ for $0\le v\le 6$ and $v\ne 4$.
		\item[$\bullet$] $p_{1,1}^4=p^{n-2}-p^{\frac{n}{2}}$, $p_{1,2}^4=p_{1,3}^4=\frac{(p-1)}{2}p^{n-2}$ and $p_{1,v}^4=0$ for $4\le v\le 6$.
		\item[$\bullet$] $p_{2,2}^4=\frac{(p-1)}{4}p^{\frac{n}{2}-2}(p^{\frac{n}{2}+1}-p^{\frac{n}{2}}-2p^2)$, $p_{2,3}^4=\frac{(p-1)^2}{4}p^{n-2}$ and $p_{2,v}^4=0$ for $4\le v\le 6$.
		\item[$\bullet$] $p_{3,3}^4=\frac{(p-1)}{4}p^{\frac{n}{2}-2}(p^{\frac{n}{2}+1}-p^{\frac{n}{2}}-2p^2)$ and $p_{3,v}^4=0$ for $4\le v\le 6$.
		\item[$\bullet$] $p_{4,4}^4=p^{\frac{n}{2}-1}-2$ and $p_{4,5}^4=p_{4,6}^4=0$.
		\item[$\bullet$] $p_{5,5}^4=\frac{(p-1)}{2}(p^{\frac{n}{2}}+p^{\frac{n}{2}-1})$ and $p_{5,6}^4=0$.
		\item[$\bullet$] $p_{6,6}^4=\frac{(p-1)}{2}(p^{\frac{n}{2}}+p^{\frac{n}{2}-1})$.
	\end{itemize}
\item [$(\romannumeral6)$] $w=5$
\begin{itemize}
	\item [$\bullet$] $p_{0,5}^5=1$ and $p_{0,v}^5=0$ for $0\le v\le 6$ and $v\ne 5$.
	\item[$\bullet$] $p_{1,1}^5=p^{n-2}-p^{\frac{n}{2}-1}$, $p_{1,2}^5=p_{1,3}^5=\frac{(p-1)}{2}p^{\frac{n}{2}-2}(p^{\frac{n}{2}}-p)$ and $p_{1,v}^5=0$ for $4\le v\le 6$.
	\item[$\bullet$] $p_{2,2}^5=p_{2,3}^5=\frac{(p-1)^2}{4}p^{\frac{n}{2}-2}(p^{\frac{n}{2}}-p)$ and $p_{2,v}^5=0$ for $4\le v\le6$.
	\item[$\bullet$] $p_{3,3}^5=\frac{(p-1)^2}{4}p^{\frac{n}{2}-2}(p^{\frac{n}{2}}-p)$ and $p_{3,v}^5=0$ for $4\le v\le 6$.
	\item[$\bullet$] $p_{4,4}^5=0$, $p_{4,5}^5=p^{\frac{n}{2}-1}-1$ and $p_{4,6}^5=0$.
	\item[$\bullet$] $p_{5,5}^5=\frac{1}{4}p^{\frac{n}{2}-1}(p^2-5)$ and $p_{5,6}^5=\frac{(p-1)}{4}(p^{\frac{n}{2}}+p^{\frac{n}{2}-1})$.
	\item[$\bullet$] $p_{6,6}^5=\frac{(p-1)}{4}(p^{\frac{n}{2}}+p^{\frac{n}{2}-1})$.
\end{itemize}	
\item [$(\rm{\romannumeral7})$] $w=6$
\begin{itemize}
	\item [$\bullet$] $p_{0,6}^6=1$ and $p_{0,v}^6=0$ for $0\le v\le 5$.
	\item[$\bullet$] $p_{1,1}^6=p^{n-2}-p^{\frac{n}{2}-1}$, $p_{1,2}^6=p_{1,3}^6=\frac{(p-1)}{2}p^{\frac{n}{2}-2}(p^{\frac{n}{2}}-p)$ and $p_{1,v}^6=0$ for $4\le v\le 6$.
	\item[$\bullet$] $p_{2,2}^6=p_{2,3}^6=\frac{(p-1)^2}{4}p^{\frac{n}{2}-2}(p^{\frac{n}{2}}-p)$ and $p_{2,v}^6=0$ for $4\le v\le 6$.
	\item[$\bullet$] $p_{3,3}^6=\frac{(p-1)^2}{4}p^{\frac{n}{2}-2}(p^{\frac{n}{2}}-p)$ and $p_{3,v}^6=0$ for $4\le v\le 6$.
	\item[$\bullet$] $p_{4,4}^6=p_{4,5}^6=0$ and $p_{4,6}^6=p^{\frac{n}{2}-1}-1$.
	\item[$\bullet$] $p_{5,5}^6=\frac{(p-1)^2}{4}p^{\frac{n}{2}-1}$ and $p_{5,6}^6=\frac{(p^2-1)}{4}p^{\frac{n}{2}-1}$.
	\item[$\bullet$] $p_{6,6}^6=\frac{(p^2-5)}{4}p^{\frac{n}{2}-1}$.
\end{itemize}
\end{itemize}
7. The first and second eigenmatrices, the intersection numbers and the Krein parameters of the association scheme induced by  $U_7$.

Note that the first and second eigenmatrices of the association scheme induced by $U_7$ are the same. The first (second) eigenmatrix is given in Table 10.
\begin{table}[h]
	\vspace{-15pt}
	\centering
	\caption{The first (second) eigenmatrix of the association scheme induced by $U_7$}
	\renewcommand\arraystretch{1.5}	
	\resizebox{\textwidth}{!}{
		\begin{tabular}{|m{2.1cm}<{\centering}|c|c|c|c|c|}
			\hline
			\diagbox{$i$}{{\tiny $P_7(ij)\ (Q_7(ij))$}}{$j$}&1&2&3&4&5\\
			\hline
			1&1&$p^{n-1}-p^{\frac{n}{2}}$&$(p-1)(p^{n-1}-p^{\frac{n}{2}})$&$p^{\frac{n}{2}-1}-1$&$(p-1)(p^{\frac{n}{2}}+p^{\frac{n}{2}-1})$\\
			\hline
			2&1&$(p-1)p^{\frac{n}{2}-1}$&$-(p-1)p^{\frac{n}{2}-1}$&$-1$&0\\
			\hline
			3&1&$-p^{\frac{n}{2}-1}$&$p^{\frac{n}{2}-1}$&$-1$&0\\
			\hline
			4&1&$-p^{\frac{n}{2}}$&$-(p-1)p^{\frac{n}{2}}$&$p^{\frac{n}{2}-1}-1$&$(p-1)(p^{\frac{n}{2}-1}+p^{\frac{n}{2}})$\\
			\hline
			5&1&0&0&$p^{\frac{n}{2}-1}-1$&$-p^{\frac{n}{2}-1}$\\
			\hline
	\end{tabular}}
	\vspace{-15pt}
\end{table}

Since $p_{u,v}^{w}=p_{v,u}^{w}$ for any $u,v,w\in\{0,1,2,3,4\}$, we only give the values of the  intersection numbers $p_{u,v}^w$ for $u\le v$ in the following five cases. The Krein parameter $q_{u,v}^w$  is the same as the intersection number $p_{u,v}^w$ for any $u,v,w\in\{0,1,2,3,4\}$.
\begin{itemize}
	\item [$(\rm{\romannumeral1})$] $w=0$
	\begin{itemize}
		\item [$\bullet$] $u=v$\\
		$p_{0,0}^0=1$, $p_{1,1}^0=p^{n-1}-p^{\frac{n}{2}}$, $p_{2,2}^0=(p-1)(p^{n-1}-p^{\frac{n}{2}})$, $p_{3,3}^0=p^{\frac{n}{2}-1}-1$ and $p_{4,4}^0=(p-1)(p^{\frac{n}{2}}+p^{\frac{n}{2}-1})$.
		\item[$\bullet$] $u\ne v$, $p_{u,v}^0=0$ for $0\le u<v\le 4$.
	\end{itemize}
\item [$(\rm{\romannumeral2})$] $w=1$
\begin{itemize}
	\item [$\bullet$] $p_{0,1}^1=1$ and $p_{0,v}^1=0$ for $0\le v\le 4$ and $v\ne 1$.
	\item[$\bullet$] $p_{1,1}^1=p^{\frac{n}{2}-2}(p^{\frac{n}{2}}-5p+p^2+2)$, $p_{1,2}^1=p^{\frac{n}{2}-2}(p-1)(p^{\frac{n}{2}}-3p+2)$, $p_{1,3}^1=p^{\frac{n}{2}-2}-1$ and $p_{1,4}^1=(p-1)(p^{\frac{n}{2}-1}+p^{\frac{n}{2}-2})$.
	\item[$\bullet$] $p_{2,2}^1=p^{\frac{n}{2}-2}(p-1)(p^{\frac{n}{2}+1}+3p-p^{\frac{n}{2}}-2p^2-2)$, $p_{2,3}^1=(p-1)p^{\frac{n}{2}-2}$ and $p_{2,4}^1=p^{\frac{n}{2}-2}(p-1)^2(p+1)$.
	\item[$\bullet$] $p_{3,3}^1=p_{3,4}^1=p_{4,4}^1=0$.
\end{itemize}	
	\item [$(\rm{\romannumeral3})$] $w=2$
	\begin{itemize}
		\item [$\bullet$] $p_{0,2}^2=1$ and $p_{0,v}^2=0$ for $0\le v\le4$ and $v\ne 2$.
		\item[$\bullet$] $p_{1,1}^2=p^{\frac{n}{2}-2}(p^{\frac{n}{2}}-3p+2)$, $p_{1,2}^2=p^{\frac{n}{2}-2}(p^{\frac{n}{2}+1}-p^{\frac{n}{2}}-2p^2+3p+2)$, $p_{1,3}^2=p^{\frac{n}{2}-2}$ and $p_{1,4}^2=p^{\frac{n}{2}-2}(p^2-1)$.
		\item[$\bullet$] $p_{2,2}^2=p^{\frac{n}{2}-2}(p^{\frac{n}{2}+2}-2p^{\frac{n}{2}+1}+p^{\frac{n}{2}}-2p^3+4p^2-3p+2)$, $p_{2,3}^2=p^{\frac{n}{2}-1}-p^{\frac{n}{2}-2}-1$ and $p_{2,4}^2=p^{\frac{n}{2}-2}(p-1)^2(p+1)$.
		\item[$\bullet$] $p_{3,3}^2=p_{3,4}^2=p_{4,4}^2=0$.
			\end{itemize}
	\item [$(\rm{\romannumeral4})$] $w=3$
	\begin{itemize}
		\item [$\bullet$] $p_{0,3}^3=1$ and $p_{0,v}^3=0$ for $0\le v\le 4$ and $v\ne 3$.
		\item[$\bullet$] $p_{1,1}^3=p^{\frac{n}{2}-2}(p^{\frac{n}{2}}-p^2)$, $p_{1,2}^3=p^{n-2}(p-1)$ and $p_{1,3}^3=p_{1,4}^3=0$.
		\item[$\bullet$] $p_{2,2}^3=p^{\frac{n}{2}-2}(p-1)(p^{\frac{n}{2}+1}-p^{\frac{n}{2}}-p^2)$ and $p_{2,3}^3=p_{2,4}^3=0$.
			\item[$\bullet$] $p_{3,3}^3=p^{\frac{n}{2}-1}-2$ and $p_{3,4}^3=0$.
			\item[$\bullet$] $p_{4,4}^3=p^{\frac{n}{2}-1}(p^2-1)$.
	\end{itemize}
\item [$(\rm{\romannumeral5})$] $w=4$
\begin{itemize}
	\item [$\bullet$] $p_{0,4}^4=1$ and $p_{0,v}^4=0$ for $0\le v\le 3$.
	\item[$\bullet$] $p_{1,1}^4=p^{\frac{n}{2}-2}(p^{\frac{n}{2}}-p)$, $p_{1,2}^4=p^{\frac{n}{2}-2}(p-1)(p^{\frac{n}{2}}-p)$ and $p_{1,3}^4=p_{1,4}^4=0$.
	\item[$\bullet$] $p_{2,2}^4=p^{\frac{n}{2}-2}(p-1)^2(p^{\frac{n}{2}}-p)$ and $p_{2,3}^4=p_{2,4}^4=0$.
	\item[$\bullet$] $p_{3,3}^4=0$ and $p_{3,4}^4=p^{\frac{n}{2}-1}-1$.
	\item[$\bullet$] $p_{4,4}^4=p^{\frac{n}{2}-1}(p^2-2)$.
\end{itemize}
\end{itemize}
8. The first and second eigenmatrices, the intersection numbers and the Krein parameters of the association scheme induced by  $U_8$.

Note that the first and second eigenmatrices of the association scheme induced by $U_8$ are the same. The first (second) eigenmatrix is given in Table 11. 
\begin{table}[h]
	\vspace{0pt}
	\centering
	\caption{The first (second) eigenmatrix of the association scheme induced by $U_8$}
	\renewcommand\arraystretch{1.5}	
	\resizebox{\textwidth}{!}{
		\begin{tabular}{|c|m{11cm}|}
			\hline
			$i$&\hspace{4.25cm} $P_8(ij)\ (Q_8(ij))$\\
			\hline
			1&$P_8(11)=1$, $P_8(12)=p^{n-1}-p^{r-1}$, $P_8(13)=P_8(14)=\frac{(p-1)}{2}(p^{n-1}-p^{r-1})$, $P_8(15)=p^{n-r}-1$, $P_8(16)=p^{r-1}-(p-1)p^{\frac{n}{2}-1}-p^{n-r}$, $P_8(17)=P_8(18)=\frac{(p-1)}{2}(p^{r-1}+p^{\frac{n}{2}-1})$\\
			\hline
			2&$P_8(21)=1$, $P_8(22)=(p-1)p^{\frac{n}{2}-1}$, $P_8(23)=P_8(24)=-\frac{(p-1)}{2}p^{\frac{n}{2}-1}$, $P_8(25)=-1$, $P_8(26)=P_8(27)=P_8(28)=0$\\
			\hline
			3&$P_8(31)=1$, $P_8(32)=-p^{\frac{n}{2}-1}$, $P_8(33)=\frac{p+1}{2}p^{\frac{n}{2}-1}$, $P_8(34)=\frac{-p+1}{2}p^{\frac{n}{2}-1}$, $P_8(35)=-1$, $P_8(36)=P_8(37)=P_8(38)=0$\\
			\hline
			4&$P_8(41)=1$, $P_8(42)=-p^{\frac{n}{2}-1}$, $P_8(43)=\frac{-p+1}{2}p^{\frac{n}{2}-1}$, $P_8(44)=\frac{p+1}{2}p^{\frac{n}{2}-1}$, $P_8(45)=-1$, $P_8(46)=P_8(47)=P_8(48)=0$\\
			\hline
			5&$P_8(51)=1$, $P_8(52)=-p^{r-1}$, $P_8(53)=P_8(54)=-\frac{(p-1)}{2}p^{r-1}$, $P_8(55)=p^{n-r}-1$, $P_8(56)=p^{r-1}-(p-1)p^{\frac{n}{2}-1}-p^{n-r}$, $P_8(57)=P_8(58)=\frac{(p-1)}{2}(p^{\frac{n}{2}-1}+p^{r-1})$\\
			\hline
			6&$P_8(61)=1$, $P_8(62)=P_8(63)=P_8(64)=0$, $P_8(65)=p^{n-r}-1$, $P_8(66)=-(p-1)p^{\frac{n}{2}-1}-p^{n-r}$, $P_8(67)=P_8(68)=\frac{p-1}{2}p^{\frac{n}{2}-1}$\\
			\hline
			7&$P_8(71)=1$, $P_8(72)=P_8(73)=P_8(74)=0$, $P_8(75)=p^{n-r}-1$, $P_8(76)=p^{\frac{n}{2}-1}-p^{n-r}$, $P_8(77)=-\frac{(p+1)}{2}p^{\frac{n}{2}-1}$, $P_8(78)=\frac{p-1}{2}p^{\frac{n}{2}-1}$\\
			\hline
			8&$P_8(81)=1$, $P_8(82)=P_8(83)=P_8(84)=0$, $P_8(85)=p^{n-r}-1$, $P_8(86)=p^{\frac{n}{2}-1}-p^{n-r}$, $P_8(87)=\frac{p-1}{2}p^{\frac{n}{2}-1}$, $P_8(88)=-\frac{(p+1)}{2}p^{\frac{n}{2}-1}$\\
			\hline
	\end{tabular}}
	\vspace{-15pt}
\end{table}

Since $p_{u,v}^{w}=p_{v,u}^{w}$ for any $u,v,w\in\{0,1,\dots,7\}$, we only give the values of the intersection numbers $p_{u,v}^w$ for $u\le v$ in the following eight cases. The Krein parameter $q_{u,v}^w$ is  the same as the intersection number $p_{u,v}^w$ for any $u,v,w\in\{0,1,\dots,7\}$.
\begin{itemize}
	\item [$(\rm{\romannumeral1})$] $w=0$
	\begin{itemize}
		\item [$\bullet$] $u=v$\\
		$p_{0,0}^0=1$, $p_{1,1}^0=p^{n-1}-p^{r-1}$, $p_{2,2}^0=p_{3,3}^0=\frac{(p-1)}{2}(p^{n-1}-p^{r-1})$, $p_{4,4}^0=p^{n-r}-1$, $p_{5,5}^0=p^{r-1}-(p-1)p^{\frac{n}{2}-1}-p^{n-r}$ and $p_{6,6}^0=p_{7,7}^0=\frac{(p-1)}{2}(p^{r-1}+p^{\frac{n}{2}-1})$.
			\item[$\bullet$] $u\ne v$, $p_{u,v}^0=0$ for $0\le u<v\le 7$.
	\end{itemize}
	\item [$(\rm{\romannumeral2})$] $w=1$
	\begin{itemize}
		\item[$\bullet$] $p_{0,1}^1=1$ and $p_{0,v}^1=0$ for $0\le v\le 7$ and $v\ne 1$.
		\item [$\bullet$] $p_{1,1}^1=p^{n-2}-2p^{r-2}-3p^{\frac{n}{2}-1}+2p^{\frac{n}{2}-2}+p^{\frac{n}{2}}$, $p_{1,2}^1=p_{1,3}^1=\frac{(p-1)}{2}(p^{n-2}-2p^{r-2}-p^{\frac{n}{2}-1}+2p^{\frac{n}{2}-2})$, $p_{1,4}^1=p^{n-r-1}-1$, $p_{1,5}^1=p^{r-2}-(p-1)p^{\frac{n}{2}-2}-p^{n-r-1}$ and $p_{1,6}^1=p_{1,7}^1=\frac{(p-1)}{2}(p^{r-2}+p^{\frac{n}{2}-2})$.
			\item[$\bullet$] $p_{2,2}^1=\frac{(p-1)}{4}(p^{n-1}+2p^{r-2}-2p^{r-1}-p^{n-2}-2p^{\frac{n}{2}-2})$, $p_{2,3}^1=\frac{(p-1)^2}{4}(p^{n-2}-2p^{r-2}+2p^{\frac{n}{2}-2})$, $p_{2,4}^1=\frac{(p-1)}{2}p^{n-r-1}$, $p_{2,5}^1=\frac{(p-1)}{2}(p^{r-2}-(p-1)p^{\frac{n}{2}-2}-p^{n-r-1})$ and  $p_{2,6}^1=p_{2,7}^1=\frac{(p-1)^2}{4}(p^{r-2}+p^{\frac{n}{2}-2})$.	
		\item[$\bullet$] $p_{3,3}^1=\frac{(p-1)}{4}(p^{n-1}+2p^{r-2}-2p^{r-1}-p^{n-2}-2p^{\frac{n}{2}-2})$, $p_{3,4}^1=\frac{(p-1)}{2}p^{n-r-1}$, $p_{3,5}^1=\frac{(p-1)}{2}(p^{r-2}-(p-1)p^{\frac{n}{2}-2}-p^{n-r-1})$ and $p_{3,6}^1=p_{3,7}^1=\frac{(p-1)^2}{4}(p^{r-2}+p^{\frac{n}{2}-2})$.
		\item[$\bullet$] $p_{u,v}^1=0$ for $4\le u\le v\le 7$.
		\end{itemize}
		\item [$(\rm{\romannumeral3})$] $w=2$
		\begin{itemize}
			\item [$\bullet$] $p_{0,2}^2=1$ and $p_{0,v}^2=0$ for $0\le v\le 7$ and $v\ne 2$.
			\item[$\bullet$] $p_{1,1}^2=p^{n-2}-2p^{r-2}
-p^{\frac{n}{2}-1}+2p^{\frac{n}{2}-2}$, $p_{1,2}^2=\frac{1}{2}(p^{n-1}+2p^{r-2}-2p^{r-1}-p^{n-2}-2p^{\frac{n}{2}-2})$, $p_{1,3}^2=\frac{(p-1)}{2}(p^{n-2}-2p^{r-2}+2p^{\frac{n}{2}-2})$, $p_{1,4}^2=p^{n-r-1}$, $p_{1,5}^2=p^{r-2}-(p-1)p^{\frac{n}{2}-2}-p^{n-r-1}$ and $p_{1,6}^2=p_{1,7}^2=\frac{(p-1)}{2}(p^{r-2}+p^{\frac{n}{2}-2})$.
\item[$\bullet$] $p_{2,2}^2=\frac{1}{4}(4p^{r-1}-2p^{n-1}+p^{n-2}-2p^{r-2}+2p^{\frac{n}{2}-1}+p^n-2p^r+2p^{\frac{n}{2}-2})$, $p_{2,3}^2=\frac{(p-1)}{4}(p^{n-1}+2p^{r-2}-2p^{r-1}-p^{n-2}-2p^{\frac{n}{2}-2})$, $p_{2,4}^2=\frac{1}{2}(p^{n-r}-p^{n-r-1}-2)$, $p_{2,5}^2=\frac{(p-1)}{2}(p^{r-2}-(p-1)p^{\frac{n}{2}-2}-p^{n-r-1})$ and $p_{2,6}^2=p_{2,7}^2=\frac{(p-1)^2}{4}(p^{r-2}+p^{\frac{n}{2}-2})$.

\item[$\bullet$] $p_{3,3}^2=\frac{(p-1)}{4}(p^{n-1}+2p^{r-2}-2p^{r-1}-p^{n-2}-2p^{\frac{n}{2}-2})$, $p_{3,4}^2=\frac{(p-1)}{2}p^{n-r-1}$, $p_{3,5}^2=\frac{(p-1)}{2}(p^{r-2}-(p-1)p^{\frac{n}{2}-2}-p^{n-r-1})$ and $p_{3,6}^2=p_{3,7}^2=\frac{(p-1)^2}{4}(p^{r-2}+p^{\frac{n}{2}-2})$.
\item[$\bullet$] $p_{u,v}^2=0$ for $4\le u\le v\le 7$.
		\end{itemize}
	\item [$(\rm{\romannumeral4})$] $w=3$
	\begin{itemize}
		\item [$\bullet$] $p_{0,3}^3=1$ and $p_{0,v}^3=0$ for $0\le v\le 7$ and $v\ne 3$.
		\item[$\bullet$] $p_{1,1}^3=p^{n-2}-2p^{r-2}-p^{\frac{n}{2}-1}+2p^{\frac{n}{2}-2}$, $p_{1,2}^3=\frac{(p-1)}{2}(p^{n-2}-2p^{r-2}+2p^{\frac{n}{2}-2})$, $p_{1,3}^3=\frac{1}{2}(p^{n-1}+2p^{r-2}-2p^{r-1}-p^{n-2}-2p^{\frac{n}{2}-2})$, $p_{1,4}^3=p^{n-r-1}$, $p_{1,5}^3=p^{r-2}-(p-1)p^{\frac{n}{2}-2}-p^{n-r-1}$ and $p_{1,6}^3=p_{1,7}^3=\frac{(p-1)}{2}(p^{r-2}+p^{\frac{n}{2}-2})$.
		\item[$\bullet$] $p_{2,2}^3=p_{2,3}^3=\frac{(p-1)}{4}(p^{n-1}+2p^{r-2}-2p^{r-1}-p^{n-2}-2p^{\frac{n}{2}-2})$, $p_{2,4}^3=\frac{(p-1)}{2}p^{n-r-1}$, $p_{2,5}^3=\frac{(p-1)}{2}(p^{r-2}-(p-1)p^{\frac{n}{2}-2}-p^{n-r-1})$ and $p_{2,6}^3=p_{2,7}^3=\frac{(p-1)^2}{4}(p^{r-2}+p^{\frac{n}{2}-2})$.
		\item[$\bullet$] $p_{3,3}^3=\frac{1}{4}(4p^{r-1}-2p^{n-1}+p^{n-2}-2p^{r-2}+2p^{\frac{n}{2}-1}+p^n-2p^r+2p^{\frac{n}{2}-2})$, $p_{3,4}^3=\frac{1}{2}(p^{n-r}-p^{n-r-1}-2)$, $p_{3,5}^3=\frac{(p-1)}{2}(p^{r-2}-(p-1)p^{\frac{n}{2}-2}-p^{n-r-1})$ and $p_{3,6}^3=p_{3,7}^3=\frac{(p-1)^2}{4}(p^{r-2}+p^{\frac{n}{2}-2})$.
		\item[$\bullet$] $p_{u,v}^3=0$ for $4\le u\le v\le 7$.
	\end{itemize}	
	\item [$(\rm{\romannumeral5})$] $w=4$
	\begin{itemize}
		\item [$\bullet$] $p_{0,4}^4=1$ and $p_{0,v}^4=0$ for $0\le v\le 7$ and $v\ne 4$.
		\item[$\bullet$] $p_{1,1}^4=p^{n-2}-p^{r-1}$, $p_{1,2}^4=p_{1,3}^4=\frac{(p-1)}{2}p^{n-2}$ and $p_{1,v}^4=0$ for $4\le v\le 7$.
		\item[$\bullet$] $p_{2,2}^4=\frac{(p-1)}{4}(p^{n-1}-2p^{r-1}-p^{n-2})$, $p_{2,3}^4=\frac{(p-1)^2}{4}p^{n-2}$ and $p_{2,v}^4=0$ for $4\le v\le 7$.
		\item[$\bullet$] $p_{3,3}^4=\frac{(p-1)}{4}(p^{n-1}-2p^{r-1}-p^{n-2})$ and $p_{3,v}^4=0$ for $4\le v\le 7$.
		\item[$\bullet$] $p_{4,4}^4=p^{n-r}-2$ and $p_{4,5}^4=p_{4,6}^4=p_{4,7}^4=0$.
		\item[$\bullet$] $p_{5,5}^4=p^{r-1}-(p-1)p^{\frac{n}{2}-1}-p^{n-r}$ and $p_{5,6}^4=p_{5,7}^4=0$.
		\item[$\bullet$] $p_{6,6}^4=\frac{(p-1)}{2}(p^{r-1}+p^{\frac{n}{2}-1})$ and $p_{6,7}^4=0$.
		\item[$\bullet$] $p_{7,7}^4=\frac{(p-1)}{2}(p^{r-1}+p^{\frac{n}{2}-1})$.
	\end{itemize}	
		\item [$(\romannumeral6)$] $w=5$
	\begin{itemize}
		\item [$\bullet$] $p_{0,5}^5=1$ and $p_{0,v}^5=0$ for $0\le v\le7$ and $v\ne 5$.
		\item[$\bullet$] $p_{1,1}^5=p^{n-2}-p^{r-2}$, $p_{1,2}^5=p_{1,3}^5=\frac{(p-1)}{2}(p^{n-2}-p^{r-2})$ and $p_{1,v}^5=0$ for $4\le v\le 7$.
		\item[$\bullet$] $p_{2,2}^5=p_{2,3}^5=\frac{(p-1)^2}{4}(p^{n-2}-p^{r-2})$ and $p_{2,v}^5=0$ for $4\le v\le 7$.
		\item[$\bullet$] $p_{3,3}^5=\frac{(p-1)^2}{4}(p^{n-2}-p^{r-2})$ and $p_{3,v}^5=0$ for $4\le v\le 7$.
		\item[$\bullet$] $p_{4,4}^5=0$, $p_{4,5}^5=p^{n-r}-1$ and $p_{4,6}^5=p_{4,7}^5=0$.
		\item[$\bullet$] $p_{5,5}^5=p^{r-2}+p^{\frac{n}{2}-1}-2p^{n-r}-p^{\frac{n}{2}}$ and $p_{5,6}^5=p_{5,7}^5=\frac{(p-1)}{2}p^{r-2}$.
			\item[$\bullet$] $p_{6,6}^5=\frac{(p-1)}{4}(p^{r-1}-p^{r-2}+2p^{\frac{n}{2}-1})$ and $p_{6,7}^5=\frac{(p-1)^2}{4}p^{r-2}$.
			\item[$\bullet$] $p_{7,7}^5=\frac{(p-1)}{4}(p^{r-1}-p^{r-2}+2p^{\frac{n}{2}-1})$.
	\end{itemize}
	\item [$(\rm{\romannumeral7})$] $w=6$	
	\begin{itemize}
		\item [$\bullet$] $p_{0,6}^6=1$ and $p_{0,v}^6=0$ for $0\le v\le 7$ and $v\ne 6$.
		\item[$\bullet$] $p_{1,1}^6=p^{n-2}-p^{r-2}$, $p_{1,2}^6=p_{1,3}^6=\frac{(p-1)}{2}(p^{n-2}-p^{r-2})$ and $p_{1,v}^6=0$ for $4\le v\le 7$.
		\item[$\bullet$] $p_{2,2}^6=p_{2,3}^6=\frac{(p-1)^2}{4}(p^{n-2}-p^{r-2})$ and $p_{2,v}^6=0$ for $4\le v\le 7$.
		\item[$\bullet$] $p_{3,3}^6=\frac{(p-1)^2}{4}(p^{n-2}-p^{r-2})$ and $p_{3,v}^6=0$ for $4\le v\le 7$.
		\item[$\bullet$] $p_{4,4}^6=p_{4,5}^6=0$, $p_{4,6}^6=p^{n-r}-1$ and $p_{4,7}^6=0$.
		\item[$\bullet$] $p_{5,5}^6=p^{r-2}-p^{\frac{n}{2}-1}$, $p_{5,6}^6=\frac{1}{2}(p^{r-1}+3p^{\frac{n}{2}-1}-2p^{n-r}-p^{r-2}-p^{\frac{n}{2}})$ and $p_{5,7}^6=\frac{(p-1)}{2}
			(p^{r-2}-p^{\frac{n}{2}-1})$.
			\item[$\bullet$] $p_{6,6}^6=\frac{1}{4}(p^{r-2}-2p^{r-1}-6p^{\frac{n}{2}-1}+p^r+2p^{\frac{n}{2}})$ and $p_{6,7}^6=\frac{(p-1)}{4}(p^{r-1}-p^{r-2}+2p^{\frac{n}{2}-1})$.
				\item[$\bullet$] $p_{7,7}^6=\frac{(p-1)}{4}(p^{r-1}-p^{r-2}+2p^{\frac{n}{2}-1})$.		
	\end{itemize}
\item[$(\rm{\romannumeral8})$] $w=7$
\begin{itemize}
	\item [$\bullet$] $p_{0,7}^7=1$ and $p_{0,v}^7=0$ for $0\le v\le 6$.
	\item[$\bullet$] $p_{1,1}^7=p^{n-2}-p^{r-2}$, $p_{1,2}^7=p_{1,3}^7=\frac{(p-1)}{2}(p^{n-2}-p^{r-2})$ and $p_{1,v}^7=0$ for $4\le v\le 7$.
	\item[$\bullet$] $p_{2,2}^7=p_{2,3}^7=\frac{(p-1)^2}{4}(p^{n-2}-p^{r-2})$ and $p_{2,v}^7=0$ for $4\le v\le 7$.
	\item[$\bullet$] $p_{3,3}^7=\frac{(p-1)^2}{4}(p^{n-2}-p^{r-2})$ and $p_{3,v}^7=0$ for $4\le v\le 7$.
	\item[$\bullet$] $p_{4,7}^7=p^{n-r}-1$ and $p_{4,4}^7=p_{4,5}^7=p_{4,6}^7=0$.
	\item[$\bullet$] $p_{5,5}^7=p^{r-2}-p^{\frac{n}{2}-1}$, $p_{5,6}^7=\frac{(p-1)}{2}(p^{r-2}-p^{\frac{n}{2}-1})$ and $p_{5,7}^7=\frac{1}{2}(p^{r-1}+3p^{\frac{n}{2}-1}-2p^{n-r}-p^{r-2}-p^{\frac{n}{2}})$.
		\item[$\bullet$] $p_{6,6}^7=p_{6,7}^7=\frac{(p-1)}{4}(p^{r-1}-p^{r-2}+2p^{\frac{n}{2}-1})$.
		\item[$\bullet$] $p_{7,7}^7=\frac{1}{4}(p^{r-2}-2p^{r-1}-6p^{\frac{n}{2}-1}+p^r+2p^{\frac{n}{2}})$.	
\end{itemize}	
\end{itemize}
9. The first and second eigenmatrices, the intersection numbers and the Krein parameters of the association scheme induced by  $U_9$.

Note that the first and second eigenmatrices of the association scheme induced by $U_9$ are the same. The first (second) eigenmatrix is given in Table 12.
\begin{table}[h]
	\vspace{-15pt}
	\centering
	\caption{The first (second) eigenmatrix of the association scheme induced by $U_9$}
	\renewcommand\arraystretch{1.5}	
	\resizebox{\textwidth}{!}{
		\begin{tabular}{|m{2.1cm}<{\centering}|c|c|m{2cm}<{\centering}|c|m{2cm}<{\centering}|m{2cm}<{\centering}|}
			\hline
			\diagbox{$i$}{{\tiny $P_9(ij)\ (Q_9(ij))$}}{$j$}&1&2&3&4&5&6\\
			\hline
			1&	1&$p^{n-1}-p^{r-1}$&$(p-1)(p^{n-1}-p^{r-1})$&$p^{n-r}-1$&$p^{r-1}-p^{n-r}-(p-1)p^{\frac{n}{2}-1}$&$(p-1)(p^{r-1}+p^{\frac{n}{2}-1})$\\
			\hline
			2&1&$(p-1)p^{\frac{n}{2}-1}$&$-(p-1)p^{\frac{n}{2}-1}$&$-1$&0&0\\
			\hline
			3&1&$-p^{\frac{n}{2}-1}$&$p^{\frac{n}{2}-1}$&$-1$&0&0\\
			\hline
			4&1&$-p^{r-1}$&$-(p-1)p^{r-1}$&$p^{n-r}-1$&$p^{r-1}-p^{n-r}-(p-1)p^{\frac{n}{2}-1}$&$(p-1)(p^{\frac{n}{2}-1}+p^{r-1})$\\
			\hline
			5&1&0&0&$p^{n-r}-1$&$-p^{n-r}-p^{\frac{n}{2}}+p^{\frac{n}{2}-1}$&$(p-1)p^{\frac{n}{2}-1}$\\
			\hline
			6&1&0&0&$p^{n-r}-1$&$p^{\frac{n}{2}-1}-p^{n-r}$&$-p^{\frac{n}{2}-1}$\\
			\hline
	\end{tabular}}
	\vspace{-15pt}
\end{table}

Since $p_{u,v}^{w}=p_{v,u}^{w}$ for any $u,v,w\in\{0,1,\dots,5\}$, we only give the values of the intersection numbers $p_{u,v}^w$ for $u\le v$ in the following six cases. The Krein parameter $q_{u,v}^w$ is the same as the intersection number $p_{u,v}^w$ for any $u,v,w\in\{0,1,\dots,5\}$.
\begin{itemize}
	\item [$(\rm{\romannumeral1})$] $w=0$
	\begin{itemize}
		\item [$\bullet$] $u=v$\\
		$p_{0,0}^0=1$, $p_{1,1}^0=p^{n-1}-p^{r-1}$, $p_{2,2}^0=(p-1)(p^{n-1}-p^{r-1})$, $p_{3,3}^0=p^{n-r}-1$, $p_{4,4}^0=p^{r-1}-(p-1)p^{\frac{n}{2}-1}-p^{n-r}$ and $p_{5,5}^0=(p-1)(p^{r-1}+p^{\frac{n}{2}-1})$.
		\item[$\bullet$] $u\ne v$, $p_{u,v}^0=0$ for $0\le u<v\le 5$.
	\end{itemize}
	\item [$(\rm{\romannumeral2})$] $w=1$
	\begin{itemize}
		\item [$\bullet$] $p_{0,1}^1=1$ and $p_{0,v}^1=0$ for $0\le v\le 5$ and $v\ne 1$.
		\item[$\bullet$] $p_{1,1}^1=p^{n-2}-2p^{r-2}-3p^{\frac{n}{2}-1}+2p^{\frac{n}{2}-2}+p^{\frac{n}{2}}$, $p_{1,2}^1=(p-1)(p^{n-2}-2p^{r-2}-p^{\frac{n}{2}-1}+2p^{\frac{n}{2}-2})$, $p_{1,3}^1=p^{n-r-1}-1$, $p_{1,4}^1=p^{r-2}-(p-1)p^{\frac{n}{2}-2}-p^{n-r-1}$ and $p_{1,5}^1=(p-1)(p^{r-2}+p^{\frac{n}{2}-2})$.
		\item[$\bullet$] $p_{2,2}^1=(p-1)(p^{n-1}-2p^{r-1}-p^{n-2}+2p^{r-2}+p^{\frac{n}{2}-1}-2p^{\frac{n}{2}-2})$, $p_{2,3}^1=(p-1)p^{n-r-1}$, $p_{2,4}^1=(p-1)(p^{r-2}-(p-1)p^{\frac{n}{2}-2}-p^{n-r-1})$ and $p_{2,5}^1=(p-1)^2(p^{r-2}+p^{\frac{n}{2}-2})$.
		\item[$\bullet$] $p_{u,v}^1=0$ for $3\le u\le v\le 5$.
	\end{itemize}
	\item [$(\rm{\romannumeral3})$] $w=2$
	\begin{itemize}
		\item [$\bullet$] $p_{0,2}^2=1$ and $p_{0,v}^2=0$ for $0\le v\le 5$ and $v\ne 2$.
		\item[$\bullet$] $p_{1,1}^2=p^{n-2}-2p^{r-2}-p^{\frac{n}{2}-1}+2p^{\frac{n}{2}-2}$, $p_{1,2}^2=p^{n-1}-2p^{r-1}-p^{n-2}+2p^{r-2}+p^{\frac{n}{2}-1}-2p^{\frac{n}{2}-2}$, $p_{1,3}^2=p^{n-r-1}$, $p_{1,4}^2=p^{r-2}-(p-1)p^{\frac{n}{2}-2}-p^{n-r-1}$ and $p_{1,5}^2=(p-1)(p^{r-2}+p^{\frac{n}{2}-2})$.
		\item[$\bullet$] $p_{2,2}^2=(p-1)^2(p^{n-2}-2p^{r-2})-p^{\frac{n}{2}-3}(p^2-2p)$, $p_{2,3}^2=p^{n-r}-p^{n-r-1}-1$, $p_{2,4}^2=(p-1)(p^{r-2}-(p-1)p^{\frac{n}{2}-2}-p^{n-r-1})$ and $p_{2,5}^2=(p-1)^2(p^{r-2}+p^{\frac{n}{2}-2})$.
		\item[$\bullet$] $p_{u,v}^2=0$ for $3\le u\le v\le 5$.
	\end{itemize}
	\item [$(\rm{\romannumeral4})$] $w=3$
	\begin{itemize}
		\item [$\bullet$] $p_{0,3}^3=1$ and $p_{0,v}^3=0$ for $0\le v\le 5$ and $v\ne 3$.
		\item[$\bullet$] $p_{1,1}^3=p^{n-2}-p^{r-1}$, $p_{1,2}^3=(p-1)p^{n-2}$ and $p_{1,v}^3=0$ for $3\le v\le 5$.
		\item[$\bullet$] $p_{2,2}^3=(p-1)(p^{n-1}-p^{r-1}-p^{n-2})$ and $p_{2,v}^3=0$ for $3\le v\le 5$.
		\item[$\bullet$] $p_{3,3}^3=p^{n-r}-2$ and $p_{3,4}^3=p_{3,5}^3=0$.
		\item[$\bullet$] $p_{4,4}^3=p^{r-1}-(p-1)p^{\frac{n}{2}-1}-p^{n-r}$ and $p_{4,5}^3=0$.
		\item[$\bullet$]$p_{5,5}^3=(p-1)(p^{r-1}+p^{\frac{n}{2}-1})$.
	\end{itemize}
	\item [$(\rm{\romannumeral5})$] $w=4$
	\begin{itemize}
		\item [$\bullet$] $p_{0,4}^4=1$ and $p_{0,v}^4=0$ for $0\le v\le 5$ and $v\ne 4$.
		\item[$\bullet$] $p_{1,1}^4=p^{n-2}-p^{r-2}$, $p_{1,2}^4=(p-1)(p^{n-2}-p^{r-2})$ and $p_{1,v}^4=0$ for $3\le v\le 5$.
		\item[$\bullet$] $p_{2,2}^4=(p-1)^2(p^{n-2}-p^{r-2})$ and $p_{2,v}^4=0$ for $3\le v\le 5$.
		\item[$\bullet$] $p_{3,3}^4=0$, $p_{3,4}^4=p^{n-r}-1$ and $p_{3,5}^4=0$.
		\item[$\bullet$] $p_{4,4}^4=p^{r-2}+p^{\frac{n}{2}-1}-2p^{n-r}-p^{\frac{n}{2}}$ and $p_{4,5}^4=(p-1)p^{r-2}$.
		\item[$\bullet$] $p_{5,5}^4=(p-1)(p^{r-1}-p^{r-2}+p^{\frac{n}{2}-1})$.
	\end{itemize}
	\item [$(\rm{\romannumeral6})$] $w=5$
	\begin{itemize}
		\item [$\bullet$] $p_{0,5}^5=1$ and $p_{0,v}^5=0$ for $0\le v\le 4$.
		\item[$\bullet$] $p_{1,1}^5=p^{n-2}-p^{r-2}$, $p_{1,2}^5=(p-1)(p^{n-2}-p^{r-2})$ and $p_{1,v}^5=0$ for $3\le v\le 5$.
		\item[$\bullet$] $p_{2,2}^5=(p-1)^2(p^{n-2}-p^{r-2})$ and $p_{2,v}^5=0$ for $3\le v\le 5$.
		\item[$\bullet$] $p_{3,3}^5=p_{3,4}^5=0$ and $p_{3,5}^5=p^{n-r}-1$.
		\item[$\bullet$] $p_{4,4}^5=p^{r-2}-p^{\frac{n}{2}-1}$ and $p_{4,5}^5=p^{r-1}+2p^{\frac{n}{2}-1}-p^{n-r}-p^{r-2}-p^{\frac{n}{2}}$.
		\item[$\bullet$] $p_{5,5}^5=p^{r-2}-2p^{r-1}-3p^{\frac{n}{2}-1}+p^r+2p^{\frac{n}{2}}$.
	\end{itemize}
\end{itemize}
10. The first and second eigenmatrices, the intersection numbers and the Krein parameters of the association scheme induced by  $U_{10}$.
	
	The first and second eigenmatrices of the association scheme induced by $U_{10}$ are given by Tables 13 and 14. 
	\begin{table}[h]
		\vspace{-15pt}
		\centering
		\caption{The first eigenmatrix of the association scheme induced by $U_{10}$}
		\renewcommand\arraystretch{1.5}	
		\resizebox{\textwidth}{!}{
			\begin{tabular}{|m{1.4cm}<{\centering}|c|c|c|c|m{3.5cm}<{\centering}|}
				\hline
				\diagbox{$i$}{$P_{10}(ij)$}{$j$}&1&2&$3\le j\le \frac{p+3}{2}$&$\frac{p+5}{2}$& $\frac{p+7}{2}\le j\le \frac{3p+3}{2}$\\
				\hline
				1&1&$p^{\frac{n-1}{2}}-1$&$2p^{\frac{n-1}{2}}$&$p^{n-1}-p^{\frac{n-1}{2}}$&$p^{n-1}-p^{\frac{n-1}{2}}$\\
				\hline
				2&1&$p^{\frac{n-1}{2}}-1$&$2p^{\frac{n-1}{2}}$&$-p^{\frac{n-1}{2}}$&$-p^{\frac{n-1}{2}}$\\
				\hline
				$3\le i\le \frac{p+3}{2}$&1&$p^{\frac{n-1}{2}}-1$&$p^{\frac{n-2}{2}}S^{(t)}(s_{j-2},s_{i-2})$&0&0\\
				\hline
				$\frac{p+5}{2}$&1&$-1$&0&0&$-\eta(j-\frac{p+5}{2})p^{\frac{n-1}{2}}$\\
				\hline
				$\frac{p+7}{2}\le i\le \frac{3p+3}{2}$&1&$-1$&0&$-\eta(i-\frac{p+5}{2})p^{\frac{n-1}{2}}$&$-p^{\frac{n-2}{2}}S^{(t)}(j-\frac{p+5}{2}, i-\frac{p+5}{2})$\\
				\hline
		\end{tabular}}
		\vspace{-15pt}
	\end{table}
	\begin{table}[h]
		\vspace{-15pt}
		\centering
		\caption{The second eigenmatrix of the association scheme induced by $U_{10}$}
		\renewcommand\arraystretch{1.5}	
		\resizebox{\textwidth}{!}{
			\begin{tabular}{|m{1.4cm}<{\centering}|c|c|c|c|m{3.5cm}<{\centering}|}
				\hline
				\diagbox{$i$}{$Q_{10}(ij)$}{$j$}&1&2&$3\le j\le \frac{p+3}{2}$&$\frac{p+5}{2}$& $\frac{p+7}{2}\le j\le \frac{3p+3}{2}$\\
				\hline
				1&1&$p^{\frac{n-1}{2}}-1$&$2p^{\frac{n-1}{2}}$&$p^{n-1}-p^{\frac{n-1}{2}}$&$p^{n-1}-p^{\frac{n-1}{2}}$\\
				\hline
				2&1&$p^{\frac{n-1}{2}}-1$&$2p^{\frac{n-1}{2}}$&$-p^{\frac{n-1}{2}}$&$-p^{\frac{n-1}{2}}$\\
				\hline
				$3\le i\le \frac{p+3}{2}$&1&$p^{\frac{n-1}{2}}-1$&$p^{\frac{n-2}{2}}S^{(h)}(s_{j-2},s_{i-2})$&0&0\\
				\hline
				$\frac{p+5}{2}$&1&$-1$&0&0&$-\eta(j-\frac{p+5}{2})p^{\frac{n-1}{2}}$\\
				\hline
				$\frac{p+7}{2}\le i\le \frac{3p+3}{2}$&1&$-1$&0&$-\eta(i-\frac{p+5}{2})p^{\frac{n-1}{2}}$&$-p^{\frac{n-2}{2}}S^{(h)}(j-\frac{p+5}{2}, i-\frac{p+5}{2})$\\
				\hline
		\end{tabular}}
		\vspace{-15pt}
	\end{table}

	Since $p_{u,v}^{w}=p_{v,u}^{w}$ for any $u,v,w\in\{0,1,\dots,\frac{3p+1}{2}\}$, we only give the values of the intersection numbers $p_{u,v}^w$ for $u\le v$ in the following eight cases. 
\begin{itemize}
	\item [$(\rm{\romannumeral1})$] $w=0$ and $0\le u\le v\le \frac{3p+1}{2}$
	\begin{itemize}
		\item[$\bullet$] $u=v$\\
		$p_{0,0}^0=1$, $p_{1,1}^0=p^{\frac{n-1}{2}-1}$, $p_{u,u}^0=2p^{\frac{n-1}{2}}$ for $2\le u\le \frac{p+1}{2}$ and $p_{u,u}^0=p^{n-1}-p^{\frac{n-1}{2}}$ for $\frac{p+3}{2}\le u\le \frac{3p+1}{2}$.
		\item[$\bullet$] $u\ne v$, $p_{u,v}^0=0$.
		\end{itemize}
	\item [$(\rm{\romannumeral2})$] $u=0$ and $1\le w\le \frac{3p+1}{2}$ and $0\le v\le \frac{3p+1}{2}$\\
	$p_{0,v}^w=\begin{cases}
	1,&\text{if}\ w=v,\\
	0,&\text{if}\ w\ne v.
	\end{cases}$
	\item [$(\rm{\romannumeral3})$] $1\le w\le \frac{p+1}{2}$ and $1\le u\le v\le \frac{p+1}{2}$
	\begin{itemize}
		\item [$\bullet$] $w=1$\\
		$p_{u,v}^1=\begin{cases}
		p^{\frac{n-1}{2}}-2, &\text{if}\ u=v=1,\\
		0,&\text{if}\ u=1\ \text{and}\ v\ge 2,\\
		2p^{\frac{n-1}{2}},&\text{if}\ u\ge 2\ \text{and}\ u=v,\\
		0,&\text{if}\ u\ge 2\ \text{and}\ u\ne v.
		\end{cases}$
		\item[$\bullet$] $2\le w\le \frac{p+1}{2}$\\
		$p_{u,v}^w=\begin{cases}
		0,&\hspace{0.5cm}\text{if}\ u=v=1,\\
		2(p^{\frac{n-3}{2}}-p^{-1})+(p^{\frac{n-5}{2}}-p^{-2})\times\\
		\sum\limits_{i\in SQ}S^{(h)}(i,s_{w-1})S^{(t)}(s_{v-1},i),&\hspace{0.5cm}\text{if}\ u=1\ \text{and}\ v\ge 2,\\
		4p^{\frac{n-3}{2}}+p^{\frac{n}{2}-3}\sum\limits_{i\in SQ}	S^{(h)}(i,s_{w-1})\times\\
	S^{(t)}(s_{u-1},i)S^{(t)}(s_{v-1},i),&\hspace{0.5cm}\text{if}\  2\le u\le v.
		\end{cases}$
	\end{itemize}
\item [$(\rm{\romannumeral4})$] $1\le w\le \frac{p+1}{2}$, $1\le u\le \frac{p+1}{2}$ and $\frac{p+3}{2}\le v\le \frac{3p+1}{2}$, $p_{u,v}^w=0$.
\item [$(\rm{\romannumeral5})$] $1\le w\le \frac{p+1}{2}$ and $\frac{p+3}{2}\le u\le v\le \frac{3p+1}{2}$\\
$p_{u,v}^w=\begin{cases}
p^{n-2}-p^{\frac{n-1}{2}},&\text{if}\ w=1\ \text{and}\ u=v,\\
	p^{n-2},&\text{if}\ w=1\ \text{and}\ u\ne v,\\
	p^{n-2}-p^{\frac{n-3}{2}},&\text{if}\ w\ge 2.
		\end{cases}$
\item [$(\rm{\romannumeral6})$] $\frac{p+3}{2}\le w\le \frac{3p+1}{2}$ and $1\le u\le v\le \frac{p+1}{2}$, $p_{u,v}^w=0$.
\item [$(\rm{\romannumeral7})$] $\frac{p+3}{2}\le w\le \frac{3p+1}{2}$, $1\le u\le \frac{p+1}{2}$ and $\frac{p+3}{2}\le v\le \frac{3p+1}{2}$\\
$p_{u,v}^w=\begin{cases}
p^{\frac{n-3}{2}}-p^{-1}-p^{-2}\sum\limits_{i\in \mathbb{F}_p}S^{(h)}(i,w-\frac{p+3}{2})S^{(t)}(v-\frac{p+3}{2},i),&\text{if}\ u=1,\\
	2p^{\frac{n-3}{2}},&\text{if}\ u\ge 2.
\end{cases}	$
\item [$(\rm{\romannumeral8})$] $\frac{p+3}{2}\le w\le \frac{3p+1}{2}$ and $\frac{p+3}{2}\le u\le v\le \frac{3p+1}{2}$\\
$p_{u,v}^w=p^{n-2}-2p^{\frac{n-3}{2}}-p^{\frac{n}{2}-3}\sum\limits_{i\in\mathbb{F}_p}S^{(h)}(i,w-\frac{p+3}{2})S^{(t)}(u-\frac{p+3}{2},i)S^{(t)}(v-\frac{p+3}{2},i)$.
\end{itemize}

Since $q_{u,v}^{w}=q_{v,u}^{w}$ for any $u,v,w\in\{0,1,\dots,\frac{3p+1}{2}\}$, we only give the values of the Krein parameters $q_{u,v}^w$ for $u\le v$ in the following eight cases. 
\begin{itemize}
	\item [$(\rm{\romannumeral1})$] $w=0$ and $0\le u\le v\le \frac{3p+1}{2}$
	\begin{itemize}
		\item[$\bullet$] $u=v$\\
		$q_{0,0}^0=1$, $q_{1,1}^0=p^{\frac{n-1}{2}-1}$, $q_{u,u}^0=2p^{\frac{n-1}{2}}$ for $2\le u\le \frac{p+1}{2}$ and $q_{u,u}^0=p^{n-1}-p^{\frac{n-1}{2}}$ for $\frac{p+3}{2}\le u\le \frac{3p+1}{2}$.
		\item[$\bullet$] $u\ne v$, $q_{u,v}^0=0$.
	\end{itemize}
	\item [$(\rm{\romannumeral2})$] $u=0$ and $1\le w\le \frac{3p+1}{2}$ and $0\le v\le \frac{3p+1}{2}$\\
	$q_{0,v}^w=\begin{cases}
	1,&\text{if}\ w=v,\\
	0,&\text{if}\ w\ne v.
	\end{cases}$
	\item [$(\rm{\romannumeral3})$] $1\le w\le \frac{p+1}{2}$ and $1\le u\le v\le \frac{p+1}{2}$
	\begin{itemize}
		\item [$\bullet$] $w=1$\\
		\newpage
		$q_{u,v}^1=\begin{cases}
		p^{\frac{n-1}{2}}-2, &\text{if}\ u=v=1,\\
		0,&\text{if}\ u=1\ \text{and}\ v\ge 2,\\
		2p^{\frac{n-1}{2}},&\text{if}\ u\ge 2\ \text{and}\ u=v,\\
		0,&\text{if}\ u\ge 2\ \text{and}\ u\ne v.
		\end{cases}$
		\item[$\bullet$] $2\le w\le \frac{p+1}{2}$\\
		$q_{u,v}^w=\begin{cases}
		0,&\hspace{0.5cm}\text{if}\ u=v=1,\\
		2(p^{\frac{n-3}{2}}-p^{-1})+(p^{\frac{n-5}{2}}-p^{-2})\times\\
		\sum\limits_{i\in SQ}S^{(t)}(i,s_{w-1})S^{(h)}(s_{v-1},i),&\hspace{0.5cm}\text{if}\ u=1\ \text{and}\ v\ge 2,\\
		4p^{\frac{n-3}{2}}+p^{\frac{n}{2}-3}\sum\limits_{i\in SQ}S^{(t)}(i,s_{w-1})\times\\
		S^{(h)}(s_{u-1},i)S^{(h)}(s_{v-1},i),&\hspace{0.5cm}\text{if}\  2\le u\le v.
		\end{cases}$
	\end{itemize}
	\item [$(\rm{\romannumeral4})$] $1\le w\le \frac{p+1}{2}$, $1\le u\le \frac{p+1}{2}$ and $\frac{p+3}{2}\le v\le \frac{3p+1}{2}$, $q_{u,v}^w=0$.
	\item [$(\rm{\romannumeral5})$] $1\le w\le \frac{p+1}{2}$ and $\frac{p+3}{2}\le u\le v\le \frac{3p+1}{2}$\\
	$q_{u,v}^w=\begin{cases}
	p^{n-2}-p^{\frac{n-1}{2}},&\text{if}\ w=1\ \text{and}\ u=v,\\
	p^{n-2},&\text{if}\ w=1\ \text{and}\ u\ne v,\\
	p^{n-2}-p^{\frac{n-3}{2}},&\text{if}\ w\ge 2.
	\end{cases}$
	\item [$(\rm{\romannumeral6})$] $\frac{p+3}{2}\le w\le \frac{3p+1}{2}$ and $1\le u\le v\le \frac{p+1}{2}$, $q_{u,v}^w=0$.
	\item [$(\rm{\romannumeral7})$] $\frac{p+3}{2}\le w\le \frac{3p+1}{2}$, $1\le u\le \frac{p+1}{2}$ and $\frac{p+3}{2}\le v\le \frac{3p+1}{2}$\\
	$q_{u,v}^w=\begin{cases}
	p^{\frac{n-3}{2}}-p^{-1}-p^{-2}\sum\limits_{i\in \mathbb{F}_p}S^{(t)}(i,w-\frac{p+3}{2})S^{(h)}(v-\frac{p+3}{2},i),&\text{if}\ u=1,\\
	2p^{\frac{n-3}{2}},&\text{if}\ u\ge 2.
	\end{cases}	$
	\item [$(\rm{\romannumeral8})$] $\frac{p+3}{2}\le w\le \frac{3p+1}{2}$ and $\frac{p+3}{2}\le u\le v\le \frac{3p+1}{2}$\\
	$q_{u,v}^w=p^{n-2}-2p^{\frac{n-3}{2}}-p^{\frac{n}{2}-3}\sum\limits_{i\in\mathbb{F}_p}S^{(t)}(i,w-\frac{p+3}{2})S^{(h)}(u-\frac{p+3}{2},i)S^{(h)}(v-\frac{p+3}{2},i)$.
\end{itemize}
	11. The first and second eigenmatrices, the intersection numbers and the  Krein parameters of the association scheme induced by $U_{11}$.

The first and second eigenmatrices of the association scheme induced by $U_{11}$ are given by Tables 15 and 16. 
\begin{table}[h]
	\vspace{-15pt}
	\centering
	\caption{The first eigenmatrix of the association scheme induced by $U_{11}$}
	\renewcommand\arraystretch{1.5}	
	\resizebox{\textwidth}{!}{
		\begin{tabular}{|m{1.4cm}<{\centering}|c|c|m{3.8cm}<{\centering}|c|m{3.5cm}<{\centering}|}
			\hline
			\diagbox{$i$}{$P_{11}(ij)$}{$j$}&1&2&$3\le j\le \frac{p+3}{2}$&$\frac{p+5}{2}$& $\frac{p+7}{2}\le j\le \frac{3p+3}{2}$\\
			\hline
			1&1&$p^{\frac{n-1}{2}}-1$&$2p^{\frac{n-1}{2}}$&$p^{n-1}-p^{\frac{n-1}{2}}$&$p^{n-1}-p^{\frac{n-1}{2}}$\\
			\hline
			2&1&$-1$&0&0&$\eta(j-\frac{p+5}{2})p^{\frac{n-1}{2}}$\\
			\hline
			$3\le i\le p+1$&1&$-1$&0&$-\eta(i-2)p^{\frac{n-1}{2}}$&$\sqrt{-1}p^{\frac{n-2}{2}}S^{(t)}(j-\frac{p+5}{2},i-2)$\\
			\hline
			$p+2$&1&$p^{\frac{n-1}{2}}-1$&$2p^{\frac{n-1}{2}}$&$-p^{\frac{n-1}{2}}$&$-p^{\frac{n-1}{2}}$\\
			\hline
			$p+3\le i\le \frac{3p+3}{2}$&1&$p^{\frac{n-1}{2}}-1$&$-\sqrt{-1}p^{\frac{n-2}{2}}S^{(t)}(n_{j-2}, s_{i-p-2})$&0&0\\
			\hline
	\end{tabular}}
	\vspace{-15pt}
\end{table}
\begin{table}[h]
	\vspace{0pt}
	\centering
	\caption{The second eigenmatrix of the association scheme induced by $U_{11}$}
	\renewcommand\arraystretch{1.5}	
	\resizebox{\textwidth}{!}{
		\begin{tabular}{|m{1.4cm}<{\centering}|c|c|m{3.55cm}<{\centering}|c|m{3.6cm}<{\centering}|}
			\hline
			\diagbox{$i$}{$Q_{11}(ij)$}{$j$}&1&2&$3\le j\le p+1$&$p+2$& $p+3\le j\le \frac{3p+3}{2}$\\
			\hline
			1&1&$p^{n-1}-p^{\frac{n-1}{2}}$&$p^{n-1}-p^{\frac{n-1}{2}}$&$p^{\frac{n-1}{2}}-1$&$2p^{\frac{n-1}{2}}$\\
			\hline
			2&1&$-p^{\frac{n-1}{2}}$&$-p^{\frac{n-1}{2}}$&$p^{\frac{n-1}{2}}-1$&$2p^{\frac{n-1}{2}}$\\
			\hline
			$3\le i\le \frac{p+3}{2}$&1&0&0&$p^{\frac{n-1}{2}}-1$&$\sqrt{-1}p^{\frac{n-2}{2}}S^{(h)}(s_{j-p-2},n_{i-2})$\\
			\hline
			$\frac{p+5}{2}$&1&0&$-\eta(j-2)p^{\frac{n-1}{2}}$&$-1$&0\\
			\hline
			$\frac{p+7}{2}\le i\le \frac{3p+3}{2}$&1&$\eta(i-\frac{p+5}{2})p^{\frac{n-1}{2}}$&$-\sqrt{-1}p^{\frac{n-2}{2}}S^{(h)}(j-2,i-\frac{p+5}{2})$&$-1$&0\\
			\hline
	\end{tabular}}
	\vspace{-15pt}
\end{table}

Since $p_{u,v}^{w}=p_{v,u}^{w}$ for any $u,v,w\in\{0,1,\dots,\frac{3p+1}{2}\}$, we only give the values of  the intersection numbers $p_{u,v}^w$ for $u\le v$ in the following eight cases.	
\begin{itemize}
	\item [$(\rm{\romannumeral1})$] $w=0$ and $0\le u\le v\le \frac{3p+1}{2}$
	\begin{itemize}
		\item [$\bullet$] $u=v$\\
		$p_{0,0}^0=1$, $p_{1,1}^0=p^{\frac{n-1}{2}}-1$, $p_{u,u}^0=2p^{\frac{n-1}{2}}$ for $2\le u\le \frac{p+1}{2}$ and $p_{u,u}^0=p^{n-1}-p^{\frac{n-1}{2}}$ for $\frac{p+3}{2}\le u\le \frac{3p+1}{2}$.
		\item[$\bullet$] $u\ne v$, $p_{u,v}^0=0$.
	\end{itemize}
	\item [$(\rm{\romannumeral2})$] $u=0$, $1\le w\le \frac{3p+1}{2}$ and $0\le v\le \frac{3p+1}{2}$\\
	$p_{0,v}^w=\begin{cases}
	1,&\text{if}\ w=v,\\
	0,&\text{if}\ w\ne v.
	\end{cases}$
		\item [$(\rm{\romannumeral3})$] $1\le w\le \frac{p+1}{2}$ and $1\le u\le v\le \frac{p+1}{2}$
		\begin{itemize}
			\item [$\bullet$] $w=1$\\
			$p_{u,v}^1=\begin{cases}
			p^{\frac{n-1}{2}}-2,&\text{if}\ u=v=1,\\
			0,&\text{if}\ u=1\ \text{and}\ v\ge 2,\\
			2p^{\frac{n-1}{2}},&\text{if}\ u\ge 2\ \text{and}\ u=v,\\
				0,&\text{if}\ u\ge 2\ \text{and}\ u\ne v.
			\end{cases}$
			\item[$\bullet$] $2\le w\le \frac{p+1}{2}$\\
			$p_{u,v}^w=\begin{cases}
			0,&\text{if}\ u=v=1,\\
			2(p^{\frac{n-3}{2}}-p^{-1})+(p^{\frac{n-5}{2}}-p^{-2})\times\\
			\sum\limits_{i\in SQ}S^{(h)}(i,n_{w-1})S^{(t)}(n_{v-1},i),&\text{if}\ u=1\ \text{and}\ v\ge 2,\\
					4p^{\frac{n-3}{2}}-\sqrt{-1}p^{\frac{n-6}{2}}\sum\limits_{i\in SQ}S^{(h)}(i,n_{w-1})\times\\
					S^{(t)}(n_{u-1},i)S^{(t)}(n_{v-1},i),&\text{if}\ 2\le u\le v.
			\end{cases}$
		\end{itemize}
	\item [$(\rm{\romannumeral4})$] $1\le w\le \frac{p+1}{2}$, $1\le u\le \frac{p+1}{2}$ and $\frac{p+3}{2}\le v\le \frac{3p+1}{2}$, $p_{u,v}^w=0$.
	\item [$(\rm{\romannumeral5})$] $1\le w\le \frac{p+1}{2}$ and $\frac{p+3}{2}\le u\le v\le \frac{3p+1}{2}$\\
	$p_{u,v}^w=\begin{cases}
	p^{n-2}-p^{\frac{n-1}{2}},&\text{if}\ w=1\ \text{and}\ u=v,\\
	p^{n-2},&\text{if}\ w=1\ \text{and}\ u\ne v,\\
	p^{n-2}-p^{\frac{n-3}{2}},&\text{if}\ w\ge 2.
	\end{cases}$
	\item [$(\rm{\romannumeral6})$] $\frac{p+3}{2}\le w\le \frac{3p+1}{2}$ and $1\le u\le v\le \frac{p+1}{2}$, $p_{u,v}^w=0$.
	\item [$(\rm{\romannumeral7})$] $\frac{p+3}{2}\le w\le \frac{3p+1}{2}$, $1\le u\le \frac{p+1}{2}$ and $\frac{p+3}{2}\le v\le \frac{3p+1}{2}$\\
	$p_{u,v}^w=\begin{cases}
	p^{\frac{n-3}{2}}-p^{-1}-p^{-2}\sum\limits_{i\in\mathbb{F}_p}S^{(h)}(i,w-\frac{p+3}{2})S^{(t)}(v-\frac{p+3}{2},i),&\text{if}\ u=1,\\
	2p^{\frac{n-3}{2}},&\text{if}\ u\ge 2.
	\end{cases}$

	\item [$(\rm{\romannumeral8})$] $\frac{p+3}{2}\le w\le \frac{3p+1}{2}$ and $\frac{p+3}{2}\le u\le v\le \frac{3p+1}{2}$\\
	$p_{u,v}^w=p^{n-2}-2p^{\frac{n-3}{2}}+\sqrt{-1}p^{\frac{n-6}{2}}\sum\limits_{i\in\mathbb{F}_p}S^{(h)}(i,w-\frac{p+3}{2})S^{(t)}(u-\frac{p+3}{2},i)S^{(t)}(v-\frac{p+3}{2},i)$.
	
\end{itemize}

Since $q_{u,v}^{w}=q_{v,u}^{w}$ for any $u,v,w\in\{0,1,\dots,\frac{3p+1}{2}\}$, we only give the values of  the Krein parameters $q_{u,v}^w$ for $u\le v$ in the following eight cases.
\begin{itemize}
	\item [$(\rm{\romannumeral1})$] $w=0$ and $0\le u\le v\le \frac{3p+1}{2}$
	\begin{itemize}
		\item [$\bullet$] $u=v$\\
		$q_{0,0}^0=1$, $q_{1,1}^0=p^{n-1}-p^{\frac{n-1}{2}}$, $q_{u,u}^0=p^{n-1}-p^{\frac{n-1}{2}}$ for $2\le u\le p$, $q_{p+1,p+1}^0=p^{\frac{n-1}{2}}-1$ and $q_{u,u}^0=2p^{\frac{n-1}{2}}$ for $p+2\le u\le \frac{3p+1}{2}$.
		\item[$\bullet$] $u\ne v$, $q_{u,v}^w=0$.
	\end{itemize}
\item [$(\rm{\romannumeral2})$] $u=0$, $1\le w\le \frac{3p+1}{2}$ and $0\le v\le \frac{3p+1}{2}$\\
	$q_{u,v}^w=\begin{cases}
	1,&\text{if}\ w=v,\\
	0,&\text{if}\ w\ne v.
	\end{cases}$
	\item [$(\rm{\romannumeral3})$] $1\le w\le p$ and $1\le u\le v\le p$\\
	$q_{u,v}^w=p^{n-2}-2p^{\frac{n-3}{2}}-\sqrt{-1}p^{\frac{n-6}{2}}\sum\limits_{i\in\mathbb{F}_p}S^{(t)}(i,w-1)S^{(h)}(u-1,i)S^{(h)}(v-1,i)$.
		\item [$(\rm{\romannumeral4})$] $1\le w\le p$, $1\le u\le p$ and $p+1\le v\le \frac{3p+1}{2}$\\
	$q_{u,v}^w=\begin{cases}
	p^{\frac{n-3}{2}}-p^{-1}-p^{-2}\sum\limits_{i\in\mathbb{F}_p}S^{(t)}(i,w-1)S^{(h)}(u-1,i),&\text{if}\ v=p+1,\\
	2p^{\frac{n-3}{2}},&\text{if}\ v\ge p+2.
	\end{cases}$
	\item [$(\rm{\romannumeral5})$] $1\le w\le p$ and $p+1\le u\le v\le \frac{3p+1}{2}$, $q_{u,v}^w=0$.
	\item [$(\rm{\romannumeral6})$] $p+1\le w\le \frac{3p+1}{2}$ and $1\le u\le v\le p$\\
	$q_{u,v}^w=\begin{cases}
	p^{n-2}-p^{\frac{n-1}{2}},&\text{if}\ w=p+1\ \text{and}\ u=v,\\
	p^{n-2},&\text{if}\ w=p+1\ \text{and}\ u\ne v,\\
	p^{n-2}-p^{\frac{n-3}{2}},&\text{if}\ w\ge p+2.
	\end{cases}$
	\item [$(\rm{\romannumeral7})$] $p+1\le w\le \frac{3p+1}{2}$, $1\le u\le p$ and $p+1\le v\le  \frac{3p+1}{2}$, $q_{u,v}^w=0$.
		\item [$(\rm{\romannumeral8})$] $p+1\le w\le \frac{3p+1}{2}$ and $p+1\le u\le v\le \frac{3p+1}{2}$
		\begin{itemize}
			\item [$\bullet$] $w=p+1$\\
			$q_{u,v}^{p+1}=\begin{cases}
			p^{\frac{n-1}{2}}-2,&\text{if}\ u=v=p+1,\\
				0,&\text{if}\ u=p+1\ \text{and}\ v\ge p+2,\\
				2p^{\frac{n-1}{2}},&\text{if}\ u\ge p+2\ \text{and}\ u=v,\\
				0,&\text{if}\ u\ge p+2\ \text{and}\ u\ne v.
			\end{cases}$
	\item[$\bullet$] $p+2\le w\le \frac{3p+1}{2}$\\
$q_{u,v}^w=\begin{cases}
0,&\text{if}\ u=v=p+1,\\
2(p^{\frac{n-3}{2}}-p^{-1})+(p^{\frac{n-5}{2}}-p^{-2})\times&\text{if}\ u=p+1\ \text{and}\\
\sum\limits_{i\in NSQ}S^{(t)}(i,s_{w-p-1})S^{(h)}(s_{v-p-1},i),&\ v\ge p+2,\\
4p^{\frac{n-3}{2}}+\sqrt{-1}p^{\frac{n-6}{2}}\sum\limits_{i\in NSQ}S^{(t)}(i,s_{w-p-1})\times\\
S^{(h)}(s_{u-p-1},i)S^{(h)}(s_{v-p-1},i),&\text{if}\ p+2\le u\le v.
\end{cases}$
		\end{itemize}
\end{itemize}
		12. The first and second eigenmatrices, the intersection numbers and the Krein parameters of the association scheme induced by  $U_{12}$.
	
	The first and second eigenmatrices of the association scheme induced by $U_{12}$ are given in Tables 17 and 18.
		\begin{table}[h]
		\vspace{-15pt}
		\centering
		\caption{The first eigenmatrix of the association scheme induced by $U_{12}$}
		\renewcommand\arraystretch{1.5}	
		\resizebox{\textwidth}{!}{
			\begin{tabular}{|m{1.4cm}<{\centering}|c|c|m{2cm}<{\centering}|m{2.3cm}<{\centering}|m{2.3cm}<{\centering}|m{2.3cm}<{\centering}|}
				\hline
				\diagbox{$i$}{$P_{12}(ij)$}{$j$}&1&2&3&$4\le j\le p+2$&$p+3$&$p+4\le j\le 2p+2$\\
				\hline
				1&1&$p^{n-r}-1$&$p^{r-1}-p^{n-r}$&$\eta(j-3)p^{\frac{n-1}{2}}+p^{r-1}$&$p^{n-1}-p^{r-1}$&$p^{n-1}-p^{r-1}$\\
				\hline
				2&1&$p^{n-r}-1$&$p^{r-1}-p^{n-r}$&$\eta(j-3)p^{\frac{n-1}{2}}+p^{r-1}$&$-p^{r-1}$&$-p^{r-1}$\\
				\hline
				3&1&$p^{n-r}-1$&$-p^{n-r}$&$\eta(j-3)p^{\frac{n-1}{2}}$&0&0\\
				\hline
				$4\le i\le p+2$&1&$p^{n-r}-1$&$\eta(i-3)p^{\frac{n-1}{2}}-p^{n-r}$&$p^{\frac{n-2}{2}}S^{(t)}(j-3,i-3)$&0&0\\
				\hline
				$p+3$&1&$-1$&0&0&0&$-\eta(j-p-3)p^{\frac{n-1}{2}}$\\
				\hline
				$p+4\le i\le 2p+2$&1&$-1$&0&0&$-\eta(i-p-3)p^{\frac{n-1}{2}}$&$-p^\frac{n-2}{2}S^{(t)}(j-p-3,i-p-3)$\\
				\hline
		\end{tabular}}
		\vspace{-15pt}
	\end{table}
	\begin{table}[h]
		\vspace{-15pt}
		\centering
		\caption{The second eigenmatrix of the association scheme induced by $U_{12}$}
		\renewcommand\arraystretch{1.5}	
		\resizebox{\textwidth}{!}{
			\begin{tabular}{|m{1.4cm}<{\centering}|c|c|m{2cm}<{\centering}|m{2.3cm}<{\centering}|m{2.3cm}<{\centering}|m{2.3cm}<{\centering}|}
				\hline
				\diagbox{$i$}{$Q_{12}(ij)$}{$j$}&1&2&3&$4\le j\le p+2$&$p+3$&$p+4\le j\le 2p+2$\\
				\hline
				1&1&$p^{n-r}-1$&$p^{r-1}-p^{n-r}$&$\eta(j-3)p^{\frac{n-1}{2}}+p^{r-1}$&$p^{n-1}-p^{r-1}$&$p^{n-1}-p^{r-1}$\\
				\hline
				2&1&$p^{n-r}-1$&$p^{r-1}-p^{n-r}$&$\eta(j-3)p^{\frac{n-1}{2}}+p^{r-1}$&$-p^{r-1}$&$-p^{r-1}$\\
				\hline
				3&1&$p^{n-r}-1$&$-p^{n-r}$&$\eta(j-3)p^{\frac{n-1}{2}}$&0&0\\
				\hline
				$4\le i\le p+2$&1&$p^{n-r}-1$&$\eta(i-3)p^{\frac{n-1}{2}}-p^{n-r}$&$p^{\frac{n-2}{2}}S^{(h)}(j-3,i-3)$&0&0\\
				\hline
				$p+3$&1&$-1$&0&0&0&$-\eta(j-p-3)p^{\frac{n-1}{2}}$\\
				\hline
				$p+4\le i\le 2p+2$&1&$-1$&0&0&$-\eta(i-p-3)p^{\frac{n-1}{2}}$&$-p^\frac{n-2}{2}S^{(h)}(j-p-3,i-p-3)$\\
				\hline
		\end{tabular}}
		\vspace{-15pt}
	\end{table}
	
	Since $p_{u,v}^{w}=p_{v,u}^{w}$ for any $u,v,w\in\{0,1,\dots,2p+1\}$, we only give the values of the intersection numbers $p_{u,v}^w$ for $u\le v$ in the following eight cases. 

	\begin{itemize}
		\item [$(\rm{\romannumeral1})$] $w=0$ and $0\le u\le v\le 2p+1$
		\begin{itemize}
			\item [$\bullet$] $u=v$\\
			$p_{0,0}^{0}=1$, $p_{1,1}^0=p^{n-r}-1$, $p_{2,2}^0=p^{r-1}-p^{n-r}$, $p_{u,u}^0=\eta(u-2)p^{\frac{n-1}{2}}+p^{r-1}$ for $3\le u\le p+1$ and $p_{u,u}^0=p^{n-1}-p^{r-1}$ for $p+2\le u\le 2p+1$.
			\item[$\bullet$] $u\ne v$, $p_{u,v}^0=0$.
		\end{itemize}
		\item [$(\rm{\romannumeral2})$] $u=0$, $1\le w\le 2p+1$ and $0\le v\le 2p+1$\\
		$p_{u,v}^w=\begin{cases}
		1,&\text{if}\ w=v,\\
		0,&\text{if}\ w\ne v.
		\end{cases}$
			\item [$(\rm{\romannumeral3})$] $1\le w\le p+1$ and $1\le u\le v\le p+1$
			\begin{itemize}
				\item [$\bullet$] $w=1$\\
				$p_{u,v}^1
				=\begin{cases}
				p^{n-r}-2,&\text{if}\ u=v=1,\\
				0,&\text{if}\ u=1\ \text{and}\ v\ge 2,\\
				p^{r-1}-p^{n-r},&\text{if}\ u=v=2,\\
				\eta(u-2)p^{\frac{n-1}{2}}+p^{r-1},&\text{if}\ u\ge 3\ \text{and}\  u=v,\\
				0,&\text{if}\ u\ge 2\ \text{and}\ u\ne v.
				\end{cases}$
				\item[$\bullet$] $w=2$\\
				$p_{u,v}^2=\begin{cases}
				p^{n-r}-1,&\text{if}\ u=1\ \text{and}\ v=2,\\
				0,&\text{if}\ u=1\ \text{and}\ v\ne 2,\\
				p^{r-2}-2p^{n-r},&\text{if}\ u=v=2,\\
				p^{r-2},&\text{if}\ u=2\ \text{and}\ v\ge 3,\\
				p^{r-2}+(\eta(u-2)+\eta(v-2))p^{\frac{n-3}{2}}+\\
				p^{\frac{n-5}{2}}\sum\limits_{i\in\mathbb{F}_p^{*}}\eta(i)S^{(t)}(u-2,i)S^{(t)}(v-2,i),&\text{if}\ 3\le u\le v.
				\end{cases}$
				\item[$\bullet$] $3\le w\le p+1$\\
				$p_{u,v}^w=\begin{cases}
				p^{n-r}-1,&\text{if}\ u=1\ \text{and}\\& w=v,\\
				0,&\text{if}\ u=1\\& \text{and}\ w\ne v,\\
				p^{r-2}-\eta(w-2)p^{\frac{n-3}{2}},&\text{if}\ u=v=2,\\
				H,&\text{if}\ u=2\\& \text{and}\ v\ge 3,\\
					p^{r-2}+(\eta(u-2)+\eta(v-2))p^{\frac{n-3}{2}}+p^{\frac{n-6}{2}}\times\\
					\sum\limits_{i\in\mathbb{F}_p}S^{(h)}(i,w-2)S^{(t)}(u-2,i)S^{(t)}(v-2,i),&\text{if}\ 3\le u\le v,
				\end{cases}$
				
			where $H=p^{r-2}-p^{n-r-1}+\eta(v-2)p^{\frac{n-3}{2}}-\eta((w-2)(v-2))p^{n-r-1}\\+p^{-2}\sum\limits_{i\in\mathbb{F}_p^*}(\eta(i)p^{\frac{n-1}{2}}-p^{n-r})S^{(h)}(i,w-2)S^{(t)}(v-2,i).$	
			\end{itemize}
			\item [$(\rm{\romannumeral4})$] $1\le w\le p+1$, $1\le u\le p+1$ and $p+2\le v\le 2p+1$, $p_{u,v}^w=0$.
			\item [$(\rm{\romannumeral5})$] $1\le w\le p+1$ and $p+2\le u\le v\le 2p+1$\\
			$p_{u,v}^w=\begin{cases}
			p^{n-2}-p^{r-1},&\text{if}\ w=1\ \text{and}\ u=v,\\
			p^{n-2},&\text{if}\ w=1\ \text{and}\ u\ne v,\\
			p^{n-2}-p^{r-2},&\text{if}\ w\ge 2.
			\end{cases}$
	\item [$(\rm{\romannumeral6})$] $p+2\le w\le 2p+1$ and $1\le u\le v\le p+1$, $p_{u,v}^w=0$.
	\item [$(\rm{\romannumeral7})$] $p+2\le w\le 2p+1$, $1\le u\le p+1$ and $p+2\le v\le 2p+1$\\
	$p_{u,v}^w=\begin{cases}
	p^{n-r-1}-p^{-1}-p^{-2}\sum\limits_{i\in\mathbb{F}_p}S^{(h)}(i,w-p-2)\times\\
	S^{(t)}(v-p-2,i),&\text{if}\ u=1,\\
	p^{r-2}-p^{n-r-1},&\text{if}\ u=2,\\
	p^{r-2}+\eta(u-2)p^{\frac{n-3}{2}},&\text{if } u\ge 3.\end{cases}$
	\item [$(\rm{\romannumeral8})$] $p+2\le w\le 2p+1$ and $p+2\le u\le v\le 2p+1$\\
	$p_{u,v}^w=p^{n-2}-2p^{r-2}-p^{\frac{n-6}{2}}\sum\limits_{i\in\mathbb{F}_p}S^{(h)}(i,w-p-2)S^{(t)}(u-p-2,i)S^{(t)}(v-p-2,i)$.
	\end{itemize}
	
Since $q_{u,v}^{w}=q_{v,u}^{w}$ for any $u,v,w\in\{0,1,\dots,2p+1\}$, we only give the values of the Krein parameters $q_{u,v}^w$ for $u\le v$ in the following eight cases. 

\begin{itemize}
	\item [$(\rm{\romannumeral1})$] $w=0$ and $0\le u\le v\le 2p+1$
	\begin{itemize}
		\item [$\bullet$] $u=v$\\
		$q_{0,0}^{0}=1$, $q_{1,1}^0=p^{n-r}-1$, $q_{2,2}^0=p^{r-1}-p^{n-r}$, $q_{u,u}^0=\eta(u-2)p^{\frac{n-1}{2}}+p^{r-1}$ for $3\le u\le p+1$ and $q_{u,u}^0=p^{n-1}-p^{r-1}$ for $p+2\le u\le 2p+1$.
		\item[$\bullet$] $u\ne v$, $q_{u,v}^0=0$.
	\end{itemize}
	\item [$(\rm{\romannumeral2})$] $u=0$, $1\le w\le 2p+1$ and $0\le v\le 2p+1$\\
	$q_{u,v}^w=\begin{cases}
	1,&\text{if}\ w=v,\\
	0,&\text{if}\ w\ne v.
	\end{cases}$
	\item [$(\rm{\romannumeral3})$] $1\le w\le p+1$ and $1\le u\le v\le p+1$
	\begin{itemize}
		\item [$\bullet$] $w=1$\\
		$q_{u,v}^1
		=\begin{cases}
		p^{n-r}-2,&\text{if}\ u=v=1,\\
		0,&\text{if}\ u=1\ \text{and}\ v\ge 2,\\
		p^{r-1}-p^{n-r},&\text{if}\ u=v=2,\\
		\eta(u-2)p^{\frac{n-1}{2}}+p^{r-1},&\text{if}\ u\ge 3\ \text{and}\  u=v,\\
		0,&\text{if}\ u\ge 2\ \text{and}\ u\ne v.
		\end{cases}$
		\item[$\bullet$] $w=2$\\
		$q_{u,v}^2=\begin{cases}
		p^{n-r}-1,&\text{if}\ u=1\ \text{and}\ v=2,\\
		0,&\text{if}\ u=1\ \text{and}\ v\ne 2,\\
		p^{r-2}-2p^{n-r},&\text{if}\ u=v=2,\\
		p^{r-2},&\text{if}\ u=2\ \text{and}\ v\ge 3,\\
		p^{r-2}+(\eta(u-2)+\eta(v-2))p^{\frac{n-3}{2}}+\\
		p^{\frac{n-5}{2}}\sum\limits_{i\in\mathbb{F}_p^{*}}\eta(i)S^{(h)}(u-2,i)S^{(h)}(v-2,i),&\text{if}\ 3\le u\le v.
		\end{cases}$
		\item[$\bullet$] $3\le w\le p+1$\\
		$q_{u,v}^w=\begin{cases}
		p^{n-r}-1,&\text{if}\ u=1\ \text{and}\\& w=v,\\
		0,&\text{if}\ u=1\ \text{and}\\& w\ne v,\\
		p^{r-2}-\eta(w-2)p^{\frac{n-3}{2}},&\text{if}\ u=v=2,\\
		H,&\text{if}\ u=2\ \text{and}\\& v\ge 3,\\
		p^{r-2}+(\eta(u-2)+\eta(v-2))p^{\frac{n-3}{2}}+p^{\frac{n-6}{2}}\times\\
		\sum\limits_{i\in\mathbb{F}_p}S^{(t)}(i,w-2)S^{(h)}(u-2,i)S^{(h)}(v-2,i),&\text{if}\ 3\le u\le v,
		\end{cases}$
		
		where $H=p^{r-2}-p^{n-r-1}+\eta(v-2)p^{\frac{n-3}{2}}-\eta((w-2)(v-2))p^{n-r-1}\\+p^{-2}\sum\limits_{i\in\mathbb{F}_p^*}(\eta(i)p^{\frac{n-1}{2}}-p^{n-r})S^{(t)}(i,w-2)S^{(h)}(v-2,i).$	
	\end{itemize}
	\item [$(\rm{\romannumeral4})$] $1\le w\le p+1$, $1\le u\le p+1$ and $p+2\le v\le 2p+1$, $q_{u,v}^w=0$.
	\item [$(\rm{\romannumeral5})$] $1\le w\le p+1$ and $p+2\le u\le v\le 2p+1$\\
	$q_{u,v}^w=\begin{cases}
	p^{n-2}-p^{r-1},&\text{if}\ w=1\ \text{and}\ u=v,\\
	p^{n-2},&\text{if}\ w=1\ \text{and}\ u\ne v,\\
	p^{n-2}-p^{r-2},&\text{if}\ w\ge 2.
	\end{cases}$
	\item [$(\rm{\romannumeral6})$] $p+2\le w\le 2p+1$ and $1\le u\le v\le p+1$, $q_{u,v}^w=0$.
	\item [$(\rm{\romannumeral7})$] $p+2\le w\le 2p+1$, $1\le u\le p+1$ and $p+2\le v\le 2p+1$\\
	$q_{u,v}^w=\begin{cases}
	p^{n-r-1}-p^{-1}-p^{-2}\sum\limits_{i\in\mathbb{F}_p}S^{(t)}(i,w-p-2)\times\\
	S^{(h)}(v-p-2,i),&\text{if}\ u=1,\\
	p^{r-2}-p^{n-r-1},&\text{if}\ u=2,\\
	p^{r-2}+\eta(u-2)p^{\frac{n-3}{2}},&\text{if } u\ge 3.\end{cases}$
	\item [$(\rm{\romannumeral8})$] $p+2\le w\le 2p+1$ and $p+2\le u\le v\le 2p+1$\\
	$q_{u,v}^w=p^{n-2}-2p^{r-2}-p^{\frac{n-6}{2}}\sum\limits_{i\in\mathbb{F}_p}S^{(t)}(i,w-p-2)S^{(h)}(u-p-2,i)S^{(h)}(v-p-2,i)$.
\end{itemize}
13. The first and second eigenmatrices, the intersection numbers and the  Krein parameters of the association scheme induced by $U_{13}$.

The first and second eigenmatrices of the association scheme induce by $U_{13}$ are given by Tables 19 and 20. 
\begin{table}[h]
	\vspace{-15pt}
	\centering
	\caption{The first eigenmatrix of the association scheme induced by $U_{13}$}
	\renewcommand\arraystretch{1.5}	
	\resizebox{\textwidth}{!}{
		\begin{tabular}{|m{1.4cm}<{\centering}|c|c|m{2cm}<{\centering}|m{2.4cm}<{\centering}|m{2.3cm}<{\centering}|m{2.3cm}<{\centering}|}
			\hline
			\diagbox{$i$}{$P_{13}(ij)$}{$j$}&1&2&3&$4\le j\le p+2$&$p+3$&$p+4\le j\le 2p+2$\\
			\hline
			1&1&$p^{n-r}-1$&$p^{r-1}-p^{n-r}$&$-\eta(j-3)p^{\frac{n-1}{2}}+p^{r-1}$&$p^{n-1}-p^{r-1}$&$p^{n-1}-p^{r-1}$\\
			\hline
			2&1&$-1$&0&0&0&$\eta(j-p-3)p^{\frac{n-1}{2}}$\\
			\hline
			$3\le i\le p+1$&1&$-1$&0&0&$-\eta(i-2)p^{\frac{n-1}{2}}$&$\sqrt{-1}p^{\frac{n-2}{2}}S^{(t)}(j-p-3,i-2)$\\
			\hline
			$p+2$&1&$p^{n-r}-1$&$p^{r-1}-p^{n-r}$&$-\eta(j-3)p^{\frac{n-1}{2}}+p^{r-1}$&$-p^{r-1}$&$-p^{r-1}$\\
			\hline
			$p+3$&1&$p^{n-r}-1$&$-p^{n-r}$&$-\eta(j-3)p^{\frac{n-1}{2}}$&0&0\\
			\hline
			$p+4\le i\le 2p+2$&1&$p^{n-r}-1$&$-p^{n-r}+\eta(i-p-3)p^{\frac{n-1}{2}}$&$-\sqrt{-1}p^{\frac{n-2}{2}}S^{(t)}(j-3,i-p-3)$&0&0\\
			\hline
	\end{tabular}}
	\vspace{-15pt}
\end{table}
\begin{table}[h]
	\vspace{0pt}
	\centering
	\caption{The second eigenmatrix of the association scheme induced by $U_{13}$}
	\renewcommand\arraystretch{1.5}	
	\resizebox{\textwidth}{!}{
		\begin{tabular}{|m{1.4cm}<{\centering}|c|m{2cm}<{\centering}|m{2.7cm}<{\centering}|c|m{2.3cm}<{\centering}|m{2.3cm}<{\centering}|}
			\hline
			\diagbox{$i$}{$Q_{13}(ij)$}{$j$}&1&2&$3\le j\le p+1$&$p+2$&$p+3$&$p+4\le j\le 2p+2$\\
			\hline
			1&1&$p^{n-1}-p^{r-1}$&$p^{n-1}-p^{r-1}$&$p^{n-r}-1$&$p^{r-1}-p^{n-r}$&${\small \eta(j-p-3)}p^{\frac{n-1}{2}}+p^{r-1}$\\
			\hline
			2&1&$-p^{r-1}$&$-p^{r-1}$&$p^{n-r}-1$&$p^{r-1}-p^{n-r}$&${\small \eta(j-p-3)}p^{\frac{n-1}{2}}+p^{r-1}$\\
			\hline
			3&1&0&0&$p^{n-r}-1$&$-p^{n-r}$&$\eta(j-p-3)p^{\frac{n-1}{2}}$\\
			\hline
			$4\le i\le p+2$&1&0&0&$p^{n-r}-1$&$-\eta(i-3)p^{\frac{n-1}{2}}-p^{n-r}$&$\sqrt{-1}p^{\frac{n-2}{2}}S^{(h)}(j-p-3,i-3)$\\
			\hline
			$p+3$&1&0&$-\eta(j-2)p^{\frac{n-1}{2}}$&$-1$&0&0\\
			\hline
			$p+4\le i\le 2p+2$&1&$\eta(i-p-3)p^{\frac{n-1}{2}}$&$-\sqrt{-1}p^{\frac{n-2}{2}}S^{(h)}(j-2,i-p-3)$&$-1$&0&0\\
			\hline
	\end{tabular}}
	\vspace{-15pt}
\end{table}

Since $p_{u,v}^{w}=p_{v,u}^{w}$ for any $u,v,w\in\{0,1,\dots,2p+1\}$, we only give the values of  the intersection numbers $p_{u,v}^w$ for $u\le v$ in the following eight cases.
\begin{itemize}
	\item [$(\rm{\romannumeral1})$] $w=0$ and $0\le u\le v\le2p+1$
\begin{itemize}
	\item [$\bullet$] $u=v$\\
	$p_{0,0}^0=1$, $p_{1,1}^0=p^{n-r}-1$, $p_{2,2}^0=p^{r-1}-p^{n-r}$, $p_{u,u}^0=p^{r-1}-\eta(u-2)p^{\frac{n-1}{2}}$ for $3\le u\le p+1$ and $p_{u,u}^0=p^{n-1}-p^{r-1}$ for $p+2\le u\le 2p+1$.
	\item[$\bullet$] $u\ne v$, $p_{u,v}^0=0$.
\end{itemize}
	\item [$(\rm{\romannumeral2})$] $u=0$, $1\le w\le 2p+1$ and $0\le v\le 2p+1$\\
	$p_{u,v}^w=\begin{cases}
	1,&\text{if } w=v,\\
	0,&\text{if } w\ne v.
	\end{cases}$
	\item [$(\rm{\romannumeral3})$] $1\le w\le p+1$ and $1\le u\le v\le p+1$
	\begin{itemize}
		\item [$\bullet$] $w=1$\\
		$p_{u,v}^1=\begin{cases}
		p^{n-r}-2,&\text{if } u=v=1,\\
		0,&\text{if } u=1\ \text{and}\ v\ge 2,\\
		p^{r-1}-p^{n-r},&\text{if}\ u=v=2,\\
		p^{r-1}-\eta(u-2)p^{\frac{n-1}{2}},&\text{if}\ u\ge 3\ \text{and}\ u=v,\\
		0,&\text{if}\ u\ge 2\ \text{and}\ u\ne v.
		\end{cases}$
		\item[$\bullet$] $w=2$\\
		$p_{u,v}^2=\begin{cases}
		p^{n-r}-1,&\text{if}\ u=1\ \text{and}\ v=2,\\
		0,&\text{if}\ u=1\ \text{and}\ v\ne 2,\\
		p^{r-2}-2p^{n-r},&\text{if}\ u=v=2,\\
		p^{r-2},&\text{if}\ u=2\ \text{and}\ v\ge 3,\\
		p^{r-2}-(\eta(u-2)+\eta(v-2))p^{\frac{n-3}{2}}-&\\
		p^{\frac{n-5}{2}}\sum\limits_{i\in\mathbb{F}_p^*}\eta(i)S^{(t)}(u-2,i)S^{(t)}(v-2,i),&\text{if}\ 3\le u\le v.
		\end{cases}$
		\item[$\bullet$] $3\le w\le p+1$\\
		$p_{u,v}^w=\begin{cases}
		p^{n-r}-1,&\text{if}\ u=1\ \text{and}\\& w=v,\\
		0,&\text{if}\ u=1\ \text{and}\\& w\ne v,\\
		p^{r-2}+\eta(w-2)p^{\frac{n-3}{2}},&\text{if}\ u=v=2,\\
	H,&\text{if}\ u=2\ \text{and}\\& v\ge 3,\\
	p^{r-2}-(\eta(u-2)+\eta(v-2))p^{\frac{n-3}{2}}-\sqrt{-1}p^{\frac{n-6}{2}}\times\\
\sum\limits_{i\in\mathbb{F}_p}S^{(h)}(i,w-2)S^{(t)}(u-2,i)S^{(t)}(v-2,i),&\text{if}\ 3\le u\le v,
		\end{cases}$
		where $H=p^{r-2}-p^{n-r-1}-\eta(v-2)p^{\frac{n-3}{2}}-\eta((w-2)(v-2))p^{n-r-1}+p^{-2}\sum\limits_{i\in\mathbb{F}_p^*}(\eta(i)p^{\frac{n-1}{2}}-p^{n-r})S^{(h)}(i,w-2)S^{(t)}(v-2,i)$.
	\end{itemize}
\item [$(\rm{\romannumeral4})$] $1\le w\le p+1$, $1\le u\le p+1$ and $p+2\le v\le 2p+1$, $p_{u,v}^w=0$.
\item [$(\rm{\romannumeral5})$] $1\le w\le p+1$ and $p+2\le u\le v\le 2p+1$\\
$p_{u,v}^w=\begin{cases}
p^{n-2}-p^{r-1},&\text{if}\ w=1\ \text{and}\ u=v,\\
p^{n-2},&\text{if}\ w=1\ \text{and}\ u\ne v,\\
p^{n-2}-p^{r-2},&\text{if}\ w\ge 2.
\end{cases}$
\item [$(\rm{\romannumeral6})$] $p+2\le w\le 2p+1$ and $1\le u\le v\le p+1$, $p_{u,v}^w=0$.
\item [$(\rm{\romannumeral7})$] $p+2\le w\le 2p+1$, $1\le u\le p+1$ and $p+2\le v\le 2p+1$\\
$p_{u,v}^w=\begin{cases}
p^{n-r-1}-p^{-1}-p^{-2}\sum\limits_{i\in\mathbb{F}_p}S^{(h)}(i,w-p-2)\times\\
S^{(t)}(v-p-2,i),&\text{if}\ u=1,\\
p^{r-2}-p^{n-r-1},&\text{if}\ u=2,\\
p^{r-2}-\eta(u-2)p^{\frac{n-3}{2}},&\text{if}\ u\ge3.
\end{cases}$
\item [$(\rm{\romannumeral8})$] $p+2\le w\le 2p+1$ and $p+2\le u\le v\le 2p+1$\\
$p_{u,v}^w=p^{n-2}-2p^{r-2}+\sqrt{-1}p^{\frac{n-6}{2}}\sum\limits_{i\in\mathbb{F}_p}S^{(h)}(i,w-p-2)S^{(t)}(u-p-2,i)S^{(t)}(v-p-2,i)$.
\end{itemize}

Since $q_{u,v}^{w}=q_{v,u}^{w}$ for any $u,v,w\in\{0,1,\dots,2p+1\}$, we only give the values of  the Krein parameters $q_{u,v}^w$ for $u\le v$ in the following eight cases.
\begin{itemize}
	\item [$(\rm{\romannumeral1})$] $w=0$ and $0\le u\le v\le 2p+1$
	\begin{itemize}
		\item [$\bullet$] $u=v$\\
		$q_{0,0}^0=1$, $q_{u,u}^0=p^{n-1}-p^{r-1}$ for $1\le u\le p$, $q_{p+1,p+1}^0=p^{n-r}-1$, $q_{p+2,p+2}^0=p^{r-1}-p^{n-r}$ and $q_{u,u}^0=p^{r-1}+\eta(u-p-2)p^{\frac{n-1}{2}}$ for $p+3\le u\le 2p+1$.
		\item[$\bullet$] $u\ne v$, $q_{u,v}^0=0$.
	\end{itemize}
\item [$(\rm{\romannumeral2})$] $u=0$, $1\le w\le 2p+1$ and $0\le v\le 2p+1$\\
$q_{0,v}^w=\begin{cases}
1,&\text{if}\ w=v,\\
0,&\text{if}\ w\ne v.
\end{cases}$
\item [$(\rm{\romannumeral3})$] $1\le w\le p$ and $1\le u\le v\le p$\\
$q_{u,v}^w=p^{n-2}-2p^{r-2}-\sqrt{-1}p^{\frac{n-6}{2}}\sum\limits_{i\in\mathbb{F}_p}S^{(t)}(i,w-1)S^{(h)}(u-1,i)S^{(h)}(v-1,i)$.
\item [$(\rm{\romannumeral4})$] $1\le w\le p$, $1\le u\le p$ and $p+1\le v\le 2p+1$\\
$q_{u,v}^w=\begin{cases}
p^{n-r-1}-p^{-1}-p^{-2}\sum\limits_{i\in\mathbb{F}_p}S^{(t)}(i,w-1)S^{(h)}(u-1,i),&\text{if}\ v=p+1,\\
p^{r-2}-p^{n-r-1},&\text{if}\ v=p+2,\\
p^{r-2}+\eta(v-p-2)p^{\frac{n-3}{2}},&\text{if}\ v\ge p+3.
\end{cases}$
\item [$(\rm{\romannumeral5})$] $1\le w\le p$ and $p+1\le u\le v\le 2p+1$, $q_{u,v}^w=0$.
\item [$(\rm{\romannumeral6})$] $p+1\le w\le 2p+1$ and $1\le u\le v\le p$\\
$q_{u,v}^w=\begin{cases}
p^{n-2}-p^{r-1},&\text{if}\ w=p+1\ \text{and}\ u=v,\\
p^{n-2},&\text{if}\ w=p+1\ \text{and}\ u\ne v,\\
p^{n-2}-p^{r-2},&\text{if}\ w\ge p+2.
\end{cases}$
\item [$(\rm{\romannumeral7})$] $p+1\le w\le 2p+1$, $1\le u\le p$ and $p+1\le v\le 2p+1$, $q_{u,v}^w=0$.
\item [$(\rm{\romannumeral8})$] $p+1\le w\le 2p+1$ and $p+1\le u\le v\le 2p+1$
\begin{itemize}
	\item [$\bullet$] $w=p+1$\\
	$q_{u,v}^{p+1}=\begin{cases}
	p^{n-r}-2,&\text{if}\ u=v=p+1,\\
	0,&\text{if}\ u=p+1\ \text{and}\ v\ge p+2,\\
	p^{r-1}-p^{n-r},&\text{if}\ u=v=p+2,\\
	p^{r-1}+\eta(u-p-2)p^{\frac{n-1}{2}},&\text{if}\ u\ge p+3\ \text{and}\ u=v,\\
	0,&\text{if}\ u\ge p+2\ \text{and}\ u\ne v.
	\end{cases}$
	\item[$\bullet$] $w=p+2$\\
	$q_{u,v}^{p+2}=\begin{cases}
	p^{n-r}-1,&\hspace{-2.9cm}\text{if}\ u=p+1\\&\hspace{-2.9cm} \text{and}\ v=p+2,\\
	0,&\hspace{-2.9cm}\text{if}\ u=p+1\\&\hspace{-2.9cm}\text{and}\ v\ne p+2,\\
	p^{r-2}-2p^{n-r},&\hspace{-2.9cm}\text{if}\ {\tiny u=v=p+2},\\
	p^{r-2},&\hspace{-2.9cm}\text{if}\ u=p+2\\&\hspace{-2.9cm} \text{and}\ v\ge p+3,\\
 p^{r-2}+
 p^{\frac{n-5}{2}}\sum\limits_{i\in\mathbb{F}_p^*}\eta(i) S^{(h)}( u-p-2,i)S^{(h)}( v-\\p-2,i)+( \eta( u-p-2)+\eta( v-p-2))p^{\frac{n-3}{2}},\text{if}\ {\tiny p+3\le u\le v}.
	\end{cases}$
	\item[$\bullet$] $p+3\le w\le 2p+1$\\
	$q_{u,v}^w=\begin{cases}
	p^{n-r}-1,&\text{if}\ u=p+1\ \text{and}\ w=v,\\
	0,&\text{if}\ u=p+1\ \text{and}\ w\ne v,\\
	p^{r-2}-p^{\frac{n-3}{2}}\eta(w-p-2),&\text{if}\ u=v=p+2,\\
	H_1,&\text{if}\ u=p+2\ \text{and}\ v\ge p+3,\\
	H_2,&\text{if}\ p+3\le u\le v,
	\end{cases}$
	
	where $H_1=p^{r-2}+\eta(v-p-2)p^{\frac{n-3}{2}}-p^{n-r-1}-\eta((w-p-2)(v-p-2))p^{n-r-1}-p^{-2}\sum\limits_{i\in\mathbb{F}_p^*}(\eta(i)p^{\frac{n-1}{2}}+p^{n-r})S^{(t)}(i,w-p-2)S^{(h)}(v-p-2,i)$ and  $H_2=p^{r-2}+(\eta(u-p-2)+\eta(v-p-2))p^{\frac{n-3}{2}}+\sqrt{-1}p^{\frac{n-6}{2}}\sum\limits_{i\in\mathbb{F}_p}S^{(t)}(i,w-p-2)S^{(h)}(u-p-2,i)S^{(h	)}(v-p-2,i)$.
\end{itemize}
\end{itemize}
14. The first and second eigenmatrices, the intersection numbers and the Krein parameters of the association scheme induced by  $U_{14}$.

The first and second eigenmatrices of the association scheme induced by $U_{14}$ are given by Tables 21 and 22.
\begin{table}[h]
	\vspace{0pt}
	\centering
	\caption{The first eigenmatrix of the association scheme induced by $U_{14}$}
	\renewcommand\arraystretch{1.5}	
	\resizebox{\textwidth}{!}{
		\begin{tabular}{|m{1.4cm}<{\centering}|c|c|c|c|m{3.8cm}<{\centering}|}
			\hline
			\diagbox{$i$}{$P_{14}(ij)$}{$j$}&1&2&$3\le j\le p+1$&$p+2$&$p+3\le j\le \frac{3p+3}{2}$\\
			\hline
			1&1&$p^{n-1}-p^{\frac{n-1}{2}}$&$p^{n-1}-p^{\frac{n-1}{2}}$&$p^{\frac{n-1}{2}}-1$&$2p^{\frac{n-1}{2}}$\\
			\hline
			2&1&0&$\eta(j-2)p^{\frac{n-1}{2}}$&$-1$&0\\
			\hline
			$3\le i\le p+1$&1&$\eta(i-2)p^{\frac{n-1}{2}}$&$p^{\frac{n-2}{2}}S^{(t)}(j-2,i-2)$&$-1$&0\\
			\hline
			$p+2$&1&$-p^{\frac{n-1}{2}}$&$-p^{\frac{n-1}{2}}$&$p^{\frac{n-1}{2}}-1$&$2p^{\frac{n-1}{2}}$\\
			\hline
			$p+3\le i\le\frac{3p+3}{2}$&1&0&0&$p^{\frac{n-1}{2}}-1$&$-p^{\frac{n-2}{2}}S^{(t)}(n_{j-p-2},n_{i-p-2})$\\
			\hline
	\end{tabular}}
	\vspace{0pt}
\end{table}
\begin{table}[h]
	\vspace{2pt}
	\centering
	\caption{The second eigenmatrix of the association scheme induced by $U_{14}$}
	\renewcommand\arraystretch{1.5}	
	\resizebox{\textwidth}{!}{
		\begin{tabular}{|m{1.4cm}<{\centering}|c|c|c|c|m{3.8cm}<{\centering}|}
			\hline
			\diagbox{$i$}{$Q_{14}(ij)$}{$j$}&1&2&$3\le j\le p+1$&$p+2$&$p+3\le j\le \frac{3p+3}{2}$\\
			\hline
			1&1&$p^{n-1}-p^{\frac{n-1}{2}}$&$p^{n-1}-p^{\frac{n-1}{2}}$&$p^{\frac{n-1}{2}}-1$&$2p^{\frac{n-1}{2}}$\\
			\hline
			2&1&0&$\eta(j-2)p^{\frac{n-1}{2}}$&$-1$&0\\
			\hline
			$3\le i\le p+1$&1&$\eta(i-2)p^{\frac{n-1}{2}}$&$p^{\frac{n-2}{2}}S^{(h)}(j-2,i-2)$&$-1$&0\\
			\hline
			$p+2$&1&$-p^{\frac{n-1}{2}}$&$-p^{\frac{n-1}{2}}$&$p^{\frac{n-1}{2}}-1$&$2p^{\frac{n-1}{2}}$\\
			\hline
			$p+3\le i\le \frac{3p+3}{2}$&1&0&0&$p^{\frac{n-1}{2}}-1$&$-p^{\frac{n-2}{2}}S^{(h)}(n_{j-p-2},n_{i-p-2})$\\
			\hline
	\end{tabular}}
	\vspace{-15pt}
\end{table}

Since $p_{u,v}^{w}=p_{v,u}^{w}$ for any $u,v,w\in\{0,1,\dots,\frac{3p+1}{2}\}$, we only give the values of the intersection numbers $p_{u,v}^w$ for $u\le v$ in the following eight cases.

\begin{itemize}
	\item [$(\rm{\romannumeral1})$] $w=0$ and $0\le u\le v\le \frac{3p+1}{2}$
	\begin{itemize}
		\item [$\bullet$] $u=v$\\
		$p_{0,0}^0=1$, $p_{u,u}^0=p^{n-1}-p^{\frac{n-1}{2}}$ for $1\le u\le p$, $p_{p+1,p+1}^0=p^{\frac{n-1}{2}}-1$ and $p_{u,u}^0=2p^{\frac{n-1}{2}}$ for $p+2\le u\le \frac{3p+1}{2}$.
		\item[$\bullet$] $u\ne v$, $p_{u,v}^0=0$.
	\end{itemize}
\item [$(\rm{\romannumeral2})$]$u=0$, $1\le w\le \frac{3p+1}{2}$ and $0\le v\le \frac{3p+1}{2}$\\
$p_{0,v}^w=\begin{cases}
1,&\text{if}\ w=v,\\
0,&\text{if}\ w\ne v.
\end{cases}$
\item [$(\rm{\romannumeral3})$] $1\le w\le p$ and $1\le u\le v\le p$\\
$p_{u,v}^w=p^{n-2}-2p^{\frac{n-3}{2}}+p^{\frac{n-6}{2}}\sum\limits_{i\in\mathbb{F}_p}S^{(h)}(i,w-1)S^{(t)}(u-1,i)S^{(t)}(v-1,i)$.
\item [$(\rm{\romannumeral4})$] $1\le w\le p$, $1\le u\le p$ and $p+1\le v\le \frac{3p+1}{2}$\\
$p_{u,v}^w=\begin{cases}
p^{\frac{n-3}{2}}-p^{-1}-p^{-2}\sum\limits_{i\in\mathbb{F}_p}S^{(h)}(i,w-1)S^{(t)}(u-1,i),&\text{if}\ v=p+1,\\
2p^{\frac{n-3}{2}},&\text{if}\ v\ge p+2. 
\end{cases}$
\item [$(\rm{\romannumeral5})$] $1\le w\le p$ and $p+1\le u\le v\le \frac{3p+1}{2}$, $p_{u,v}^w=0$.
\item [$(\rm{\romannumeral6})$] $p+1\le w\le \frac{3p+1}{2}$ and $1\le u\le v\le p$\\
$p_{u,v}^w=\begin{cases}
p^{n-2}-p^{\frac{n-1}{2}},&\text{if}\ w=p+1\ \text{and}\ u=v,\\
p^{n-2},&\text{if}\ w=p+1\ \text{and}\ u\ne v,\\
p^{n-2}-p^{\frac{n-3}{2}},&\text{if}\ w\ge p+2.
\end{cases}$
\item [$(\rm{\romannumeral7})$] $p+1\le w\le \frac{3p+1}{2}$, $1\le u\le p$ and $p+1\le v\le \frac{3p+1}{2}$, $p_{u,v}^w=0$.
\item [$(\rm{\romannumeral8})$] $p+1\le w\le  \frac{3p+1}{2}$ and $p+1\le u\le v\le \frac{3p+1}{2}$
\begin{itemize}
	\item [$\bullet$] $w=p+1$\\
	\newpage
	$p_{u,v}^{p+1}=\begin{cases}
	p^{\frac{n-1}{2}}-2,&\text{if}\ u=v=p+1,\\
	0,&\text{if}\ u=p+1\ \text{and}\ v\ge p+2,\\
	2p^{\frac{n-1}{2}},&\text{if}\ u\ge p+2\ \text{and}\ u=v,\\
	0,&\text{if}\ u\ge p+2\ \text{and}\ u\ne v.
	\end{cases}$
	\item[$\bullet$] $p+2\le w\le \frac{3p+1}{2}$\\
	$p_{u,v}^w=\begin{cases}
	0,&\text{if}\ u=v=p+1,\\
	2(p^{\frac{n-3}{2}}-p^{-1})+(p^{\frac{n-5}{2}}-p^{-2})\times&\text{if}\ u=p+1\ \text{and}\\
	\sum\limits_{i\in NSQ}S^{(h)}(i,n_{w-p-1})S^{(t)}(n_{v-p-1},i),& v\ge p+2,\\
	4p^{\frac{n-3}{2}}-p^{\frac{n-6}{2}}\sum\limits_{i\in NSQ}S^{(h)}(i,n_{w-p-1})\times&\\
	S^{(t)}(n_{u-p-1},i)S^{(t)}(n_{v-p-1},i),&\text{if}\ p+2\le u\le v.
	\end{cases}$
\end{itemize}
\end{itemize}

Since $q_{u,v}^{w}=q_{v,u}^{w}$ for any $u,v,w\in\{0,1,\dots,\frac{3p+1}{2}\}$, we only give the values of the Krein parameters $p_{u,v}^w$ for $u\le v$ in the following eight cases.

\begin{itemize}
	\item [$(\rm{\romannumeral1})$] $w=0$ and $0\le u\le v\le \frac{3p+1}{2}$
	\begin{itemize}
		\item [$\bullet$] $u=v$\\
		$q_{0,0}^0=1$, $q_{u,u}^0=p^{n-1}-p^{\frac{n-1}{2}}$ for $1\le u\le p$, $q_{p+1,p+1}^0=p^{\frac{n-1}{2}}-1$ and $q_{u,u}^0=2p^{\frac{n-1}{2}}$ for $p+2\le u\le \frac{3p+1}{2}$.
		\item[$\bullet$] $u\ne v$, $q_{u,v}^0=0$.
	\end{itemize}
	\item [$(\rm{\romannumeral2})$]$u=0$, $1\le w\le \frac{3p+1}{2}$ and $0\le v\le \frac{3p+1}{2}$\\
	$q_{0,v}^w=\begin{cases}
	1,&\text{if}\ w=v,\\
	0,&\text{if}\ w\ne v.
	\end{cases}$
	\item [$(\rm{\romannumeral3})$] $1\le w\le p$ and $1\le u\le v\le p$\\
	$q_{u,v}^w=p^{n-2}-2p^{\frac{n-3}{2}}+p^{\frac{n-6}{2}}\sum\limits_{i\in\mathbb{F}_p}S^{(t)}(i,w-1)S^{(h)}(u-1,i)S^{(h)}(v-1,i)$.
	\item [$(\rm{\romannumeral4})$] $1\le w\le p$, $1\le u\le p$ and $p+1\le v\le \frac{3p+1}{2}$\\
	$q_{u,v}^w=\begin{cases}
	p^{\frac{n-3}{2}}-p^{-1}-p^{-2}\sum\limits_{i\in\mathbb{F}_p}S^{(t)}(i,w-1)S^{(h)}(u-1,i),&\text{if}\ v=p+1,\\
	2p^{\frac{n-3}{2}},&\text{if}\ v\ge p+2. 
	\end{cases}$
	\item [$(\rm{\romannumeral5})$] $1\le w\le p$ and $p+1\le u\le v\le \frac{3p+1}{2}$, $q_{u,v}^w=0$.
	\item [$(\rm{\romannumeral6})$] $p+1\le w\le \frac{3p+1}{2}$ and $1\le u\le v\le p$\\
	$q_{u,v}^w=\begin{cases}
	p^{n-2}-p^{\frac{n-1}{2}},&\text{if}\ w=p+1\ \text{and}\ u=v,\\
	p^{n-2},&\text{if}\ w=p+1\ \text{and}\ u\ne v,\\
	p^{n-2}-p^{\frac{n-3}{2}},&\text{if}\ w\ge p+2.
	\end{cases}$
	\item [$(\rm{\romannumeral7})$] $p+1\le w\le \frac{3p+1}{2}$, $1\le u\le p$ and $p+1\le v\le \frac{3p+1}{2}$, $q_{u,v}^w=0$.
	\item [$(\rm{\romannumeral8})$] $p+1\le w\le\frac{3p+1}{2}$ and $p+1\le u\le v\le \frac{3p+1}{2}$
	\begin{itemize}
		\item [$\bullet$] $w=p+1$\\
	\newpage
		$q_{u,v}^{p+1}=\begin{cases}
		p^{\frac{n-1}{2}}-2,&\text{if}\ u=v=p+1,\\
		0,&\text{if}\ u=p+1\ \text{and}\ v\ge p+2,\\
		2p^{\frac{n-1}{2}},&\text{if}\ u\ge p+2\ \text{and}\ u=v,\\
		0,&\text{if}\ u\ge p+2\ \text{and}\ u\ne v.
		\end{cases}$
		\item[$\bullet$] $p+2\le w\le \frac{3p+1}{2}$\\
		$q_{u,v}^w=\begin{cases}
		0,&\text{if}\ u=v=p+1,\\
		2(p^{\frac{n-3}{2}}-p^{-1})+(p^{\frac{n-5}{2}}-p^{-2})\times&\text{if}\ u=p+1\ \text{and}\\
		\sum\limits_{i\in NSQ}S^{(t)}(i,n_{w-p-1})S^{(h)}(n_{v-p-1},i),& v\ge p+2,\\
		4p^{\frac{n-3}{2}}-p^{\frac{n-6}{2}}\sum\limits_{i\in NSQ}S^{(t)}(i,n_{w-p-1})\times\\
		S^{(h)}(n_{u-p-1},i)S^{(h)}(n_{v-p-1},i),&\text{if}\ p+2\le u\le v.
		\end{cases}$
	\end{itemize}
\end{itemize}
15. The first and second eigenmatrices, the intersection numbers and the Krein parameters of the association scheme induced by  $U_{15}$.

The first and second eigenmatrices of the association scheme induced by $U_{15}$ are given in Tables 23 and 24.

\begin{table}[h]
	\vspace{-15pt}
	\centering
	\caption{The first eigenmatrix of the association scheme induced by $U_{15}$}
	\renewcommand\arraystretch{1.5}	
	\resizebox{\textwidth}{!}{
		\begin{tabular}{|m{1.4cm}<{\centering}|c|c|m{3.55cm}<{\centering}|c|m{3.8cm}<{\centering}|}
			\hline
			\diagbox{$i$}{$P_{15}(ij)$}{$j$}&1&2&$3\le j\le p+1$&$p+2$& $p+3\le j\le \frac{3p+3}{2}$\\
			\hline
			1&1&$p^{n-1}-p^{\frac{n-1}{2}}$&$p^{n-1}-p^{\frac{n-1}{2}}$&$p^{\frac{n-1}{2}}-1$&$2p^{\frac{n-1}{2}}$\\
			\hline
			2&1&$-p^{\frac{n-1}{2}}$&$-p^{\frac{n-1}{2}}$&$p^{\frac{n-1}{2}}-1$&$2p^{\frac{n-1}{2}}$\\
			\hline
			$3\le i\le \frac{p+3}{2}$&1&0&0&$p^{\frac{n-1}{2}}-1$&$\sqrt{-1}p^{\frac{n-2}{2}}S^{(t)}(s_{j-p-2},n_{i-2})$\\
			\hline
			$\frac{p+5}{2}$&1&0&$-\eta(j-2)p^{\frac{n-1}{2}}$&$-1$&0\\
			\hline
			$\frac{p+7}{2}\le i\le \frac{3p+3}{2}$&1&$\eta(i-\frac{p+5}{2})p^{\frac{n-1}{2}}$&$-\sqrt{-1}p^{\frac{n-2}{2}}S^{(t)}(j-2,i-\frac{p+5}{2})$&$-1$&0\\
			\hline
	\end{tabular}}
	\vspace{-15pt}
\end{table}

\begin{table}[h]
	\vspace{-15pt}
	\centering
	\caption{The second eigenmatrix of the association scheme induced by $U_{15}$}
	\renewcommand\arraystretch{1.5}	
	\resizebox{\textwidth}{!}{
		\begin{tabular}{|m{1.4cm}<{\centering}|c|c|m{3.85cm}<{\centering}|c|m{3.8cm}<{\centering}|}
			\hline
			\diagbox{$i$}{$Q_{15}(ij)$}{$j$}&1&2&$3\le j\le \frac{p+3}{2}$&$\frac{p+5}{2}$& $\frac{p+7}{2}\le j\le \frac{3p+3}{2}$\\
			\hline
			1&1&$p^{\frac{n-1}{2}}-1$&$2p^{\frac{n-1}{2}}$&$p^{n-1}-p^{\frac{n-1}{2}}$&$p^{n-1}-p^{\frac{n-1}{2}}$\\
			\hline
			2&1&$-1$&0&0&$\eta(j-\frac{p+5}{2})p^{\frac{n-1}{2}}$\\
			\hline
			$3\le i\le p+1$&1&$-1$&0&$-\eta(i-2)p^{\frac{n-1}{2}}$&$\sqrt{-1}p^{\frac{n-2}{2}}S^{(h)}(j-\frac{p+5}{2},i-2)$\\
			\hline
			$p+2$&1&$p^{\frac{n-1}{2}}-1$&$2p^{\frac{n-1}{2}}$&$-p^{\frac{n-1}{2}}$&$-p^{\frac{n-1}{2}}$\\
			\hline
			$p+3\le i\le \frac{3p+3}{2}$&1&$p^{\frac{n-1}{2}}-1$&$-\sqrt{-1}p^{\frac{n-2}{2}}S^{(h)}(n_{j-2}, s_{i-p-2})$&0&0\\
			\hline
	\end{tabular}}
	\vspace{-15pt}
\end{table}
Since $p_{u,v}^{w}=p_{v,u}^{w}$ for any $u,v,w\in\{0,1,\dots,\frac{3p+1}{2}\}$, we only give the values of  the intersection numbers $p_{u,v}^w$ for $u\le v$ in the following eight cases.
\begin{itemize}
	\item [$(\rm{\romannumeral1})$] $w=0$ and $0\le u\le v\le \frac{3p+1}{2}$
	\begin{itemize}
		\item [$\bullet$] $u=v$\\
		$p_{0,0}^0=1$, $p_{u,u}^0=p^{n-1}-p^{\frac{n-1}{2}}$ for $1\le u\le p$, $p_{p+1,p+1}^0=p^{\frac{n-1}{2}}-1$ and $p_{u,u}^0=2p^{\frac{n-1}{2}}$ for $p+2\le u\le \frac{3p+1}{2}$.
		\item[$\bullet$] $u\ne v$, $p_{u,v}^w=0$.
	\end{itemize}
	\item [$(\rm{\romannumeral2})$] $u=0$, $1\le w\le \frac{3p+1}{2}$ and $0\le v\le \frac{3p+1}{2}$\\
	$p_{u,v}^w=\begin{cases}
	1,&\text{if}\ w=v,\\
	0,&\text{if}\ w\ne v.
	\end{cases}$
	\item [$(\rm{\romannumeral3})$] $1\le w\le p$ and $1\le u\le v\le p$\\
	$p_{u,v}^w=p^{n-2}-2p^{\frac{n-3}{2}}-\sqrt{-1}p^{\frac{n-6}{2}}\sum\limits_{i\in\mathbb{F}_p}S^{(h)}(i,w-1)S^{(t)}(u-1,i)S^{(t)}(v-1,i)$.
	\item [$(\rm{\romannumeral4})$] $1\le w\le p$, $1\le u\le p$ and $p+1\le v\le \frac{3p+1}{2}$\\
	$p_{u,v}^w=\begin{cases}
	p^{\frac{n-3}{2}}-p^{-1}-p^{-2}\sum\limits_{i\in\mathbb{F}_p}S^{(h)}(i,w-1)S^{(t)}(u-1,i),&\text{if}\ v=p+1,\\
	2p^{\frac{n-3}{2}},&\text{if}\ v\ge p+2.
	\end{cases}$
	\item [$(\rm{\romannumeral5})$] $1\le w\le p$ and $p+1\le u\le v\le \frac{3p+1}{2}$, $p_{u,v}^w=0$.
	\item [$(\rm{\romannumeral6})$] $p+1\le w\le \frac{3p+1}{2}$ and $1\le u\le v\le p$\\
	$p_{u,v}^w=\begin{cases}
	p^{n-2}-p^{\frac{n-1}{2}},&\text{if}\ w=p+1\ \text{and}\ u=v,\\
	p^{n-2},&\text{if}\ w=p+1\ \text{and}\ u\ne v,\\
	p^{n-2}-p^{\frac{n-3}{2}},&\text{if}\ w\ge p+2.
	\end{cases}$
	\item [$(\rm{\romannumeral7})$] $p+1\le w\le \frac{3p+1}{2}$, $1\le u\le p$ and $p+1\le v\le  \frac{3p+1}{2}$, $p_{u,v}^w=0$.
	\item [$(\rm{\romannumeral8})$] $p+1\le w\le \frac{3p+1}{2}$ and $p+1\le u\le v\le \frac{3p+1}{2}$
	\begin{itemize}
		\item [$\bullet$] $w=p+1$\\
		$p_{u,v}^{p+1}=\begin{cases}
		p^{\frac{n-1}{2}}-2,&\text{if}\ u=v=p+1,\\
		0,&\text{if}\ u=p+1\ \text{and}\ v\ge p+2,\\
		2p^{\frac{n-1}{2}},&\text{if}\ u\ge p+2\ \text{and}\ u=v,\\
		0,&\text{if}\ u\ge p+2\ \text{and}\ u\ne v.
		\end{cases}$
		\item[$\bullet$] $p+2\le w\le \frac{3p+1}{2}$\\
		$p_{u,v}^w=\begin{cases}
		0,&\text{if}\ u=v=p+1,\\
		2(p^{\frac{n-3}{2}}-p^{-1})+(p^{\frac{n-5}{2}}-p^{-2})\times&\text{if}\ u=p+1\ \text{and}\\
		\sum\limits_{i\in NSQ}S^{(h)}(i,s_{w-p-1})S^{(t)}(s_{v-p-1},i),&\ v\ge p+2,\\
		4p^{\frac{n-3}{2}}+\sqrt{-1}p^{\frac{n-6}{2}}\sum\limits_{i\in NSQ}S^{(h)}(i,s_{w-p-1})\times&\\
		S^{(t)}(s_{u-p-1},i)S^{(t)}(s_{v-p-1},i),&\text{if}\ p+2\le u\le v.
		\end{cases}$
	\end{itemize}
\end{itemize}

Since $q_{u,v}^{w}=q_{v,u}^{w}$ for any $u,v,w\in\{0,1,\dots,\frac{3p+1}{2}\}$, we only give the values of  the Krein parameters $q_{u,v}^w$ for $u\le v$ in the following eight cases.	
\begin{itemize}
	\item [$(\rm{\romannumeral1})$] $w=0$ and $0\le u\le v\le \frac{3p+1}{2}$
	\begin{itemize}
		\item [$\bullet$] $u=v$\\
		$q_{0,0}^0=1$, $q_{1,1}^0=p^{\frac{n-1}{2}}-1$, $q_{u,u}^0=2p^{\frac{n-1}{2}}$ for $2\le u\le \frac{p+1}{2}$ and $q_{u,u}^0=p^{n-1}-p^{\frac{n-1}{2}}$ for $\frac{p+3}{2}\le u\le \frac{3p+1}{2}$.
		\item[$\bullet$] $u\ne v$, $q_{u,v}^0=0$.
	\end{itemize}
	\item [$(\rm{\romannumeral2})$] $u=0$, $1\le w\le \frac{3p+1}{2}$ and $0\le v\le \frac{3p+1}{2}$\\
	$q_{0,v}^w=\begin{cases}
	1,&\text{if}\ w=v,\\
	0,&\text{if}\ w\ne v.
	\end{cases}$
	\item [$(\rm{\romannumeral3})$] $1\le w\le \frac{p+1}{2}$ and $1\le u\le v\le \frac{p+1}{2}$
	\begin{itemize}
		\item [$\bullet$] $w=1$\\
		$q_{u,v}^1=\begin{cases}
		p^{\frac{n-1}{2}}-2,&\text{if}\ u=v=1,\\
		0,&\text{if}\ u=1\ \text{and}\ v\ge 2,\\
		2p^{\frac{n-1}{2}},&\text{if}\ u\ge 2\ \text{and}\ u=v,\\
		0,&\text{if}\ u\ge 2\ \text{and}\ u\ne v.
		\end{cases}$
		\item[$\bullet$] $2\le w\le \frac{p+1}{2}$\\
		$q_{u,v}^w=\begin{cases}
		0,&\text{if}\ u=v=1,\\
		2(p^{\frac{n-3}{2}}-p^{-1})+(p^{\frac{n-5}{2}}-p^{-2})\\
		\sum\limits_{i\in SQ}S^{(t)}(i,n_{w-1})S^{(h)}(n_{v-1},i)\times,&\text{if}\ u=1\ \text{and}\ v\ge 2,\\
		4p^{\frac{n-3}{2}}-\sqrt{-1}p^{\frac{n-6}{2}}\sum\limits_{i\in SQ}S^{(t)}(i,n_{w-1})\times\\
		S^{(h)}(n_{u-1},i)S^{(h)}(n_{v-1},i),&\text{if}\ 2\le u\le v.
		\end{cases}$
	\end{itemize}
	\item [$(\rm{\romannumeral4})$] $1\le w\le \frac{p+1}{2}$, $1\le u\le \frac{p+1}{2}$ and $\frac{p+3}{2}\le v\le \frac{3p+1}{2}$, $q_{u,v}^w=0$.
	\item [$(\rm{\romannumeral5})$] $1\le w\le \frac{p+1}{2}$ and $\frac{p+3}{2}\le u\le v\le \frac{3p+1}{2}$\\
	$q_{u,v}^w=\begin{cases}
	p^{n-2}-p^{\frac{n-1}{2}},&\text{if}\ w=1\ \text{and}\ u=v,\\
	p^{n-2},&\text{if}\ w=1\ \text{and}\ u\ne v,\\
	p^{n-2}-p^{\frac{n-3}{2}},&\text{if}\ w\ge 2.
	\end{cases}$
	\item [$(\rm{\romannumeral6})$] $\frac{p+3}{2}\le w\le \frac{3p+1}{2}$ and $1\le u\le v\le \frac{p+1}{2}$, $q_{u,v}^w=0$.
	\item [$(\rm{\romannumeral7})$] $\frac{p+3}{2}\le w\le \frac{3p+1}{2}$, $1\le u\le \frac{p+1}{2}$ and $\frac{p+3}{2}\le v\le \frac{3p+1}{2}$\\
	$q_{u,v}^w=\begin{cases}
	p^{\frac{n-3}{2}}-p^{-1}-p^{-2}\sum\limits_{i\in\mathbb{F}_p}S^{(t)}(i,w-\frac{p+3}{2})S^{(h)}(v-\frac{p+3}{2},i),&\text{if}\ u=1,\\
	2p^{\frac{n-3}{2}},&\text{if}\ u\ge 2.
	\end{cases}$
	
	\item [$(\rm{\romannumeral8})$] $\frac{p+3}{2}\le w\le \frac{3p+1}{2}$ and $\frac{p+3}{2}\le u\le v\le \frac{3p+1}{2}$\\
	$q_{u,v}^w=p^{n-2}-2p^{\frac{n-3}{2}}+\sqrt{-1}p^{\frac{n-6}{2}}\sum\limits_{i\in\mathbb{F}_p}S^{(t)}(i,w-\frac{p+3}{2})S^{(h)}(u-\frac{p+3}{2},i)S^{(h)}(v-\frac{p+3}{2},i)$.
	
\end{itemize}
16. The first and second eigenmatrices, the intersection numbers and the Krein parameters of the association scheme induced by  $U_{16}$.

The first and second eigenmatrices of the association scheme induced by $U_{16}$ are given in Tables 25 and 26.
\begin{table}[h]
	\vspace{0pt}
	\centering
	\caption{The first eigenmatrix of the association scheme induced by $U_{16}$}
	\renewcommand\arraystretch{1.5}	
	\resizebox{\textwidth}{!}{
		\begin{tabular}{|m{1.4cm}<{\centering}|c|m{1.6cm}<{\centering}|m{2.3cm}<{\centering}|m{1.05cm}<{\centering}|m{2.5cm}<{\centering}|m{2.55cm}<{\centering}|}
			\hline
			\diagbox{$i$}{$P_{16}(ij)$}{$j$}&1&2&$3\le j\le p+1$&$p+2$&$p+3$&$p+4\le j\le 2p+2$\\
			\hline
			1&1&$p^{n-1}-p^{r-1}$&$p^{n-1}-p^{r-1}$&$p^{n-r}-1$&$p^{r-1}-p^{n-r}$&$-{\small \eta(j-p-3)}p^{\frac{n-1}{2}}+p^{r-1}$\\
			\hline
			2&1&0&$\eta(j-2)p^{\frac{n-1}{2}}$&$-1$&0&0\\
			\hline
			$3\le i\le p+1$&1&$\eta(i-2)p^{\frac{n-1}{2}}$&$p^{\frac{n-2}{2}}S^{(t)}(j-2,i-2)$&$-1$&0&0\\
			\hline
			$p+2$&1&$-p^{r-1}$&$-p^{r-1}$&$p^{n-r}-1$&$p^{r-1}-p^{n-r}$&$-{\small \eta(j-p-3)}p^{\frac{n-1}{2}}+p^{r-1}$\\
			\hline
			$p+3$&1&0&0&$p^{n-r}-1$&$-p^{n-r}$&$-\eta(j-p-3)p^{\frac{n-1}{2}}$\\
			\hline
			$p+4\le i\le 2p+2$&1&0&0&$p^{n-r}-1$&${\small -\eta(i-p-3)}p^{\frac{n-1}{2}}-p^{n-r}$&$-p^{\frac{n-2}{2}}S^{(t)}(j-p-3,i-p-3)$\\
			\hline
	\end{tabular}}
	\vspace{0pt}
\end{table}
\begin{table}[h]
	\vspace{5pt}
	\centering
	\caption{The second eigenmatrix of the association scheme induced by $U_{16}$}
	\renewcommand\arraystretch{1.5}	
	\resizebox{\textwidth}{!}{
		\begin{tabular}{|m{1.4cm}<{\centering}|c|m{1.6cm}<{\centering}|m{2.3cm}<{\centering}|m{1.05cm}<{\centering}|m{2.5cm}<{\centering}|m{2.55cm}<{\centering}|}
			\hline
			\diagbox{$i$}{$Q_{16}(ij)$}{$j$}&1&2&$3\le j\le p+1$&$p+2$&$p+3$&$p+4\le j\le 2p+2$\\
			\hline
			1&1&$p^{n-1}-p^{r-1}$&$p^{n-1}-p^{r-1}$&$p^{n-r}-1$&$p^{r-1}-p^{n-r}$&$-{\small \eta(j-p-3)}p^{\frac{n-1}{2}}+p^{r-1}$\\
			\hline
			2&1&0&$\eta(j-2)p^{\frac{n-1}{2}}$&$-1$&0&0\\
			\hline
			$3\le i\le p+1$&1&$\eta(i-2)p^{\frac{n-1}{2}}$&$p^{\frac{n-2}{2}}S^{(h)}(j-2,i-2)$&$-1$&0&0\\
			\hline
			$p+2$&1&$-p^{r-1}$&$-p^{r-1}$&$p^{n-r}-1$&$p^{r-1}-p^{n-r}$&$-{\small \eta(j-p-3)}p^{\frac{n-1}{2}}+p^{r-1}$\\
			\hline
			$p+3$&1&0&0&$p^{n-r}-1$&$-p^{n-r}$&$-\eta(j-p-3)p^{\frac{n-1}{2}}$\\
			\hline
			$p+4\le i\le 2p+2$&1&0&0&$p^{n-r}-1$&${\small -\eta(i-p-3)}p^{\frac{n-1}{2}}-p^{n-r}$&$-p^{\frac{n-2}{2}}S^{(h)}(j-p-3,i-p-3)$\\
			\hline
	\end{tabular}}
	\vspace{-15pt}
\end{table}

Since $p_{u,v}^{w}=p_{v,u}^{w}$ for any $u,v,w\in\{0,1,\dots,2p+1\}$, we only give the values of the intersection numbers $p_{u,v}^w$ for $u\le v$ in the following eight cases. 
\begin{itemize}
	\item [$(\rm{\romannumeral1})$] $w=0$ and $0\le u\le v\le 2p+1$
	\begin{itemize}
		\item [$\bullet$] $u=v$\\
		$p_{0,0}^0=1$, $p_{u,u}^0=p^{n-1}-p^{r-1}$ for $1\le u\le p$, $p_{p+1,p+1}^0=p^{n-r}-1$, $p_{p+2,p+2}^0=p^{r-1}-p^{n-r}$ and $p_{u,u}^0=p^{r-1}-\eta(u-p-2)p^{\frac{n-1}{2}}$.
		\item[$\bullet$] $u\ne v$, $p_{u,v}^0=0$.
	\end{itemize}
\item [$(\rm{\romannumeral2})$] $u=0$, $1\le w\le 2p+1$ and $0\le v\le 2p+1$\\
$p_{0,v}^w=\begin{cases}
1,&\text{if}\ w=v,\\
0,&\text{if}\ w\ne v.
\end{cases}$
\item [$(\rm{\romannumeral3})$] $1\le w\le p$ and $1\le u\le v\le p$\\
$p_{u,v}^w=p^{n-2}-2p^{r-2}+p^{\frac{n-6}{2}}\sum\limits_{i\in\mathbb{F}_p}S^{(h)}(i,w-1)S^{(t)}(u-1,i)S^{(t)}(v-1,i)$.\\
\item [$(\rm{\romannumeral4})$] $1\le w\le p$, $1\le u\le p$ and $p+1\le v\le 2p+1$\\
$p_{u,v}^w=\begin{cases}
p^{n-r-1}-p^{-1}-p^{-2}\sum\limits_{i\in\mathbb{F}_p}S^{(h)}(i,w-1)S^{(t)}(u-1,i),&\text{if}\ v=p+1,\\
p^{r-2}-p^{n-r-1},&\text{if}\ v=p+2,\\
p^{r-2}-\eta(v-p-2)p^{\frac{n-3}{2}},&\text{if}\ v\ge p+3.
\end{cases}$
\item [$(\rm{\romannumeral5})$] $1\le w\le p$ and $p+1\le u\le v\le 2p+1$, $p_{u,v}^w=0$.
\item [$(\rm{\romannumeral6})$] $p+1\le w\le 2p+1$ and $1\le u\le v\le p$\\
$p_{u,v}^w=\begin{cases}
p^{n-2}-p^{r-1},&\text{if}\ w=p+1\ \text{and}\ u=v,\\
p^{n-2},&\text{if}\ w=p+1\ \text{and}\ u\ne v,\\
p^{n-2}-p^{r-2},&\text{if}\ w\ge p+2.
\end{cases}$
\item [$(\rm{\romannumeral7})$] $p+1\le w\le 2p+1$, $1\le u\le p$ and $p+1\le v\le 2p+1$, $p_{u,v}^w=0$.
\item [$(\rm{\romannumeral8})$] $p+1\le w\le 2p+1$ and $p+1\le u\le v\le 2p+1$
\begin{itemize}
	\item[$\bullet$] $w=p+1$\\
\newpage
	$p_{u,v}^{p+1}=\begin{cases}
	p^{n-r}-2,&\text{if}\ u=v=p+1,\\
	0,&\text{if}\ u=p+1\ \text{and}\ v\ge p+2,\\
	p^{r-1}-p^{n-r},&\text{if}\ u=v=p+2,\\
	p^{r-1}-\eta(u-p-2)p^{\frac{n-1}{2}},&\text{if}\ u\ge p+3\ \text{and}\ u=v,\\
	0,&\text{if}\ u\ge p+2\ \text{and}\ u\ne v.
	\end{cases}$
	\item[$\bullet$] $w=p+2$\\
	$p_{u,v}^{p+2}=\begin{cases}
	p^{n-r}-1,&\hspace{-2.9cm}\text{if}\ u=p+1\ \text{and}\\&\hspace{-2.9cm} v=p+2,\\
	0,&\hspace{-2.9cm}\text{if}\ u=p+1\ \text{and}\\&\hspace{-2.9cm} v\ne p+2,\\
	p^{r-2}-2p^{n-r},&\hspace{-2.9cm}\text{if}\ u=v=p+2,\\
	p^{r-2},&\hspace{-2.9cm}\text{if}\ u=p+2\ \text{and}\\&\hspace{-2.9cm} v\ge p+3,\\
 p^{r-2}-p^{\frac{n-5}{2}}\sum\limits_{i\in\mathbb{F}_p^*}\eta(i)S^{(t)}(u-p-2,i)S^{(t)}(v-\\p-2,i)-(\eta(u-p-2)+\eta(v-p-2))p^{\frac{n-3}{2}},
 	\text{if}\  p+3\le u\le v.
	\end{cases}$
	\item[$\bullet$] $p+3\le w\le 2p+1$\\
	$p_{u,v}^w=\begin{cases}
	p^{n-r}-1,&\text{if}\ u=p+1\ \text{and}\ w=v,\\
	0,&\text{if}\ u=p+1\ \text{and}\ w\ne v,\\
	p^{r-2}+\eta(w-p-2)p^{\frac{n-3}{2}},&\text{if}\ u=v=p+2,\\
	H_1,&\text{if}\ u=p+2\ \text{and}\ v\ge p+3,\\
	H_2,&\text{if}\ p+3\le u\le v,
	\end{cases}$
	
	where $H_1=p^{r-2}-p^{n-r-1}-\eta(v-p-2)p^{\frac{n-3}{2}}-\eta((w-p-2)(v-p-2))p^{n-r-1}-p^{-2}\sum\limits_{i\in\mathbb{F}_p^*}(\eta(i)p^{\frac{n-1}{2}}+p^{n-r})S^{(h)}(i,w-p-2)S^{(t)}(v-p-2,i)$ and $H_2=p^{r-2}-(\eta(u-p-2)+\eta(v-p-2))p^{\frac{n-3}{2}}-p^{\frac{n-6}{2}}\sum\limits_{i\in\mathbb{F}_p}S^{(h)}(i,w-p-2)S^{(t)}(u-p-2,i)S^{(t)}(v-p-2,i)$.
\end{itemize}
\end{itemize}

Since $q_{u,v}^{w}=q_{v,u}^{w}$ for any $u,v,w\in\{0,1,\dots,2p+1\}$, we only give the values of the Krein parameters $q_{u,v}^w$ for $u\le v$ in the following eight cases. 
\begin{itemize}
	\item [$(\rm{\romannumeral1})$] $w=0$ and $0\le u\le v\le 2p+1$
	\begin{itemize}
		\item [$\bullet$] $u=v$\\
		$q_{0,0}^0=1$, $q_{u,u}^0=p^{n-1}-p^{r-1}$ for $1\le u\le p$, $q_{p+1,p+1}^0=p^{n-r}-1$, $q_{p+2,p+2}^0=p^{r-1}-p^{n-r}$ and $q_{u,u}^0=p^{r-1}-\eta(u-p-2)p^{\frac{n-1}{2}}$.
		\item[$\bullet$] $u\ne v$, $q_{u,v}^0=0$.
	\end{itemize}
	\item [$(\rm{\romannumeral2})$] $u=0$, $1\le w\le 2p+1$ and $0\le v\le 2p+1$\\
	$q_{0,v}^w=\begin{cases}
	1,&\text{if}\ w=v,\\
	0,&\text{if}\ w\ne v.
	\end{cases}$
	\item [$(\rm{\romannumeral3})$] $1\le w\le p$ and $1\le u\le v\le p$\\
	$q_{u,v}^w=p^{n-2}-2p^{r-2}+p^{\frac{n-6}{2}}\sum\limits_{i\in\mathbb{F}_p}S^{(t)}(i,w-1)S^{(h)}(u-1,i)S^{(h)}(v-1,i)$.\\
	\item [$(\rm{\romannumeral4})$] $1\le w\le p$, $1\le u\le p$ and $p+1\le v\le 2p+1$\\
	$q_{u,v}^w=\begin{cases}
	p^{n-r-1}-p^{-1}-p^{-2}\sum\limits_{i\in\mathbb{F}_p}S^{(t)}(i,w-1)S^{(h)}(u-1,i),&\text{if}\ v=p+1,\\
	p^{r-2}-p^{n-r-1},&\text{if}\ v=p+2,\\
	p^{r-2}-\eta(v-p-2)p^{\frac{n-3}{2}},&\text{if}\ v\ge p+3.
	\end{cases}$
	\item [$(\rm{\romannumeral5})$] $1\le w\le p$ and $p+1\le u\le v\le 2p+1$, $q_{u,v}^w=0$.
	\item [$(\rm{\romannumeral6})$] $p+1\le w\le 2p+1$ and $1\le u\le v\le p$\\
	$q_{u,v}^w=\begin{cases}
	p^{n-2}-p^{r-1},&\text{if}\ w=p+1\ \text{and}\ u=v,\\
	p^{n-2},&\text{if}\ w=p+1\ \text{and}\ u\ne v,\\
	p^{n-2}-p^{r-2},&\text{if}\ w\ge p+2.
	\end{cases}$
	\item [$(\rm{\romannumeral7})$] $p+1\le w\le 2p+1$, $1\le u\le p$ and $p+1\le v\le 2p+1$, $q_{u,v}^w=0$.
	\item [$(\rm{\romannumeral8})$] $p+1\le w\le 2p+1$ and $p+1\le u\le v\le 2p+1$
	\begin{itemize}
		\item[$\bullet$] $w=p+1$\\
		$q_{u,v}^{p+1}=\begin{cases}
		p^{n-r}-2,&\text{if}\ u=v=p+1,\\
		0,&\text{if}\ u=p+1\ \text{and}\ v\ge p+2,\\
		p^{r-1}-p^{n-r},&\text{if}\ u=v=p+2,\\
		p^{r-1}-\eta(u-p-2)p^{\frac{n-1}{2}},&\text{if}\ u\ge p+3\ \text{and}\ u=v,\\
		0,&\text{if}\ u\ge p+2\ \text{and}\ u\ne v.
		\end{cases}$
		\item[$\bullet$] $w=p+2$\\
		$q_{u,v}^{p+2}=\begin{cases}
		p^{n-r}-1,&\hspace{-2.9cm}\text{if}\ u=p+1\ \text{and}\\&\hspace{-2.9cm} v=p+2,\\
		0,&\hspace{-2.9cm}\text{if}\ u=p+1\ \text{and}\\&\hspace{-2.9cm} v\ne p+2,\\
		p^{r-2}-2p^{n-r},&\hspace{-2.9cm}\text{if}\ u=v=p+2,\\
		p^{r-2},&\hspace{-2.9cm}\text{if}\ u=p+2\ \text{and}\\&\hspace{-2.9cm} v\ge p+3,\\
		 p^{r-2}-p^{\frac{n-5}{2}}\sum\limits_{i\in\mathbb{F}_p^*}\eta(i)S^{(h)}(u-p-2,i)S^{(h)}(v-\\p-2,i)-(\eta(u-p-2)+\eta(v-p-2))p^{\frac{n-3}{2}},
		\text{if}\  p+3\le u\le v.
		\end{cases}$
		\item[$\bullet$] $p+3\le w\le 2p+1$\\
		$q_{u,v}^w=\begin{cases}
		p^{n-r}-1,&\text{if}\ u=p+1\ \text{and}\ w=v,\\
		0,&\text{if}\ u=p+1\ \text{and}\ w\ne v,\\
		p^{r-2}+\eta(w-p-2)p^{\frac{n-3}{2}},&\text{if}\ u=v=p+2,\\
		H_1,&\text{if}\ u=p+2\ \text{and}\ v\ge p+3,\\
		H_2,&\text{if}\ p+3\le u\le v,
		\end{cases}$
		
		where $H_1=p^{r-2}-p^{n-r-1}-\eta(v-p-2)p^{\frac{n-3}{2}}-\eta((w-p-2)(v-p-2))p^{n-r-1}-p^{-2}\sum\limits_{i\in\mathbb{F}_p^*}(\eta(i)p^{\frac{n-1}{2}}+p^{n-r})S^{(t)}(i,w-p-2)S^{(h)}(v-p-2,i)$ and $H_2=p^{r-2}-(\eta(u-p-2)+\eta(v-p-2))p^{\frac{n-3}{2}}-p^{\frac{n-6}{2}}\sum\limits_{i\in\mathbb{F}_p}S^{(t)}(i,w-p-2)S^{(h)}(u-p-2,i)S^{(h)}(v-p-2,i)$.
	\end{itemize}
\end{itemize}
17. The first and second eigenmatrices, the intersection numbers and the Krein parameters of the association scheme induced by  $U_{17}$.

The first and second eigenmatrices of the association scheme induced by $U_{17}$ are given in Tables 27 and 28.
\begin{table}[h]
	\vspace{-15pt}
	\centering
	\caption{The first eigenmatrix of the association scheme induced by $U_{17}$}
	\renewcommand\arraystretch{1.5}	
	\resizebox{\textwidth}{!}{
		\begin{tabular}{|m{1.4cm}<{\centering}|c|m{2.3cm}<{\centering}|m{2.7cm}<{\centering}|c|m{2.3cm}<{\centering}|m{2.3cm}<{\centering}|}
			\hline
			\diagbox{$i$}{$P_{17}(ij)$}{$j$}&1&2&$3\le j\le p+1$&$p+2$&$p+3$&$p+4\le j\le 2p+2$\\
			\hline
			1&1&$p^{n-1}-p^{r-1}$&$p^{n-1}-p^{r-1}$&$p^{n-r}-1$&$p^{r-1}-p^{n-r}$&${\small \eta(j-p-3)}p^{\frac{n-1}{2}}+p^{r-1}$\\
			\hline
			2&1&$-p^{r-1}$&$-p^{r-1}$&$p^{n-r}-1$&$p^{r-1}-p^{n-r}$&${\small \eta(j-p-3)}p^{\frac{n-1}{2}}+p^{r-1}$\\
			\hline
			3&1&0&0&$p^{n-r}-1$&$-p^{n-r}$&$\eta(j-p-3)p^{\frac{n-1}{2}}$\\
			\hline
			$4\le i\le p+2$&1&0&0&$p^{n-r}-1$&$-\eta(i-3)p^{\frac{n-1}{2}}-p^{n-r}$&$\sqrt{-1}p^{\frac{n-2}{2}}S^{(t)}(j-p-3,i-3)$\\
			\hline
			$p+3$&1&0&$-\eta(j-2)p^{\frac{n-1}{2}}$&$-1$&0&0\\
			\hline
			$p+4\le i\le 2p+2$&1&$\eta(i-p-3)p^{\frac{n-1}{2}}$&$-\sqrt{-1}p^{\frac{n-2}{2}}S^{(t)}(j-2,i-p-3)$&$-1$&0&0\\
			\hline
	\end{tabular}}
	\vspace{-15pt}
\end{table}
\begin{table}[h]
	\vspace{-15pt}
	\centering
	\caption{The second eigenmatrix of the association scheme induced by $U_{17}$}
	\renewcommand\arraystretch{1.5}	
	\resizebox{\textwidth}{!}{
		\begin{tabular}{|m{1.4cm}<{\centering}|c|c|m{2.3cm}<{\centering}|m{2.7cm}<{\centering}|m{2.3cm}<{\centering}|m{2.3cm}<{\centering}|}
			\hline
			\diagbox{$i$}{$Q_{17}(ij)$}{$j$}&1&2&3&$4\le j\le p+2$&$p+3$&$p+4\le j\le 2p+2$\\
			\hline
			1&1&$p^{n-r}-1$&$p^{r-1}-p^{n-r}$&$-\eta(j-3)p^{\frac{n-1}{2}}+p^{r-1}$&$p^{n-1}-p^{r-1}$&$p^{n-1}-p^{r-1}$\\
			\hline
			2&1&$-1$&0&0&0&$\eta(j-p-3)p^{\frac{n-1}{2}}$\\
			\hline
			$3\le i\le p+1$&1&$-1$&0&0&$-\eta(i-2)p^{\frac{n-1}{2}}$&$\sqrt{-1}p^{\frac{n-2}{2}}S^{(h)}(j-p-3,i-2)$\\
			\hline
			$p+2$&1&$p^{n-r}-1$&$p^{r-1}-p^{n-r}$&$-\eta(j-3)p^{\frac{n-1}{2}}+p^{r-1}$&$-p^{r-1}$&$-p^{r-1}$\\
			\hline
			$p+3$&1&$p^{n-r}-1$&$-p^{n-r}$&$-\eta(j-3)p^{\frac{n-1}{2}}$&0&0\\
			\hline
			$p+4\le i\le 2p+2$&1&$p^{n-r}-1$&$-p^{n-r}+\eta(i-p-3)p^{\frac{n-1}{2}}$&$-\sqrt{-1}p^{\frac{n-2}{2}}S^{(h)}(j-3,i-p-3)$&0&0\\
			\hline
	\end{tabular}}
	\vspace{-15pt}
\end{table}

Since $p_{u,v}^{w}=p_{v,u}^{w}$ for any $u,v,w\in\{0,1,\dots,2p+1\}$, we only give the values of  the intersection parameters $p_{u,v}^w$ for $u\le v$ in the following eight cases.
\begin{itemize}
	\item [$(\rm{\romannumeral1})$] $w=0$ and $0\le u\le v\le 2p+1$
	\begin{itemize}
		\item [$\bullet$] $u=v$\\
		$p_{0,0}^0=1$, $p_{u,u}^0=p^{n-1}-p^{r-1}$ for $1\le u\le p$, $p_{p+1,p+1}^0=p^{n-r}-1$, $p_{p+2,p+2}^0=p^{r-1}-p^{n-r}$ and $p_{u,u}^0=p^{r-1}+\eta(u-p-2)p^{\frac{n-1}{2}}$ for $p+3\le u\le 2p+1$.
		\item[$\bullet$] $u\ne v$, $p_{u,v}^0=0$.
	\end{itemize}
	\item [$(\rm{\romannumeral2})$] $u=0$, $1\le w\le 2p+1$ and $0\le v\le 2p+1$\\
	$p_{0,v}^w=\begin{cases}
	1,&\text{if}\ w=v,\\
	0,&\text{if}\ w\ne v.
	\end{cases}$
	\item [$(\rm{\romannumeral3})$] $1\le w\le p$ and $1\le u\le v\le p$\\
	$p_{u,v}^w=p^{n-2}-2p^{r-2}-\sqrt{-1}p^{\frac{n-6}{2}}\sum\limits_{i\in\mathbb{F}_p}S^{(h)}(i,w-1)S^{(t)}(u-1,i)S^{(t)}(v-1,i)$.
	\item [$(\rm{\romannumeral4})$] $1\le w\le p$, $1\le u\le p$ and $p+1\le v\le 2p+1$\\
	$p_{u,v}^w=\begin{cases}
	p^{n-r-1}-p^{-1}-p^{-2}\sum\limits_{i\in\mathbb{F}_p}S^{(h)}(i,w-1)S^{(t)}(u-1,i),&\text{if}\ v=p+1,\\
	p^{r-2}-p^{n-r-1},&\text{if}\ v=p+2,\\
	p^{r-2}+\eta(v-p-2)p^{\frac{n-3}{2}},&\text{if}\ v\ge p+3.
	\end{cases}$
	\item [$(\rm{\romannumeral5})$] $1\le w\le p$ and $p+1\le u\le v\le 2p+1$, $p_{u,v}^w=0$.
	\item [$(\rm{\romannumeral6})$] $p+1\le w\le 2p+1$ and $1\le u\le v\le p$\\
	$p_{u,v}^w=\begin{cases}
	p^{n-2}-p^{r-1},&\text{if}\ w=p+1\ \text{and}\ u=v,\\
	p^{n-2},&\text{if}\ w=p+1\ \text{and}\ u\ne v,\\
	p^{n-2}-p^{r-2},&\text{if}\ w\ge p+2.
	\end{cases}$
	\item [$(\rm{\romannumeral7})$] $p+1\le w\le 2p+1$, $1\le u\le p$ and $p+1\le v\le 2p+1$, $p_{u,v}^w=0$.
	\item [$(\rm{\romannumeral8})$] $p+1\le w\le 2p+1$ and $p+1\le u\le v\le 2p+1$
	\begin{itemize}
		\item [$\bullet$] $w=p+1$\\
		$p_{u,v}^{p+1}=\begin{cases}
		p^{n-r}-2,&\text{if}\ u=v=p+1,\\
		0,&\text{if}\ u=p+1\ \text{and}\ v\ge p+2,\\
		p^{r-1}-p^{n-r},&\text{if}\ u=v=p+2,\\
		p^{r-1}+\eta(u-p-2)p^{\frac{n-1}{2}},&\text{if}\ u\ge p+3\ \text{and}\ u=v,\\
		0,&\text{if}\ u\ge p+2\ \text{and}\ u\ne v.
		\end{cases}$
		\item[$\bullet$] $w=p+2$\\
		$p_{u,v}^{p+2}=\begin{cases}
		p^{n-r}-1,&\hspace{-2.9cm}\text{if}\ u=p+1\\&\hspace{-2.9cm} \text{and}\ v=p+2,\\
		0,&\hspace{-2.9cm}\text{if}\ u=p+1\\&\hspace{-2.9cm} \text{and}\ v\ne p+2,\\
		p^{r-2}-2p^{n-r},&\hspace{-2.9cm}\text{if}\ {\small u=v=p+2},\\
		p^{r-2},&\hspace{-2.9cm}\text{if}\ u=p+2\\&\hspace{-2.9cm} \text{and}\ v\ge p+3,\\
		p^{r-2}+
		p^{\frac{n-5}{2}}\sum\limits_{i\in\mathbb{F}_p^*}\eta(i)S^{(t)}(u-p-2,i)S^{(t)}(v-\\p-2,i) +(\eta(u-p-2)+\eta(v-p-2))p^{\frac{n-3}{2}},\text{if}\  p+3\le u\le v.
		\end{cases}$
		\item[$\bullet$] $p+3\le w\le 2p+1$\\
		$p_{u,v}^w=\begin{cases}
		p^{n-r}-1,&\text{if}\ u=p+1\ \text{and}\ w=v,\\
		0,&\text{if}\ u=p+1\ \text{and}\ w\ne v,\\
		p^{r-2}-p^{\frac{n-3}{2}}\eta(w-p-2),&\text{if}\ u=v=p+2,\\
		H_1,&\text{if}\ u=p+2\ \text{and}\ v\ge p+3,\\
		H_2,&\text{if}\ p+3\le u\le v,
		\end{cases}$
		
		where $H_1=p^{r-2}+\eta(v-p-2)p^{\frac{n-3}{2}}-p^{n-r-1}-\eta((w-p-2)(v-p-2))p^{n-r-1}-p^{-2}\sum\limits_{i\in\mathbb{F}_p^*}(\eta(i)p^{\frac{n-1}{2}}+p^{n-r})S^{(h)}(i,w-p-2)S^{(t)}(v-p-2,i)$ and  $H_2=p^{r-2}+(\eta(u-p-2)+\eta(v-p-2))p^{\frac{n-3}{2}}+\sqrt{-1}p^{\frac{n-6}{2}}\sum\limits_{i\in\mathbb{F}_p}S^{(h)}(i,w-p-2)S^{(t)}(u-p-2,i)S^{(t)}(v-p-2,i)$.
	\end{itemize}
\end{itemize}

Since $q_{u,v}^{w}=q_{v,u}^{w}$ for any $u,v,w\in\{0,1,\dots,2p+1\}$, we only give the values of  the Krein parameters $q_{u,v}^w$ for $u\le v$ in the following eight cases.
\begin{itemize}
	\item [$(\rm{\romannumeral1})$] $w=0$ and $0\le u\le v\le2p+1$
	\begin{itemize}
		\item [$\bullet$] $u=v$\\
		$q_{0,0}^0=1$, $q_{1,1}^0=p^{n-r}-1$, $q_{2,2}^0=p^{r-1}-p^{n-r}$, $q_{u,u}^0=p^{r-1}-\eta(u-2)p^{\frac{n-1}{2}}$ for $3\le u\le p+1$ and $q_{u,u}^0=p^{n-1}-p^{r-1}$ for $p+2\le u\le 2p+1$.
		\item[$\bullet$] $u\ne v$, $q_{u,v}^0=0$.
	\end{itemize}
	\item [$(\rm{\romannumeral2})$] $u=0$, $1\le w\le 2p+1$ and $0\le v\le 2p+1$\\
	$q_{u,v}^w=\begin{cases}
	1,&\text{if } w=v,\\
	0,&\text{if } w\ne v.
	\end{cases}$
	\item [$(\rm{\romannumeral3})$] $1\le w\le p+1$ and $1\le u\le v\le p+1$
	\begin{itemize}
		\item [$\bullet$] $w=1$\\
		$q_{u,v}^1=\begin{cases}
		p^{n-r}-2,&\text{if } u=v=1,\\
		0,&\text{if } u=1\ \text{and}\ v\ge 2,\\
		p^{r-1}-p^{n-r},&\text{if}\ u=v=2,\\
		p^{r-1}-\eta(u-2)p^{\frac{n-1}{2}},&\text{if}\ u\ge 3\ \text{and}\ u=v,\\
		0,&\text{if}\ u\ge 2\ \text{and}\ u\ne v.
		\end{cases}$
		\item[$\bullet$] $w=2$\\
		$q_{u,v}^2=\begin{cases}
		p^{n-r}-1,&\text{if}\ u=1\ \text{and}\ v=2,\\
		0,&\text{if}\ u=1\ \text{and}\ v\ne 2,\\
		p^{r-2}-2p^{n-r},&\text{if}\ u=v=2,\\
		p^{r-2},&\text{if}\ u=2\ \text{and}\ v\ge 3,\\
		p^{r-2}-(\eta(u-2)+\eta(v-2))p^{\frac{n-3}{2}}-\\
		p^{\frac{n-5}{2}}\sum\limits_{i\in\mathbb{F}_p^*}\eta(i)S^{(h)}(u-2,i)S^{(h)}(v-2,i),&\text{if}\ 3\le u\le v.
		\end{cases}$
		\item[$\bullet$] $3\le w\le p+1$\\
		$q_{u,v}^w=\begin{cases}
		p^{n-r}-1,&\text{if}\ u=1\ \text{and}\\& w=v,\\
		0,&\text{if}\ u=1\ \text{and}\\& w\ne v,\\
		p^{r-2}+\eta(w-2)p^{\frac{n-3}{2}},&\text{if}\ u=v=2,\\
		H,&\text{if}\ u=2\ \text{and}\\& v\ge 3,\\
		p^{r-2}-(\eta(u-2)+\eta(v-2))p^{\frac{n-3}{2}}-\sqrt{-1}	p^{\frac{n-6}{2}}&\text{if}\ 3\le u\le v,\\
		\sum\limits_{i\in\mathbb{F}_p}S^{(t)}(i,w-2)S^{(h)}(u-2,i)S^{(h)}(v-2,i),
		\end{cases}$
		where $H=p^{r-2}-p^{n-r-1}-\eta(v-2)p^{\frac{n-3}{2}}-\eta((w-2)(v-2))p^{n-r-1}+p^{-2}\sum\limits_{i\in\mathbb{F}_p^*}(\eta(i)p^{\frac{n-1}{2}}-p^{n-r})S^{(t)}(i,w-2)S^{(h)}(v-2,i)$.
	\end{itemize}
	\item [$(\rm{\romannumeral4})$] $1\le w\le p+1$, $1\le u\le p+1$ and $p+2\le v\le 2p+1$, $q_{u,v}^w=0$.
	\item [$(\rm{\romannumeral5})$] $1\le w\le p+1$ and $p+2\le u\le v\le 2p+1$\\
	\newpage
	$q_{u,v}^w=\begin{cases}
	p^{n-2}-p^{r-1},&\text{if}\ w=1\ \text{and}\ u=v,\\
	p^{n-2},&\text{if}\ w=1\ \text{and}\ u\ne v,\\
	p^{n-2}-p^{r-2},&\text{if}\ w\ge 2.
	\end{cases}$
	\item [$(\rm{\romannumeral6})$] $p+2\le w\le 2p+1$ and $1\le u\le v\le p+1$, $q_{u,v}^w=0$.
	\item [$(\rm{\romannumeral7})$] $p+2\le w\le 2p+1$, $1\le u\le p+1$ and $p+2\le v\le 2p+1$\\
	$q_{u,v}^w=\begin{cases}
	p^{n-r-1}-p^{-1}-p^{-2}\sum\limits_{i\in\mathbb{F}_p}S^{(t)}(i,w-p-2)\times\\
	S^{(h)}(v-p-2,i),&\text{if}\ u=1,\\
	p^{r-2}-p^{n-r-1},&\text{if}\ u=2,\\
	p^{r-2}-\eta(u-2)p^{\frac{n-3}{2}},&\text{if}\ u\ge3.
	\end{cases}$
	\item [$(\rm{\romannumeral8})$] $p+2\le w\le 2p+1$ and $p+2\le u\le v\le 2p+1$\\
	$q_{u,v}^w=p^{n-2}-2p^{r-2}+\sqrt{-1}p^{\frac{n-6}{2}}\sum\limits_{i\in\mathbb{F}_p}S^{(t)}(i,w-p-2)S^{(h)}(u-p-2,i)S^{(h)}(v-p-2,i)$.
\end{itemize}
18. The first and second eigenmatrices, the intersection numbers and the Krein parameters of the association scheme induced by  $U_{18}$.

Note that the first and second eigenmatrices of the association scheme induced by $U_{18}$ are the same. The first (second) eigenmatrix of the association scheme induced by $U_{18}$ is given in Table 29.
\begin{table}[h]
	\vspace{-15pt}
	\centering
	\caption{The first (second) eigenmatrix of the association scheme induced by $U_{18}$}
	\renewcommand\arraystretch{2}	
	\resizebox{\textwidth}{!}{
		\begin{tabular}{|m{2.3cm}<{\centering}|c|c|c|c|c|c|
			}
			\hline
			\diagbox{$i$}{{\tiny $P_{18}(ij)\ (Q_{18}(ij))$}}{$j$}&1&2&3&4&5&6\\
			\hline
			1&1&$p^{\frac{n-1}{2}}-1$&$(p-1)p^{\frac{n-1}{2}}$&$p^{n-1}-p^{\frac{n-1}{2}}$&$\frac{(p-1)}{2}(p^{n-1}-p^{\frac{n-1}{2}})$&$\frac{(p-1)}{2}(p^{n-1}-p^{\frac{n-1}{2}})$\\
			\hline
			2&1&$p^{\frac{n-1}{2}}-1$&$(p-1)p^{\frac{n-1}{2}}$&$-p^{\frac{n-1}{2}}$&$-\frac{(p-1)}{2}p^{\frac{n-1}{2}}$&$-\frac{(p-1)}{2}p^{\frac{n-1}{2}}$\\
			\hline
			3&1&$p^{\frac{n-1}{2}}-1$&$-p^{\frac{n-1}{2}}$&0&0&0\\
			\hline
			4&1&$-1$&0&0&$-\frac{(p-1)}{2}p^{\frac{n-1}{2}}$&$\frac{(p-1)}{2}p^{\frac{n-1}{2}}$\\
			\hline
			5&1&$-1$&0&$-p^{\frac{n-1}{2}}$&$p^{\frac{n-1}{2}}$&0\\
			\hline
			6&1&$-1$&0&$p^{\frac{n-1}{2}}$&0&$-p^{\frac{n-1}{2}}$\\
			\hline
	\end{tabular}}
	\vspace{-15pt}
\end{table}

Since $p_{u,v}^{w}=p_{v,u}^{w}$ for any $u,v,w\in\{0,1,\dots,5\}$, we only give the values of the intersection numbers $p_{u,v}^w$ for $u\le v$ in the following six cases. The Krein parameter $q_{u,v}^w$ is the same as the intersection number $p_{u,v}^w$ for any $u,v,w\in\{0,1,\dots,5\}$.
\begin{itemize}
	\item [$(\rm{\romannumeral1})$] $w=0$
	\begin{itemize}
		\item [$\bullet$] $u=v$\\
		$p_{0,0}^0=1$, $p_{1,1}^0=p^{\frac{n-1}{2}}-1$, $p_{2,2}^0=(p-1)p^{\frac{n-1}{2}}$, $p_{3,3}^0=p^{n-1}-p^{\frac{n-1}{2}}$ and $p_{4,4}^0=p_{5,5}^0=\frac{(p-1)}{2}(p^{n-1}-p^{\frac{n-1}{2}})$.
		\item[$\bullet$] $u\ne v$, $p_{u,v}^0=0$ for $0\le u<v\le 5$.
	\end{itemize}
	\item [$(\rm{\romannumeral2})$] $w=1$
	\begin{itemize}
		\item [$\bullet$] $p_{0,1}^1=1$ and $p_{0,v}^1=0$ for $0\le v\le 5$ and $v\ne 1$.
		\item[$\bullet$] $p_{1,1}^1=p^{\frac{n-1}{2}}-2$ and $p_{1,v}^1=0$ for $2\le v\le 5$.
		\item[$\bullet$] $p_{2,2}^1=(p-1)p^{\frac{n-1}{2}}$ and $p_{2,v}^1=0$ for $3\le v\le 5$.
		\item[$\bullet$] $p_{3,3}^1=p^{n-2}-p^{\frac{n-1}{2}}$ and $p_{3,4}^1=p_{3,5}^1=\frac{(p-1)}{2}p^{n-2}$.
			\item[$\bullet$] $p_{4,4}^1=\frac{(p-1)}{4}(p^{n-1}-p^{n-2}-2p^{\frac{n-1}{2}})$ and $p_{4,5}^1=\frac{(p-1)^2}{4}p^{n-2}$.
			\item[$\bullet$] $p_{5,5}^1=\frac{(p-1)}{4}(p^{n-1}-p^{n-2}-2p^{\frac{n-1}{2}})$.
	\end{itemize}
	\item [$(\rm{\romannumeral3})$] $w=2$
	\begin{itemize}
		\item [$\bullet$] $p_{0,2}^2=1$ and $p_{0,v}^2=0$ for $0\le v\le 5$ and $v\ne 2$.
		\item[$\bullet$] $p_{1,2}^2=p^{\frac{n-1}{2}}-1$ and $p_{1,v}^2=0$ for $1\le v\le 5$ and $v\ne 2$.
		\item[$\bullet$] $p_{2,2}^2=(p-2)p^{\frac{n-1}{2}}$ and $p_{2,v}^2=0$ for $3\le v\le 5$.
		\item[$\bullet$] $p_{3,3}^2=p^{n-2}-p^{\frac{n-3}{2}}$ and $p_{3,4}^2=p_{3,5}^2=\frac{(p-1)}{2}(p^{n-2}-p^{\frac{n-3}{2}})$.
		\item[$\bullet$] $p_{4,4}^2=p_{4,5}^2=\frac{(p-1)^2}{4}(p^{n-2}-p^{\frac{n-3}{2}})$.
	\item[$\bullet$] $p_{5,5}^2=\frac{(p-1)^2}{4}(p^{n-2}-p^{\frac{n-3}{2}})$.
	\end{itemize}
\item [$(\rm{\romannumeral4})$] $w=3$
\begin{itemize}
	\item [$\bullet$] $p_{0,3}^3=1$ and $p_{0,v}^3=0$ for $0\le v\le 5$ and $v\ne 3$.
	\item[$\bullet$] $p_{1,1}^3=p_{1,2}^3=0$, $p_{1,3}^3=p^{\frac{n-3}{2}}-1$ and $p_{1,4}^3=p_{1,5}^3=\frac{(p-1)}{2}p^{\frac{n-3}{2}}$.
	\item[$\bullet$] $p_{2,2}^3=0$, $p_{2,3}^3=(p-1)p^{\frac{n-3}{2}}$ and $p_{2,4}^3=p_{2,5}^3=\frac{(p-1)^2}{2}p^{\frac{n-3}{2}}$.
	\item[$\bullet$] $p_{3,3}^3=p^{n-2}-2p^{\frac{n-3}{2}}$, $p_{3,4}^3=\frac{(p-1)}{2}(p^{n-2}-p^{\frac{n-3}{2}})$ and $p_{3,5}^3=\frac{(p-1)}{2}(p^{n-2}-3p^{\frac{n-3}{2}})$.
	\item[$\bullet$] $p_{4,4}^3=\frac{(p-1)}{4}(p^{n-1}-p^{n-2}-2p^{\frac{n-1}{2}})$ and $p_{4,5}^3=\frac{(p-1)^2}{4}(p^{n-2}-2p^{\frac{n-3}{2}})$.
	\item[$\bullet$] $p_{5,5}^3=\frac{(p-1)}{4}(p^{n-1}-p^{n-2}+4p^{\frac{n-3}{2}}-2p^{\frac{n-1}{2}})$.
\end{itemize}
\item [$(\rm{\romannumeral5})$] $w=4$
\begin{itemize}
	\item [$\bullet$] $p_{0,4}^4=1$ and $p_{0,v}^4=0$ for $0\le v\le 5$ and $v\ne 4$.
	\item[$\bullet$] $p_{1,1}^4=p_{1,2}^4=0$, $p_{1,3}^4=p^{\frac{n-3}{2}}$, $p_{1,4}^4=\frac{1}{2}(p^{\frac{n-1}{2}}-p^{\frac{n-3}{2}}-2)$ and $p_{1,5}^4=\frac{(p-1)}{2}p^{\frac{n-3}{2}}$.
	\item[$\bullet$] $p_{2,2}^4=0$, $p_{2,3}^4=(p-1)p^{\frac{n-3}{2}}$,  $p_{2,4}^4=\frac{(p-1)^2}{2}p^{\frac{n-3}{2}}$ and $p_{2,5}^4=\frac{(p-1)^2}{2}p^{\frac{n-3}{2}}$.
	\item[$\bullet$] $p_{3,3}^4=p^{n-2}-p^{\frac{n-3}{2}}$, $p_{3,4}^4=\frac{1}{2}(p^{n-1}-p^{n-2}-2p^{\frac{n-1}{2}})$ and $p_{3,5}^4=\frac{(p-1)}{2}(p^{n-2}-2p^{\frac{n-3}{2}})$.
	\item[$\bullet$] $p_{4,4}^4=\frac{1}{4}(p^{n-2}-2p^{n-1}+p^n+p^{\frac{n-3}{2}}+6p^{\frac{n-1}{2}}-3p^{\frac{n+1}{2}})$ and $p_{4,5}^4=\frac{(p-1)^2}{4}(p^{n-2}-p^{\frac{n-3}{2}})$.
	\item[$\bullet$] $p_{5,5}^4=\frac{(p-1)^2}{4}(p^{n-2}-3p^{\frac{n-3}{2}})$.
\end{itemize}
\item [$(\rm{\romannumeral6})$] $w=5$
\begin{itemize}
	\item [$\bullet$] $p_{0,5}^5=1$ and $p_{0,v}^5=0$ for $0\le v\le 4$.
	\item[$\bullet$] $p_{1,1}^5=p_{1,2}^5=0$, $p_{1,3}^5=p^{\frac{n-3}{2}}$, $p_{1,4}^5=\frac{(p-1)}{2}p^{\frac{n-3}{2}}$ and $p_{1,5}^5=\frac{1}{2}(p^{\frac{n-1}{2}}-p^{\frac{n-3}{2}}-2)$.
	\item[$\bullet$] $p_{2,2}^5=0$, $p_{2,3}^5=(p-1)p^{\frac{n-3}{2}}$ and $p_{2,4}^5=p_{2,5}^5=\frac{(p-1)^2}{2}p^{\frac{n-3}{2}}$.
		\item[$\bullet$] $p_{3,3}^5=p^{n-2}-3p^{\frac{n-3}{2}}$, $p_{3,4}^5=\frac{(p-1)}{2}(p^{n-2}-2p^{\frac{n-3}{2}})$ and $p_{3,5}^5=\frac{1}{2}(p^{n-1}-p^{n-2}+4p^{\frac{n-3}{2}}-2p^{\frac{n-1}{2}})$.
		\item[$\bullet$] $p_{4,4}^5=\frac{(p-1)^2}{4}(p^{n-2}-p^{\frac{n-3}{2}})$ and $p_{4,5}^5=\frac{(p-1)^2}{4}(p^{n-2}-3p^{\frac{n-3}{2}})$.
		\item[$\bullet$] $p_{5,5}^5=\frac{1}{4}(p^n+p^{n-2}+2p^{\frac{n-1}{2}}-2p^{n-1}-5p^{\frac{n-3}{2}}-p^{\frac{n+1}{2}})$.
	\end{itemize}
\end{itemize}
19. The first and second eigenmatrices, the intersection numbers and the  Krein parameters of the association scheme induced by $U_{19}$.

The first and second eigenmatrices of the association scheme induced by $U_{19}$ are given in Tables 30 and 31.

Since $p_{u,v}^{w}=p_{v,u}^{w}$ for any $u,v,w\in\{0,1,\dots,5\}$, we only give the values of  the intersection numbers $p_{u,v}^w$ for $u\le v$ in the following six cases.
\begin{table}[h]
	\vspace{-15pt}
	\centering
	\caption{The first eigenmatrix of the association scheme induced by $U_{19}$}
	\renewcommand\arraystretch{2}	
	\resizebox{\textwidth}{!}{
		\begin{tabular}{|m{2.3cm}<{\centering}|c|c|c|c|c|c|
			}
			\hline
			\diagbox{$i$}{ $P_{19}(ij)$}{$j$}&1&2&3&4&5&6\\
			\hline
			1&1&$p^{\frac{n-1}{2}}-1$&$(p-1)p^{\frac{n-1}{2}}$&$p^{n-1}-p^{\frac{n-1}{2}}$&$\frac{(p-1)}{2}(p^{n-1}-p^{\frac{n-1}{2}})$&$\frac{(p-1)}{2}(p^{n-1}-p^{\frac{n-1}{2}})$\\
			\hline
			2&1&$-1$&0&0&$\frac{(p-1)}{2}p^{\frac{n-1
				}{2}}$&$-\frac{(p-1)}{2}p^{\frac{n-1}{2}}$\\
			\hline
			3&1&$-1$&0&$-p^{\frac{n-1}{2}}$&0&$p^{\frac{n-1}{2}}$\\
			\hline
			4&1&$-1$&0&$p^{\frac{n-1}{2}}$&$-p^{\frac{n-1}{2}}$&0\\
			\hline
			5&1&$p^{\frac{n-1}{2}}-1$&$(p-1)p^{\frac{n-1}{2}}$&$-p^{\frac{n-1}{2}}$&$-\frac{(p-1)}{2}p^{\frac{n-1}{2}}$&$-\frac{(p-1)}{2}p^{\frac{n-1}{2}}$\\
			\hline
			6&1&$p^{\frac{n-1}{2}}-1$&$-p^{\frac{n-1}{2}}$&0&0&0\\
			\hline
	\end{tabular}}
\vspace{-15pt}
\end{table}

\begin{table}[h]
	\vspace{-15pt}
	\centering
	\caption{The second eigenmatrix of the association scheme induced by $U_{19}$}
	\renewcommand\arraystretch{2}	
	\resizebox{\textwidth}{!}{
		\begin{tabular}{|m{2.3cm}<{\centering}|c|c|c|c|c|c|
			}
			\hline
			\diagbox{$i$}{ $Q_{19}(ij)$}{$j$}&1&2&3&4&5&6\\
			\hline
			1&1&$p^{n-1}-p^{\frac{n-1}{2}}$&$\frac{(p-1)}{2}(p^{n-1}-p^{\frac{n-1}{2}})$&$\frac{(p-1)}{2}(p^{n-1}-p^{\frac{n-1}{2}})$&$p^{\frac{n-1}{2}}-1$&$(p-1)p^{\frac{n-1}{2}}$\\
			\hline
			2&1&$-p^{\frac{n-1}{2}}$&$-\frac{(p-1)}{2}p^{\frac{n-1}{2}}$&$-\frac{(p-1)}{2}p^{\frac{n-1}{2}}$&$p^{\frac{n-1}{2}}-1$&$(p-1)p^{\frac{n-1}{2}}$\\
			\hline
			3&1&0&0&0&$p^{\frac{n-1}{2}}-1$&$-p^{\frac{n-1}{2}}$\\
			\hline
			4&1&0&$-\frac{(p-1)}{2}p^{\frac{n-1}{2}}$&$\frac{(p-1)}{2}p^{\frac{n-1}{2}}$&$-1$&0\\
			\hline
			5&1&$p^{\frac{n-1}{2}}$&0&$-p^{\frac{n-1}{2}}$&$-1$&0\\
			\hline
			6&1&$-p^{\frac{n-1}{2}}$&$p^{\frac{n-1}{2}}$&0&$-1$&0\\
			\hline
	\end{tabular}}
	\vspace{-15pt}
\end{table}

	\begin{itemize}
		\item [$(\rm{\romannumeral1})$] $w=0$
\begin{itemize}
	\item [$\bullet$] $u=v$\\
$p_{0,0}^0=1$, $p_{1,1}^0=p^{\frac{n-1}{2}}-1$, $p_{2,2}^0=(p-1)p^{\frac{n-1}{2}}$, $p_{3,3}^0=p^{n-1}-p^{\frac{n-1}{2}}$ and $p_{4,4}^0=p_{5,5}^0=\frac{(p-1)}{2}(p^{n-1}-p^{\frac{n-1}{2}})$.
\item[$\bullet$] $u\ne v$, $p_{u,v}^0=0$ for $0\le u<v\le 5$.
\end{itemize}
	\item [$(\rm{\romannumeral2})$] $w=1$
	\begin{itemize}
		\item [$\bullet$] $p_{0,1}^1=1$ and $p_{0,v}^1=0$ for $0\le v\le 5$ and $v\ne 5$.
		\item[$\bullet$] $p_{1,1}^1=p^{\frac{n-1}{2}}-2$ and $p_{1,v}^1=0$ for $2\le v\le 5$.
		\item[$\bullet$] $p_{2,2}^1=(p-1)p^{\frac{n-1}{2}}$ and $p_{2,v}^1=0$ for $3\le v\le 5$.
		\item[$\bullet$] $p_{3,3}^1=p^{n-2}-p^{\frac{n-1}{2}}$ and $p_{3,4}^1=p_{3,5}^1=\frac{(p-1)}{2}p^{n-2}$.
		\item[$\bullet$] $p_{4,4}^1=\frac{(p-1)}{4}(p^{n-1}-p^{n-2}-2p^{\frac{n-1}{2}})$ and $p_{4,5}^1=\frac{(p-1)^2}{4}p^{n-2}$.
		\item[$\bullet$] $p_{5,5}^1=\frac{(p-1)}{4}(p^{n-1}-p^{n-2}-2p^{\frac{n-1}{2}})$.
				\end{itemize}
	\item [$(\rm{\romannumeral3})$] $w=2$ 
	\begin{itemize}
		\item [$\bullet$] $p_{0,2}^2=1$ and $p_{0,v}^2=0$ for $0\le v\le 5$ and $v\ne 2$.
		\item[$\bullet$] $p_{1,1}^2=0$, $p_{1,2}^2=p^{\frac{n-1}{2}}-1$ and $p_{1,v}^2=0$ for $3\le v\le 5$.
		\item[$\bullet$] $p_{2,2}^2=(p-2)p^{\frac{n-1}{2}}$ and $p_{2,v}^2=0$ for $3\le v\le 5$.
		\item[$\bullet$] $p_{3,3}^2=p^{n-2}-p^{\frac{n-3}{2}}$ and $p_{3,4}^2=p_{3,5}^2=\frac{(p-1)}{2}(p^{n-2}-p^{\frac{n-3}{2}})$.
			\item[$\bullet$] $p_{4,4}^2=p_{4,5}^2=\frac{(p-1)^2}{4}(p^{n-2}-p^{\frac{n-3}{2}})$.
			\item[$\bullet$] $p_{5,5}^2=\frac{(p-1)^2}{4}(p^{n-2}-p^{\frac{n-3}{2}})$.
	\end{itemize}
	\item [$(\rm{\romannumeral4})$] $w=3$
	\begin{itemize}
		\item [$\bullet$] $p_{0,3}^3=1$ and $p_{0,v}^3=0$ for $0\le v\le 5$ and $v\ne 3$.
		\item[$\bullet$] $p_{1,1}^3=p_{1,2}^3=0$, $p_{1,3}^3=p^{\frac{n-3}{2}}-1$ and $p_{1,4}^3=p_{1,5}^3=\frac{(p-1)}{2}p^{\frac{n-3}{2}}$.
		\item[$\bullet$] $p_{2,2}^3=0$, $p_{2,3}^3=(p-1)p^{\frac{n-3}{2}}$ and $p_{2,4}^3=p_{2,5}^3=\frac{(p-1)^2}{2}p^{\frac{n-3}{2}}$.
		\item[$\bullet$] $p_{3,3}^3=p^{n-2}-2p^{\frac{n-3}{2}}$, $p_{3,4}^3=\frac{(p-1)}{2}(p^{n-2}-3p^{\frac{n-3}{2}})$ and $p_{3,5}^3=\frac{(p-1)}{2}(p^{n-2}-p^{\frac{n-3}{2}})$.
		\item[$\bullet$] $p_{4,4}^3=\frac{(p-1)}{4}(p^{n-1}-p^{n-2}+4p^{\frac{n-3}{2}}-2p^{\frac{n-1}{2}})$ and $p_{4,5}^3=\frac{(p-1)^2}{4}(p^{n-2}-2p^{\frac{n-3}{2}})$.
		\item[$\bullet$] $p_{5,5}^3=\frac{(p-1)}{4}(p^{n-1}-p^{n-2}-2p^{\frac{n-1}{2}})$.
	\end{itemize}
	\item [$(\rm{\romannumeral5})$] $w=4$
	\begin{itemize}
		\item [$\bullet$] $p_{0,4}^4=1$ and $p_{0,v}^4=0$ for $0\le v\le 5$ and $v\ne 4$.
		\item[$\bullet$] $p_{1,1}^4=p_{1,2}^4=0$, $p_{1,3}^4=p^{\frac{n-3}{2}}$,  $p_{1,4}^4=\frac{1}{2}(p^{\frac{n-1}{2}}-p^{\frac{n-3}{2}}-2)$ and $p_{1,5}^4=\frac{(p-1)}{2}p^{\frac{n-3}{2}}$.
		\item[$\bullet$] $p_{2,2}^4=0$, $p_{2,3}^4=(p-1)p^{\frac{n-3}{2}}$ and $p_{2,4}^4=p_{2,5}^4=\frac{(p-1)^2}{2}p^{\frac{n-3}{2}}$.
		\item[$\bullet$] $p_{3,3}^4=p^{n-2}-3p^{\frac{n-3}{2}}$, $p_{3,4}^4=\frac{1}{2}(p^{n-1}-p^{n-2}+4p^{\frac{n-3}{2}}-2p^{\frac{n-1}{2}})$ and $p_{3,5}^4=\frac{(p-1)}{2}(p^{n-2}-2p^{\frac{n-3}{2}})$.
			\item[$\bullet$] $p_{4,4}^4=\frac{1}{4}(p^n+p^{n-2}+2p^{\frac{n-1}{2}}-2p^{n-1}-5p^{\frac{n-3}{2}}-p^{\frac{n+1}{2}})$ and $p_{4,5}^4=\frac{(p-1)^2}{4}(p^{n-2}-3p^{\frac{n-3}{2}})$.
			\item[$\bullet$] $p_{5,5}^4=\frac{(p-1)^2}{4}(p^{n-2}-p^{\frac{n-3}{2}})$.
	\end{itemize}
	\item [$(\rm{\romannumeral6})$] $w=5$
	\begin{itemize}
		\item [$\bullet$] $p_{0,5}^5=1$ and $p_{0,v}^5=0$ for $0\le v\le 4$.
		\item[$\bullet$] $p_{1,1}^5=p_{1,2}^5=0$, $p_{1,3}^5=p^{\frac{n-3}{2}}$, $p_{1,4}^5=\frac{(p-1)}{2}p^{\frac{n-3}{2}}$ and $p_{1,5}^5=\frac{1}{2}(p^{\frac{n-1}{2}}-p^{\frac{n-3}{2}}-2)$.
		\item[$\bullet$] $p_{2,2}^5=0$, $p_{2,3}^5=(p-1)p^{\frac{n-3}{2}}$ and $p_{2,4}^5=p_{2,5}^5=\frac{(p-1)^2}{2}p^{\frac{n-3}{2}}$.
		\item[$\bullet$] $p_{3,3}^5=p^{n-2}-p^{\frac{n-3}{2}}$, $p_{3,4}^5=\frac{(p-1)}{2}(p^{n-2}-2p^{\frac{n-3}{2}})$ and $p_{3,5}^5=\frac{1}{2}(p^{n-1}-p^{n-2}-2p^{\frac{n-1}{2}})$.
		\item[$\bullet$] $p_{4,4}^5=\frac{(p-1)^2}{4}(p^{n-2}-3p^{\frac{n-3}{2}})$ and $p_{4,5}^5=\frac{(p-1)^2}{4}(p^{n-2}-p^{\frac{n-3}{2}})$.
		\item[$\bullet$] $p_{5,5}^5=\frac{1}{4}(p^{n-2}-2p^{n-1}+p^n+p^{\frac{n-3}{2}}+6p^{\frac{n-1}{2}}-3p^{\frac{n+1}{2}})$.
		
			\end{itemize}
\end{itemize}

Since $q_{u,v}^{w}=q_{v,u}^{w}$ for any $u,v,w\in\{0,1,\dots,5\}$, we only give the values of  the Krein parameters $q_{u,v}^w$ for $u\le v$ in the following six cases.
\begin{itemize}
	\item [$(\rm{\romannumeral1})$] $w=0$
	\begin{itemize}
		\item [$\bullet$] $u=v$\\
		$q_{0,0}^0=1$, $q_{1,1}^0=p^{n-1}-p^{\frac{n-1}{2}}$, $q_{2,2}^0=q_{3,3}^0=\frac{(p-1)}{2}(p^{n-1}-p^{\frac{n-1}{2}})$, $q_{4,4}^0=p^{\frac{n-1}{2}}-1$ and $q_{5,5}^0=(p-1)p^{\frac{n-1}{2}}$.
		\item[$\bullet$] $u\ne v$, $q_{u,v}^0=0$  for $0\le u<v\le 5$.
	\end{itemize}
	\item [$(\rm{\romannumeral2})$] $w=1$
	\begin{itemize}
	\item[$\bullet$] $q_{0,1}^1=1$ and $q_{0,v}^1=0$ for $0\le v\le 5$ and $v\ne 1$.
	\item[$\bullet$] $q_{1,1}^1=p^{n-2}-2p^{\frac{n-3}{2}}$, $q_{1,2}^1=\frac{(p-1)}{2}(p^{n-2}-p^{\frac{n-3}{2}})$ and $q_{1,3}^1=\frac{(p-1)}{2}(p^{n-2}-3p^{\frac{n-3}{2}})$, $q_{1,4}^1=p^{\frac{n-3}{2}}-1$ and $q_{1,5}^1=(p-1)p^{\frac{n-3}{2}}$.
	\item[$\bullet$] $q_{2,2}^1=\frac{(p-1)}{4}(p^{n-1}-p^{n-2}-2p^{\frac{n-1}{2}})$, $q_{2,3}^1=\frac{(p-1)^2}{4}(p^{n-2}-2p^{\frac{n-3}{2}})$, $q_{2,4}^1=\frac{(p-1)}{2}p^{\frac{n-3}{2}}$ and $q_{2,5}^1=\frac{(p-1)^2}{2}p^{\frac{n-3}{2}}$.
	\item[$\bullet$] $q_{3,3}^1=\frac{(p-1)}{4}(p^{n-1}-p^{n-2}+4p^{\frac{n-3}{2}}-2p^{\frac{n-1}{2}})$, $q_{3,4}^1=\frac{(p-1)}{2}p^{\frac{n-3}{2}}$ and $q_{3,5}^1=\frac{(p-1)^2}{2}p^{\frac{n-3}{2}}$.
	\item[$\bullet$] $q_{4,4}^1=q_{4,5}^1=q_{5,5}^1=0$.
	\end{itemize}
\item [$(\rm{\romannumeral3})$] $w=2$
\begin{itemize}
	\item [$\bullet$] $q_{0,2}^2=1$ and $q_{0,v}^2=0$ for $0\le v\le 5$ and $v\ne 2$.
	\item[$\bullet$] $q_{1,1}^2=p^{n-2}-p^{\frac{n-3}{2}}$, $q_{1,2}^2=\frac{1}{2}(p^{n-1}-p^{n-2}-2p^{\frac{n-1}{2}})$, $q_{1,3}^2=\frac{(p-1)}{2}(p^{n-2}-2p^{\frac{n-3}{2}})$, $q_{1,4}^2=p^{\frac{n-3}{2}}$ and $q_{1,5}^2=(p-1)p^{\frac{n-3}{2}}$.
	\item[$\bullet$] $q_{2,2}^2=\frac{1}{4}(p^{n-2}-2p^{n-1}+p^n+p^{\frac{n-3}{2}}+6p^{\frac{n-1}{2}}-3p^{\frac{n+1}{2}})$, $q_{2,3}^2=\frac{(p-1)^2}{4}(p^{n-2}-p^{\frac{n-3}{2}})$, $q_{2,4}^2=\frac{1}{2}(p^{\frac{n-1}{2}}-p^{\frac{n-3}{2}}-2)$ and $q_{2,5}^2=\frac{(p-1)^2}{2}p^{\frac{n-3}{2}}$.
		\item[$\bullet$] $q_{3,3}^2=\frac{(p-1)^2}{4}(p^{n-2}-3p^{\frac{n-3}{2}})$, $q_{3,4}^2=\frac{(p-1)}{2}p^{\frac{n-3}{2}}$ and $q_{3,5}^2=\frac{(p-1)^2}{2}p^{\frac{n-3}{2}}$.
		\item[$\bullet$] $q_{4,4}^2=q_{4,5}^2=q_{5,5}^2=0$.
\end{itemize}
\item [$(\rm{\romannumeral4})$] $w=3$
\begin{itemize}
	\item [$\bullet$] $q_{0,3}^3=1$ and $q_{0,v}^3=0$ for $0\le v\le 5$ and $v\ne 3$.
	\item[$\bullet$] $q_{1,1}^3=p^{n-2}-3p^{\frac{n-3}{2}}$, $q_{1,2}^3=\frac{(p-1)}{2}(p^{n-2}-2p^{\frac{n-3}{2}})$, $q_{1,3}^3=\frac{1}{2}(p^{n-1}-p^{n-2}+4p^{\frac{n-3}{2}}-2p^{\frac{n-1}{2}})$, $q_{1,4}^3=p^{\frac{n-3}{2}}$ and $q_{1,5}^3=(p-1)p^{\frac{n-3}{2}}$.
	\item[$\bullet$] $q_{2,2}^3=\frac{(p-1)^2}{4}(p^{n-2}-p^{\frac{n-3}{2}})$, $q_{2,3}^3=\frac{(p-1)^2}{4}(p^{n-2}-3p^{\frac{n-3}{2}})$, $q_{2,4}^3=\frac{(p-1)}{2}p^{\frac{n-3}{2}}$ and $q_{2,5}^3=\frac{(p-1)^2}{2}p^{\frac{n-3}{2}}$.
	\item[$\bullet$] $q_{3,3}^3=\frac{1}{4}(p^{n-2}+p^n+2p^{\frac{n-1}{2}}-2p^{n-1}-5p^{\frac{n-3}{2}}-p^{\frac{n+1}{2}})$, $q_{3,4}^3=\frac{1}{2}(p^{\frac{n-1}{2}}-p^{\frac{n-3}{2}}-2)$ and $q_{3,5}^3=\frac{(p-1)^2}{2}p^{\frac{n-3}{2}}$.
	\item[$\bullet$] $q_{4,4}^3=q_{4,5}^3=q_{5,5}^3=0$.
\end{itemize}
\item [$(\rm{\romannumeral5})$] $w=4$
\begin{itemize}
	\item [$\bullet$] $q_{0,4}^4=1$ and $q_{0,v}^4=0$ for $0\le v\le 5$ and $v\ne 4$.
	\item[$\bullet$] $q_{1,1}^4=p^{n-2}-p^{\frac{n-1}{2}}$, $q_{1,2}^4=q_{1,3}^4=\frac{(p-1)}{2}p^{n-2}$ and $q_{1,4}^4=q_{1,5}^4=0$.
	\item[$\bullet$] $q_{2,2}^4=\frac{(p-1)}{4}(p^{n-1}-p^{n-2}-2p^{\frac{n-1}{2}})$, $q_{2,3}^4=\frac{(p-1)^2}{4}p^{n-2}$ and $q_{2,4}^4=q_{2,5}^4=0$.
	\item[$\bullet$] $q_{3,3}^4=\frac{(p-1)}{4}(p^{n-1}-p^{n-2}-2p^{\frac{n-1}{2}})$ and $q_{3,4}^4=q_{3,5}^4=0$.
	\item[$\bullet$] $q_{4,4}^4=p^{\frac{n-1}{2}}-2$ and $q_{4,5}^4=0$.
	\item[$\bullet$] $q_{5,5}^4=(p-1)p^{\frac{n-1}{2}}$.
	\end{itemize}
	\item [$(\rm{\romannumeral6})$] $w=5$
	\begin{itemize}
		\item [$\bullet$] $q_{0,5}^5=1$ and $q_{0,v}^5=0$ for $0\le v\le 4$.
		\item[$\bullet$] $q_{1,1}^5=p^{n-2}-p^{\frac{n-3}{2}}$, $q_{1,2}^5=q_{1,3}^5=\frac{(p-1)}{2}(p^{n-2}-p^{\frac{n-3}{2}})$ and $q_{1,4}^5=q_{1,5}^5=0$.
		\item[$\bullet$] $q_{2,2}^5=q_{2,3}^5=\frac{(p-1)^2}{4}(p^{n-2}-p^{\frac{n-3}{2}})$ and $q_{2,4}^5=q_{2,5}^5=0$.
		\item[$\bullet$] $q_{3,3}^5=\frac{(p-1)^2}{4}(p^{n-2}-p^{\frac{n-3}{2}})$ and $q_{3,4}^5=q_{3,5}^5=0$.
		\item[$\bullet$] $q_{4,4}^5=0$ and $q_{4,5}^5=p^{\frac{n-1}{2}}-1$.
		\item[$\bullet$] $q_{5,5}^5=(p-2)p^{\frac{n-1}{2}}$.
	\end{itemize}
\end{itemize}
	20. The first and second eigenmatrices, the intersection numbers and the Krein parameters of the association scheme induced by  $U_{20}$.

Note that the first and second eigenmatrices of the association scheme induced by $U_{20}$ are the same. The first (second) eigenmatrix of the association scheme induced by $U_{20}$ is given in Table 32.
\begin{table}[h]
	\vspace{-15pt}
	\centering
	\caption{The first (second) eigenmatrix of the association scheme induced by $U_{20}$}
	\renewcommand\arraystretch{1.25}	
	\resizebox{\textwidth}{!}{
		\begin{tabular}{|c|m{11cm}|}
			\hline
			$i$&\hspace{4.25cm} $P_{20}(ij)\ (Q_{20} (ij))$\\
			\hline
			1&$P_{20}(11)=1$, $P_{20}(12)=p^{n-r}-1$, $P_{20}(13)=p^{r-1}-p^{n-r}$, $P_{20}(14)=\frac{(p-1)}{2}(p^{r-1}+p^{\frac{n-1}{2}})$, $P_{20}(15)=\frac{(p-1)}{2}(p^{r-1}-p^{\frac{n-1}{2}})$, $P_{20}(16)=p^{n-1}-p^{r-1}$, $P_{20}(17)=P_{20}(18)=\frac{(p-1)}{2}(p^{n-1}-p^{r-1})$\\
			\hline
			2&$P_{20}(21)=1$, $P_{20}(22)=p^{n-r}-1$, $P_{20}(23)=p^{r-1}-p^{n-r}$, $P_{20}(24)=\frac{(p-1)}{2}(p^{r-1}+p^{\frac{n-1}{2}})$, $P_{20}(25)=\frac{(p-1)}{2}(p^{r-1}-p^{\frac{n-1}{2}})$, $P_{20}(26)=-p^{r-1}$, $P_{20}(27)=P_{20}(28)=-\frac{(p-1)}{2}p^{r-1}$\\
			\hline
			3&$P_{20}(31)=1$, $P_{20}(32)=p^{n-r}-1$, $P_{20}(33)=-p^{n-r}$, $P_{20}(34)=\frac{(p-1)}{2}p^{\frac{n-1}{2}}$, $P_{20}(35)=-\frac{(p-1)}{2}p^{\frac{n-1}{2}}$, $P_{20}(36)=P_{20}(37)=P_{20}(38)=0$\\
			\hline
			4&$P_{20}(41)=1$, $P_{20}(42)=p^{n-r}-1$, $P_{20}(43)=p^{\frac{n-1}{2}}-p^{n-r}$, $P_{20}(44)=-p^{\frac{n-1}{2}}$, $P_{20}(45)=P_{20}(46)=P_{20}(47)=P_{20}(48)=0$\\
			\hline
			5&$P_{20}(51)=1$, $P_{20}(52)=p^{n-r}-1$, $P_{20}(53)=-p^{\frac{n-1}{2}}-p^{n-r}$, $P_{20}(54)=0$, $P_{20}(55)=p^{\frac{n-1}{2}}$, $P_{20}(56)=P_{20}(57)=P_{20}(58)=0$\\
			\hline
			6&$P_{20}(61)=1$, $P_{20}(62)=-1$, $P_{20}(63)=P_{20}(64)=P_{20}(65)=P_{20}(66)=0$, $P_{20}(67)=-\frac{(p-1)}{2}p^{\frac{n-1}{2}}$, $P_{20}(68)=\frac{(p-1)}{2}p^{\frac{n-1}{2}}$\\
			\hline
			7&$P_{20}(71)=1$, $P_{20}(72)=-1$, $P_{20}(73)=P_{20}(74)=P_{20}(75)=0$, $P_{20}(76)=-p^{\frac{n-1}{2}}$, $P_{20}(77)=p^{\frac{n-1}{2}}$, $P_{20}(78)=0$\\
			\hline
			8&$P_{20}(81)=1$, $P_{20}(82)=-1$, $P_{20}(83)=P_{20}(84)=P_{20}(85)=0$, $P_{20}(86)=p^{\frac{n-1}{2}}$, $P_{20}(87)=0$, $P_{20}(88)=-p^{\frac{n-1}{2}}$\\
			\hline			
	\end{tabular}}
	\vspace{-15pt}
\end{table}
	
	Since $p_{u,v}^{w}=p_{v,u}^{w}$ for any $u,v,w\in\{0,1,\dots,7\}$, we only give the values of the intersection numbers $p_{u,v}^w$ for $u\le v$ in the following eight cases. The Krein parameter $q_{u,v}^w$ is the same as the intersection number $p_{u,v}^w$ for any $u,v,w\in\{0,1,\dots,7\}$.
	\begin{itemize}
		\item[$(\rm{\romannumeral1})$] $w=0$
		\begin{itemize}
			\item [$\bullet$] $u=v$\\
			$p_{0,0}^0=1$, $p_{1,1}^0=p^{n-r}-1$, $p_{2,2}^0=p^{r-1}-p^{n-r}$, $p_{3,3}^0=\frac{(p-1)}{2}(p^{r-1}+p^{\frac{n-1}{2}})$, $p_{4,4}^0=\frac{(p-1)}{2}(p^{r-1}-p^{\frac{n-1}{2}})$, $p_{5,5}^0=p^{n-1}-p^{r-1}$ and  $p_{6,6}^0=p_{7,7}^0=\frac{(p-1)}{2}(p^{n-1}-p^{r-1})$.
			\item[$\bullet$] $u\ne v$, $p_{u,v}^0=0$ for $0\le u<v\le 7$.
		\end{itemize}
			\item[$(\rm{\romannumeral2})$] $w=1$
		\begin{itemize}
			\item [$\bullet$] $p_{0,1}^1=1$ and $p_{0,v}^1=0$ for $0\le v\le 7$ and $v\ne 1$.
			\item[$\bullet$] $p_{1,1}^1=p^{n-r}-2$ and $p_{1,v}^1=0$ for $2\le v\le 7$.
			\item[$\bullet$] $p_{2,2}^1=p^{r-1}-p^{n-r}$ and $p_{2,v}^1=0$ for $3\le v\le 7$.
			\item[$\bullet$] $p_{3,3}^1=\frac{(p-1)}{2}(p^{r-1}+p^{\frac{n-1}{2}})$ and $p_{3,v}^1=0$ for $4\le v\le 7$.
			\item[$\bullet$] $p_{4,4}^1=\frac{(p-1)}{2}(p^{r-1}-p^{\frac{n-1}{2}})$ and $p_{4,v}^1=0$ for $5\le v\le 7$.
			\item[$\bullet$] $p_{5,5}^1=p^{n-2}-p^{r-1}$ and $p_{5,6}^1=p_{5,7}^1=\frac{(p-1)}{2}p^{n-2}$.
			\item[$\bullet$] $p_{6,6}^1=\frac{(p-1)}{4}(p^{n-1}-2p^{r-1}-p^{n-2})$ and $p_{6,7}^1=\frac{(p-1)^2}{4}p^{n-2}$.
			\item[$\bullet$] $p_{7,7}^1=\frac{(p-1)}{4}(p^{n-1}-p^{n-2}-2p^{r-1})$.
		\end{itemize}
			\item[$(\rm{\romannumeral3})$] $w=2$
			\begin{itemize}
				\item [$\bullet$] $p_{0,2}^2=1$ and $p_{0,v}^2=0$ for $0\le v\le 7$ and $v\ne 2$.
				\item[$\bullet$] $p_{1,2}^2=p^{n-r}-1$ and $p_{1,v}^2=0$ for $1\le v\le 7$ and $v\ne 2$.
				\item[$\bullet$] $p_{2,2}^2=p^{r-2}-2p^{n-r}$, $p_{2,3}^2=p_{2,4}^2=\frac{(p-1)}{2}p^{r-2}$ and $p_{2,v}^2=0$ for $5\le v\le 7$.
				\item[$\bullet$] $p_{3,3}^2=\frac{(p-1)}{4}(p^{r-1}-p^{r-2}+2p^{\frac{n-1}{2}})$, $p_{3,4}^2=\frac{(p-1)^2}{4}p^{r-2}$ and $p_{3,v}^2=0$ for $5\le v\le 7$.
				\item[$\bullet$] $p_{4,4}^2=\frac{(p-1)}{4}(p^{r-1}-p^{r-2}-2p^{\frac{n-1}{2}})$ and $p_{4,v}^2=0$ for $5\le v\le 7$.
				\item[$\bullet$] $p_{5,5}^2=p^{n-2}-p^{r-2}$ and $p_{5,6}^2=p_{5,7}^2=\frac{(p-1)}{2}(p^{n-2}-p^{r-2})$.
				\item[$\bullet$] $p_{6,6}^2=p_{6,7}^2=\frac{(p-1)^2}{4}(p^{n-2}-p^{r-2})$.
				\item[$\bullet$] $p_{7,7}^2=\frac{(p-1)^2}{4}(p^{n-2}-p^{r-2})$.
			\end{itemize}
			\item[$(\rm{\romannumeral4})$] $w=3$	
			\begin{itemize}
				\item [$\bullet$] $p_{0,3}^3=1$ and $p_{0,v}^3=0$ for $0\le v\le 7$ and $v\ne 3$.
				\item[$\bullet$] $p_{1,3}^3=p^{n-r}-1$ and $p_{1,v}^3=0$ for $1\le v\le 7$ and $v\ne 3$.
				\item[$\bullet$] $p_{2,2}^3=p^{r-2}-p^{\frac{n-3}{2}}$, $p_{2,3}^3=\frac{1}{2}(p^{r-1}+p^{\frac{n-3}{2}}+p^{\frac{n-1}{2}}-2p^{n-r}-p^{r-2})$, $p_{2,4}^3=\frac{(p-1)}{2}(p^{r-2}-p^{\frac{n-3}{2}})$ and $p_{2,v}^3=0$ for $5\le v\le 7$.
				\item[$\bullet$] $p_{3,3}^3=\frac{1}{4}(p^r+p^{r-2}+3p^{\frac{n+1}{2}}-2p^{r-1}-p^{\frac{n-3}{2}}-6p^{\frac{n-1}{2}})$, $p_{3,4}^3=\frac{(p-1)^2}{4}(p^{r-2}-p^{\frac{n-3}{2}})$ and $p_{3,v}^3=0$ for $5\le v\le 7$.
				\item[$\bullet$] $p_{4,4}^3=\frac{(p-1)^2}{4}(p^{r-2}-p^{\frac{n-3}{2}})$ and $p_{4,v}^3=0$ for $5\le v\le 7$.
				\item[$\bullet$] $p_{5,5}^3=p^{n-2}-p^{r-2}$ and $p_{5,6}^3=p_{5,7}^3=\frac{(p-1)}{2}(p^{n-2}-p^{r-2})$.
				\item[$\bullet$] $p_{6,6}^3=p_{6,7}^3=\frac{(p-1)^2}{4}(p^{n-2}-p^{r-2})$.
				\item[$\bullet$] $p_{7,7}^3=\frac{(p-1)^2}{4}
				(p^{n-2}-p^{r-2})$.
			\end{itemize}
			\item[$(\rm{\romannumeral5})$] $w=4$
			\begin{itemize}
				\item [$\bullet$] $p_{0,4}^4=1$ and $p_{0,v}^4=0$ for $0\le v\le 7$ and $v\ne 4$.
				\item[$\bullet$] $p_{1,4}^4=p^{n-r}-1$ and $p_{1,v}^4=0$ for $1\le v\le 7$ and $v\ne 4$.
				\item[$\bullet$] $p_{2,2}
				^4=p^{r-2}+p^{\frac{n-3}{2}}$, $p_{2,3}^4=\frac{(p-1)}{2}(p^{r-2}+p^{\frac{n-3}{2}})$, $p_{2,4}^4=\frac{1}{2}(p^{r-1}-p^{\frac{n-3}{2}}-p^{\frac{n-1}{2}}-2p^{n-r}-p^{r-2})$ and $p_{2,v}^4=0$ for $5\le v\le 7$.
			\item[$\bullet$] $p_{3,3}^4=p_{3,4}^4=\frac{(p-1)^2}{4}(p^{r-2}+p^{\frac{n-3}{2}})$ and $p_{3,v}^4=0$ for $5\le v\le 7$.
			\item[$\bullet$] $p_{4,4}^4=\frac{1}{4}(p^{r-2}-2p^{r-1}+p^r+p^{\frac{n-3}{2}}+6p^{\frac{n-1}{2}}-3p^{\frac{n+1}{2}})$ and $p_{4,v}^4=0$ for $5\le v\le 7$.
			\item[$\bullet$] $p_{5,5}^4=p^{n-2}-p^{r-2}$ and $p_{5,6}^4=p_{5,7}^4=\frac{(p-1)}{2}(p^{n-2}-p^{r-2})$.
			\item[$\bullet$] $p_{6,6}^4=p_{6,7}^4=\frac{(p-1)^2}{4}(p^{n-2}-p^{r-2})$.
			\item[$\bullet$] $p_{7,7}^4=\frac{(p-1)^2}{4}(p^{n-2}-p^{r-2})$.
	\end{itemize}
		\item[$(\rm{\romannumeral6})$] $w=5$
		\begin{itemize}
			\item [$\bullet$] $p_{0,5}^5=1$ and $p_{0,v}^5=0$ for $0\le v\le 7$ and $v\ne 5$.
			\item[$\bullet$] $p_{1,v}^5=0$ for $1\le v\le 4$, $p_{1,5}^5=p^{n-r-1}-1$ and  $p_{1,6}^5=p_{1,7}^5=\frac{(p-1)}{2}p^{n-r-1}$.
			\item[$\bullet$] $p_{2,v}^5=0$ for $2\le v\le 4$, $p_{2,5}^2=p^{r-2}-p^{n-r-1}$ and $p_{2,6}^5=p_{2,7}^5=\frac{(p-1)}{2}(p^{r-2}-p^{n-r-1})$.
			\item[$\bullet$] $p_{3,3}^5=p_{3,4}
		^5=0$, $p_{3,5}^5=\frac{(p-1)}{2}(p^{r-2}+p^{\frac{n-3}{2}})$ and $p_{3,6}^5=p_{3,7}^5=\frac{(p-1)^2}{4}(p^{r-2}+p^{\frac{n-3}{2}})$.
		\item[$\bullet$] $p_{4,4}^5=0$, $p_{4,5}
				^5=\frac{(p-1)}{2}(p^{r-2}-p^{\frac{n-3}{2}})$ and $p_{4,6}^5=p_{4,7}^5=\frac{(p-1)^2}{4}(p^{r-2}-p^{\frac{n-3}{2}})$.
				\item[$\bullet$] $p_{5,5}^5=p^{n-2}-2p^{r-2}$, $p_{5,6}^5=\frac{(p-1)}{2}(p^{n-2}-2p^{r-2}+p^{\frac{n-3}{2}})$ and $p_{5,7}^5=\frac{(p-1)}{2}(p^{n-2}-2p^{r-2}-p^{\frac{n-3}{2}})$.
				\item[$\bullet$] $p_{6,6}^5=\frac{(p-1)^2}{4}(p^{n-2}-2p^{r-2})-\frac{(p-1)}{2}p^{\frac{n-3}{2}}$ and $p_{6,7}^5=\frac{(p-1)^2}{4}(p^{n-2}-2p^{r-2})$.
				\item[$\bullet$] $p_{7,7}^5=\frac{(p-1)^2}{4}(p^{n-2}-2p^{r-2})+\frac{(p-1)}{2}p^{\frac{n-3}{2}}$.					
		\end{itemize}		
		\item[$(\rm{\romannumeral7})$] $w=6$	
		\begin{itemize}
			\item [$\bullet$] $p_{0,6}^6=1$ and $p_{0,v}^6=0$ for $0\le v\le 7$ and $v\ne 6$.
			\item[$\bullet$] $p_{1,v}^6=0$ for $1\le v\le 4$, $p_{1,5}^6=p^{n-r-1}$, $p_{1,6}^6=\frac{1}{2}(p^{n-r}-p^{n-r-1}-2)$ and $p_{1,7}^6=\frac{(p-1)}{2}p^{n-r-1}$.
			\item[$\bullet$] $p_{2,v}^6=0$ for $2\le v\le 4$, $p_{2,5}^6=p^{r-2}-p^{n-r-1}$ and  $p_{2,6}^6=p_{2,7}^6=\frac{(p-1)}{2}(p^{r-2}-p^{n-r-1})$.
			\item[$\bullet$] $p_{3,3}^6=p_{3,4}^6=0$, $p_{3,5}^6=\frac{(p-1)}{2}(p^{r-2}+p^{\frac{n-3}{2}})$ and $p_{3,6}^6=p_{3,7}^6=\frac{(p-1)^2}{4}(p^{r-2}+p^{\frac{n-3}{2}})$.
			\item[$\bullet$] $p_{4,4}^6=0$, $p_{4,5}^6=\frac{(p-1)}{2}(p^{r-2}-p^{\frac{n-3}{2}})$ and $p_{4,6}^6=p_{4,7}^6=\frac{(p-1)^2}{4}(p^{r-2}-p^{\frac{n-3}{2}})$.
			\item[$\bullet$] $p_{5,5}^6=p^{n-2}-2p^{r-2}+p^{\frac{n-3}{2}}$, $p_{5,6}^6=\frac{(p-1)}{2}(p^{n-2}-2p^{r-2})-p^{\frac{n-3}{2}}$ and $p_{5,7}^6=\frac{(p-1)}{2}(p^{n-2}-2p^{r-2})$.
			\item[$\bullet$] $p_{6,6}^6=\frac{(p-1)^2}{4}(p^{n-2}-2p^{r-2})-\frac{p^2-2p-3}{4}p^{\frac{n-3}{2}}$ and $p_{6,7}^6=\frac{(p-1)^2}{4}(p^{n-2}-2p^{r-2}+p^{\frac{n-3}{2}})$.
			\item[$\bullet$] $p_{7,7}^6=\frac{(p-1)^2}{4}(p^{n-2}-2p^{r-2}-p^{\frac{n-3}{2}})$.
		\end{itemize}	
	\item[$(\rm{\romannumeral8})$] $w=7$
	\begin{itemize}
		\item [$\bullet$] $p_{0,7}^7=1$ and $p_{0,v}^7=0$ for $0\le v\le 6$.
		\item[$\bullet$] $p_{1,v}^7=0$ for $1\le v\le 4$, $p_{1,5}^7=p^{n-r-1}$, $p_{1,6}^7=\frac{(p-1)}{2}p^{n-r-1}$ and $p_{1,7}^7=\frac{1}{2}(p^{n-r}-p^{n-r-1}-2)$.
		\item[$\bullet$] $p_{2,v}^7=0$ for $2\le v\le 4$, $p_{2,5}^7=p^{r-2}-p^{n-r-1}$ and $p_{2,6}^7=p_{2,7}^7=\frac{(p-1)}{2}(p^{r-2}-p^{n-r-1})$.
			\item[$\bullet$] $p_{3,3}^7=p_{3,4}^7=0$, $p_{3,5}^7=\frac{(p-1)}{2}(p^{r-2}+p^{\frac{n-3}{2}})$ and $p_{3,6}^7=p_{3,7}^7=\frac{(p-1)^2}{4}(p^{r-2}+p^{\frac{n-3}{2}})$.
			\item[$\bullet$] $p_{4,4}^7=0$, $p_{4,5}^7=\frac{(p-1)}{2}(p^{r-2}-p^{\frac{n-3}{2}})$ and $p_{4,6}^7=p_{4,7}^7=\frac{(p-1)^2}{4}(p^{r-2}-p^{\frac{n-3}{2}})$.
			\item[$\bullet$] $p_{5,5}^7=p^{n-2}-2p^{r-2}-p^{\frac{n-3}{2}}$, $p_{5,6}^7=\frac{(p-1)}{2}(p^{n-2}-2p^{r-2})$ and $p_{5,7}^7=p^{\frac{n-3}{2}}+\frac{(p-1)}{2}(p^{n-2}-2p^{r-2})$.
				\item[$\bullet$] $p_{6,6}^7=\frac{(p-1)^2}{4}(p^{n-2}-2p^{r-2}+p^{\frac{n-3}{2}})$, $p_{6,7}^7=\frac{(p-1)^2}{4}(p^{n-2}-2p^{r-2}-p^{\frac{n-3}{2}})$.
				\item[$\bullet$] $p_{7,7}^7=\frac{(p-1)^2}{4}(p^{n-2}-2p^{r-2})+\frac{p^2-2p-3}{4}p^{\frac{n-3}{2}}$.
						\end{itemize}
		\end{itemize}
21. The first and second eigenmatrices, the intersection numbers and the  Krein parameters of the association scheme induced by $U_{21}$.

The first and second eigenmatrices  of the association scheme induced by $U_{21}$ are given in Tables 33 and 34.
\begin{table}[h]
	\vspace{-15pt}
	\centering
	\caption{The first eigenmatrix of the association scheme induced by $U_{21}$}
	\renewcommand\arraystretch{1.25}	
	\resizebox{\textwidth}{!}{
		\begin{tabular}{|c|m{11cm}|}
			\hline
			$i$&\hspace{4.25cm} $P_{21}(ij)$\\
			\hline
			1&$P_{21}(11)=1$, $P_{21}(12)=p^{n-r}-1$, $P_{21}(13)=p^{r-1}-p^{n-r}$, $P_{21}(14)=\frac{(p-1)}{2}(p^{r-1}-p^{\frac{n-1}{2}})$, $P_{21}(15)=\frac{(p-1)}{2}(p^{r-1}+p^{\frac{n-1}{2}})$, $P_{21}(16)=p^{n-1}-p^{r-1}$, $P_{21}(17)=P_{21}(18)=\frac{(p-1)}{2}(p^{n-1}-p^{r-1})$\\
			\hline
			2&$P_{21}(21)=1$, $P_{21}(22)=-1$, $P_{21}(23)=P_{21}(24)=P_{21}(25)=P_{21}(26)=0$, $P_{21}(27)=\frac{(p-1)}{2}p^{\frac{n-1}{2}}$, $P_{21}(28)=-\frac{(p-1)}{2}p^{\frac{n-1}{2}}$\\
			\hline
			3&$P_{21}(31)=1$, $P_{21}(32)=-1$, $P_{21}(33)=P_{21}(34)=P_{21}(35)=0$, $P_{21}(36)=-p^{\frac{n-1}{2}}$, $P_{21}(37)=0$, $P_{21}(38)=p^{\frac{n-1}{2}}$\\
			\hline
			4&$P_{21}(41)=1$, $P_{21}(42)=-1$, $P_{21}(43)=P_{21}(44)=P_{21}(45)=0$, $P_{21}(46)=p^{\frac{n-1}{2}}$, $P_{21}(47)=-p^{\frac{n-1}{2}}$, $P_{21}(48)=0$\\
			\hline
			5&$P_{21}(51)=1$, $P_{21}(52)=p^{n-r}-1$, $P_{21}(53)=p^{r-1}-p^{n-r}$, $P_{21}(54)=\frac{(p-1)}{2}(p^{r-1}-p^{\frac{n-1}{2}})$, $P_{21}(55)=\frac{(p-1)}{2}(p^{r-1}+p^{\frac{n-1}{2}})$, $P_{21}(56)=-p^{r-1}$, $P_{21}(57)=P_{21}(58)=-\frac{(p-1)}{2}p^{r-1}$\\
			\hline
			6&$P_{21}(61)=1$, $P_{21}(62)=p^{n-r}-1$, $P_{21}(63)=-p^{n-r}$, $P_{21}(64)=-\frac{(p-1)}{2}p^{\frac{n-1}{2}}$, $P_{21}(65)=\frac{(p-1)}{2}p^{\frac{n-1}{2}}$, $P_{21}(66)=P_{21}(67)=P_{21}(68)=0$\\
			\hline
			7&$P_{21}(71)=1$, $P_{21}(72)=p^{n-r}-1$, $P_{21}(73)=p^{\frac{n-1}{2}}-p^{n-r}$, $P_{21}(74)=0$, $P_{21}(75)=-p^{\frac{n-1}{2}}$, $P_{21}(76)=P_{21}(77)=P_{21}(78)=0$\\
			\hline
			8&$P_{21}(81)=1$, $P_{21}(82)=p^{n-r}-1$, $P_{21}(83)=-p^{\frac{n-1}{2}}-p^{n-r}$, $P_{21}(84)=p^{\frac{n-1}{2}}$, $P_{21}(85)=P_{21}(86)=P_{21}(87)=P_{21}(88)=0$\\
			\hline
	\end{tabular}}
	\vspace{-15pt}
\end{table}
\begin{table}[h]
	\vspace{0pt}
	\centering
	\caption{The second eigenmatrix of the association scheme induced by $U_{21}$}
	\renewcommand\arraystretch{1.5}	
	\resizebox{\textwidth}{!}{
		\begin{tabular}{|c|m{11cm}|}
			\hline
			$i$&\hspace{4.25cm} $Q_{21}(ij)$\\
			\hline
			1&$Q_{21}(11)=1$, $Q_{21}(12)=p^{n-1}-p^{r-1}$, $Q_{21}(13)=Q_{21}(14)=\frac{(p-1)}{2}(p^{n-1}-p^{r-1})$, $Q_{21}(15)=p^{n-r}-1$, $Q_{21}(16)=p^{r-1}-p^{n-r}$, $Q_{21}(17)=\frac{(p-1)}{2}(p^{r-1}+p^{\frac{n-1}{2}})$, $Q_{21}(18)=\frac{(p-1)}{2}(p^{r-1}-p^{\frac{n-1}{2}})$\\
			\hline
			2&$Q_{21}(21)=1$, $Q_{21}(22)=-p^{r-1}$, $Q_{21}(23)=Q_{21}(24)=-\frac{(p-1)}{2}p^{r-1}$, $Q_{21}(25)=p^{n-r}-1$, $Q_{21}(26)=p^{r-1}-p^{n-r}$, $Q_{21}(27)=\frac{(p-1)}{2}(p^{r-1}+p^{\frac{n-1}{2}})$, $Q_{21}(28)=\frac{(p-1)}{2}(p^{r-1}-p^{\frac{n-1}{2}})$\\
			\hline
			3&$Q_{21}(31)=1$, $Q_{21}(32)=Q_{21}(33)=Q_{21}(34)=0$, $Q_{21}(35)=p^{n-r}-1$, $Q_{21}(36)=-p^{n-r}$, $Q_{21}(37)=\frac{(p-1)}{2}p^{\frac{n-1}{2}}$, $Q_{21}(38)=-\frac{(p-1)}{2}p^{\frac{n-1}{2}}$\\
			\hline
			4&$Q_{21}(41)=1$, $Q_{21}(42)=Q_{21}(43)=Q_{21}(44)=0$, $Q_{21}(45)=p^{n-r}-1$, $Q_{21}(46)=-p^{\frac{n-1}{2}}-p^{n-r}$, $Q_{21}(47)=0$, $Q_{21}(48)=p^{\frac{n-1}{2}}$\\
			\hline
			5&$Q_{21}(51)=1$, $Q_{21}(52)=Q_{21}(53)=Q_{21}(54)=0$, $Q_{21}(55)=p^{n-r}-1$, $Q_{21}(56)=p^{\frac{n-1}{2}}-p^{n-r}$, $Q_{21}(57)=-p^{\frac{n-1}{2}}$, $Q_{21}(58)=0$\\
			\hline
			6&$Q_{21}(61)=1$, $Q_{21}(62)=0$, $Q_{21}(63)=-\frac{(p-1)}{2}p^{\frac{n-1}{2}}$, $Q_{21}(64)=\frac{(p-1)}{2}p^{\frac{n-1}{2}}$, $Q_{21}(65)=-1$, $Q_{21}(66)=Q_{21}(67)=Q_{21}(68)=0$\\
			\hline
			7&$Q_{21}(71)=1$, $Q_{21}(72)=p^{\frac{n-1}{2}}$, $Q_{21}(73)=0$, $Q_{21}(74)=-p^{\frac{n-1}{2}}$, $Q_{21}(75)=-1$, $Q_{21}(76)=Q_{21}(77)=Q_{21}(78)=0$\\
			\hline
			8&$Q_{21}(81)=1$, $Q_{21}(82)=-p^{\frac{n-1}{2}}$, $Q_{21}(83)=p^{\frac{n-1}{2}}$, $Q_{21}(84)=0$, $Q_{21}(85)=-1$, $Q_{21}(86)=Q_{21}(87)=Q_{21}(88)=0$\\
			\hline	
		\end{tabular}
	}
	\vspace{-15pt}
\end{table}

Since $p_{u,v}^{w}=p_{v,u}^{w}$ for any $u,v,w\in\{0,1,\dots,7\}$, we only give the values of  the intersection numbers $p_{u,v}^w$ for $u\le v$ in the following eight cases.
\begin{itemize}
	\item [$(\rm{\romannumeral 1})$] $w=0$
	\begin{itemize}
		\item [$\bullet$] $u=v$\\
		$p_{0,0}^0=1$, $p_{1,1}^0=p^{n-r}-1$, $p_{2,2}^0=p^{r-1}-p^{n-r}$, $p_{3,3}^0=\frac{(p-1)}{2}(p^{r-1}-p^{\frac{n-1}{2}})$, $p_{4,4}^0=\frac{(p-1)}{2}(p^{r-1}+p^{\frac{n-1}{2}})$, $p_{5,5}^0=p^{n-1}-p^{r-1}$ and $p_{6,6}^0=p_{7,7}^0=\frac{(p-1)}{2}(p^{n-1}-p^{r-1})$.
		\item[$\bullet$] $u\ne v$, $p_{u,v}^0=0$ for $0\le u<v\le 7$.
	\end{itemize}
	\item [$(\rm{\romannumeral 2})$] $w=1$
	\begin{itemize}
		\item [$\bullet$] $p_{0,1}^1=1$ and $p_{0,v}^1=0$ for $0\le v\le 7$ and $v\ne 1$.
		\item[$\bullet$] $p_{1,1}^1=p^{n-r}-2$ and $p_{1,v}^1=0$ for $2\le v\le 7$.
		\item[$\bullet$] $p_{2,2}^1=p^{r-1}-p^{n-r}$ and $p_{2,v}^1=0$ for $3\le v\le 7$.
		\item[$\bullet$] $p_{3,3}^1=\frac{(p-1)}{2}(p^{r-1}-p^{\frac{n-1}{2}})$ and $p_{3,v}^1=0$ for $4\le v\le 7$.
		\item[$\bullet$] $p_{4,4}^1=\frac{(p-1)}{2}(p^{r-1}+p^{\frac{n-1}{2}})$ and $p_{4,v}^1=0$ for $5\le v\le 7$.
		\item[$\bullet$] $p_{5,5}^1=p^{n-2}-p^{r-1}$ and $p_{5,6}^1=p_{5,7}^1=\frac{(p-1)}{2}p^{n-2}$.
		\item[$\bullet$] $p_{6,6}^1=\frac{(p-1)}{4}(p^{n-1}-p^{n-2}-2p^{r-1})$ and $p_{6,7}^1=\frac{(p-1)^2}{4}p^{n-2}$.
		\item[$\bullet$] $p_{7,7}^1=\frac{(p-1)}{4}(p^{n-1}-p^{n-2}-2p^{r-1})$.	
	\end{itemize}
\item [$(\rm{\romannumeral 3})$] $w=2$
\begin{itemize}
	\item [$\bullet$] $p_{0,2}^2=1$ and $p_{0,v}^2=0$ for $0\le v\le 7$ and $v\ne 2$.
	\item[$\bullet$] $p_{1,2}^2=p^{n-r}-1$ and $p_{1,v}^2=0$ for $1\le v\le 7$ and $v\ne 2$.
	\item[$\bullet$] $p_{2,2}^2=p^{r-2}-2p^{n-r}$,  $p_{2,3}^2=p_{2,4}^2=\frac{(p-1)}{2}p^{r-2}$ and $p_{2,v}^2=0$ for $5\le v\le 7$.
	\item[$\bullet$] $p_{3,3}^2=\frac{(p-1)}{4}(p^{r-1}-p^{r-2}-2p^{\frac{n-1}{2}})$, $p_{3,4}^2=\frac{(p-1)^2}{4}p^{r-2}$ and $p_{3,v}^2=0$ for $5\le v\le 7$.
	\item[$\bullet$] $p_{4,4}^2=\frac{(p-1)}{4}(p^{r-1}-p^{r-2}+2p^{\frac{n-1}{2}})$ and $p_{4,v}^2=0$ for $5\le v\le 7$.
	\item[$\bullet$] $p_{5,5}^2=p^{n-2}-p^{r-2}$ and $p_{5,6}^2=p_{5,7}^2=\frac{(p-1)}{2}(p^{n-2}-p^{r-2})$.
	\item[$\bullet$] $p_{6,6}^2=p_{6,7}^2=\frac{(p-1)^2}{4}(p^{n-2}-p^{r-2})$.
	\item[$\bullet$] $p_{7,7}^2=\frac{(p-1)^2}{4}(p^{n-2}-p^{r-2})$.
\end{itemize}
	\item [$(\rm{\romannumeral 4})$] $w=3$
	\begin{itemize}
		\item [$\bullet$] $p_{0,3}^3=1$ and $p_{0,v}^3=0$ for $0\le v\le 7$ and $v\ne 3$.
		\item[$\bullet$] $p_{1,3}^3=p^{n-r}-1$ and $p_{1,v}^3=0$ for $1\le v\le 7$ and $v\ne 3$.
		\newpage
		\item[$\bullet$] $p_{2,2}^3=p^{r-2}+p^{\frac{n-3}{2}}$, $p_{2,3}^3=\frac{1}{2}(p^{r-1}-p^{\frac{n-3}{2}}-p^{\frac{n-1}{2}}-2p^{n-r}-p^{r-2})$, $p_{2,4}^3=\frac{(p-1)}{2}(p^{r-2}+p^{\frac{n-3}{2}})$ and $p_{2,v}^3=0$ for $5\le v\le 7$.
		\item[$\bullet$] $p_{3,3}^3=\frac{1}{4}(p^{r-2}-2p^{r-1}+p^{r}+p^{\frac{n-3}{2}}+6p^{\frac{n-1}{2}}-3p^{\frac{n+1}{2}})$, $p_{3,4}^3=\frac{(p-1)^2}{4}(p^{r-2}+p^{\frac{n-3}{2}})$ and $p_{3,v}^3=0$ for $5\le v\le 7$.
		\item[$\bullet$] $p_{4,4}^3=\frac{(p-1)^2}{4}(p^{r-2}+p^{\frac{n-3}{2}})$ and $p_{4,v}^3=0$ for $5\le v\le 7$.
		\item[$\bullet$] $p_{5,5}^3=p^{n-2}-p^{r-2}$ and $p_{5,6}^3=p_{5,7}^3=\frac{(p-1)}{2}(p^{n-2}-p^{r-2})$.
		\item[$\bullet$] $p_{6,6}^3=p_{6,7}^3=\frac{(p-1)^2}{4}(p^{n-2}-p^{r-2})$.
		\item[$\bullet$] $p_{7,7}^3=\frac{(p-1)^2}{4}(p^{n-2}-p^{r-2})$.
	\end{itemize}
	\item [$(\rm{\romannumeral 5})$] $w=4$
	\begin{itemize}
		\item [$\bullet$] $p_{0,4}^4=1$ and $p_{0,v}^4=0$ for $0\le v\le 7$ and $v\ne 4$.
		\item[$\bullet$] $p_{1,4}^4=p^{n-r}-1$ and $p_{1,v}^4=0$ for $1\le v\le 7$ and $v\ne 4$.
		\item[$\bullet$] $p_{2,2}^4=p^{r-2}-p^{\frac{n-3}{2}}$, $p_{2,3}^4=\frac{(p-1)}{2}(p^{r-2}-p^{\frac{n-3}{2}})$, $p_{2,4}^4=\frac{1}{2}(p^{r-1}+p^{\frac{n-3}{2}}+p^{\frac{n-1}{2}}-2p^{n-r}-p^{r-2})$ and $p_{2,v}^4=0$ for $5\le v\le 7$.
		\item[$\bullet$] $p_{3,3}^4=p_{3,4}^4=\frac{(p-1)^2}{4}(p^{r-2}-p^{\frac{n-3}{2}})$ and $p_{3,v}^4=0$ for $5\le v\le 7$.
		\item[$\bullet$] $p_{4,4}^4=\frac{1}{4}(p^{r-2}+p^r+3p^{\frac{n+1}{2}}-2p^{r-1}-p^{\frac{n-3}{2}}-6p^{\frac{n-1}{2}})$ and $p_{4,v}^4=0$ for $5\le v\le 7$.
		\item[$\bullet$] $p_{5,5}^4=p^{n-2}-p^{r-2}$ and $p_{5,6}^4=p_{5,7}^4=\frac{(p-1)}{2}(p^{n-2}-p^{r-2})$.
		\item[$\bullet$] $p_{6,6}^4=p_{6,7}^4=\frac{(p-1)^2}{4}(p^{n-2}-p^{r-2})$.
		\item[$\bullet$] $p_{7,7}^4=\frac{(p-1)^2}{4}(p^{n-2}-p^{r-2})$.
	\end{itemize}
\item [$(\rm{\romannumeral 6})$] $w=5$
\begin{itemize}
	\item [$\bullet$] $p_{0,5}^5=1$ and $p_{0,v}^5=0$ for $0\le v\le 7$ and $v\ne 5$.
	\item[$\bullet$] $p_{1,v}^5=0$ for $1\le v\le 4$, $p_{1,5}^5=p^{n-r-1}-1$ and $p_{1,6}^5=p_{1,7}^5=\frac{(p-1)}{2}p^{n-r-1}$.
	\item[$\bullet$] $p_{2,v}^5=0$ for $2\le v\le 4$, $p_{2,5}^5=p^{r-2}-p^{n-r-1}$ and $p_{2,6}^5=p_{2,7}^5=\frac{(p-1)}{2}(p^{r-2}-p^{n-r-1})$.
	\item[$\bullet$] $p_{3,3}^5=p_{3,4}^5=0$, $p_{3,5}^5=\frac{(p-1)}{2}(p^{r-2}-p^{\frac{n-3}{2}})$ and $p_{3,6}^5=p_{3,7}^5=\frac{(p-1)^2}{4}(p^{r-2}-p^{\frac{n-3}{2}})$.
	\item[$\bullet$] $p_{4,4}^5=0$, $p_{4,5}^5=\frac{(p-1)}{2}(p^{r-2}+p^{\frac{n-3}{2}})$ and $p_{4,6}^5=p_{4,7}^5=\frac{(p-1)^2}{4}(p^{r-2}+p^{\frac{n-3}{2}})$.
	\item[$\bullet$] $p_{5,5}^5=p^{n-2}-2p^{r-2}$, $p_{5,6}^5=\frac{(p-1)}{2}(p^{n-2}-2p^{r-2}-p^{\frac{n-3}{2}})$ and $p_{5,7}^5=\frac{(p-1)}{2}(p^{n-2}-2p^{r-2}+p^{\frac{n-3}{2}})$.
	\item[$\bullet$] $p_{6,6}^5=\frac{(p-1)}{2}p^{\frac{n-3}{2}}+\frac{(p-1)^2}{4}(p^{n-2}-2p^{r-2})$ and $p_{6,7}^5=\frac{(p-1)^2}{4}(p^{n-2}-2p^{r-2})$.
		\item[$\bullet$] $p_{7,7}^5=\frac{(p-1)^2}{4}(p^{n-2}-2p^{r-2})-\frac{(p-1)}{2}p^{\frac{n-3}{2}}$.
\end{itemize}
\item [$(\rm{\romannumeral 7})$] $w=6$
\begin{itemize}
	\item [$\bullet$] $p_{0,6}^6=1$ and $p_{0,v}^6=0$ for $0\le v\le 7$ and $v\ne 6$.
	\item[$\bullet$] $p_{1,v}^6=0$ for $1\le v\le 4$, $p_{1,5}^6=p^{n-r-1}$, $p_{1,6}^6=\frac{1}{2}(p^{n-r}-p^{n-r-1}-2)$ and $p_{1,7}^6=\frac{(p-1)}{2}p^{n-r-1}$.
	\item[$\bullet$] $p_{2,v}^6=0$ for $2\le v\le 4$, $p_{2,5}^6=p^{r-2}-p^{n-r-1}$ and $p_{2,6}^6=p_{2,7}^6=\frac{(p-1)}{2}(p^{r-2}-p^{n-r-1})$.
	\item[$\bullet$] $p_{3,3}^6=p_{3,4}^6=0$, $p_{3,5}^6=\frac{(p-1)}{2}(p^{r-2}-p^{\frac{n-3}{2}})$ and $p_{3,6}^6=p_{3,7}^6=\frac{(p-1)^2}{4}(p^{r-2}-p^{\frac{n-3}{2}})$.
	\item[$\bullet$] $p_{4,4}^6=0$, $p_{4,5}^6=\frac{(p-1)}{2}(p^{r-2}+p^{\frac{n-3}{2}})$ and $p_{4,6}^6=p_{4,7}^6=\frac{(p-1)^2}{4}(p^{r-2}+p^{\frac{n-3}{2}})$.
		\item[$\bullet$] $p_{5,5}^6=p^{n-2}-2p^{r-2}-p^{\frac{n-3}{2}}$, $p_{5,6}^6=p^{\frac{n-3}{2}}+\frac{(p-1)}{2}(p^{n-2}-2p^{r-2})$ and $p_{5,7}^6=\frac{(p-1)}{2}(p^{n-2}-2p^{r-2})$
		\item[$\bullet$] $p_{6,6}^6=\frac{p^2-2p-3}{4}p^{\frac{n-3}{2}}+\frac{(p-1)^2}{4}(p^{n-2}-2p^{r-2})$ and $p_{6,7}^6=\frac{(p-1)^2}{4}(p^{n-2}-2p^{r-2}-p^{\frac{n-3}{2}})$.
		\item[$\bullet$] $p_{7,7}^6=\frac{(p-1)^2}{4}(p^{n-2}-2p^{r-2}+p^{\frac{n-3}{2}})$.
\end{itemize}
\item [$(\rm{\romannumeral 8})$] $w=7$
\begin{itemize}
	\item [$\bullet$] $p_{0,7}^7=1$ and $p_{0,v}^7=0$ for $0\le v\le 6$.
	\item[$\bullet$] $p_{1,v}^7=0$ for $1\le v\le 4$, $p_{1,5}^7=p^{n-r-1}$, $p_{1,6}^7=\frac{(p-1)}{2}p^{n-r-1}$ and $p_{1,7}^7=\frac{1}{2}(p^{n-r}-p^{n-r-1}-2)$.
	\item[$\bullet$] $p_{2,v}^7=0$ for $2\le v\le 4$, $p_{2,5}^7=p^{r-2}-p^{n-r-1}$ and $p_{2,6}^7=p_{2,7}^7=\frac{(p-1)}{2}(p^{r-2}-p^{n-r-1})$.
	\item[$\bullet$] $p_{3,3}^7=p_{3,4}^7=0$, $p_{3,5}^7=\frac{(p-1)}{2}(p^{r-2}-p^{\frac{n-3}{2}})$ and $p_{3,6}^7=p_{3,7}^7=\frac{(p-1)^2}{4}(p^{r-2}-p^{\frac{n-3}{2}})$.
	\item[$\bullet$] $p_{4,4}^7=0$, $p_{4,5}^7=\frac{(p-1)}{2}(p^{r-2}+p^{\frac{n-3}{2}})$ and $p_{4,6}^7=p_{4,7}^7=\frac{(p-1)^2}{4}(p^{r-2}+p^{\frac{n-3}{2}})$.
	\item[$\bullet$] $p_{5,5}^7=p^{n-2}-2p^{r-2}+p^{\frac{n-3}{2}}$, $p_{5,6}^7=\frac{(p-1)}{2}(p^{n-2}-2p^{r-2})$ and $p_{5,7}^7=\frac{(p-1)}{2}(p^{n-2}-2p^{r-2})-p^{\frac{n-3}{2}}$.
	\item[$\bullet$] $p_{6,6}^7=\frac{(p-1)^2}{4}(p^{n-2}-2p^{r-2}-p^{\frac{n-3}{2}})$ and $p_{6,7}^7=\frac{(p-1)^2}{4}(p^{n-2}-2p^{r-2}+p^{\frac{n-3}{2}})$.
	\item[$\bullet$] $p_{7,7}^7=\frac{(p-1)^2}{4}(p^{n-2}-2p^{r-2})-\frac{p^2-2p-3}{4}p^{\frac{n-3}{2}}$.
	\end{itemize}
\end{itemize}

Since $q_{u,v}^{w}=q_{v,u}^{w}$ for any $u,v,w\in\{0,1,\dots,7\}$, we only give the values of  the intersection numbers $q_{u,v}^w$ for $u\le v$ in the following eight cases.
\begin{itemize}
	\item [$(\rm{\romannumeral 1})$] $w=0$
	\begin{itemize}
		\item [$\bullet$] $u=v$\\
		$q_{0,0}^0=1$, $q_{1,1}^0=p^{n-1}-p^{r-1}$, $q_{2,2}^0=q_{3,3}^0=\frac{(p-1)}{2}(p^{n-1}-p^{r-1})$, $q_{4,4}^0=p^{n-r}-1$, $q_{5,5}^0=p^{r-1}-p^{n-r}$, $q_{6,6}^0=\frac{(p-1)}{2}(p^{r-1}+p^{\frac{n-1}{2}})$ and $q_{7,7}^0=\frac{(p-1)}{2}(p^{r-1}-p^{\frac{n-1}{2}})$.
		\item[$\bullet$] $u\ne v$, $q_{u,v}^0=0$ for $0\le u<v\le 7$.
	\end{itemize}
	\item [$(\rm{\romannumeral 2})$] $w=1$
	\begin{itemize}
		\item [$\bullet$] $q_{0,1}^1=1$ and $q_{0,v}^1=0$ for $0\le v\le 7$ and $v\ne 1$.
		\item[$\bullet$] $q_{1,1}^1=p^{n-2}-2p^{r-2}$, $q_{1,2}^1=\frac{(p-1)}{2}(p^{n-2}-2p^{r-2}+p^{\frac{n-3}{2}})$, $q_{1,3}^1=\frac{(p-1)}{2}(p^{n-2}-2p^{r-2}-p^{\frac{n-3}{2}})$, $q_{1,4}^1=p^{n-r-1}-1$, $q_{1,5}^1=p^{r-2}-p^{n-r-1}$, $q_{1,6}^1=\frac{(p-1)}{2}(p^{r-2}+p^{\frac{n-3}{2}})$ and $q_{1,7}^1=\frac{(p-1)}{2}(p^{r-2}-p^{\frac{n-3}{2}})$.\\
		\newpage
		\item[$\bullet$] $q_{2,2}^1=\frac{(p-1)^2}{4}(p^{n-2}-2p^{r-2})-\frac{(p-1)}{2}p^{\frac{n-3}{2}}$, $q_{2,3}^1=\frac{(p-1)^2}{4}(p^{n-2}-2p^{r-2})$, $q_{2,4}^1=\frac{(p-1)}{2}p^{n-r-1}$, $q_{2,5}^1=\frac{(p-1)}{2}(p^{r-2}-p^{n-r-1})$, $q_{2,6}^1=\frac{(p-1)^2}{4}(p^{r-2}+p^{\frac{n-3}{2}})$ and $q_{2,7}^1=\frac{(p-1)^2}{4}(p^{r-2}-p^{\frac{n-3}{2}})$.
		\item[$\bullet$] $q_{3,3}^1=\frac{(p-1)}{2}p^{\frac{n-3}{2}}+\frac{(p-1)^2}{4}(p^{n-2}-2p^{r-2})$, $q_{3,4}^1=\frac{(p-1)}{2}p^{n-r-1}$, $q_{3,5}^1=\frac{(p-1)}{2}(p^{r-2}-p^{n-r-1})$, $q_{3,6}^1=\frac{(p-1)^2}{4}(p^{r-2}+p^{\frac{n-3}{2}})$ and $q_{3,7}^1=\frac{(p-1)^2}{4}(p^{r-2}-p^{\frac{n-3}{2}})$.
		\item[$\bullet$] $q_{u,v}^1=0$ for $4\le u\le v\le 7$.
	\end{itemize}
	\item [$(\rm{\romannumeral 3})$] $w=2$
	\begin{itemize}
		\item [$\bullet$] $q_{0,2}^2=1$ and $q_{0,v}^2=0$ for $0\le v\le 7$ and $v\ne 2$.
		\item[$\bullet$] $q_{1,1}^2=p^{n-2}-2p^{r-2}+p^{\frac{n-3}{2}}$, $q_{1,2}^2=\frac{(p-1)}{2}(p^{n-2}-2p^{r-2})-p^{\frac{n-3}{2}}$, $q_{1,3}^2=\frac{(p-1)}{2}(p^{n-2}-2p^{r-2})$, $q_{1,4}^2=p^{n-r-1}$, $q_{1,5}^2=p^{r-2}-p^{n-r-1}$, $q_{1,6}^2=\frac{(p-1)}{2}(p^{r-2}+p^{\frac{n-3}{2}})$ and $q_{1,7}^2=\frac{(p-1)}{2}(p^{r-2}-p^{\frac{n-3}{2}})$.
		\item[$\bullet$] 
		$q_{2,2}^2=\frac{(p-1)^2}{4}(p^{n-2}-2p^{r-2})-\frac{p^2-2p-3}{4}p^{\frac{n-3}{2}}$, $q_{2,3}^2=\frac{(p-1)^2}{4}(p^{n-2}-2p^{r-2}+p^{\frac{n-3}{2}})$, $q_{2,4}^2=\frac{1}{2}(p^{n-r}-p^{n-r-1}-2)$, $q_{2,5}^2=\frac{(p-1)}{2}(p^{r-2}-p^{n-r-1})$, $q_{2,6}^2=\frac{(p-1)^2}{4}(p^{r-2}+p^{\frac{n-3}{2}})$ and $q_{2,7}^2=\frac{(p-1)^2}{4}(p^{r-2}-p^{\frac{n-3}{2}})$.
	\item[$\bullet$] $q_{3,3}^2=\frac{(p-1)^2}{4}(p^{n-2}-2p^{r-2}-p^{\frac{n-3}{2}})$, $q_{3,4}^2=\frac{(p-1)}{2}p^{n-r-1}$, $q_{3,5}^2=\frac{(p-1)}{2}(p^{r-2}\\-p^{n-r-1})$, $q_{3,6}^2=\frac{(p-1)^2}{4}(p^{r-2}+p^{\frac{n-3}{2}})$, $q_{3,7}^2=\frac{(p-1)^2}{4}(p^{r-2}-p^{\frac{n-3}{2}})$.
	\item[$\bullet$] $q_{u,v}^2=0$ for $4\le u\le v\le 7$.
		\end{itemize}
		\item [$(\rm{\romannumeral 4})$] $w=3$
		\begin{itemize}
			\item [$\bullet$] $q_{0,3}^3=1$ and $q_{0,v}^3=0$ for $0\le v\le 7$ and $v\ne 3$.
			\item[$\bullet$] $q_{1,1}^3=p^{n-2}-2p^{r-2}-p^{\frac{n-3}{2}}$, $q_{1,2}^3=\frac{(p-1)}{2}(p^{n-2}-2p^{r-2})$, $q_{1,3}^3=p^{\frac{n-3}{2}}+\frac{(p-1)}{2}(p^{n-2}-2p^{r-2})$, $q_{1,4}^3=p^{n-r-1}$, $q_{1,5}^3=p^{r-2}-p^{n-r-1}$, $q_{1,6}^3=\frac{(p-1)}{2}(p^{r-2}+p^{\frac{n-3}{2}})$ and $q_{1,7}^3=\frac{(p-1)}{2}(p^{r-2}-p^{\frac{n-3}{2}})$.
			\item[$\bullet$] $q_{2,2}^3=\frac{(p-1)^2}{4}(p^{n-2}-2p^{r-2}+p^{\frac{n-3}{2}})$, $q_{2,3}^3=\frac{(p-1)^2}{4}(p^{n-2}-2p^{r-2}-p^{\frac{n-3}{2}})$, $q_{2,4}^3=\frac{(p-1)}{2}p^{n-r-1}$, $q_{2,5}^3=\frac{(p-1)}{2}(p^{r-2}-p^{n-r-1})$, $q_{2,6}^3=\frac{(p-1)^2}{4}(p^{r-2}+p^{\frac{n-3}{2}})$, $q_{2,7}^3=\frac{(p-1)^2}{4}(p^{r-2}-p^{\frac{n-3}{2}})$.
			\item[$\bullet$] $q_{3,3}^3=\frac{p^2-2p-3}{4}p^{\frac{n-3}{2}}+\frac{(p-1)^2}{4}(p^{n-2}-2p^{r-2})$, $q_{3,4}^3=\frac{1}{2}(p^{n-r}-p^{n-r-1}-2)$, $q_{3,5}^3=\frac{(p-1)}{2}(p^{r-2}-p^{n-r-1})$, $q_{3,6}^3=\frac{(p-1)^2}{4}(p^{r-2}+p^{\frac{n-3}{2}})$ and $q_{3,7}^3=\frac{(p-1)^2}{4}(p^{r-2}-p^{\frac{n-3}{2}})$.
			\item[$\bullet$] $q_{u,v}^3=0$ for $4\le u\le v\le 7$.
		\end{itemize}
		\item [$(\rm{\romannumeral 5})$] $w=4$
		\begin{itemize}
			\item [$\bullet$] $q_{0,4}^4=1$ and $q_{0,v}^4=0$ for $0\le v\le 7$ and $v\ne 4$.
			\item[$\bullet$] $q_{1,1}^4=p^{n-2}-p^{r-1}$, $q_{1,2}^4=q_{1,3}^4=\frac{(p-1)}{2}p^{n-2}$ and $q_{1,v}^4=0$ for $4\le v\le 7$.
			\item[$\bullet$] $q_{2,2}^4=\frac{(p-1)}{4}(p^{n-1}-2p^{r-1}-p^{n-2})$, $q_{2,3}^4=\frac{(p-1)^2}{4}p^{n-2}$ and $q_{2,v}^4=0$ for $4\le v\le 7$.
			\item[$\bullet$] $q_{3,3}^4=\frac{(p-1)}{4}(p^{n-1}-2p^{r-1}-p^{n-2})$ and $q_{3,v}^4=0$ for $4\le v\le 7$.
			\item[$\bullet$] $q_{4,4}^4=p^{n-r}-2$ and $q_{4,v}^4=0$ for $5\le v\le 7$.
			\item[$\bullet$] $q_{5,5}^4=p^{r-1}-p^{n-r}$ and $q_{5,6}^4=q_{5,7}^4=0$.
			\item[$\bullet$] $q_{6,6}^4=\frac{(p-1)}{2}(p^{r-1}+p^{\frac{n-1}{2}})$ and $q_{6,7}^4=0$.
			\item[$\bullet$] $q_{7,7}^4=\frac{(p-1)}{2}(p^{r-1}-p^{\frac{n-1}{2}})$.
		\end{itemize}
		\item [$(\rm{\romannumeral 6})$] $w=5$
		\begin{itemize}
			\item [$\bullet$] $q_{0,5}^5=1$ and $q_{0,v}^5=0$ for $0\le v\le 7$ and $v\ne 5$.
			\item[$\bullet$] $q_{1,1}^5=p^{n-2}-p^{r-2}$, $q_{1,2}^5=q_{1,3}^5=\frac{(p-1)}{2}(p^{n-2}-p^{r-2})$ and $q_{1,v}^5=0$ for $4\le v\le 7$.
			\item[$\bullet$] $q_{2,2}^5=q_{2,3}^5=\frac{(p-1)^2}{4}(p^{n-2}-p^{r-2})$ and $q_{2,v}^5=0$ for $4\le v\le 7$.
			\item[$\bullet$] $q_{3,3}^5=\frac{(p-1)^2}{4}(p^{n-2}-p^{r-2})$ and $q_{3,v}^5=0$ for $4\le v\le 7$.
			\item[$\bullet$] $q_{4,5}^5=p^{n-r}-1$ and $q_{4,v}^5=0$ for $4\le v\le 7$ and $v\ne 5$.
			\item[$\bullet$] $q_{5,5}^5=p^{r-2}-2p^{n-r}$ and $q_{5,6}^5=q_{5,7}^5=\frac{(p-1)}{2}p^{r-2}$.
			\item[$\bullet$] $q_{6,6}^5=\frac{(p-1)}{4}(p^{r-1}-p^{r-2}+2p^{\frac{n-1}{2}})$ and $q_{6,7}^5=\frac{(p-1)^2}{4}p^{r-2}$.
			\item[$\bullet$] $q_{7,7}^5=\frac{(p-1)}{4}(p^{r-1}-p^{r-2}-2p^{\frac{n-1}{2}})$.
		\end{itemize}
		\item [$(\rm{\romannumeral 7})$] $w=6$
		\begin{itemize}
			\item [$\bullet$] $q_{0,6}^6=1$ and $q_{0,v}^6=0$ for $0\le v\le 7$ and $v\ne 6$.
			\item[$\bullet$] $q_{1,1}^6=p^{n-2}-p^{r-2}$, $q_{1,2}^6=q_{1,3}^6=\frac{(p-1)}{2}(p^{n-2}-p^{r-2})$ and $q_{1,v}^6=0$ for $4\le v\le 7$.
			\item[$\bullet$] $q_{2,2}^6=q_{2,3}^6=\frac{(p-1)^2}{4}(p^{n-2}-p^{r-2})$ and $q_{2,v}^6=0$ for $4\le v\le 7$.
			\item[$\bullet$] $q_{3,3}^6=\frac{(p-1)^2}{4}(p^{n-2}-p^{r-2})$ and $q_{3,v}^6=0$ for $4\le v\le 7$.
			\item[$\bullet$] $q_{4,6}^6=p^{n-r}-1$ and $q_{4,v}^6=0$ for $4\le v\le 7$ and $v\ne 6$.
			\item[$\bullet$] $q_{5,5}^6=p^{r-2}-p^{\frac{n-3}{2}}$, $q_{5,6}^6=\frac{1}{2}(p^{r-1}+p^{\frac{n-3}{2}}+p^{\frac{n-1}{2}}-2p^{n-r}-p^{r-2})$ and $q_{5,7}^6=\frac{(p-1)}{2}(p^{r-2}-p^{\frac{n-3}{2}})$.
			\item[$\bullet$] $q_{6,6}^6=\frac{1}{4}(p^{r-2}+p^r+3p^{\frac{n+1}{2}}-2p^{r-1}-p^{\frac{n-3}{2}}-6p^{\frac{n-1}{2}})$ and $q_{6,7}^6=\frac{(p-1)^2}{4}(p^{r-2}-p^{\frac{n-3}{2}})$.
			\item[$\bullet$] $q_{7,7}^6=\frac{(p-1)^2}{4}(p^{r-2}-p^{\frac{n-3}{2}})$.
		\end{itemize}
		\item [$(\rm{\romannumeral 8})$] $w=7$
		\begin{itemize}
			\item [$\bullet$] $q_{0,7}^7=1$ and $q_{0,v}^7=0$ for $0\le v\le 6$.
			\item[$\bullet$] $q_{1,1}^7=p^{n-2}-p^{r-2}$, $q_{1,2}^7=q_{1,3}^7=\frac{(p-1)}{2}(p^{n-2}-p^{r-2})$ and $q_{1,v}^7=0$ for $4\le v\le 7$.
			\item[$\bullet$] $q_{2,2}^7=q_{2,3}^7=\frac{(p-1)^2}{4}(p^{n-2}-p^{r-2})$ and $q_{2,v}^7=0$ for $4\le v\le 7$.
			\item[$\bullet$] $q_{3,3}^7=\frac{(p-1)^2}{4}(p^{n-2}-p^{r-2})$ and $q_{3,v}^7=0$ for $4\le v\le 7$.
			\item[$\bullet$] $q_{4,v}^7=0$ for $4\le v\le 6$ and $q_{4,7}^7=p^{n-r}-1$.
			\item[$\bullet$] $q_{5,5}^7=p^{r-2}+p^{\frac{n-3}{2}}$, $q_{5,6}^7=\frac{(p-1)}{2}(p^{r-2}+p^{\frac{n-3}{2}})$ and $q_{5,7}^7=\frac{1}{2}(p^{r-1}-p^{\frac{n-3}{2}}-p^{\frac{n-1}{2}}-2p^{n-r}-p^{r-2})$.
			\item[$\bullet$] $q_{6,6}^7=q_{6,7}^7=\frac{(p-1)^2}{4}(p^{r-2}+p^{\frac{n-3}{2}})$.
			\item[$\bullet$] $q_{7,7}^7=\frac{1}{4}(p^{r-2}-2p^{r-1}+p^r+p^{\frac{n-3}{2}}+6p^{\frac{n-1}{2}}-3p^{\frac{n+1}{2}})$.		\end{itemize}
\end{itemize}
22. The first and second eigenmatrices, the intersection numbers and the Krein parameters of the association scheme induced by  $U_{22}$.

Note that the first and second eigenmatrices of the association scheme induced by $U_{22}$ are the same. The first (second) eigenmatrix of the association scheme induced by $U_{22}$ is given in Table 35.
\begin{table}[h]
\vspace{0pt}
\centering
\caption{The first (second) eigenmatrix of the association scheme induced by $U_{22}$}
\renewcommand\arraystretch{2}	
\resizebox{\textwidth}{!}{
	\begin{tabular}{|m{2.3cm}<{\centering}|c|c|c|c|c|c|
		}
		\hline
		\diagbox{$i$}{{\tiny $P_{22}(ij)\ (Q_{22}(ij))$}}{$j$}&1&2&3&4&5&6\\
		\hline
		1&1&$p^{n-1}-p^{\frac{n-1}{2}}$&$\frac{(p-1)}{2}(p^{n-1}-p^{\frac{n-1}{2}})$&$\frac{(p-1)}{2}(p^{n-1}-p^{\frac{n-1}{2}})$&$p^{\frac{n-1}{2}}-1$&$(p-1)p^{\frac{n-1}{2}}$\\
		\hline
		2&1&0&$\frac{(p-1)}{2}p^{\frac{n-1}{2}}$&$-\frac{(p-1)}{2}p^{\frac{n-1}{2}}$&$-1$&0\\
		\hline
		3&1&$p^{\frac{n-1}{2}}$&$-p^{\frac{n-1}{2}}$&0&$-1$&0\\
		\hline
		4&1&$-p^{\frac{n-1}{2}}$&0&$p^{\frac{n-1}{2}}$&$-1$&0\\
		\hline
		5&1&$-p^{\frac{n-1}{2}}$&$-\frac{(p-1)}{2}p^{\frac{n-1}{2}}$&$-\frac{(p-1)}{2}p^{\frac{n-1}{2}}$&$p^{\frac{n-1}{2}}-1$&$(p-1)p^{\frac{n-1}{2}}$\\
		\hline
		6&1&0&0&0&$p^{\frac{n-1}{2}}-1$&$-p^{\frac{n-1}{2}}$\\
		\hline
\end{tabular}}
\vspace{-15pt}
\end{table}

	Since $p_{u,v}^{w}=p_{v,u}^{w}$ for any $u,v,w\in\{0,1,\dots,5\}$, we only give the values of the intersection numbers $p_{u,v}^w$ for $u\le v$ in the following six cases. The Krein parameter $q_{u,v}^w$ is the same as the intersection number $p_{u,v}^w$ for any $u,v,w\in\{0,1,\dots,5\}$.
	\begin{itemize}
		\item [$(\rm{\romannumeral1})$] $w=0$
		\begin{itemize}
			\item [$\bullet$] $u=v$\\
			$p_{0,0}^0=1$, $p_{1,1}^0=p^{n-1}-p^{\frac{n-1}{2}}$, $p_{2,2}^0=p_{3,3}^0=\frac{(p-1)}{2}(p^{n-1}-p^{\frac{n-1}{2}})$, $p_{4,4}^0=p^{\frac{n-1}{2}}-1$ and $p_{5,5}^0=(p-1)p^{\frac{n-1}{2}}$.
			\item[$\bullet$] $u\ne v$, $p_{u,v}^0=0$ for $0\le u<v\le 5$.			
		\end{itemize}
\item [$(\rm{\romannumeral2})$] $w=1$
\begin{itemize}
	\item [$\bullet$] $p_{0,1}^1=1$ and $p_{0,v}^1=0$ for $0\le v\le 5$ and $v\ne 1$.
	\item[$\bullet$] $p_{1,1}^1=p^{n-2}-2p^{\frac{n-3}{2}}$, $p_{1,2}^1=\frac{(p-1)}{2}(p^{n-2}-3p^{\frac{n-3}{2}})$, $p_{1,3}^1=\frac{(p-1)}{2}(p^{n-2}-p^{\frac{n-3}{2}})$, $p_{1,4}^1=p^{\frac{n-3}{2}}-1$ and $p_{1,5}^1=(p-1)p^{\frac{n-3}{2}}$.
	\item[$\bullet$] $p_{2,2}^1=\frac{(p-1)}{4}(p^{n-1}-p^{n-2}+4p^{\frac{n-3}{2}}-2p^{\frac{n-1}{2}})$, $p_{2,3}^1=\frac{(p-1)^2}{4}(p^{n-2}-2p^{\frac{n-3}{2}})$, $p_{2,4}^1=\frac{(p-1)}{2}p^{\frac{n-3}{2}}$ and $p_{2,5}^1=\frac{(p-1)^2}{2}p^{\frac{n-3}{2}}$.
	\item[$\bullet$] $p_{3,3}^1=\frac{(p-1)}{4}(p^{n-1}-p^{n-2}-2p^{\frac{n-1}{2}})$, $p_{3,4}^1=\frac{(p-1)}{2}p^{\frac{n-3}{2}}$ and $p_{3,5}^1=\frac{(p-1)^2}{2}p^{\frac{n-3}{2}}$.
	\item[$\bullet$] $p_{4,4}^1=p_{4,5}^1=p_{5,5}^1=0$.
\end{itemize}
\item [$(\rm{\romannumeral3})$] $w=2$
\begin{itemize}
	\item [$\bullet$] $p_{0,2}^2=1$ and $p_{0,v}^2=0$ for $0\le v\le 5$ and $v\ne 2$.
	\item[$\bullet$] $p_{1,1}^2=p^{n-2}-3p^{\frac{n-3}{2}}$, $p_{1,2}^2=\frac{1}{2}(p^{n-1}-p^{n-2}+4p^{\frac{n-3}{2}}-2p^{\frac{n-1}{2}})$, $p_{1,3}^2=\frac{(p-1)}{2}(p^{n-2}-2p^{\frac{n-3}{2}})$, $p_{1,4}^2=p^{\frac{n-3}{2}}$ and $p_{1,5}^2=(p-1)p^{\frac{n-3}{2}}$.
	\item[$\bullet$] $p_{2,2}^2=\frac{1}{4}(p^{n-2}+p^n+2p^{\frac{n-1}{2}}-2p^{n-1}-5p^{\frac{n-3}{2}}-p^{\frac{n+1}{2}})$, $p_{2,3}^2=\frac{(p-1)^2}{4}(p^{n-2}-3p^{\frac{n-3}{2}})$, $p_{2,4}^2=\frac{1}{2}(p^{\frac{n-1}{2}}-p^{\frac{n-3}{2}}-2)$ and $p_{2,5}^2=\frac{(p-1)^2}{2}p^{\frac{n-3}{2}}$.
	\item[$\bullet$] $p_{3,3}^2=\frac{(p-1)^2}{4}(p^{n-2}-p^{\frac{n-3}{2}})$, $p_{3,4}^2=\frac{(p-1)}{2}p^{\frac{n-3}{2}}$ and $p_{3,5}^2=\frac{(p-1)^2}{2}p^{\frac{n-3}{2}}$.
	\item[$\bullet$] $p_{4,4}^2=p_{4,5}^2=p_{5,5}^2=0$.
\end{itemize}
\item [$(\rm{\romannumeral4})$] $w=3$
\begin{itemize}
	\item [$\bullet$] $p_{0,3}^3=1$ and $p_{0,v}^3=0$ for $0\le v\le 5$ and $v\ne 3$.
	\item[$\bullet$] $p_{1,1}^3=p^{n-2}-p^{\frac{n-3}{2}}$, $p_{1,2}^3=\frac{(p-1)}{2}(p^{n-2}-2p^{\frac{n-3}{2}})$, $p_{1,3}^3=\frac{1}{2}(p^{n-1}-p^{n-2}-2p^{\frac{n-1}{2}})$, $p_{1,4}^3=p^{\frac{n-3}{2}}$ and $p_{1,5}^3=(p-1)p^{\frac{n-3}{2}}$.
	\item[$\bullet$] $p_{2,2}^3=\frac{(p-1)^2}{4}(p^{n-2}-3p^{\frac{n-3}{2}})$, $p_{2,3}^3=\frac{(p-1)^2}{4}(p^{n-2}-p^{\frac{n-3}{2}})$, $p_{2,4}^3=\frac{(p-1)}{2}p^{\frac{n-3}{2}}$ and $p_{2,5}^3=\frac{(p-1)^2}{2}p^{\frac{n-3}{2}}$.
\item[$\bullet$] $p_{3,3}^3=\frac{1}{4}(p^{n-2}-2p^{n-1}+p^n+p^{\frac{n-3}{2}}+6p^{\frac{n-1}{2}}-3p^{\frac{n+1}{2}})$, $p_{3,4}^3=\frac{1}{2}(p^{\frac{n-1}{2}}-p^{\frac{n-3}{2}}-2)$ and $p_{3,5}^3=\frac{(p-1)^2}{2}p^{\frac{n-3}{2}}$.
		\item[$\bullet$] $p_{4,4}^3=p_{4,5}^3=p_{5,5}^3=0$.
	\end{itemize}
\item [$(\rm{\romannumeral5})$] $w=4$
\begin{itemize}
	\item [$\bullet$] $p_{0,4}^4=1$ and $p_{0,v}^4=0$ for $0\le v\le 5$ and $v\ne 4$.
	\item[$\bullet$] $p_{1,1}^4=p^{n-2}-p^{\frac{n-1}{2}}$, $p_{1,2}^4=p_{1,3}^4=\frac{(p-1)}{2}p^{n-2}$ and $p_{1,4}^4=p_{1,5}^4=0$.
	\item[$\bullet$] $p_{2,2}^4=\frac{(p-1)}{4}(p^{n-1}-p^{n-2}-2p^{\frac{n-1}{2}})$, $p_{2,3}^4=\frac{(p-1)^2}{4}p^{n-2}$ and $p_{2,4}^4=p_{2,5}^4=0$.
	\item[$\bullet$] $p_{3,3}^4=\frac{(p-1)}{4}(p^{n-1}-p^{n-2}-2p^{\frac{n-1}{2}})$ and $p_{3,4}^4=p_{3,5}^4=0$.
	\item[$\bullet$] $p_{4,4}^4=p^{\frac{n-1}{2}}-2$ and $p_{4,5}^4=0$.
	\item[$\bullet$] $p_{5,5}^4=(p-1)p^{\frac{n-1}{2}}$. 
\end{itemize}
\item [$(\rm{\romannumeral6})$] $w=5$
\begin{itemize}
	\item [$\bullet$] $p_{0,5}^5=1$ and $p_{0,v}^5=0$ for $0\le v\le 4$.
	\item[$\bullet$] $p_{1,1}^5=p^{n-2}-p^{\frac{n-3}{2}}$, $p_{1,2}^5=p_{1,3}^5=\frac{(p-1)}{2}(p^{n-2}-p^{\frac{n-3}{2}})$ and $p_{1,4}^5=p_{1,5}^5=0$.
	\item[$\bullet$] $p_{2,2}^5=p_{2,3}^5=\frac{(p-1)^2}{4}(p^{n-2}-p^{\frac{n-3}{2}})$ and $p_{2,4}^5=p_{2,5}^5=0$.
	\item[$\bullet$] $p_{3,3}^5=\frac{(p-1)^2}{4}(p^{n-2}-p^{\frac{n-3}{2}})$ and $p_{3,4}^5=p_{3,5}^5=0$.
	\item[$\bullet$] $p_{4,4}^5=0$ and $p_{4,5}^5=p^{\frac{n-1}{2}}-1$.
	\item[$\bullet$] $p_{5,5}^5=(p-2)p^{\frac{n-1}{2}}$.
\end{itemize}
	\end{itemize}
23. The first and second eigenmatrices, the intersection numbers and the Krein parameters of the association scheme induced by  $U_{23}$.

For fixed parameters $p$ and $n$, the first (respectively the second) eigenmatrix of association scheme induced by $U_{23}$ is the same as the second (respectively the first) eigenmatrix of the association scheme induced by $U_{19}$, and the intersection number $p_{u,v}^w$ (respectively the Krein parameter $q_{u,v}^w$) of the association scheme induced by $U_{23}$ is the same as the Krein parameter $q_{u,v}^w$ (respectively the intersection number $p_{u,v}^w$) of the association scheme induced by $U_{19}$.\\
24. The first and second eigenmatrices, the intersection numbers and the Krein parameters of the association scheme induced by  $U_{24}$.

Note that the first and second eigenmatrices of the association scheme induced by $U_{24}$ are the same. The first (second) eigenmatrix  of the association scheme induced by $U_{24}$ is given in Table 36.

	Since $p_{u,v}^{w}=p_{v,u}^{w}$ for any $u,v,w\in\{0,1,\dots,7\}$, we only give the values of the intersection numbers $p_{u,v}^w$ for $u\le v$ in the following eight cases. The Krein parameter $q_{u,v}^w$ is the same as the intersection number $p_{u,v}^w$ for any $u,v,w\in\{0,1,\dots,7\}$.	
\begin{table}[h]
	\vspace{0pt}
	\centering
	\caption{The first (second) eigenmatrix of the association scheme induced by $U_{24}$}
	\renewcommand\arraystretch{1.5}	
	\resizebox{\textwidth}{!}{
		\begin{tabular}{|c|m{11cm}|}
			\hline
			$i$&\hspace{4.25cm} $P_{24}(ij)\ (Q_{24}(ij))$\\
			\hline
			1&$P_{24}(11)=1$, $P_{24}(12)=p^{n-1}-p^{r-1}$, $P_{24}(13)=P_{24}(14)=\frac{(p-1)}{2}(p^{n-1}-p^{r-1})$, $P_{24}(15)=p^{n-r}-1$, $P_{24}(16)=p^{r-1}-p^{n-r}$, $P_{24}(17)=\frac{(p-1)}{2}(p^{r-1}-p^{\frac{n-1}{2}})$, $P_{24}(28)=\frac{(p-1)}{2}(p^{r-1}+p^{\frac{n-1}{2}})$\\
			\hline
			2&$P_{24}(21)=1$, $P_{24}(22)=0$, $P_{24}(23)=\frac{(p-1)}{2}p^{\frac{n-1}{2}}$, $P_{24}(24)=-\frac{(p-1)}{2}p^{\frac{n-1}{2}}$, $P_{24}(25)=-1$, $P_{24}(26)=P_{24}(27)=P_{24}(28)=0$\\
			\hline
			3&$P_{24}(31)=1$, $P_{24}(32)=p^{\frac{n-1}{2}}$, $P_{24}(33)=-p^{\frac{n-1}{2}}$, $P_{24}(34)=0$, $P_{24}(35)=-1$, $P_{24}(36)=P_{24}(37)=P_{24}(38)=0$\\
			\hline
			4&$P_{24}(41)=1$, $P_{24}(42)=-p^{\frac{n-1}{2}}$, $P_{24}(43)=0$, $P_{24}(44)=p^{\frac{n-1}{2}}$, $P_{24}(45)=-1$, $P_{24}(46)=P_{24}(47)=P_{24}(48)=0$\\
			\hline
			5&$P_{24}(51)=1$, $P_{24}(52)=-p^{r-1}$, $P_{24}(53)=-\frac{(p-1)}{2}p^{r-1}$, $P_{24}(54)=-\frac{(p-1)}{2}p^{r-1}$, $P_{24}(55)=p^{n-r}-1$, $P_{24}(56)=p^{r-1}-p^{n-r}$, $P_{24}(57)=\frac{(p-1)}{2}(p^{r-1}-p^{\frac{n-1}{2}})$, $P_{24}(58)=\frac{(p-1)}{2}(p^{r-1}+p^{\frac{n-1}{2}})$\\
			\hline
			6&$P_{24}(61)=1$, $P_{24}(62)=P_{24}(63)=P_{24}(64)=0$, $P_{24}(65)=p^{n-r}-1$, $P_{24}(66)=-p^{n-r}$, $P_{24}(67)=-\frac{(p-1)}{2}p^{\frac{n-1}{2}}$, $P_{24}(68)=\frac{(p-1)}{2}p^{\frac{n-1}{2}}$\\
			\hline
			7&$P_{24}(71)=1$, $P_{24}(72)=P_{24}(73)=P_{24}(74)=0$, $P_{24}(75)=p^{n-r}-1$, $P_{24}(76)=-p^{\frac{n-1}{2}}-p^{n-r}$, $P_{24}(77)=p^{\frac{n-1}{2}}$, $P_{24}(78)=0$\\
			\hline
			8&$P_{24}(81)=1$, $P_{24}(82)=P_{24}(83)=P_{24}(84)=0$, $P_{24}(85)=p^{n-r}-1$, $P_{24}(86)=p^{\frac{n-1}{2}}-p^{n-r}$, $P_{24}(87)=0$, $P_{24}(88)=-p^{\frac{n-1}{2}}$\\
			\hline
	\end{tabular}}
	\vspace{-15pt}			
\end{table}
\begin{itemize}
	\item [$(\rm{\romannumeral1})$] $w=0$
	\begin{itemize}
		\item [$\bullet$] $u=v$\\
		$p_{0,0}^0=1$, $p_{1,1}^0=p^{n-1}-p^{r-1}$, $p_{2,2}^0=p_{3,3}^0=\frac{(p-1)}{2}(p^{n-1}-p^{r-1})$, $p_{4,4}^0=p^{n-r}-1$, $p_{5,5}^0=p^{r-1}-p^{n-r}$, $p_{6,6}^0=\frac{(p-1)}{2}(p^{r-1}-p^{\frac{n-1}{2}})$ and $p_{7,7}^0=\frac{(p-1)}{2}(p^{r-1}+p^{\frac{n-1}{2}})$.
	\item[$\bullet$] $u\ne v$, $p_{u,v}^0=0$.
	\end{itemize}
\item [$(\rm{\romannumeral2})$] $w=1$
\begin{itemize}
	\item [$\bullet$] $p_{0,1}^1=1$ and $p_{0,v}^1=0$ for $0\le v\le 7$ and $v\ne 1$.
	\item[$\bullet$] $p_{1,1}^1=p^{n-2}-2p^{r-2}$, $p_{1,2}^1=\frac{(p-1)}{2}(p^{n-2}-2p^{r-2}-p^{\frac{n-3}{2}})$, $p_{1,3}^1=\frac{(p-1)}{2}
(p^{n-2}-2p^{r-2}+p^{\frac{n-3}{2}})$, $p_{1,4}^1=p^{n-r-1}-1$, $p_{1,5}^1=p^{r-2}-p^{n-r-1}$, $p_{1,6}^1=\frac{(p-1)}{2}(p^{r-2}-p^{\frac{n-3}{2}})$ and $p_{1,7}^1=\frac{(p-1)}{2}(p^{r-2}+p^{\frac{n-3}{2}})$.
\item[$\bullet$] $p_{2,2}^1=\frac{(p-1)}{2}p^{\frac{n-3}{2}}+(p^{n-2}-2p^{r-2})\frac{(p-1)^2}{4}$, $p_{2,3}^1=\frac{(p-1)^2}{4}(p^{n-2}-2p^{r-2})$, $p_{2,4}^1=\frac{(p-1)}{2}p^{n-r-1}$, $p_{2,5}^1=\frac{(p-1)}{2}(p^{r-2}-p^{n-r-1})$, $p_{2,6}^1=\frac{(p-1)^2}{4}(p^{r-2}-p^{\frac{n-3}{2}})$ and $p_{2,7}^1=\frac{(p-1)^2}{4}(p^{r-2}+p^{\frac{n-3}{2}})$.
\item[$\bullet$] $p_{3,3}^1=\frac{(p-1)^2}{4}(p^{n-2}-2p^{r-2})-\frac{(p-1)}{2}p^{\frac{n-3}{2}}$, $p_{3,4}^1=\frac{(p-1)}{2}p^{n-r-1}$, $p_{3,5}^1=\frac{(p-1)}{2}(p^{r-2}-p^{n-r-1})$, $p_{3,6}^1=\frac{(p-1)^2}{4}(p^{r-2}-p^{\frac{n-3}{2}})$ and $p_{3,7}^1=\frac{(p-1)^2}{4}(p^{r-2}\\+p^{\frac{n-3}{2}})$.
\item[$\bullet$] $p_{u,v}^1=0$ for $4\le u\le v\le 7$.
\end{itemize}
\item [$(\rm{\romannumeral3})$] $w=2$
\begin{itemize}
	\item [$\bullet$] $p_{0,2}^2=1$ and $p_{0,v}^2=0$ for $0\le v\le 7$ and $v\ne 2$.\\
	\newpage
	\item[$\bullet$] $p_{1,1}^2=p^{n-2}-2p^{r-2}-p^{\frac{n-3}{2}}$, $p_{1,2}^2=p^{\frac{n-3}{2}}+\frac{(p-1)}{2}(p^{n-2}-2p^{r-2})$, $p_{1,3}^2=\frac{(p-1)}{2}(p^{n-2}-2p^{r-2})$, $p_{1,4}^2=p^{n-r-1}$, $p_{1,5}^2=p^{r-2}-p^{n-r-1}$, $p_{1,6}^2=\frac{(p-1)}{2}(p^{r-2}-p^{\frac{n-3}{2}})$ and $p_{1,7}^2=\frac{(p-1)}{2}(p^{r-2}+p^{\frac{n-3}{2}})$.
	\item[$\bullet$] $p_{2,2}^2=\frac{p^2-2p-3}{4}p^{\frac{n-3}{2}}+\frac{(p-1)^2}{4}(p^{n-2}-2p^{r-2})$, $p_{2,3}^2=\frac{(p-1)^2}{4}(p^{n-2}-2p^{r-2}-p^{\frac{n-3}{2}})$, $p_{2,4}^2=\frac{1}{2}(p^{n-r}-p^{n-r-1}-2)$, $p_{2,5}^2=\frac{(p-1)}{2}(p^{r-2}-p^{n-r-1})$, $p_{2,6}^2=\frac{(p-1)^2}{4}(p^{r-2}-p^{\frac{n-3}{2}})$ and $p_{2,7}^2=\frac{(p-1)^2}{4}(p^{r-2}+p^{\frac{n-3}{2}})$.
	\item[$\bullet$] $p_{3,3}^2=\frac{(p-1)^2}{4}(p^{n-2}-2p^{r-2}+p^{\frac{n-3}{2}})$, $p_{3,4}^2=\frac{(p-1)}{2}p^{n-r-1}$, $p_{3,5}^2=\frac{(p-1)}{2}\\(p^{r-2}-p^{n-r-1})$, $p_{3,6}^2=\frac{(p-1)^2}{4}(p^{r-2}-p^{\frac{n-3}{2}})$ and $p_{3,7}^2=\frac{(p-1)^2}{4}(p^{r-2}+p^{\frac{n-3}{2}})$.
	\item[$\bullet$] $p_{u,v}^2=0$ for $4\le u\le v\le 7$.	
\end{itemize}
\item [$(\rm{\romannumeral4})$] $w=3$
\begin{itemize}
	\item [$\bullet$] $p_{0,3}^3=1$ and $p_{0,v}^3=0$ for $0\le v\le 7$ and $v\ne 3$.
	\item[$\bullet$] $p_{1,1}^3=p^{n-2}-2p^{r-2}+p^{\frac{n-3}{2}}$, $p_{1,2}^3=\frac{(p-1)}{2}(p^{n-2}-2p^{r-2})$, $p_{1,3}^3=\frac{(p-1)}{2}\\(p^{n-2}-2p^{r-2})-p^{\frac{n-3}{2}}$, $p_{1,4}^3=p^{n-r-1}$, $p_{1,5}^3=p^{r-2}-p^{n-r-1}$, $p_{1,6}^3=\frac{(p-1)}{2}(p^{r-2}-p^{\frac{n-3}{2}})$ and $p_{1,7}^3=\frac{(p-1)}{2}(p^{r-2}+p^{\frac{n-3}{2}})$.
	\item[$\bullet$] $p_{2,2}^3=\frac{(p-1)^2}{4}(p^{n-2}-2p^{r-2}-p^{\frac{n-3}{2}})$, $p_{2,3}^3=\frac{(p-1)^2}{4}(p^{n-2}-2p^{r-2}+p^{\frac{n-3}{2}})$, $p_{2,4}^3=\frac{(p-1)}{2}p^{n-r-1}$, $p_{2,5}^3=\frac{(p-1)}{2}(p^{r-2}-p^{n-r-1})$, $p_{2,6}^3=\frac{(p-1)^2}{4}(p^{r-2}-p^{\frac{n-3}{2}})$ and $p_{2,7}^3=\frac{(p-1)^2}{4}(p^{r-2}+p^{\frac{n-3}{2}})$.
	\item[$\bullet$] $p_{3,3}^3=\frac{(p-1)^2}{4}(p^{n-2}-2p^{r-2})-\frac{p^2-2p-3}{4}p^{\frac{n-3}{2}}$, $p_{3,4}^3=\frac{1}{2}(p^{n-r}-p^{n-r-1}-2)$, $p_{3,5}^3=\frac{(p-1)}{2}(p^{r-2}-p^{n-r-1})$, $p_{3,6}^3=\frac{(p-1)^2}{4}(p^{r-2}-p^{\frac{n-3}{2}})$ and $p_{3,7}^3=\frac{(p-1)^2}{4}(p^{r-2}+p^{\frac{n-3}{2}})$.
	\item[$\bullet$] $p_{u,v}^3=0$ for $4\le u\le v\le 7$.
\end{itemize}
	\item [$(\rm{\romannumeral5})$] $w=4$	
	\begin{itemize}
		\item [$\bullet$] $p_{0,4}^4=1$ and $p_{0,v}^4=0$ for $0\le v\le 7$ and $v\ne 4$.
		\item[$\bullet$] $p_{1,1}^4=p^{n-2}-p^{r-1}$, $p_{1,2}^4=p_{1,3}^4=\frac{(p-1)}{2}p^{n-2}$ and $p_{1,v}^4=0$ for $4\le v\le 7$.
		\item[$\bullet$] $p_{2,2}^4=\frac{(p-1)}{4}(p^{n-1}-p^{n-2}-2p^{r-1})$, $p_{2,3}^4=\frac{(p-1)^2}{4}p^{n-2}$ and $p_{2,v}^4=0$ for $4\le v\le 7$.
		\item[$\bullet$] $p_{3,3}^4=\frac{(p-1)}{4}(p^{n-1}-p^{n-2}-2p^{r-1})$ and $p_{3,v}^4=0$ for $4\le v\le 7$.
		\item [$\bullet$] $p_{4,4}^4=p^{n-r}-2$ and $p_{4,5}^4=p_{4,6}^4=p_{4,7}^4=0$.
		\item[$\bullet$] $p_{5,5}^4=p^{r-1}-p^{n-r}$ and $p_{5,6}^4=p_{5,7}^4=0$.
		\item[$\bullet$] $p_{6,6}^4=\frac{(p-1)}{2}(p^{r-1}-p^{\frac{n-1}{2}})$ and $p_{6,7}^4=0$.
		\item[$\bullet$] $p_{7,7}^4=\frac{(p-1)}{2}(p^{r-1}+p^{\frac{n-1}{2}})$.
	\end{itemize}
\item [$(\rm{\romannumeral6})$] $w=5$
\begin{itemize}
	\item [$\bullet$] $p_{0,5}^5=1$ and $p_{0,v}^5=0$ for $0\le v\le 7$ and $v\ne 5$.
	\item[$\bullet$] $p_{1,1}^5=p^{n-2}-p^{r-2}$, $p_{1,2}^5=p_{1,3}^5=\frac{(p-1)}{2}(p^{n-2}-p^{r-2})$ and $p_{1,v}^5=0$ for $4\le v\le 7$.
	\item[$\bullet$] $p_{2,2}^5=p_{2,3}^5=\frac{(p-1)^2}{4}(p^{n-2}-p^{r-2})$ and $p_{2,v}^5=0$ for $4\le v\le 7$.
	\item[$\bullet$] $p_{3,3}^5=\frac{(p-1)^2}{4}(p^{n-2}-p^{r-2})$ and $p_{3,v}^5=0$ for $4\le v\le 7$.
	\item[$\bullet$] $p_{4,5}^5=p^{n-r}-1$ and $p_{4,v}^5=0$ for $4\le v\le 7$ and $v\ne 5$.
	\item[$\bullet$] $p_{5,5}^5=p^{r-2}-2p^{n-r}$ and $p_{5,6}^5=p_{5,7}^5=\frac{(p-1)}{2}p^{r-2}$.
	\item[$\bullet$] $p_{6,6}^5=\frac{(p-1)}{4}(p^{r-1}-p^{r-2}-2p^{\frac{n-1}{2}})$ and $p_{6,7}^5=\frac{(p-1)^2}{4}p^{r-2}$.
	\item[$\bullet$] $p_{7,7}^5=\frac{(p-1)}{4}(p^{r-1}-p^{r-2}+2p^{\frac{n-1}{2}})$.	
	\end{itemize}
\item [$(\rm{\romannumeral7})$] $w=6$
\begin{itemize}
	\item [$\bullet$] $p_{0,6}^6=1$ and $p_{0,v}^6=0$ for $0\le v\le 7$ and $v\ne 6$.
	\item[$\bullet$] $p_{1,1}^6=p^{n-2}-p^{r-2}$, $p_{1,2}^6=p_{1,3}^6=\frac{(p-1)}{2}(p^{n-2}-p^{r-2})$ and $p_{1,v}^6=0$ for $4\le v\le 7$.
	\item[$\bullet$] $p_{2,2}^6=p_{2,3}^6=\frac{(p-1)^2}{4}(p^{n-2}-p^{r-2})$ and $p_{2,v}^6=0$ for $4\le v\le 7$.
	\item[$\bullet$] $p_{3,3}^6=\frac{(p-1)^2}{4}(p^{n-2}-p^{r-2})$ and $p_{3,v}^6=0$ for $4\le v\le 7$.
	\item[$\bullet$] $p_{4,6}^6=p^{n-r}-1$ and $p_{4,v}^6=0$ for $4\le v\le 7$ and $v\ne 6$.
	\item[$\bullet$] $p_{5,5}^6=p^{r-2}+p^{\frac{n-3}{2}}$, $p_{5,6}^6=\frac{1}{2}(p^{r-1}-p^{\frac{n-3}{2}}-p^{\frac{n-1}{2}}-2p^{n-r}-p^{r-2})$ and $p_{5,7}^6=\frac{(p-1)}{2}(p^{r-2}+p^{\frac{n-3}{2}})$.
	\item[$\bullet$] $p_{6,6}^6=\frac{1}{4}(p^{r-2}-2p^{r-1}+p^r+p^{\frac{n-3}{2}}+6p^{\frac{n-1}{2}}-3p^{\frac{n+1}{2}})$ and $p_{6,7}^6=\frac{(p-1)^2}{4}(p^{r-2}+p^{\frac{n-3}{2}})$.
	\item[$\bullet$] $p_{7,7}^6=\frac{(p-1)^2}{4}(p^{r-2}+p^{\frac{n-3}{2}})$.
\end{itemize}
\item [$(\rm{\romannumeral8})$] $w=7$
\begin{itemize}
	\item [$\bullet$] $p_{0,7}^7=1$ and $p_{0,v}^7=0$ for $0\le v\le 6$.
	\item[$\bullet$] $p_{1,1}^7=p^{n-2}-p^{r-2}$, $p_{1,2}^7=p_{1,3}^7=\frac{(p-1)}{2}(p^{n-2}-p^{r-2})$ and $p_{1,v}^7=0$ for $4\le v\le 7$.
	\item[$\bullet$] $p_{2,2}^7=p_{2,3}^7=\frac{(p-1)^2}{4}(p^{n-2}-p^{r-2})$ and $p_{2,v}^7=0$ for $4\le v\le 7$.
	\item[$\bullet$] $p_{3,3}^7=\frac{(p-1)^2}{4}(p^{n-2}-p^{r-2})$ and $p_{3,v}^7=0$ for $4\le v\le 7$.
	\item[$\bullet$] $p_{4,7}^7=p^{n-r}-1$ and $p_{4,v}^7=0$ for $4\le v\le 6$.
	\item[$\bullet$] $p_{5,5}^7=p^{r-2}-p^{\frac{n-3}{2}}$, $p_{5,6}^7=\frac{(p-1)}{2}(p^{r-2}-p^{\frac{n-3}{2}})$ and $p_{5,7}^7=\frac{1}{2}(p^{r-1}+p^{\frac{n-3}{2}}+p^{\frac{n-1}{2}}-2p^{n-r}-p^{r-2})$.
	\item[$\bullet$] $p_{6,6}^7=p_{6,7}^7=\frac{(p-1)^2}{4}(p^{r-2}-p^{\frac{n-3}{2}})$.
	\item[$\bullet$] $p_{7,7}^7=\frac{1}{4}(p^{r-2}+p^r+3p^{\frac{n+1}{2}}-2p^{r-1}-p^{\frac{n-3}{2}}-6p^{\frac{n-1}{2}})$.
\end{itemize}
\end{itemize}
25. The first and second eigenmatrices, the intersection numbers and the Krein parameters of the association scheme induced by  $U_{25}$.

For fixed parameters $p$, $n$ and $r$, the first (respectively the second) eigenmatrix of association scheme induced by $U_{25}$ is the same as the second (respectively the first) eigenmatrix of the association scheme induced by $U_{21}$, and the intersection number $p_{u,v}^w$ (the Krein parameter $q_{u,v}^w$) of the association scheme induced by $U_{25}$ is the same as the Krein parameter $q_{u,v}^w$ (the intersection number $p_{u,v}^w$) of the association scheme induced by $U_{21}$.

\end{document}